\documentclass[11pt,reqno]{amsart}
\usepackage{amsmath,amsbsy,amssymb,amscd,amsfonts,latexsym,amstext,delarray, amsmath,color,caption}
\usepackage{tikz,lscape}
\usetikzlibrary{matrix,arrows}
\usepackage{graphicx}
\usepackage{multicol}
\usepackage{scrextend}
\usepackage{mathtools}
\usepackage{amssymb} 
\usepackage{calligra}
\usepackage{calrsfs}
\usepackage{eucal}
\usepackage{calrsfs}
\usepackage{epsfig}
\usepackage{xcolor,import}
\usepackage{transparent}
\usepackage{tikz-cd}
\usepackage{epstopdf}
\usepackage{subcaption}
\usepackage{enumitem}
\usepackage{float}
\numberwithin{equation}{section}
\numberwithin{figure}{section}
\usepackage{cleveref}
\usepackage{mathrsfs} 
\DeclareFontFamily{U}{mathb}{\hyphenchar\font45}
\DeclareFontShape{U}{mathb}{m}{n}{
      <5> <6> <7> <8> <9> <10> gen * mathb
      <10.95> mathb10 <12> <14.4> <17.28> <20.74> <24.88> mathb12
      }{}
\DeclareSymbolFont{mathb}{U}{mathb}{m}{n}

\DeclareMathSymbol{\precneq}{3}{mathb}{"AC}

\usepackage{mathtools}

\usepackage{graphicx}
\usepackage{tikz}
\tikzset{
  font={\fontsize{10pt}{12}\selectfont}}
  \usetikzlibrary{arrows.meta,arrows}
  \tikzset{>=latex}
\usepackage{caption, cleveref} 
\usepackage{xfrac}  
\usetikzlibrary{arrows, fit, backgrounds, patterns, positioning, calc, intersections, decorations.markings, decorations.pathmorphing, decorations.pathreplacing, shapes.misc}
\usetikzlibrary{shapes.multipart}
\usepackage[utf8]{inputenc}
\usepackage{multicol}
\usepackage{subcaption}
\usepackage{enumitem}

\DeclareMathAlphabet{\pazocal}{OMS}{zplm}{m}{n}

\input xy

\xyoption{all}
    \advance \topmargin by -\headheight
    \advance \topmargin by -\headsep

    \evensidemargin \oddsidemargin
\makeatletter
\newcommand{\leqnomode}{\tagsleft@true\let\veqno\@@leqno}
\newcommand{\reqnomode}{\tagsleft@false\let\veqno\@@eqno}
\makeatother

\newtheorem {theorem}    {Theorem}[section]

\theoremstyle{definition}
\newtheorem {lemma}      [theorem]    {Lemma}

\theoremstyle{definition}
\newtheorem{definition}[theorem]{Definition}
\theoremstyle{definition}
\newtheorem{remark}[theorem]{Remark}

\newcommand{\defeq}{\vcentcolon=}
\newcommand{\gen}[1]{\langle #1 \rangle}
\newcommand\numberthis{\addtocounter{equation}{1}\tag{\theequation}}

\def\TM{{\rm TM}}

\def\a{\alpha}                
\def\b{\beta}

\def\eps{\varepsilon}

\def\a{\alpha}
\def\b{\beta}

\def\N{\mathbb{N}}     % natural numbers
\def\lab{{\text{Lab}}}

\renewcommand{\int}{\mathrm{int}}

\def\b{\beta}

\def\vertexradius{.1}
\def\vertex(#1){\fill (#1) circle (\vertexradius)}

\begin{document}

%%%%%%%%%%%%%%%%%%%%%%%%%%%%%%%%%%%%%%%%%%%%%%%%%%%%%%%%%%%%%%%%%
%%%%%%%%%%%%%%%%%%%%%%%%%%%%%%%%%%%%%%%%%%%%%%%%%%%%%%%%%%%%%%%%%
\title{\bf Quasi-isometric Higman Embeddings and the Dehn Function}
\maketitle
\begin{center}

Francis Wagner

\end{center}

\bigskip

\begin{center}

\textbf{Abstract}

\end{center}

\begin{addmargin}[2em]{2em}

This is the first of a sequence of papers devoted to studying the link between the complexity of the Word Problem for a finitely generated recursively presented group $G$ and the isoperimetric functions of the finitely presented groups in which $G$ embeds.  We prove here that if a finitely generated group has a presentation $\pazocal{P}$ whose relators can be enumerated by a computational model satisfying certain technical requirements, then the group embeds quasi-isometrically into a finitely presented group whose Dehn function is bounded above by a function of the model's computational complexity and the Dehn function of $\pazocal{P}$.  This generalizes a previous result of the author pertaining to the embeddings of free Burnside groups and gives a recipe for establishing such Higman embeddings into groups with desired geometric properties.  As an example of the use of this embedding scheme, we find a substantial improvement to the seminal result of Birget, Ol'shanskii, Rips, and Sapir showing that the Word Problem of a finitely generated group is in class NP if and only if the group embeds into a finitely presented group with polynomial Dehn function.  %In a subsequent paper, we will use to find an `optimal' refinement of the seminal results of Birget, Ol'shanskii, Rips, and Sapir.

	%We study the link between the complexity of the word problem for a finitely generated group and the isoperimetric functions of the finitely presented groups in which it embeds.  In doing so, we show that for every $\eps>0$ and finitely generated group $G$, if there exists a non-deterministic Turing machine deciding the word problem of $G$ with time function $T(n)$ such that $T(n)^2$ is super-additive, then $G$ embeds in a finitely presented group with Dehn function bounded above by $T(n)^{2+\eps}$.  This provides a sharp upper bound for the `essential Dehn function' of general finitely generated groups, an invariant which is introduced herein.  We also study properties of the embedding, showing that it can be chosen to be a quasi-isometric embedding satisfying the congruence extension property.
%	Let $G$ be a group with a finite generating set $X$ and let $M$ be a non-deterministic Turing machine enumerating the set of all words over $X\cup X^{-1}$ which represent the trivial element of $G$.  Then for all $\eps>0$, we construct a finitely presented group $H_\eps$ in which $G$ embeds such that the Dehn function of $H_\eps$ is asymptotically bounded above by $\TM_M^{2+\eps}$, where $\TM_M$ denotes the time complexity of $M$.

\end{addmargin}

\bigskip

%%%%%%%%%%%%%%%%%%%%%%%%%%%%%%%%%%%%%%%%%%%%%%%%%%%%%%%%%%%%%%%%%

\section{Introduction}

An $S$-machine is a computational model resembling a non-deterministic multi-tape Turing machine invented for the purpose of constructing finitely presented groups with desirable geometric properties.  See \cite{SBR} and \cite{CW} for the interpretation as a computational model and \cite{BORS}, \cite{OS00}, \cite{OS01}, \cite{OSconj}, \cite{OS06}, \cite{OS19}, \cite{OS22}, \cite{OSnon-amen}, and \cite{W} for applications of this theory.

The goal of this paper is to refine the techniques developed for this tool to study the Dehn functions of the finitely presented groups into which a given recursively presented group embeds.  In particular, we achieve the following refinement of the Higman embedding theorem \cite{Higman}:

\begin{theorem} \label{theorem-higman}

Let $X$ be a finite generating set for a group $G$.  Suppose there exists a multi-tape, non-deterministic Turing machine $\pazocal{T}$ which enumerates the set of words in $X\cup X^{-1}$ which represent the identity in $G$.  Letting $\TM_\pazocal{T}$ denote the time function of $\pazocal{T}$, suppose there exists a superadditive function $f:\N\to[0,\infty)$ {\frenchspacing (i.e. $f(n+m)\geq f(n)+f(m)$ for all $n,m\in\N$)} such that $\TM_\pazocal{T}\preceq f$.  Then for any $\eps>0$, there exists a quasi-isometric embedding of $G$ into a finitely presented group $H_\eps$ such that $\delta_{H_\eps}(n)\preceq n^4+f(n)^{2+\eps}$.

\end{theorem}

\Cref{theorem-higman} amounts to a significant refinement of the momentous result of Birget, Ol'shanskii, Rips, and Sapir in \cite{BORS}, which obtains an embedding of $G$ into a finitely presented group with Dehn function $n^2f(n^2)^4$.  %It should be noted, however, that in that setting it is assumed that $f(n)^4$ is superadditive rather than $f$ itself; with that said, the assumption that $f$ is superadditive here bypasses the presence of the factor of $n^2$ in the Dehn function.

\Cref{theorem-higman} is a corollary of the following more technical result that relies on two concepts developed herein.  The first of these concepts is the theory of a `Move machine', an $S$-machine with a specific computational makeup whose associated groups possess a certain structural property.  The second is the theory of a `mass condition' satisfied by a presentation with finite generating set, a generalization of the Dehn function of the presentation.  

The definitions of these concepts are discussed further in Sections \ref{sec-intro-move} and \ref{sec-intro-mass}.  For the precise formulation of a Move machine and the `move condition' referenced in the statement, see Sections \ref{sec-move} and \ref{sec-Move-conditions}.

\begin{theorem} \label{main-theorem}

For $i=1,2$, let $\pazocal{P}_i=\gen{X\mid\pazocal{R}_i}$ be two presentations for a group $G$ with $|X|<\infty$, $\pazocal{R}_2\subseteq\pazocal{R}_1$, and $1\notin\pazocal{R}_i$.  Let $f_1$ and $f_2$ be non-decreasing functions on the nonnegative reals such that $f_1(n)\geq n$ and $f_2(n)\geq1$ for all $n\in\N$.  Suppose:

\begin{enumerate}

\item Every non-trivial word over $X\cup X^{-1}$ which represents the identity element of $G$ has length at least some sufficiently large constant $M$.

\item There exists a $(\pazocal{R}_1,f_1,C)$-Move machine satisfying the $(n^2f_2(n),G)$-move condition.

\item There exists an $S$-machine $\textbf{S}$ that recognizes $\pazocal{R}_1$ and satisfies $\TM_\textbf{S}\preceq f_1$.

\item The presentation $\pazocal{P}_2$ satisfies the $(f_1(n)^2,n^2f_2(n))$-mass condition.

\end{enumerate}

Then there exists a quasi-isometric (indeed bi-Lipschitz) embedding of $G$ into a finitely presented group $H$ such that $\delta_H(n)\preceq n^2f_2(n)$.

\end{theorem}

Note that the sets of relators $\pazocal{R}_i$ are not assumed to be finite, and indeed the group $G$ need not possess any finite presentation.

Further, note that \Cref{main-theorem} can be viewed as a extensive generalization of the main result of the present author in \cite{W}: 

Given a sufficiently large exponent $n$, let $B$ be the free Burnside group with exponent $n$ and finite generating set $X$.  

\begin{enumerate}

\item The graded small-cancellation arguments of \cite{Ivanov} show that the non-trivial words representing the identity in $B$ have length at least $n$.  

\item The treatment of `impeding trapezia' in \cite{W} show that the machine $\textbf{M}_3(1)$ constructed therein (restricted to the base $P_0Q_0P_1$) is an $(F(X),n,0)$-Move machine satisfying the $(n^2,B)$-move condition.

\item The machine $\textbf{M}_1$ constructed in \cite{W} is shown to enumerate the `standard' set of relators of $B$ in linear time.  

\item The `mass' defined in \cite{W} of the diagrams over the inductive presentation of $B$ constructed in \cite{Ivanov} (which has generating set $X$ and is obtained by carefully selecting which of the `standard' relators to include) is shown to be bounded in a way that implies this presentation satisfies the $(n^2,n^2)$-mass condition. %Hence, the main theorem can be viewed as a specific case of \Cref{main-theorem}.

\end{enumerate}

Thus, \Cref{main-theorem} implies the existence of a quasi-isometric embedding of $B$ into a finitely presented group with quadratic Dehn function, recovering the main theorem of \cite{W}.

It is useful to observe that condition (1) in the statement of \Cref{main-theorem} is not restrictive.  Indeed, we obtain the following result by applying the present author's initial embedding in \cite{WMal}:

\begin{theorem} \label{main-corollary}

For $i=1,2$, let $\pazocal{P}_i=\gen{X\mid\pazocal{R}_i}$ be two presentations for a group $G$ with $|X|<\infty$, $\pazocal{R}_2\subseteq\pazocal{R}_1$, and $1\notin\pazocal{R}_i$.  Let $f_1$ and $f_2$ be non-decreasing functions on the nonnegative reals such that $f_1(n)\geq n$ and $f_2(n)\geq1$ for all $n\in\N$.  Suppose:

\begin{enumerate}

\item For any finite alphabet $Y$, there exists a $(F(Y),f_1,C)$-Move machine satisfying the computational $c$-move condition for some constant $c$.

\item There exists an $S$-machine $\textbf{S}$ that recognizes $\pazocal{R}_1$ and satisfies $\TM_\textbf{S}\preceq f_1$.

\item The presentation $\pazocal{P}_2$ satisfies the $(f_1(n)^2,n^2f_2(n))$-mass condition.

\end{enumerate}

Then there exists a quasi-isometric (indeed bi-Lipschitz) embedding of $G$ into a finitely presented group $H$ such that $\delta_H(n)\preceq n^2f_2(n)$.

\end{theorem}

The `computational $c$-move condition' referenced in condition (1) of the statement of \Cref{main-corollary} is defined in \Cref{sec-move} and shown in \Cref{sec-Move-conditions} to assure that the Move machine satisfies the $(f,\Gamma)$-move condition for any appropriate choices of function $f$ and group $\Gamma=\gen{Y}$.  Hence, \Cref{theorem-higman} follows as a corollary of \Cref{main-corollary} through the main theorem of Chornomaz and the present author in \cite{CW} and the construction herein of a simple $S$-machine $\textbf{Move}_{1,X}$ which is shown to be a $(F(X),n^2,0)$-Move machine satisfying the computational $2$-move condition (see Section 14).

A subsequent paper will be dedicated to the inductive construction (using $\textbf{Move}_{1,X}$ as a base) of a $(F(X),n^{1+1/k},1)$-Move machine $\textbf{Move}_{k,X}$ satisfying the computational $c$-move condition.  Applying \Cref{main-corollary}, this construction will produce a refinement of \Cref{theorem-higman} that is almost optimal in a particular sense.

We emphasize, however, that the ambient groups arising through the construction of $\textbf{Move}_{k,X}$ will not generally be able to have quadratic Dehn function, a property that is plainly in stark contrast to the specific case of the present author's construction in \cite{W}.  As the Move machine in \cite{W} does not satisfy the computational $c$-move condition for any $c$, the prospect of facilitating a further optimal refinement is the main motivation for the definition of the move condition and the statement of \Cref{main-theorem}.

Finally, we observe that the embeddings constructed for the proofs of these statements enjoy several additional notable properties.  For example, it can be added to each statement that the embeddings are Frattini and satisfy the Congruence Extension Property.  It is thus the aim of this paper to provide a `recipe' (or as desired a `black box') for constructing a Higman embedding which satisfies several desirable algebraic and topological properties borne out of combinatorial and algorithmic properties of the initial group.

\medskip

\subsection{Background} \label{sec-intro-background} \

If $X$ is a finite generating set for a group $G$, then a presentation $\pazocal{P}=\gen{X\mid\pazocal{R}}$ for $G$ is said to be \textit{recursive} if $\pazocal{R}$ is a recursively enumerable set of words in $X\cup X^{-1}$.  In this case, $G$ is said to be a \textit{recursively presented} group.  It can be readily seen that a finitely generated subgroup of a finitely presented group is recursively presented.  The converse to this, {\frenchspacing i.e. that} any finitely generated recursively presented group can be embedded into a finitely presented group, is a deep and celebrated result of Higman \cite{Higman}.

In the years since Higman's embedding theorem, a great deal of work has gone into analyzing its consequences and limitations, with numerous distinguished results that can be expressed as refinements of this central theorem.  However, as many of these results were achieved through the general technique of associating a finitely presented group to a Turing machine \cite{Aanderaa}, many desirable geometric characteristics were unable to be realized in such a refinement.

This changed in the 1990s with Sapir's invention of the \textit{$S$-machine} \cite{SBR}.  There are many equivalent interpretations of $S$-machines (see \cite{S06}, \cite{CW}), but the two vital to this manuscript are as follows: (1) An $S$-machine can be viewed as a rewriting system resembling a multi-tape, non-deterministic Turing machine that works with group words, and (2) An $S$-machine can be viewed as a multiple HNN-extension of a free group.  With these interpretations, it was shown that $S$-machines form a `robust' computational model in that any non-deterministic Turing machine can be emulated by an $S$-machine \cite{SBR}, while these machines also provide a relational framework which is convenient for study through meticulous and well-established group-theoretic techniques.  Accordingly, $S$-machines facilitate the construction of finitely presented groups satisfying various desirable properties.

A major motivation for the invention of the $S$-machine was the study of the \textit{Dehn function} of groups.  Given a presentation $\pazocal{P}$ for a group $G$ with associated finite generating set $X$, the Dehn function $\delta_\pazocal{P}$ is the smallest isoperimetric function for the Cayley $2$-complex of $G$ with respect to this presentation.  Dehn functions are considered up to a natural equivalence $\sim$ on non-decreasing functions $f:\N\to[0,\infty)$ induced by the preorder $\preceq$ such that $f\preceq g$ if and only if there exists a constant $C\in\N$ such that for all $n\in\N$,
$$f(n)\leq Cg(Cn)+Cn+C$$
This equivalence is natural for the consideration of Dehn functions in the sense that the Dehn function of two finite presentations of the same group (or indeed two presentations of quasi-isometric groups) are equivalent.  As such, if $G$ is finitely presented then one can define the Dehn function of $G$, $\delta_G$, as the equivalence class of the Dehn function of any of its finite presentations.

The Dehn function has many applications in the study of finitely presented groups.  For example, the Dehn function quantifies a notion of curvature in these groups: A finitely presented group $G$ is hyperbolic if and only if $\delta_G\sim n$ \cite{Gromov}, while groups possessing forms of nonpositive curvature generally have quadratic Dehn function.  Further, it is well-known that the Dehn function of the fundamental group of a compact Riemannian manifold $M$ is equivalent to the smallest isoperimetric function of the universal cover $\tilde{M}$.

The Dehn function also has close a connection to a particular computational property of the group: The Word Problem for a finitely presented group is decidable if and only if its Dehn function is a computable function, and moreover the Dehn function of a group bounds the (non-deterministic) computational complexity of the Word Problem.  

A partial converse to this was established through the invention of $S$-machines \cite{BORS}: If $X$ is a finite generating set of a group $G$, then for any multi-tape, non-deterministic Turing machine $\pazocal{T}$  which enumerates the set of words over $X\cup X^{-1}$ which represent the identity, there exists a finitely presented group into which $G$ embeds whose Dehn function is polynomially bounded by the complexity of $\pazocal{T}$.  

This seminal result gives a strong link between the algorithmic concept of the Word Problem and the geometric concept of the Dehn function, producing an example of a geometric refinement of the Higman embedding theorem.  However, for groups whose Word Problem is in the NP complexity class, the nature of the `polynomial bound' of this result leaves room for potential refinement.  For example, a group for which there exists a linear time non-deterministic algorithm for deciding the Word Problem is only promised an embedding into a finitely presented group with Dehn function bounded above by $n^{10}$ begging the question:

\begin{center}

Can this bound be improved?  In particular, if there exists a linear time non-deterministic algorithm for deciding the Word Problem of a finitely generated group $G$, then does there exist an embedding of $G$ into a finitely presented group whose Dehn function is at most $n^k$ for some $k<10$?

%can this bound be improved, so that the algorithm given by the Dehn function of this ambient group is closer to the optimal one?

\end{center}

A specific case of this question was addressed by the present author in \cite{W}: 

Using the graded small-cancellation theory of Ol'shanskii \cite{O}, one can find a linear time non-deterministic algorithm deciding the Word Problem of the infinite free Burnside groups $B(m,n)$ for sufficiently large exponent, so that the above question asks what functions can be realized as the Dehn function of a finitely presented group in which such a group embeds.

%For $n>1$ and $m\in\N$, the \textit{free Burnside group} $B(m,n)$ is the quotient $F(X)/F(X)^n$, where $|X|=m$ and $F(X)^n$ is the subgroup of $F(X)$ generated by all words of the form $w^n$.  As a solution to the \textit{bounded Burnside problem}, Novikov and Adian showed that $B(m,n)$ is an infinite group for $m>1$ and $n\geq4381$ odd \cite{Novikov-Adian}. Adian later improved the bound on $n$ to $n\geq665$ \cite{Adian}. %In 1982, Ol'shanskii provided a simpler geometric proof that $B(m,n)$ is infinite for $m>1$ and sufficiently large odd $n$ ($n>10^{10}$), as well as proving the existence of the so-called Tarski monster groups [14].  
%%With respect to even exponents, Ivanov refined the geometric methods of Ol'shanskii to obtain an analogous result for $n$ divisible by $2^9$ and sufficiently large (at least $2^{48}$) [8], while 
%Lysenok proved a similar result for even exponents with $n=16k\geq8000$ \cite{Lysenok}.
%
%As it can be seen through the graded small-cancellation theory of Ol'shanskii \cite{O} that there exists a linear time non-deterministic algorithm deciding the Word Problem of the infinite free Burnside groups of large enough exponent, the question arises as to what functions can be realized as the Dehn function of a finitely presented group in which such a group embeds.  

This question only becomes more compelling when one investigates other related geometric considerations: Ghys and de la Harpe showed that no hyperbolic group can contain an infinite torsion subgroup \cite{Ghys-delaHarpe}, so that the isoperimetric gap (\cite{Gromov}, \cite{Bowditch}, \cite{O91}) implies the minimal function that can be realized as such is quadratic; moreover, many generalizations of the notion of nonpositive curvature in groups produce classes of groups for which the existence of an analogue of the result of Ghys and de la Harpe is known.

However, the present author showed that the class of quadratic Dehn function groups definitively does not possess such an analogue: For sufficiently large $n$, $B(m,n)$ embeds quasi-isometrically into a finitely presented group with quadratic Dehn function.

It is the goal of this manuscript to adapt the techniques of \cite{W} to the much more general setting of \Cref{main-theorem}.  As a main application, this result will be shown to imply the refinement \Cref{theorem-higman} of the seminal result of \cite{BORS}, and will then be used in a subsequent paper alongside the emulation result of \cite{CW} to exemplify a `quasi-optimal' refinement of this result.  

Note that while the refinements of \cite{BORS} will improve the general bound on the Dehn function in which a given group may embed, there are specific environments where the techniques herein can give better bounds than the general bound promised.  For example, even the `quasi-optimal' refinement that will be established in the forthcoming paper will only guarantee that a free Burnside group of sufficiently large exponent can be embedded into a finitely presented group with `almost-quadratic' Dehn function, whereas the specific application in \cite{W} improves this to exactly quadratic.  Hence, this paper will also serve as a recipe for the construction of various interesting embeddings, particularly those that may potentially solve several of the questions left open by this sequence related to the so-called `essential Dehn function' introduced in the sequel.

\medskip

%%%%%%%%%%%%%%%%%%%%%%%%%%%%%%%%%%%%%%%%%%%%%%%%%%%

\subsection{$S$-machines} \label{sec-intro-S} \

In order to put the main results of this paper in proper perspective, some technical terminology regarding $S$-machines is necessary.  We outline here only the necessary details and in as simple of terms as possible; for full definitions, see \Cref{sec-S-machines}.

The reader may think of an $S$-machine as a multi-tape, non-deterministic Turing machine which works with words from the free group rather than those from the free monoid.  However, $S$-machines may also be interpreted as multiple HNN-extensions of the free group, and it is through this interpretation that one can realize the constructions from which this model has yielded many groundbreaking results.

To bridge this gap, we present $S$-machines as rewriting systems.

Let $Q$ be a finite set of possible `states' and $A$ be a finite `alphabet'.  Duplicate the set of states to create two sets $Q_0$ and $Q_1$ in bijection with $Q$.  

Consider a word of the form $W\equiv q_0~w~q_1$ (where henceforth $\equiv$ indicates visual equality) where $w\in F(A)$ and the letters $q_i\in Q_i$ correspond to a particular state $q$ in $Q$.  A command $\theta$ of a machine, called an \textit{$S$-rule}, operates on $W$ by looking at its contents and then, should they be `suitable', by rewriting it as the reduced word $W'=q_0'~xwy~q_1'$ where $q_i'\in Q_i$ correspond to some fixed state $q'\in Q$ and $x,y\in F(A)$ are some fixed (perhaps empty) words.

One can interpret this rewriting as taking place on a bi-infinite tape \`{a} la the operation of a Turing machine.  However, one can also realize this rewriting in a group-theoretic setting.

For this, taking the aforementioned letters as generators of a group, impose the following conjugation relations:

\begin{itemize}

\item $q_0^\theta=q_0'x$

\item $a^\theta=a$ for all letters $a$ comprising $w$

\item $q_1^\theta=yq_1'$

\end{itemize}

\begin{figure}[H]
	\centering
	\includegraphics[scale=1.25]{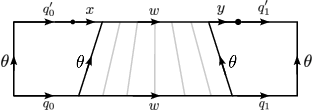}
	\caption{} 
	\label{fig-one-sector-rule}       
\end{figure} 

Then it follows that $W^\theta=W'$.  This is seen in the van Kampen diagram of \Cref{fig-one-sector-rule}, which can be interpreted as a `relational strip'.  Written along the `bottom' of this strip is the word $W$, while $W'$ is written along the `top'.  Meanwhile, each cell corresponds to one of the relations of the form above, with the shaded lines indicating the presence of relators of the second type (of which there are $\|w\|$).

Note, however, that if $v\in F(A)$ is comprised only of letters appearing in $w$, then per these relations $\theta$ also conjugates the word $V\equiv q_0~v~q_1$ to the word $V'=q_0'~xvy~q_1'$.  As such, in the interpretation of $S$-machines as computational models, a rule is in some sense `blind' to what is written on the tape.  With that said, some control can be gained by limiting what letters are allowed to comprise the word written on the tape.  This is done by restricting the relations of the second type above, producing a `domain' of the rule $\theta$.

Further, note that this relation also implies $(W')^{\theta^{-1}}=W$ in the group, suggesting the existence of an `inverse rule'.  This rule is realized by necessitating that the letters of $x$ and $y$ are in the domain of $\theta^{\pm1}$ and considering $\theta^{-1}$ as the $S$-rule with the same domain as $\theta$ that takes a word of the form $q_0'~u~q_1'$ to the reduced word $q_0~x^{-1}uy^{-1}~q_1$.

The setup we are concerned with, however, is a multi-tape one.  For this, we simply add more copies of the set of states, i.e $Q_0,Q_1,\dots,Q_N$, and consider words of the form $q_0w_1q_1w_2\dots w_Nq_N$.  In an attempt to ensure that these words progress in the `correct order' when applying a rule, we introduce multiple copies of a rule $\theta$ indexed by the tapes and replace the conjugation relations of the state letters with ones that incrementally change the copy of $\theta$.  For example, the first relation above would now appear as $\theta_0^{-1}q_0\theta_1=q_0'x$.  Note that this change still allows the application of $\theta$ to be realized as a `relational band' as in \Cref{fig-one-sector-rule}.
%take different copies of the letter $\theta$, i.e $\theta_0,\theta_1,\dots,\theta_N$, and replace the conjugation relations with those of the form $\theta_{i-1}^{-1}q_i\theta_i=y_{i-1}q_i'x_i$ and $a^{\theta_i}=a$ for appropriate $a$ from the $i$-th tape.

An $S$-machine is then defined to be a finite, symmetric set of rules which apply to a finite set of states with corresponding finite tape alphabets.  

Naturally, a `computation' of such a machine is the consecutive application of a number of its rules to an `admissible word' to which it applies.  As such, computations can be realized diagrammatically by stacking the `relational bands' corresponding to each application on top of one another, forming a diagram called a `trapezium' as seen in \Cref{fig-trapezium}.

\begin{figure}[H]
	\centering
	\includegraphics[scale=1.25]{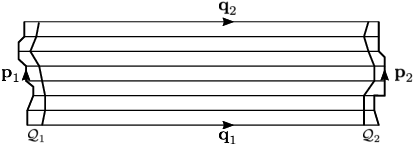}
	\caption{} 
	\label{fig-trapezium}       
\end{figure} 

Note that non-determinacy is inherent in the setup of this model: Given a rule $\theta$ and a word $W$ to which it applies, the rule $\theta^{-1}$ is applicable to the result $W'$ of applying $\theta$ to $W$, in which case the application would again be $W$.  Hence, in order for a computation to `progress' in any meaningful way, we must allow multiple rules to apply to the same word.

By fixing `start' and `end' states and an `input tape', the $S$-machine can be seen to `recognize' a language analogous to how a non-deterministic multi-tape Turing machine enumerates a language.  The `time function' of the $S$-machine is then defined in just the same way as it is for such a Turing machine.

One complication arising from this setup, however, is that from a relational viewpoint a rule may apply to a `non-standard' reduced word like $q_i^{-1}~w_i~q_i$.  To account for this, we must consider a wider range of `admissible words', including those in which the state letters do not progress in the `correct order', but for which there exists a diagrammatic analogue of a `trapezium'.

To be precise, the `base' of an admissible word $W$ is the (unreduced) concatenation of the letters $Q_i^{\pm1}$ in the order in which the state letters of $W$ appear.  As the application of a rule to an admissible word does not change the base of the word, there is a base associated to any computation, and by extension to any trapezium. 

The `standard base' of the machine is the base $Q_0\dots Q_N$, {\frenchspacing i.e. the} base that progresses in the `correct order' and contains every letter exactly once.  As the name suggests, computations in the standard base are those that are interpretable as computations of the particular multi-tape computational model first imagined.  A more general type of base that is of interest in what follows is a `full base'.  Roughly speaking, a base is full if either it starts and ends with the same letter or if it is not a proper subword of the base of another admissible word.

\subsection{Move machines} \label{sec-intro-move} \

Fix three finite sets $X$, $Y$, and $Z$ such that there exists a fixed bijection between $X$ and $Y$.  %Given a reduced word $v\in F(Y\cup Z)$, define:
%
%\begin{itemize}
%
%\item $|v|_Z$ to be the number of letters of $v$ which are from $Z\cup Z^{-1}$
%
%\item $\bar{v}$ to be the word over $X\cup X^{-1}$ obtained from $v$ by deleting all letters from $Z\cup Z^{-1}$ and replacing the letters of $Y^{\pm1}$ with the corresponding letters of $X^{\pm1}$.
%
%\end{itemize}

Let $\textbf{S}$ be an $S$-machine with standard base $Q_0\dots Q_N$ such that the alphabet of the `first' tape between $Q_0$ and $Q_1$ is $X$ and that of the `last' tape between $Q_{N-1}$ and $Q_N$ is $Y\cup Z$.

Suppose $\pazocal{C}$ is a computation of $\textbf{S}$ in the standard base that begins in the start state with all tapes empty except for the first tape, on which the word $w\in F(X)$ is written.  Further, suppose $\pazocal{C}$ terminates in the end state with all tapes empty except for the last one, on which the word $v\in F(Y\cup Z)$ is written.  

Then, $\pazocal{C}$ is called a `move computation of $w$' if, disregarding the letters of $Z\cup Z^{-1}$, $v$ is freely equal to the word corresponding to $w$ with respect to the fixed bijection.  If these words are actually the same (and not just freely equal) and the number of letters from $Z\cup Z^{-1}$ comprising $v$ is at most $C\|w\|$, then $\pazocal{C}$ is called a `$C$-move computation of $w$'.

Now, fix a subset $R\subseteq F(X)$ and a function $f:\N\to\N$.  Then, $\textbf{S}$ is called a `$(R,f,C)$-Move machine' if:

\begin{itemize}

\item All rules of $\textbf{S}$ have full domain in the first tape and change this sector only on the right.

\item Given a reduced computation of $\textbf{S}$ in the standard base that begins in the start state and terminates in the end state with all tapes empty except the last tape is a move computation of some word.

\item For every $w\in R$, there exists a $C$-move computation of $w$ with length asymptotically bounded by $f(\|w\|)$.

\item Any reduced computation of $\textbf{S}$ in the standard base that begins and ends in the start state with all tapes empty except the first tape must be empty.

\item For any reduced computation of $\textbf{S}$ in the standard base that begins and ends in the end state with all tapes empty except the last tape, the corresponding words are freely equal if the letters of $Z\cup Z^{-1}$ are disregarded.

\item Given a computation $\pazocal{C}:W_0\to\dots\to W_t$ of $\textbf{S}$ in a full base, the size of the tape words of $W_i$ is bounded above by the sizes of those of $W_0$ and $W_t$.

\end{itemize}

%Suppose that any computation of $\textbf{S}$ that begins in the start state and terminates in the end state is a $C'$-move computation of some word $w$ and some $C'$.

%Then, given a subset $R\subseteq F(X)$ and a function $f:\N\to\N$, the machine $\textbf{S}$ is called a `$(R,f,C)$-Move machine' if for each $w\in R$, there exists a $C$-move computation of $w$ whose length is asymptotically bounded by $f(\|w\|)$.

The term `Move machine' is indicative of the key property that, up to the insertion of a bounded number of `noise' in the way of letters of $Z\cup Z^{-1}$, there are computations of the machine which `move' the words of $R$ from the first sector to the last in time $f$.

Suppose that in addition to the properties above, for any reduced computation $\pazocal{C}:W_0\to\dots\to W_t$ of $\textbf{S}$ in a full base, the number of applied rules that change tape words corresponding to the first sector of the standard base is bounded by a multiple $c$ of the sizes tape words of $W_0$ and $W_t$ corresponding to other sectors.  Then $\textbf{S}$ is said to satisfy the `computational $c$-move condition'.

Now, fix a normal subgroup $N$ of $F(X)$ and let $\Omega$ be the set of non-trivial cyclically reduced words in $N$.  Generalizing the notion diagrammatic notions of the previous section, an `$a$-trapezium over $\Omega$' is a diagram formed by allowing the presence of `$a$-cells' whose boundary labels are words in $\Omega$ between the `relational bands' of a trapezium corresponding to a computation of $\textbf{S}$.

%Fix a finite alphabet $A$ and a set $\Omega\subseteq F(A)$.  Let $\textbf{M}$ be an $S$-machine with standard base $Q_0Q_1\dots Q_NQ_{N+1}$ and so that the alphabet of the tape between $Q_0$ and $Q_1$ is identified with $A$.  Generalizing the notion above, an `$a$-trapezium mod $\Omega$' is a diagram formed by allowing `$a$-cells' corresponding to words in $\Omega$ between the `relational bands' of a trapezium corresponding to a computation of $\textbf{M}$.  

Note that by construction a base may be assigned to an $a$-trapezium in just the same manner as it is to a trapezium.  Moreover, one can assign the `height' of the trapezium to be the number of bands that comprises it, while its `top length' and `bottom length' are the lengths of the words written on the corresponding bands.

It is useful to count the areas of the cells of an $a$-trapezium in a non-uniform way.  In particular, given a non-decreasing function $\phi:\N\to\N$, the `$\phi$-weight' of an $a$-cell $\pi$ is taken to be $\phi(|\partial\pi|)$, while the $\phi$-weight of all other types of cells are taken to be 1.  Naturally, this is extended by setting the $\phi$-weight of an $a$-trapezium to be the sum of the $\phi$-weights of its cells.

Then, given a non-decreasing function $f:\N\to\N$ with $\phi\sim f$, the Move machine $\textbf{S}$ is said to satisfy the `$(f,G)$-move condition' for $G\cong F(X)/N$ if there exists another non-decreasing function $\psi:\N\to\N$ with $\psi\sim f$ and a constant $K>0$ such that for any $a$-trapezium $\Delta$ over $\Omega$ with full base, there exists another such $a$-trapezium with the same boundary label and $\phi$-weight at most 
$$Kh(|\textbf{bot}(\Delta)|+|\textbf{top}(\Delta)|)+\psi(|\textbf{bot}(\Delta)|+|\textbf{top}(\Delta)|)$$ where $h$ is the height of $\Delta$, $|\textbf{bot}(\Delta)|$ is the bottom length of $\Delta$, and $|\textbf{top}(\Delta)|$ is the top length.

The $(f,G)$-move condition is the central property that assures the Dehn function of the group constructed herein has the desired bound.  It is shown in \Cref{sec-Move-conditions}, however, that any Move machine that satisfies a computational move condition also satisfies the $(f,G)$-move condition for the appropriate $f$ and $G$.  The computational move condition is thus useful for providing a straightforward method (though not a necessary one) for verifying the move condition.

\medskip

\subsection{Mass conditions} \label{sec-intro-mass} \

Finally, we describe here the definition of the `mass condition'.

Fix a presentation $\pazocal{P}=\gen{X\mid\pazocal{R}}$ for a group $G$ with $|X|<\infty$ and a non-decreasing function $\rho_1:\N\to[0,\infty)$.  If $\Delta$ is a van Kampen diagram over $\pazocal{P}$, then define the `$\rho_1$(n)-mass' of a cell $\pi$ of $\Delta$ to be $\rho_1(|\partial\pi|)$.  We extend this naturally by setting the `$\rho_1(n)$-mass' of $\Delta$ to be the sum of the $\rho_1(n)$-masses of its cells.

Then, given a function $\rho_2:\N\to[0,\infty)$, the presentation $\pazocal{P}$ is said to satisfy the `$(\rho_1,\rho_2)$-mass condition' if for every $C_1>0$ there exists $C_2>0$ such that for every word $w$ over $X$ that represents the identity in $G$, there exists a van Kampen diagram over $\pazocal{P}$ with boundary label $w$ whose $\rho_1(C_1n)$-mass is at most $C_2\rho_2(C_2\|w\|)$.

Note that if $\rho_1$ is the constant function $\rho_1(n)\equiv1$, then the $\rho_1(n)$-mass of a diagram is simply its area.  As such, the $(\rho_1,\rho_2)$-mass condition would in this case simply indicate that the Dehn function of the presentation is bounded above by $\rho_2$.

\medskip

%%%%%%%%%%%%%%%%%%%%%%%%%%%%%%%%%%%%%%%%%%%%%%%%%%%

\subsection{Outline of contents} \

In Section 2, we recall the definition and relevant basic properties in the theory of $S$-machines, before giving a full introduction to the concept of Move machines in Section 3.  In Sections 5 and 6, we construct the main machine $\textbf{M}$ for the purpose of proving \Cref{main-theorem}, using the inductive parameters listed in Section 4 as a guide for the makeup.  In Sections 7-11, we study the groups associated to $\textbf{M}$ using van Kampen diagrams, proving that the finitely presented group $G(\textbf{M})$ is suitable for the proof of \Cref{main-theorem}.  Note that the arguments given in these sections are analogous to those of the present author in \cite{W}.  Finally, we complete the proofs of the main theorems in Sections 12-14.

\medskip

\textbf{Acknowledgements:} The author would like to acknowledge Jingyin Huang for several useful discussions relating to the content of this paper.  The author would also like to thank Bogdan Chornomaz and Alexander Ol'shanskii for their significant contributions to many of the pieces that went into this paper.

\medskip

%%%%%%%%%%%%%%%%%%%%%%%%%%%%%%%%%%%%%%%%%%%%%%%%%%%

\section{$S$-machines} \label{sec-S-machines}

\subsection{Rewriting systems} \

In this section, we give a formal definition of the computational model of $S$-machines.  As in \Cref{sec-intro-S}, this definition is presented as a rewriting system of group words, following the conventions of \cite{BORS}, \cite{O18}, \cite{OS01}, \cite{OSconj}, \cite{OS06}, \cite{OS19}, \cite{SBR}, \cite{W}, and others.

A \textit{hardware} is a pair of finite sets $(Y,Q)$ with a fixed partition $Q=\sqcup_{i=0}^NQ_i$ and $Y=\sqcup_{i=1}^NY_i$ is called.  The subsets $Q_i$ and $Y_i$ are called the \textit{parts} of $Q$ and $Y$, respectively.  Moreover, any letter of $Q_i\sqcup Q_i^{-1}$ is called a \textit{state letter}, while any letter of $Y_i\sqcup Y_i^{-1}$ is called a \textit{tape letter} or \textit{$a$-letter}.

An \textit{admissible word} $W$ over $(Y,Q)$ is a reduced word over $(Y\cup Q)^{\pm1}$ of the form $q_0^{\eps_0}w_1q_1^{\eps_1}\dots w_sq_s^{\eps_s}$ such that for each $i\in\{1,\dots,s\}$, there exists an integer $j(i)\in\{1,\dots,N\}$ such that subword $q_{i-1}^{\eps_{i-1}}w_iq_i^{\eps_i}$ is either:

\begin{enumerate}

\item an element of $(Q_{j(i)-1}F(Y_{j(i)})Q_{j(i)})^{\pm1}$,

\item of the form $qwq^{-1}$ where $q\in Q_{j(i)-1}$ and $w\in F(Y_{j(i)})$, or

\item is of the form $q^{-1}wq$ where $q\in Q_{j(i)}$ and $w\in F(Y_{j(i)})$.

\end{enumerate}

In this case, the \textit{base} of $W$ is the (possibly unreduced) word $Q_{j(0)}^{\eps_0}Q_{j(1)}^{\eps_1}\dots Q_{j(s)}^{\eps_s}$, while the subword $q_{i-1}^{\eps_{i-1}}w_iq_i^{\eps_i}$ is called the \textit{$Q_{j(i)-1}^{\eps_{i-1}}Q_{j(i)}^{\eps_i}$-sector} of $W$.  Note that an admissible word may have many different subwords which are sectors of the same name.  

The \textit{standard base} of $(Y,Q)$ is the base $Q_0Q_1\dots Q_N$, while a \textit{configuration} is an admissible word over $(Y,Q)$ whose base is standard.  Conversely,  the base $B$ of an admissible word is said to be \textit{full} if it satisfies one of the following conditions:

\begin{itemize}

\item $B$ or $B^{-1}$ is the standard base

\item The first or last letters of $B$ are the same and no proper subword of $B$ satisfies this condition.  In this case, $B$ is said to be `faulty'.

\item The first letter of $B$ is $Q_0$, the last letter is $Q_0^{-1}$, and $B$ has no faulty subwords.

\item The first letter of $B$ is $Q_N^{-1}$, the last letter is $Q_N$, and $B$ has no faulty subwords.

\end{itemize}

Suppose $W$ is an admissible word whose base $B\equiv xvx$ is faulty. If $v$ has the form $v_1yv_2$ for some letter $y$, then the word $B'\equiv yv_2xv_1y$ is also a faulty base of an admissible word $W'$ satisfying $|W'|_a=|W|_a$. In this case, $B'$ is called a \textit{cyclic permutation} of $B$. 

The \textit{$a$-length} of the admissible word $W$ is the number $|W|_a$ of $a$-letters comprising $W$.  The \textit{$q$-length} $|W|_q$ is defined similarly.

An \textit{$S$-rule} on the hardware $(Y,Q)$ is a rewriting rule $\theta$ with the following associated information:

\begin{itemize}

\item A subset $Q(\theta)=\sqcup_{i=0}^N Q_i(\theta)$ of $Q$ such that each $Q_i(\theta)$ is a singleton $\{q_i\}$ consisting of a letter from $Q_i$.

\item A subset $Q'(\theta)=\sqcup_{i=0}^N Q_i'(\theta)$ of $Q$ such that each $Q_i'(\theta)$ is a singleton $\{q_i'\}$ consisting of a letter from $Q_i$.

\item A subset $Y(\theta)=\sqcup_{i=1}^N Y_i(\theta)$ of $Y$ such that each $Y_i(\theta)$ is a subset of $Y_i$.

\item A set of words $\{\a_{i,\theta},\omega_{i,\theta}\in F(Y_i(\theta)):i=1,\dots,N\}$

\end{itemize}

An admissible word $W$ is said to be \textit{$\theta$-admissible} if every letter comprising it is a letter of $Q(\theta)^{\pm1}$ or $Y(\theta)^{\pm1}$.  In this case, the result of applying $\theta$ to $W$ is the admissible word $W\cdot\theta$ obtained by simultaneously:

\begin{itemize}

\item Replacing any occurrence of the $q_i^{\pm1}$ with the subword $(\omega_{i,\theta}q_i'\a_{i+1,\theta})^{\pm1}$, where for completeness $\omega_{0,\theta}$ and $\a_{N,\theta}$ are empty.

\item Making any necessary free reductions so that the word is again reduced.

\item Making any necessary trims so that the word is again admissible.

\end{itemize}

It is common to represent $\theta$ in this case with the notation:
$$\theta=[q_0\to q_0'\a_{1,\theta}, \ q_1\to \omega_{1,\theta}q_1'\a_{2,\theta}, \ \dots, \ q_{N-1}\to \omega_{N-1,\theta}q_{N-1}'\a_{N,\theta}, \ q_N\to \omega_{N,\theta}q_N']$$
If $Y_i(\theta)=\emptyset$, then $\theta$ is said to \textit{lock the $Q_{i-1}Q_i$-sector}.  In this case, $\a_{i,\theta}$ and $\omega_{i,\theta}$ are both necessarily trivial, while the notation $q_{i-1}\xrightarrow{\ell}\omega_{i-1,\theta}q_i'$ is used in the representation of the rule.

Note that the notation defined above does not quite capture the full details of the rule, as it omits the \textit{domain} $Y(\theta)$.  In most cases, however, each set $Y_i(\theta)$ is either taken to be $Y_i$ or $\emptyset$, so that the use of the locking notation suffices.  As such, this is taken as an implicit assumption, with any other situation explicitly detailed.

The \textit{inverse $S$-rule} of $\theta$ is then the $S$-rule $\theta^{-1}$ on $(Y,Q)$ given by taking:

\begin{itemize}

\item $Q(\theta^{-1})=Q'(\theta)$ and $Q'(\theta^{-1})=Q(\theta)$,

\item $Y(\theta^{-1})=Y(\theta)$, and

\item $\omega_{\theta^{-1},i}=\omega_{\theta,i}^{-1}$ and $\a_{\theta^{-1},i}=\a_{\theta,i}^{-1}$.

\end{itemize}

This terminology is justified by noting that an admissible word $W$ is $\theta$-admissible if and only if $W\cdot\theta$ is $\theta^{-1}$-admissible, with $(W\cdot\theta)\cdot\theta^{-1}\equiv W$.

An \textit{$S$-machine} is then a rewriting system $\textbf{S}$ with a fixed hardware $(Y,Q)$ and a \textit{software} consisting of a finite symmetric set of $S$-rules $\Theta(\textbf{S})$ over $(Y,Q)$.  It is convenient to partition the hardware $\Theta(\textbf{S})=\Theta^+(\textbf{S})\sqcup\Theta^-(\textbf{S})$ such that $\theta\in\Theta^+(\textbf{S})$ if and only if $\theta^{-1}\in\Theta^-(\textbf{S})$; in this case, $\theta$ is called a \textit{positive rule} of $\textbf{S}$, while its inverse is called a \textit{negative rule}.

A \textit{computation} of $\textbf{S}$ is a finite sequence $\pazocal{C}:W_0\to\dots\to W_t$ of admissible words with fixed rules $\theta_1,\dots,\theta_t\in \Theta(\textbf{S})$ such that $W_{i-1}\cdot\theta_i\equiv W_i$.  In this case, the number $t$ is called the \textit{length} of $\pazocal{C}$ and the word $\theta_1\dots\theta_t$ is called the \textit{history} of $\pazocal{C}$.  The computation $\pazocal{C}$ is said to be \textit{reduced} if its history is a reduced word in $F(\Theta^+(\textbf{S}))$; note that any computation can be made reduced without altering its initial or terminal admissible words simply by deleting any pairs of consecutive mutually inverse rules.

Given an $S$-machine $\textbf{S}$ with hardware $(Y,Q)$, it is common to fix in every part $Q_i$ of $Q$ a \textit{start} and an \textit{end} state letter.  A configuration of $\textbf{S}$ is called a \textit{start} or \textit{end} configuration if all the state letters that comprise it are start or end state letters, respectively.  

Given a reduced computation $\pazocal{C}:W_0\to\dots\to W_t$ of $\textbf{S}$ with faulty base $B$ and history $H$, for any cyclic permutation $B'$ of $B$ there exists a reduced computation $\pazocal{C}':W_0'\to\dots\to W_t'$ with base $B'$ and history $H$ and so that $|W_j|_a=|W_j'|_a$ for all $0\leq j\leq t$.

The $S$-machine $\textbf{S}$ is said to be \textit{recognizing} if there are fixed sectors $Q_{i-1}Q_i$ which are deemed to be \textit{input sectors}.  A start configuration $W$ of the recognizing $S$-machine $\textbf{S}$ is called an \textit{input configuration} if all of its non-input sectors have empty tape word.   In this case, the word obtained from $W$ by deleting all of its state letters is called its \textit{input}.  In contrast, the \textit{accept configuration} is the end configuration for which every sector has empty tape word.

A configuration $W$ of the recognizing $S$-machine $\textbf{S}$ is said to be \textit{accepted} if there exists an \textit{accepting computation} which begins with the configuration $W$ and terminates with the accept configuration.  If $W$ is an input configuration, then its input is also said to be accepted.

For a configuration $W$ accepted by $\textbf{S}$, define $\TM_{\textbf{S}}(W)$ be the minimal time of a computation that accepts $W$.  The \textit{time function} of $\textbf{S}$ is then the non-decreasing function $\TM_{\textbf{S}}:\N\to\N$ given by: $$\TM_{\textbf{S}}(n)=\max\{\TM_{\textbf{S}}(W): W\text{ is an accepted input configuration of } \textbf{S}, \ |W|_a\leq n\}$$

If two recognizing $S$-machines have the same language of accepted words and $\Theta$-equivalent time functions, then they are said to be \textit{equivalent}.

The following statement simplifies how one approaches the rules of a recognizing $S$-machine:

\medskip

\begin{lemma} [Lemma 2.1 of \cite{O18}] \label{simplify rules}
 
Every recognizing $S$-machine $\textbf{S}$ is equivalent to a recognizing $S$-machine such that $|\a_{\theta,i}|_a,|\omega_{\theta,i}|_a\leq1$ for every rule $\theta$.

\end{lemma}

As a result of Lemma \ref{simplify rules} and the assumption that every $a$-letter of $\a_{\theta,i}$ and $\omega_{\theta,i}$ is in the domain of $\theta$, it may be assumed that each part of every rule of an $S$-machine is of the form $q_i\to bq_i'a$ where $\|a\|,\|b\|\leq1$ (note that the corresponding part of $\theta^{-1}$ is then interpreted as $q_i'\to b^{-1}q_ia^{-1}$).

In fact, the analogous statement of \cite{O18} implies that one can enforce the stronger condition that $\|a\|+\|b\|\leq1$.  However, it will be convenient to allow $\|a\|=\|b\|=1$ in the definitions of the rules of the $S$-machines constructed in future sections, so we use the weaker statement of \Cref{simplify rules} for consistency.

\medskip

%%%%%%%%%%%%%%%%%%%%%%%%%%%%%%%%%%%%%%%%%%%%%%%%%%%

\subsection{Some elementary properties of $S$-machines} \

This section serves to recall several features of $S$-machines which follow from the definition adopted above and will serve useful in the analysis of the machines constructed in the ensuing sections.

First, the following statement is an immediate consequence of the definition of admissible words.

\begin{lemma} \label{locked sectors}

If the rule $\theta$ locks the $Q_iQ_{i+1}$-sector, i.e it has a part $q_i\xrightarrow{\ell}aq_i'$ for some $q_i,q_i'\in Q_i$, then the base of any $\theta$-admissible word has no subword of the form $Q_iQ_i^{-1}$ or $Q_{i+1}^{-1}Q_{i+1}$.

\end{lemma}

Through the rest of this manuscript, we will often use copies of words over disjoint alphabets. To be precise, let $A$ and $B$ be disjoint alphabets and $\varphi:A\to B$ be an injection.  Then for any word $a_1^{\eps_1}\dots a_k^{\eps_k}$ with $a_i\in A$ and $\eps_i\in\{\pm1\}$, its \textit{copy} over the alphabet $B$ formed by $\varphi$ is the word $\varphi(a_1)^{\eps_1}\dots\varphi(a_k)^{\eps_k}$. Typically, the injection defining the copy will be contextually clear.

Alternatively, a copy of an alphabet $A$ is a disjoint alphabet $A'$ which is in one-to-one correspondence with $A$. For a word over $A$, its copy over $A'$ is defined by the bijection defining the correspondence between the alphabets.

%The following are properties of some simple computations in $S$-machines that are fundamental to the proofs presented in the next two sections. They are stated here without proof, with a reference provided for their proofs in previous literature.

\begin{lemma} [Lemma 2.7 of \cite{O18}] \label{multiply one letter}

Let $X_i$ be a subset of $Y_i\cup Y_i^{-1}$ with $X_i\cap X_i^{-1}=\emptyset$.  Let $\Theta_i^+$ be a set of positive rules in correspondence with $X_i$ such that each rule multiplies the $Q_{i-1}Q_i$-sector by the corresponding letter on the left (respectively on the right).  Let $\pazocal{C}:W_0\to\dots\to W_t$ be a reduced computation with base $Q_{i-1}Q_i$ and history $H\in F(\Theta_i^+)$.  Denote the tape word of $W_j$ as $u_j$ for each $0\leq j\leq t$.  Then:

%Let $\pazocal{C}:W_0\to\dots\to W_t$ be a reduced computation, where $W_0$ is an admissible word with the two-letter base $Q_iQ_{i+1}$. Suppose that each rule of $\pazocal{C}$ multiplies the $Q_iQ_{i+1}$-sector by a letter on the left (respectively right). Suppose further that different rules multiply this sector by different letters. Then:

\begin{enumerate} [label=({\alph*})]

\item $H$ is the natural copy of the reduced form of the word $u_tu_0^{-1}$ read from right to left (respectively the word $u_0^{-1}u_t$ read left to right). In particular, if $u_0\equiv u_t$, then the computation is empty

\item $\|H\|\leq\|u_0\|+\|u_t\|$

\item if $\|u_{j-1}\|<\|u_j\|$ for some $1\leq j\leq t-1$, then $\|u_j\|<\|u_{j+1}\|$

\item $\|u_j\|\leq\max(\|u_0\|,\|u_t\|)$

\end{enumerate}

\end{lemma}

Note that in the setting of \Cref{multiply one letter}, for any two words $w_1,w_2\in F(Y_i)$ comprised entirely of letters from $X_i\cup X_i^{-1}$, there exists a computation of the same form as that described in \Cref{multiply one letter} with initial and terminal tape words $w_1$ and $w_2$, respectively.

%\begin{proof}
%
%$(a)$ Note that $u_j$ is freely equal to $\a_ju_{j-1}$ (or $u_{j-1}\a_j$), where $\a_j$ is the letter corresponding to the $j$-th letter of $H$. So, $u_t=\bar{H}'u_0$ (or $u_0H'$), where $\bar{H}'$ (respectively $H'$) is a copy of $H$ read right to left (respectively left to right). So, since $H$ is reduced, the statement follows.
%
%$(b)$ This follows immediately from $(a)$.
%
%$(c)$ Suppose the rules multiply the sector on the left. If $\|\a_ju_{j-1}\|=\|u_j\|>\|u_{j-1}\|$, then $\a_ju_{j-1}$ must be its reduced form. So, $u_{j+1}=\a_{j+1}\a_ju_{j-1}$ can only be unreduced if $\a_{j+1}$ and $\a_j$ cancel; but this would imply that the corresponding rules are mutually inverse, contradicting the assumption that $\pazocal{C}$ is reduced.
%
%The analogous argument applies if the rules multiply the sector on the right.
%
%$(d)$ Fix $0\leq r\leq t$ such that $\|u_r\|$ is minimal. Then either $\|u_r\|=\dots=\|u_t\|$ or there exists $r\leq s\leq t-1$ such that $\|u_r\|=\dots=\|u_s\|<\|u_{s+1}\|$. In the latter case, applying part $(c)$ to the subcomputation $W_s\to\dots\to W_t$ then gives $\|u_s\|<\|u_{s+1}\|<\dots<\|u_t\|$. So, $\|u_j\|\leq\|u_t\|$ for all $r\leq j\leq t$. 
%
%The analogous argument applied to the inverse computation $W_r\to\dots\to W_0$ then shows that $\|u_j\|\leq\|u_0\|$ for all $0\leq j\leq r$.
%
%\end{proof}

\begin{lemma} [Lemma 2.8 of \cite{O18}] \label{multiply two letters}

Let $X_\ell$ and $X_r$ be disjoint subsets of $Y_i$ which are copies of some set $X$.  Let $\Theta_i^+$ be a set of positive rules in correspondence with $X$ such that each rule multiplies the $Q_{i-1}Q_i$-sector on the left by the corresponding letter's copy in $X_\ell$ (or its inverse) and on the right by the copy in $X_r$ (or its inverse).  Let $\pazocal{C}:W_0\to\dots\to W_t$ be a reduced computation with base $Q_{i-1}Q_i$ and history $H\in F(\Theta_i^+)$.  Denote the tape word of $W_j$ as $u_j$ for each $0\leq j\leq t$.  Then:

%Let $W$ be an admissible word with base $Q_iQ_{i+1}$ and $X_\ell,X_r$ be disjoint subsets of $Y_i$. Let $\pazocal{C}:W\equiv W_0\to\dots\to W_t$ be a reduced computation and denote the tape word of $W_j$ by $u_j$ for each $0\leq j\leq t$. Suppose that each rule of $\pazocal{C}$ multiplies the $Q_iQ_{i+1}$-sector by a letter of $X_\ell$ on the left and a letter of $X_r$ on the right, with different rules multiplying by different letters. Then:

\begin{enumerate}[label=({\alph*})]

\item if $\|u_{j-1}\|<\|u_j\|$ for some $0\leq j\leq t-1$, then $\|u_j\|<\|u_{j+1}\|$

\item $\|u_j\|\leq\max(\|u_0\|,\|u_t\|)$ for each $j$

\item $\|H\|\leq\frac{1}{2}(\|u_0\|+\|u_t\|)$.

\end{enumerate}

\end{lemma}

%\begin{proof}
%
%$(a)$ Note that $u_j=\a_ju_{j-1}\b_j$ for some letters $\a_j$ and $\b_j$. So, if $\a_j$ (or $\b_j$) cancels with the first (or last) letter of $u_{j-1}$, then $\|u_j\|\leq\|u_{j-1}\|$. Assuming $\|u_{j-1}\|<\|u_j\|$ then implies that $\a_ju_{j-1}\b_j$ is the reduced form of $u_j$.
%
%Then $u_{j+1}=\a_{j+1}u_j\b_{j+1}$ for some letters $\a_{j+1}$ and $\b_{j+1}$. As in the proof of Lemma \ref{multiply one letter}$(c)$, a cancellation in this word would then contradict the assumption that $\pazocal{C}$ is reduced, so that $\|u_j\|<\|u_{j+1}\|$.
%
%$(b)$ Using $(a)$, this follows from an argument similar to the one presented as the proof of Lemma \ref{multiply one letter}$(d)$.
%
%$(c)$ Note that $u_t=v u_0 w$ for some $v\in F(X_\ell)$ and $w\in F(X_r)$ satisfying $\|v\|=\|w\|=t$.
%
%As $X_\ell$ and $X_r$ are disjoint alphabets, all reductions in this expression for $u_t$ must come from $u_0$. As a result, $\|u_t\|\geq\|v\|+\|w\|-\|u_0\|=2t-\|u_0\|$.
%
%\end{proof}

\begin{lemma} [Lemma 3.6 of \cite{OS12}] \label{unreduced base}

Suppose $\pazocal{C}:W_0\to\dots\to W_t$ is a reduced computation of an $S$-machine with base $Q_iQ_i^{-1}$ (respectively $Q_i^{-1}Q_i$). For $0\leq j\leq t$, let $u_j$ be the tape word of $W_j$. Suppose each rule of $\pazocal{C}$ multiplies the $Q_iQ_{i+1}$-sector (respectively the $Q_{i-1}Q_i$-sector) by a letter from the left (respectively from the right), with different rules corresponding to different letters. Then $\|u_j\|\leq\max(\|u_0\|,\|u_t\|)$ for all $j$ and the history of $\pazocal{C}$ has the form $H_1H_2^\ell H_3$, where $\ell\geq0$, $\|H_2\|\leq\min(\|u_0\|,\|u_t\|)$, $\|H_1\|\leq\|u_0\|/2$, and $\|H_3\|\leq\|u_t\|/2$.

\end{lemma}

\medskip

%%%%%%%%%%%%%%%%%%%%%%%%%%%%%%%%%%%%%%%%%%%%%%%%%%%

\subsection{Primitive machines} \

One of the most useful tools in shaping the computational makeup of an $S$-machine is the implementation of \textit{primitive machines}.  These are simple $S$-machines that `run' the state letters past a certain tape word and back, locking the sector in between.  In practice, this type of machine is used as a submachine to force the base of an admissible word to be reduced in order to carry out specific types of computations.

The first primitive machine is the machine customarily named $\textbf{LR}(Y)$ for some alphabet $Y$.  The standard base of the machine is $QPR$, with $Q=\{q_1,q_2\}$, $P=\{p_1,p_2\}$, and $R=\{r_1,r_2\}$.  The state letters with subscript $1$ are the start state letters, while those with subscript $2$ are the end state letters.  The tape alphabets $Y_1$ and $Y_2$ of the machine are two copies of $Y$.  For simplicity, the letter of $Y_i$ corresponding to the letter $y\in Y$ is denoted $y_i$.

The positive rules of $\textbf{LR}(Y)$ are defined as follows:

\begin{itemize}

\item For every $y\in Y$, a rule $\tau_1(y)=[p_1\to p_1, \ q_1\to y_1^{-1}q_1y_2, \ r_1\to r_1]$

\item The \textit{connecting rule} $\zeta=[p_1\xrightarrow{\ell} p_2, \ q_1\to q_2, \ r_1\to r_2]$

\item For every $y\in Y$, a rule $\tau_2(y)=[p_2\to p_2, \ q_2\to y_1q_2y_2^{-1}, \ r_2\to r_2]$

\end{itemize}

For any configuration $W$ of $\textbf{LR}(Y)$, one can associate a reduced word over $Y\cup Y^{-1}$ by deleting the state letters and taking the natural copies of the tape letters.  Note that the application of any rule does not alter this associated word.  Uses of this or a similar observation is called a \text{projection argument}.

One can interpret the function of the rules as follows: The rules $\tau_1(y)^{\pm1}$ are used to delete the letters from the $PQ$-sector and move them to the $QR$-sector; this continues until the $PQ$-sector is empty, at which point the connecting rule may be applied; from there, the rules $\tau_2(y)^{\pm1}$ are used to move the letters back from the $QR$-sector to the $PQ$-sector.  This informal description motivate the naming of the machine, as one can view the state letter of $P$ as `moving left toward $Q$ and then right toward $R$'.

These notions are made precise by the following statement, which can be understood simply through a projection argument:

\begin{lemma}[Lemma 3.1 of \cite{O18}] \label{primitive computations}

Let $\pazocal{C}:W_0\to\cdots\to W_t$ be a reduced computation of $\textbf{LR}(Y)$ in the standard base. Then:

\begin{enumerate}[label=({\arabic*})]

\item if $|W_{i-1}|_a<|W_i|_a$ for some $1\leq i\leq t-1$, then $|W_i|_a<|W_{i+1}|_a$

\item $|W_i|_a\leq\max(|W_0|_a,|W_t|_a)$ for each $i$

\item Suppose $W_0\equiv q_1~u~p_1r_1$ and $W_t\equiv q_2~v~p_2r_2$ for some $u,v\in F(Y_1)$.  Then $u\equiv v$, $|W_i|_a=\|u\|\defeq\ell$ for each $i$, $t=2\ell+1$, and the $QP$-sector is locked in the rule $W_\ell\to W_{\ell+1}$. Moreover, letting $\bar{u}$ be the word obtained by reading $u$ from right to left, the history $H$ of $\pazocal{C}$ is a copy of $\bar{u}\zeta u$.

\item if $W_0\equiv q_j~u~p_jr_j$ and $W_t\equiv q_j~v~p_jr_j$ for some $u,v$ and $j\in\{1,2\}$, then $u\equiv v$ and the computation is empty (i.e $t=0$)

\item if $W_0$ is of the form $q_j~u~p_jr_j$ or $q_jp_j~v~r_j$ for $j\in\{1,2\}$, then $|W_i|_a\geq|W_0|_a$ for every $i$.

\end{enumerate}

\end{lemma}

It is a useful observation that for any $u\in F(Y)$, there exists a computation of $\textbf{LR}(Y)$ of the form detailed in \Cref{primitive computations}(3).  

The next statement helps understand computations of the machine with an unreduced base:

\begin{lemma} [Lemma 3.4 of \cite{OS19}] \label{primitive unreduced} Suppose $W_0\to\dots\to W_t$ is a reduced computation of $\textbf{LR}(Y)$ with base $QPP^{-1}Q^{-1}$ (or $R^{-1}P^{-1}PR$) such that $W_0\equiv q_jp_j~u~p_j^{-1}q_j^{-1}$ (or $W_0\equiv r_j^{-1}p_j^{-1}~v~p_jr_j$) for $j\in\{1,2\}$ and some word $u$ (or $v$). Then $|W_0|_a\leq\dots\leq|W_t|_a$ and all state letters of $W_t$ have the index $j$.

\end{lemma}

The next primitive machine is the analogous machine $\textbf{RL}(Y)$ which instead `runs' the state letter right and then left.  To be precise, identifying the hardware of this machine with that of $\textbf{LR}(Y)$, the positive rules of $\textbf{RL}(Y)$ are:

\begin{itemize}

\item For every $y\in Y$, a rule $\tau_s(y)=[p_s\to p_s, \ q_s\to y_1q_sy_2^{-1}, \ r_s\to r_s]$

\item The \textit{connecting rule} $\xi=[p_s\to p_f, \ q_s\xrightarrow{\ell} q_f, \ r_s\to r_f]$

\item For every $y\in Y$, a rule $\tau_f(y)=[p_f\to p_f, \ q_f\to y_1^{-1}q_fy_2, \ r_f\to r_f]$

\end{itemize}

There are obvious analogues of Lemmas \ref{primitive computations} and \ref{primitive unreduced} in the setting of $\textbf{RL}(Y)$, which can be verified in much the same ways.

More generally, a primitive machine is the concatenation of one of the machines above with itself a number of times.  

For example, for a fixed $k\geq1$, the primitive machine $\textbf{LR}_k(Y)$ has standard base $QPR$ with $Q=\{q_i\}_{i=1}^{2k}$, $P=\{p_i\}_{i=1}^{2k}$, and $R=\{r_i\}_{i=1}^{2k}$, while the tape alphabets are again copies of $Y$.  The state letters with subscript $1$ are again the start state letters, while those with subscript $2k$ are now taken as the end state letters.  The positive rules of $\textbf{LR}_k(Y)$ are then as follows:

\begin{itemize}

\item For every $i\in\{1,\dots,k\}$ and every $y\in Y$, the rule $$\tau_{2i-1}(y)=[p_{2i-1}\to p_{2i-1}, \ q_{2i-1}\to y_1^{-1}q_{2i-1}y_2, \ r_{2i-1}\to r_{2i-1}]$$

\item For every $i\in\{1,\dots,k\}$, the connecting rule $$\zeta_{2i-1}=\{p_{2i-1}\xrightarrow{\ell} p_{2i}, \ q_{2i-1}\to q_{2i}, \ r_{2i-1}\to r_{2i}]$$

\item For every $i\in\{1,\dots,k\}$ and every $y\in Y$, the rule $$\tau_{2i}(y)=[p_{2i}\to p_{2i}, \ q_{2i}\to y_1q_{2i}y_2^{-1}, \ r_{2i}\to r_{2i}]$$

\item For every $i\in\{1,\dots,k-1\}$, the connecting rule $$\zeta_{2i}=[p_{2i}\to p_{2i+1}, \ q_{2i}\xrightarrow{\ell} q_{2i+1}, \ r_{2i}\to r_{2i+1}]$$

\end{itemize}

Note that $\textbf{LR}_1(Y)=\textbf{LR}(Y)$.

The analogous primitive machine $\textbf{RL}_k(Y)$ is defined similarly.

The analogues of Lemmas \ref{primitive computations} and \ref{primitive unreduced} hold for the machines $\textbf{LR}_k(Y)$ and $\textbf{RL}_k(Y)$. For example, the following is the analogue of Lemma \ref{primitive computations}(3):

\begin{lemma} \label{LR_k analogue}

Let $\pazocal{C}:W_0\to\dots\to W_t$ be a reduced computation of $\textbf{LR}_k(Y)$ in the standard base. If $W_0\equiv q_0^{(1)}up_{{\color{white}1} }^{(1)}q_1^{(1)}$ and $W_t\equiv q_0^{(2k)}vp_{{\color{white}1} }^{(2k)}q_1^{(2k)}$ for some $u,v\in F(Y^{(1)})$, then $u\equiv v$, $|W_i|_a=|W_0|_a\defeq l$ for each $i$, and $t=2lk+2k-1$.

\end{lemma}

As with its analogue, it is useful to observe that for any $u\in F(Y)$, there exists a computation of $\textbf{LR}_k(Y)$ of the form detailed in \Cref{LR_k analogue}.

When the alphabet $Y$ is contextually clear, it is convenient to omit it from the names of the machines. So, there will be reference in subsequent constructions to the machines $\textbf{LR}$, $\textbf{RL}$, $\textbf{LR}_k$, and $\textbf{RL}_k$.

\medskip

%%%%%%%%%%%%%%%%%%%%%%%%%%%%%%%%%%%%%%%%%%%%%%%%%%

\section{Move machines} \label{sec-move}

In this section, we discuss the critical tool of `Move machines'.

As discussed in the introduction, an $S$-machine is called a Move machine if it possesses a certain computational property with respect to certain parameters.  However, it is also imperative to the hypotheses of \Cref{main-theorem} that the Move machine satisfy a certain `move condition' with respect to more parameters, a highly technical property relating to van Kampen diagrams over a particular group presentation associated to the machine.

To keep the focus of the discussion of this section on $S$-machines and away from their associated groups, though, we introduce a stronger, entirely computational property called the `computational $c$-move condition'.  In \Cref{sec-Move-conditions}, we will see that a Move machine satisfying the computational $c$-move condition satisfies the $(f,G)$-move condition for any relevant parameters.

Finally, given an alphabet $X$, we construct in this section the machine $\textbf{Move}_{1,X}$ and show that it is a $(F(X),n^2,0)$-Move machine which satisfies the computational $c$-move condition.

%As indicated in the introduction, there are two properties that an $S$-machine must satisfy in order to be considered a move machine: The first, (Mv1), is a computational property that motivates the name of this class of machine, while the second, (Mv2), is a highly technical property relating to van Kampen diagrams over particular groups associated to the machine.

%To keep the focus of the discussion on $S$-machines and away from their associated gropus, though, we introduce in this section a different class of machines called `Computational Move machines'.  This class is given by replacing property (Mv2) with a more computational property called (CMv2).  In Section {\color{blue} 7}, we will see that a machine satisfying properties (Mv1) and (CMv2) satisfies properties (Mv1) and (Mv2), so that Computational Move machines are indeed themselves Move machines.

%Finally, given an alphabet $X$, we construct the machine $\textbf{Move}_{1,X}$ and show that it is a $(n,n,n^2,0)$-computational move machine with respect to $(X,\Omega,R)$ for any subsets $\Omega$ and $R$ of $F(X)$.

\subsection{Move machines and the computational move condition} \

Let $X$, $Y$, and $Z$ be finite sets such that $X$ is in correspondence with $Y$.  Given a reduced word $v\in F(Y\cup Z)$, define:

\begin{itemize}

\item $\bar{v}$ as the (perhaps unreduced) word over $X\cup X^{-1}$ obtained from $v$ by deleting all letters from $Z\cup Z^{-1}$ and replacing the letters of $Y^{\pm1}$ with the corresponding letters of $X^{\pm1}$.

\item $|v|_Z$ as the number of letters of $v$ which are from $Z\cup Z^{-1}$

\end{itemize}

Let $\textbf{S}$ be an $S$-machine with standard base $Q_0\dots Q_N$ such that the alphabet of the input $Q_0Q_1$-sector is $X$ and that of the $Q_{N-1}Q_N$-sector is $Y\cup Z$.  %For each $i\in\{0,\dots,N\}$, let $q_i^s$ and $q_i^f$ be the start and end state letters, respectively, of $Q_i$.

For any $v\in F(Y\cup Z)$, let $A_\textbf{S}(v)$ be the end configuration which has $v$ written in its $Q_{N-1}Q_N$-sector and every other sector empty.

Given an admissible word $W$ of $\textbf{S}$, the number of tape letters from the input tape alphabet (or its inverse) comprising $W$ is denoted $|W|_I$, while the number of other tape letters in $W$ is denoted $|W|_{NI}=|W|_a-|W|_I$.

%Assume that no rule of $\textbf{S}$ multiplies the input sector, i.e. the $Q_0Q_1$-sector, by a letter from the left.  A rule which multiplies this sector by a letter on the right is called a \textit{moving rule} of the machine.  Given a reduced computation of $\textbf{S}$ with history $H$, the number of moving rules comprising $H$ is denoted $|H|_{mv}$.

Suppose $\pazocal{C}:W_0\to\dots\to W_t$ is a reduced computation of $\textbf{S}$ in the standard base such that $W_0$ is the input configuration corresponding to $w$ and $W_t\equiv A_\textbf{S}(v)$.  Then $\pazocal{C}$ is called a \textit{setup computation}.

%\begin{itemize}
%
%\item $W_0\equiv q_0^s ~ w ~ q_1^s\dots q_{N-1}^s q_N^s$
%
%\item $W_t\equiv q_0^f q_1^f \dots q_{N-1}^f ~ v ~ q_N^f$
%
%\end{itemize}

%Then, if $\bar{v}=w$ and $|v|_Z\leq C\|w\|$, then $\pazocal{C}$ is called a \textit{$C$-move computation of $w$}.  More generally, $\pazocal{C}$ is called a \textit{move computation of $w$} for unspecified $C$.

%If $\bar{v}=w$, then $\pazocal{C}$ is called a \textit{move computation of $w$}; if also $|v|_Z\leq C\|w\|$, then $\pazocal{C}$ is called a \textit{$C$-move computation of $w$}.

If in this case $\bar{v}$ is freely equal to $w$, then $\pazocal{C}$ is called a \textit{move computation of $w$}.  More specifically, given $C\in\N$, the move computation $\pazocal{C}$ is called a \textit{$C$-move computation of $w$} if $\bar{v}\equiv w$ and $|v|_Z\leq C\|w\|$.

%On the other hand, given $v_1,v_2\in F(Y\cup Z)$, a reduced computation of $\textbf{S}$ between the configurations $A_\textbf{S}(v_1)$ and $A_\textbf{S}(v_2)$ is called an \textit{} computation.

Now, fix a subset $R\subseteq F(X)$, a constant $C\in\N$, and a non-decreasing function $f:\N\to\N$.  Suppose:

\begin{enumerate}[label=(Mv{\arabic*})]

\item No rule of $\textbf{S}$ multiplies the input sector by a letter on the left.

\item Every rule of $\textbf{S}$ has domain $X$ in the input sector.

\item Any setup computation of $\textbf{S}$ is a move computation of some word over $X\cup X^{-1}$.

\item There exists a non-decreasing function $g:\N\to\N$ with $g\sim f$ such that for every $w\in R$, there exists a $C$-move computation of $w$ of length $\leq g(\|w\|)$.

\item Any reduced computation between input configurations is empty.

\item Given $u,v\in F(Y\cup Z)$, if there exists a reduced computation of $\textbf{S}$ between $A_\textbf{S}(u)$ and $A_\textbf{S}(v)$, then $\bar{u}$ is freely equal to $\bar{v}$.

\item There exists $d\in\N$ such that for all reduced computations $\pazocal{D}:V_0\to\dots\to V_s$ with a full base, $|V_i|_a\leq d\max(|V_0|_a,|V_s|_a)$ for all $i$.

\item There exists $k\in\N$ such that for all reduced computations $\pazocal{D}:V_0\to\dots\to V_s$ with a full base, $|V_i|_{NI}\leq k\max(|V_0|_{NI},|V_s|_{NI})$ for all $i$.

\end{enumerate}

Then $\textbf{S}$ is called a \textit{$(R,f,C)$-Move machine}.

%Now, suppose every computation $\pazocal{C}:W_0\to\dots\to W_t$ of $\textbf{S}$ in the standard base such that $W_0$ is a start configuration and $W_t$ is an end configuration is a move computation.  Then, given a subset $R\subseteq F(X)$ and a non-decreasing function $f:\N\to\N$, the $S$-machine $\textbf{S}$ is called a \textit{$(R,f,C)$-Move machine} if there exists a function $g\sim f$ such that for each $w\in R$, there exists a $C$-move computation of $w$ $\pazocal{C}:W_0\to\dots\to W_t$ satisfying $t\leq g(\|w\|)$.

%The first machine $\textbf{M}_1$ is the $(f_1,f_2,f_3,C)$-move machine with respect to $(X,\Omega,\pazocal{R}_1)$ in the hypotheses of \Cref{main-theorem}.   However, instead of outlining here the relevant properties of the machine, we instead construct a particular machine $\textbf{Move}_{1,X}$ which will be shown to be a $(n,n,n^2,0)$-move machine with respect to $(X,\Omega,\pazocal{R}_1)$ for any $\pazocal{R}_1$ and $\Omega$. 
%
%The standard base of $\textbf{Move}_{1,X}$ is $Q_0Q_1Q_2Q_3$, while the tape alphabet of each sector is a copy of $X$.  

\begin{remark}

The presence of the alphabet $Z$ is to facilitate the introduction of `noise' when performing a move computation.  The benefit of this will not be seen in practice in this manuscript, as $Z$ is empty in the machine $\textbf{Move}_{1,X}$ constructed herein.  However, the alphabet $Z$ is critical for the inductive construction of the machines $\textbf{Move}_{k,X}$ in the sequel, justifying our general approach to Move machines herein.

\end{remark}

\begin{remark}

Using projection arguments, \Cref{multiply one letter}, and \Cref{unreduced base}, then one can show that the machine $\textbf{M}_3(1)$ constructed in \cite{W} (when restricted to the base $P_0Q_0P_1$) is a $(F(X),n,0)$-Move machine.  Through our main construction in Sections \ref{sec-auxiliary} and \ref{sec-main-machine}, we will comment on each of the defining properties (Mv1)-(Mv8) as their importance becomes apparent.

\end{remark}

For any machine satisfying property (Mv1), a rule which multiplies the input sector by a letter on the right is called a \textit{moving rule}.  Given a reduced computation of such a machine with history $H$, the number of moving rules comprising $H$ is denoted $|H|_{mv}$.

Given the base $B$ of an admissible word of $\textbf{S}$, denote $I(B)$ as the number of occurrences of $Q_1^{\pm1}$ in $B$, with the caveat that we disregard $Q_1$ as the first letter or $Q_1^{-1}$ as the last letter of $B$.

A reduced computation $\pazocal{D}:V_0\to\dots\to V_s$ of the Move machine $\textbf{S}$ with base $B$ and history $H$ is called a \textit{$c$-narrow computation} if $I(B)|H|_{mv}\leq c\max(|V_0|_{NI},|V_s|_{NI})$.

%Let $\textbf{S}$ be a $(R,f,C)$-Move machine as above and let $\pazocal{D}:V_0\to\dots\to V_s$ be a computation with base $B$ and history $H$.  Let $I(B)$ be the number of occurrences of $Q_1^{\pm1}$ in $B$ disregarding $Q_1$ as the first letter or $Q_1^{-1}$ as the last letter.  Further, let $|H|_{mv}$ be the number of moving rules comprising $H$.  Then $\pazocal{D}$ is called a \textit{$c$-narrow computation} if:
%
%\begin{enumerate}
%
%\item $I(B)|H|_{mv}\leq c\max(|V_0|_{NI},|V_s|_{NI})+4c$
%
%\item $|V_i|_a\leq c\max(|V_0|_a,|V_s|_a)+4c$ for all $i$.
%
%\end{enumerate}

Then, the Move machine $\textbf{S}$ is said to \textit{satisfy the computational $c$-move condition} if every computation of $\textbf{S}$ in a full base is a $c$-narrow computation.

\begin{remark}

Note that the parameters defining a given Move machine are not necessarily optimal: A $(R,f,C)$-Move machine is necessarily a $(R',f',C')$-Move machine if $R'\subseteq R$, $f\preceq f'$, and $C'\geq C$.  Similarly, if such a Move machine satisfies the computational $c$-move condition, then it also satisfies the computational $c'$-move condition for any $c'\geq c$.

\end{remark}

\subsection{One-letter Machines} \

As a preliminary step in the construction of a Move machine satisfying the computational $c$-move condition, we introduce here a class of machines that will serve as `submachines' of $\textbf{Move}_{1,X}$ in particular sense.  Given a finite set $X$ and a fixed element $y\in X$, this $S$-machine is called the \textit{one-letter machine} $\textbf{One}_{X,y}$.  

The standard base of $\textbf{One}_{X,y}$ is $Q_0Q_1Q_2Q_3$, while the tape alphabet of the $Q_{i-1}Q_i$-sector is a copy $X_i$ of $X$.  Each part of the state letters is denoted by $Q_i=\{q_i(s),q_i(1),\dots,q_i(6),q_i(f)\}$, where $q_i(s)$ and $q_i(f)$ function as the start and end letters, respectively, of the given part of the hardware.

The software of the machine can be understood by the following informal description of the manner in which a `typical' computation is designed to run:

\begin{itemize}

\item The positive rule $\sigma_{s}$ changes every state letter from $q_i(s)$ to $q_i(1)$ and locks the $Q_1Q_2$-sector.  

\item The machine then operates on the subword $Q_1Q_2Q_3$ of the standard base as the machine $\textbf{RL}(X)$, with the letters $q_i(1)$ and $q_i(2)$ functioning as the start and end letters of this portion of the machine.

\item The positive rule $\tau_s$ changes every state letter from $q_i(2)$ to $q_i(3)$ and locks the $Q_1Q_2$-sector.

\item The positive rule $\theta_y=[q_0(3)\to q_0(4), \ q_1(3)\xrightarrow{\ell} y_1^{-1}q_1(4), \ q_2(3)\to q_2(4)y_3, \ q_3(3)\to q_3(4)]$, where $y_i$ is the copy of $y$ in the tape alphabet $X_i$.

\item The positive rule $\tau_{f}$ changes every state letter from $q_i(4)$ to $q_i(5)$ and locks the $Q_1Q_2$-sector.

\item The machine then operates on the subword $Q_1Q_2Q_3$ of the standard base as the machine $\textbf{RL}(X)$, with the letters $q_i(5)$ and $q_i(6)$ functioning as the start and end letters of this portion of the machine.

\item The positive rule $\sigma_f$ changes every state letter from $q_i(6)$ to $q_i(f)$ and locks the $Q_1Q_2$-sector.  

\end{itemize}

Note that the machine $\textbf{One}_{X,y}$ satisfies conditions (Mv1) and (Mv2), while the only rules of its software which are moving rules are $\theta_y^{\pm1}$.

Moreover, there is an analogue of the projection arguments of primitive machines for these machines: Given a configuration $W$ of $\textbf{One}_{X,y}$, define the \textit{projection} $\pi(W)$ to be the reduced word over $X\cup X^{-1}$ obtained from $W$ by removing the state letters, taking the natural copy of the tape words, and reducing.  The next statement then follows immediately from the construction:

\begin{lemma} \label{one-letter projection}

For any rule $\theta$ of $\textbf{One}_{X,y}$, if $W$ is a configuration which is $\theta$-admissible, then $\pi(W\cdot\theta)\equiv\pi(W)$.

\end{lemma}

Through projection arguments given by \Cref{one-letter projection}, the next statements follow in much the same way as \Cref{primitive computations}:

\begin{lemma} \label{one-letter standard}

Let $\pazocal{C}:W_0\to\dots\to W_t$ be a reduced computation of $\textbf{One}_{X,y}$ with base $Q_1Q_2Q_3$ and history $H$.

\begin{enumerate}

\item Suppose the state letters comprising $W_0$ are start letters and those comprising $W_t$ are end letters.  Then, there exists $w\in F(X)$ such that $W_0\equiv q_1(s)q_2(s)~ w_3 ~ q_3(s)$ and $W_t= q_1(f) q_2(f) ~ y_3w_3 ~ q_3(f)$, where $w_3$ and $y_3$ are the natural copies of $w$ and $y$, respectively, over $X_3$.  Moreover, $\pazocal{C}$ is uniquely determined by $w$, $t=7+2\|w\|+2\|yw\|\leq 4\|w\|+9$, and $|H|_{mv}=1$.

\item If the state letters comprising both $W_0$ and $W_t$ are all start letters or all end letters, then $t=0$.

\item $|H|_{mv}\leq1$.

\item $|W_i|_a\leq\max(|W_0|_a,|W_t|_a)$ for all $i$.

%\item If the state letters comprising $W_0$ are all start or all end letters, then $|W_i|_a\leq|W_t|_a+1$ for all $i$.

\end{enumerate}

\end{lemma}

\begin{lemma} \label{one-letter standard mv}

Let $\pazocal{C}:W_0\to\dots\to W_t$ be a reduced computation of $\textbf{One}_{X,y}$ with base $Q_1Q_2Q_3$ and history $H$.  Then $\max(|W_0|_a,|W_t|_a)\geq|H|_{mv}$.

\end{lemma}

\begin{proof}

By \Cref{one-letter standard}(3), it suffices to assume that $|H|_{mv}=1$.

Let $j\in\{1,\dots,t\}$ such that $W_{j-1}\to W_j$ is the transition corresponding to the move rule $\theta_y^{\pm1}$.  Then there exists $w\in F(X)$ such that $W_{j-1}\equiv q_1(3)q_2(3) ~ w_3 ~ q_3(3)$ and $W_j=q_1(4)q_2(4) ~ y_3w_3 ~ q_3(4)$ or vice versa.  In particular, $\max(|W_{j-1}|_a,|W_j|_a)=\max(\|w\|,\|yw\|)\geq1$.

Hence, the statement follows by \Cref{one-letter standard}(4).

\end{proof}

%The next statement follows similarly:

\begin{lemma} \label{one-letter standard 2}

Let $\pazocal{C}:W_0\to\dots\to W_t$ is a reduced computation of $\textbf{One}_{X,y}$ with base $Q_1Q_2Q_3$ and history $H$.  Fix $w\in F(X)$ and let $w_3$ be the natural copy of $w$ over $X_3$.  Suppose either:

\begin{itemize}

\item $W_0\equiv q_1(s)q_2(s) ~ w_3 ~ q_3(s)$ and $y^{-1}$ is not a prefix of $w$, or

\item $W_0\equiv q_1(f)q_2(f) ~ w_3 ~ q_3(f)$ and $y$ is not a prefix of $w$.

\end{itemize}

Then for any suffix $H'$ of $H$, $|W_t\cdot(H')^{-1}|_a\leq|W_t|_a-|H'|_{mv}$.

\end{lemma}

%\begin{lemma} \label{one-letter standard 2}
%
%Let $\pazocal{C}:W_0\to\dots\to W_t$ be a reduced computation of $\textbf{One}_{X,y}$ with base $Q_1Q_2Q_3$ and history $H$.
%
%\begin{enumerate}
%
%\item
%
%\end{enumerate}
%
%\end{lemma}
%
%\begin{proof}
%
%
%
%\end{proof}

\begin{lemma} \label{one-letter unreduced locked sector}

Let $\pazocal{C}:W_0\to\dots\to W_t$ be a reduced computation of $\textbf{One}_{X,y}$ with base $B$ and history $H$.  Suppose either:

\begin{enumerate}[label=(\roman*)]

\item $B$ is the subword of a full base and contains a two-letter subword of the form $Q_1Q_1^{-1}$, or 

\item $B$ is a full base containing a two-letter subword of the form $Q_2^{-1}Q_2$.  

\end{enumerate}

Then:

\begin{enumerate}[label=(\alph*)]

\item $|W_i|_a\leq2\max(|W_0|_a,|W_t|_a)$ for all $i$.

\item $|H|_{mv}=0$.

\item If every state letter comprising $W_0$ is either a start letter or an end letter, then $t=0$.

\end{enumerate}

\end{lemma}

\begin{proof}

By \Cref{locked sectors}, $\pazocal{C}$ must operate entirely as a single copy of $\textbf{RL}(X)$.  Statements (b) and (c) then follow immediately.

If $B$ satisfies (i), then it cannot contain a letter of the form $Q_2^{\pm1}$.  As such, no rule of $\pazocal{C}$ changes the tape words of the admissible words, so that $|W_0|_a=\dots=|W_t|_a$.

Otherwise, if $B$ satisfies (ii), then it is either $Q_3^{-1}Q_2^{-1}Q_2Q_3$ or a cyclic permutation of $Q_2^{-1}Q_2Q_2^{-1}$.  In the former case, (a) follows immediately from \Cref{primitive unreduced}; in the latter, the restriction of $\pazocal{C}$ to any two-letter subword of $B$ satisfies the hypotheses of \Cref{unreduced base}, again implying (a).

\end{proof}

\begin{lemma} \label{one-letter unreduced last sector}

Let $\pazocal{C}:W_0\to\dots\to W_t$ be a reduced computation of $\textbf{One}_{X,y}$ with base $B$ and history $H$.  Suppose either:

\begin{enumerate}[label=(\roman*)]

\item $B$ is a full base containing a two-letter subword of the form $Q_2Q_2^{-1}$.

\item $B=Q_1Q_2Q_2^{-1}Q_1^{-1}$.

\end{enumerate}

Then:

\begin{enumerate}[label=(\alph*)]

\item $|W_i|_a\leq2\max(|W_0|_a,|W_t|_a)$ for all $i$.

\item $|H|_{mv}\leq1$.

\item Suppose every state letter comprising $W_0$ is either a start letter or an end letter.  Then $|W_0|_a\leq\dots\leq|W_t|_a$.  Moreover, if every state letter comprising $W_t$ is also either a start or an end letter, then $t=0$.

\end{enumerate}

\end{lemma}

\begin{proof}

By \Cref{one-letter unreduced locked sector}, it suffices to assume that $B$ contains no subword of the form $Q_1Q_1^{-1}$ or $Q_2^{-1}Q_2$.  In particular, in either case (perhaps a cyclic permutation of) $B$ must contain a subword of the form $Q_1Q_2Q_2^{-1}Q_1^{-1}$.

The statement then follows by applying Lemmas \ref{primitive unreduced} and \ref{multiply one letter} to the restriction of any subcomputation of $\pazocal{C}$ operating entirely as $\textbf{RL}$ to the appropriate subwords of $B$.

\end{proof}

\medskip

\subsection{The machine $\textbf{Move}_{1,X}$} \

We now introduce the machine $\textbf{Move}_{1,X}$ constructed from the one-letter machines defined in the previous section.

As with the one-letter machines, the standard base of $\textbf{Move}_{1,X}$ is $Q_0Q_1Q_2Q_3$, the $Q_0Q_1$-sector is the input sector, and the tape alphabet of the $Q_{i-1}Q_i$-sector is a copy $X_i$ of the alphabet $X$.

The software of the machine can be understood by the following informal description of the manner in which a `typical' computation is designed to run:

\begin{itemize}

\item The positive rule $\theta_{sf}$ switches the state letters from the start to the end letters, locking the $Q_1Q_2$- and $Q_2Q_3$-sectors.

\item For each $y\in X$, there is a `submachine' identified with $\textbf{One}_{X,y}$.  The states of this machine corresponding to the `intermediate states' $q_i(1),\dots,q_i(6)$ of $\textbf{One}_{X,y}$ are distinct for each $y$; however, the end state letters of the machine correspond to both the start and end state letters of $\textbf{One}_{X,y}$.

\end{itemize}

Note that in constructing $\textbf{Move}_{1,X}$, the `submachines' corresponding to the one-letter machines are not quite identical to the machines of the previous section, as the start and end state letters are essentially identified.  Despite this difference, Lemmas \ref{one-letter standard}-\ref{one-letter unreduced last sector} can be used to understand the full computational makeup of this machine.

To this end, as any one-letter machine satisfies properties (Mv1) and (Mv2), $\textbf{Move}_{1,X}$ does also.

%Similarly, defining the projection of a configuration of $\textbf{Move}_{1,X}$ in much the same way as it is defined for one-letter machines, an analogue of \Cref{one-letter projection} follows immediately from the construction.  As such, the next statement follows quickly:

Further, a reduced computation $\pazocal{C}:W_0\to\dots\to W_t$ of $\textbf{Move}_{1,X}$ is called an \textit{end computation} if every state letter of both $W_0$ and $W_t$ are end letters.  Note that it follows immediately from the definition of the rules that no rule of an end computation can be $\theta_{sf}$ or $\theta_{sf}^{-1}$.

The next statement then follows immediately from \Cref{one-letter standard}:

\begin{lemma} \label{move_1 end factorization}

Let $\pazocal{C}:W_0\to\dots\to W_t$ be a nonempty end computation of $\textbf{Move}_{1,X}$ with base $Q_1Q_2Q_3$ and history $H$.
%
%\begin{enumerate}
%
%\item 
Then there exists a factorization $H\equiv H_{y_1,w_1}^{\eps_1}\dots H_{y_k,w_k}^{\eps_k}$ where $H_{y_i,w_i}$ is the copy of the history of the unique reduced computation of $\textbf{One}_{X,y_i}$ given in \Cref{one-letter standard}(1) and $\eps_i\in\{\pm1\}$.

%\item If $B$ contains a subword of the form $Q_1Q_1^{-1}$, $Q_2^{-1}Q_2$, or $Q_2Q_2^{-1}$, then $t=0$.
%
%\end{enumerate}

\end{lemma}

Note that the application of \Cref{one-letter standard} to the restriction of $\pazocal{C}$ to the base $(Q_1Q_2Q_3)^{\pm1}$ implies $w_{i+1}=y_i^{\eps_i}w_i$ for all $i$ in the setting of \Cref{move_1 end factorization}.  As such, the uniqueness of the relevant reduced computation of a one-letter machine with relevant base implies $y_i\neq y_{i+1}$ or $\eps_i=-\eps_{i+1}$ for each $i$.

\begin{lemma} \label{move_1 Mv5}

Suppose $\pazocal{C}:W_0\to\dots\to W_t$ is a reduced computation of $\textbf{Move}_{1,X}$ with base $Q_1Q_2Q_3$.  If $W_0$ and $W_t$ are both start configurations, then $t=0$.

\end{lemma}

\begin{proof}

Assuming $\pazocal{C}$ is nonempty, the definition of the rules implies there exists a factorization $H\equiv\theta_{sf}H'\theta_{sf}^{-1}$ of the history of $\pazocal{C}$.  Without loss of generality, we may assume that $H'$ contains no letter of the form $\theta_{sf}^{\pm1}$.

In particular, the subcomputation $W_1\to\dots\to W_{t-1}$ with history $H'$ is an end computation.  So, since $H$ is reduce, \Cref{move_1 end factorization} implies there exists a factorization $H'\equiv H_{y_1,w_1}^{\eps_1}\dots H_{y_k,w_k}^{\eps_k}$.

As $W_1$ is $\theta_{sf}^{-1}$-admissible, it must have empty tape word in the $Q_2Q_3$-sector, {\frenchspacing i.e. $w_1=1$}.  So, $w_{t-1}\equiv y_1^{-\eps_1}\dots y_k^{-\eps_k}$.

But then $W_{t-1}$ does not have empty tape word in the $Q_2Q_3$-sector, so that it is not $\theta_{sf}^{-1}$-admissible.

\end{proof}

\begin{lemma} \label{move_1 Mv6}

Suppose $\pazocal{C}:W_0\to\dots\to W_t$ is an end computation of $\textbf{Move}_{1,X}$ in the standard base.  If $W_0$ and $W_t$ have the same word written in the input sector, then $t=0$.

\end{lemma}

\begin{proof}

Suppose $\pazocal{C}$ is nonempty.  Then, applying \Cref{move_1 end factorization} to the restriction of $\pazocal{C}$ to the base $Q_1Q_2Q_3$, there exists a factorization $H\equiv H_{y_1,w_1}^{\eps_1}\dots H_{y_k,w_k}^{\eps_k}$ of the history $H$ of $\pazocal{C}$.  Letting $w_0$ and $w_t$ be the words written in the input sector of $W_0$ and $W_t$, then $w_t=w_0y_1^{-\eps_1}\dots y_k^{-\eps_k}$.

But $y_1^{-\eps_1}\dots y_k^{-\eps_k}$ is necessarily a reduced word, so that $w_t\neq w_0$.

\end{proof}

\begin{lemma} \label{move_1 sub-move}

Given $w\in F(X)$, there exists a unique reduced computation $\pazocal{C}_w:W_0\to\dots\to W_t$ of $\textbf{Move}_{1,X}$ with base $Q_1Q_2Q_3$ such that:

\begin{itemize}

\item Every state letter of $W_0$ is a start letter.

\item Every state letter of $W_t$ is an end letter.

\item $W_t$ has the natural copy of $w$ written in its $Q_2Q_3$-sector.

\end{itemize}

Moreover, in this case $t\leq 2\|w\|^2+7\|w\|+1$.

\end{lemma}

\begin{proof}

Let $w\equiv y_1^{\eps_1}\dots y_k^{\eps_k}\in F(X)$, where $y_i\in X$ and $\eps_i\in\{\pm1\}$.  For each $i$, let $H_{y_i,w_i}$ be the history of the unique computation of $\textbf{One}_{X,y_i}$ in base $Q_1Q_2Q_3$ corresponding to $w_i$, where $w_1=1$ and $w_i\equiv y_1^{\eps_1}\dots y_{i-1}^{\eps_{i-1}}$ for $i>1$.

%For each $i$, let $H_{y_i,w_i}$ be the history of the unique reduced computation of $\textbf{One}_{X,y_i}$ corresponding to $w_i$ given in \Cref{one-letter standard}(1).

Let $H_w\equiv\theta_{sf}H_{y_1,w_1}^{\eps_1}\dots H_{y_k,w_k}^{\eps_k}$.  Then the computation $\pazocal{C}_w$ with history $H_w$ and base $Q_1Q_2Q_3$ satisfies the statement.
%By the definition of the rules, $W_0$ must be $\theta_{sf}$-admissible, and so must have empty tape words.  \Cref{move_1 end factorization} then implies that there exists a factorization of the history $H_w$ of $\pazocal{C}_w$ of the form $H_w\equiv\theta_{sf}H_{y_1,w_1}^{\eps_1}\dots H_{y_k,w_k}^{\eps_k}$, where $H_{y_i,w_i}$ is the natural copy of the history of the unique computation of $\textbf{One}_{X,y_i}$ corresponding to $w_i$ given in \Cref{one-letter standard}(1).
%It then follows that $w_1=1$ and $w_i\equiv y_{i-1}^{\eps_{i-1}}\dots y_1^{\eps_1}$ for $i>1$.  Moreover, $w\equiv y_k^{\eps_k}w_k$, so that $H_w$ is uniquely determined by $w$.
Moreover, by \Cref{one-letter standard}, 
$$\|H\|=1+\sum_{i=1}^k\|H_{y_i,w_i}\|\leq 1+\sum_{i=1}^k(4\|w_i\|+9)=9k+1+4\sum_{i=1}^{k-1}i=2k^2+7k+1$$
%But $k=\|w\|$, implying the statement.
Now suppose $\pazocal{C}:V_0\to\dots\to V_s$ is any reduced computation satisfying the statement for $w$.  As any admissible word consisting of start state letters can be $\theta$-admissible only for $\theta=\theta_{sf}$, there must be a factorization $H\equiv\theta_{sf}H'$ of the history $H$ of $\pazocal{C}$.

By \Cref{locked sectors}, it follows that $V_0\equiv W_0$ and $V_s\equiv W_t$.

The subcomputation of $\pazocal{C}$ with history $H'$ is then an end computation, so that \Cref{move_1 sub-move} implies $H'\equiv H_{z_1,v_1}^{\delta_1}\dots H_{z_\ell,v_\ell}^{\delta_\ell}$ for some $z_i\in X$, $v_i\in F(X)$, and $\delta_i\in\{\pm1\}$.

By construction, $v_1=1$ and $v_{i+1}=z_i^{\delta_i}v_i$.  But then $w\equiv z_\ell^{\delta_\ell}\dots z_1^{\delta_1}$, so that $H\equiv H_w$.

\end{proof}

%\begin{lemma} \label{move_1 sub-move unique}
%
%Let $\pazocal{C}:W_0\to\dots\to W_t$ be a reduced computation of $\textbf{Move}_{1,X}$ with base $Q_1Q_2Q_3$ such that every state letter of $W_0$ is a start letter and ever state letter of $W_t$ is an end letter.  Then $\pazocal{C}$ is uniquely determined by the word $w\in F(X)$ such that $W_t$ has the natural copy of $w$ written in its $Q_2Q_3$-sector.  Moreover, in this case $t\leq2\|w\|^2+7\|w\|+1$.
%
%\end{lemma}
%
%\begin{proof}
%
%Let $H$ be the history of $\pazocal{C}$.  As the only rule applicable to an admissible word whose state letters are start letters, $H$ has prefix $\theta_{sf}$.
%
%By \Cref{move_1 end factorization}, $H$ then has prefix $\theta_{sf}H_{y_1,w_1}^{\eps_1}\dots H_{y_k,w_k}^{\eps_k}$ where $H_{y_i,w_i}$ is the copy of the history of the unique reduced computation $\textbf{One}_{X,y_i}$.  Let $W_0\to\dots\to W_s$ be the subcomputation with this history.
%
%By construction, $w_1=1$ and $w_{i+1}=y_i^{\eps_i}w_i$ for all $i$.  In particular, if $W_s$ is $\theta_{sf}^{-1}$-admissible, then $k=0$.  But then a subsequent letter $\theta_{sf}^{-1}$ would imply $H$ is unreduced.  Hence, $s=t$ and $H\equiv\theta_{sf}H_{y_1,w_1}^{\eps_1}\dots H_{y_k,w_k}^{\eps_k}$.  This implies $w\equiv y_k^{\eps_k}\dots y_1^{\eps_1}\in F(X)$, so that $H$ is uniquely determined by $w$.
%
%\end{proof}

\begin{lemma} \label{move_1 move}

For any pair of reduced words $w,v$ over $X\cup X^{-1}$, there exists a unique reduced computation $\pazocal{C}:W_0\to\dots\to W_t$ of $\textbf{Move}_{1,X}$ in the standard base such that:

\begin{itemize}

\item $W_0$ is a start configuration with the natural copy of $v$ written in its $Q_0Q_1$-sector.

\item $W_t$ is an end configuration with the natural copy of $w$ written in its $Q_2Q_3$-sector.

\end{itemize}

Moreover, in this case $t\leq 2\|w\|^2+7\|w\|+1$ and $W_t$ has the natural copy of $vw^{-1}$ written in its $Q_0Q_1$-sector.

%$\textbf{Move}_{1,X}$ is a $(F(X),n^2,0)$-Move machine.  Moreover, for each $w\in F(X)$ there exists a unique $0$-move computation of $w$.

\end{lemma}

\begin{proof}

%Suppose $\pazocal{C}:W_0\to\dots\to W_t$ is a reduced computation of $\textbf{Move}_{1,X}$ in the standard base such that $W_0$ is a start configuration and $W_t$ is an end configuration.  

The restriction of $\pazocal{C}$ to the base $Q_1Q_2Q_3$ satisfies the hypotheses of \Cref{move_1 sub-move}, so that $\pazocal{C}$ is uniquely determined by $w$.  The uniqueness and quadratic bound on the length of this computation then follow from \Cref{move_1 sub-move}.

A projection argument then implies $W_t$ has the natural copy of $vw^{-1}$ written in its $Q_0Q_1$-sector.

%If $\pazocal{C}$ is a $0$-move computation of $w$, then by definition the restriction of $\pazocal{C}$ to the base $Q_1Q_2Q_3$ satisfies the hypotheses of \Cref{move_1 sub-move}.  It then suffices to note that a computation with this history will indeed be a $0$-move computation.

%Let $w\equiv y_1^{\eps_1}\dots y_k^{\eps_k}\in F(X)$, where $y_i\in X$ and $\eps_i\in\{\pm1\}$.  For each $i$, let $H_{y_i,w_i}$ be the history of the unique computation of $\textbf{One}_{X,y_i}$ in base $Q_1Q_2Q_3$ corresponding to $w_i$, where $w_1=1$ and $w_i\equiv y_1^{\eps_1}\dots y_{i-1}^{\eps_{i-1}}$ for $i>1$.
%
%Let $W$ be the configuration whose state letters are all start letters and such that the natural copy of $w$ is written in the $Q_0Q_1$-sector and all other sectors empty.  Conversely, let $W'$ be the configuration whose state letters are end letters and such that the natural copy of $w$ is written in the $Q_2Q_3$-sector and all other sectors empty.
%
%Then, $H_w\equiv\theta_{sf}H_{y_1,w_1}^{\eps_1}\dots H_{y_k,w_k}^{\eps_k}$ is the history of a reduced computation between $W$ and $W'$, and so is a $0$-move computation of $w$.  By \Cref{one-letter standard}, 
%$$\|H_w\|=1+\sum_{i=1}^k\|H_{y_i,w_i}\|\leq 1+\sum_{i=1}^k(4\|w_i\|+9)=9k+1+4\sum_{i=1}^{k-1}i=2k^2+7k+1$$
%But $k=\|w\|$, implying the statement.

\end{proof}

\begin{lemma} \label{move_1 basic narrow}

Let $\pazocal{C}:W_0\to\dots\to W_t$ be a reduced computation of $\textbf{Move}_{1,X}$ with base $B$.  Suppose $B$ contains a two-letter subword of the form $Q_1Q_1^{-1}$, $Q_2^{-1}Q_2$, or $Q_2Q_2^{-1}$ and either:

\begin{enumerate}[label=(\roman*)]

\item $B$ is a full base.

\item $B=Q_1Q_1^{-1}$

\item $B=Q_1Q_2Q_2^{-1}Q_1^{-1}$.

\end{enumerate}

Then $\pazocal{C}$ is a $2$-narrow computation satisfying $|W_i|_a\leq2\max(|W_0|_a,|W_t|_a)$ for all $i$.

\end{lemma}

\begin{proof}

By \Cref{locked sectors}, no letter of the history $H$ of $\pazocal{C}$ can be $\theta_{sf}^{\pm1}$.  Further, Lemmas \ref{one-letter unreduced locked sector}(c) and \ref{one-letter unreduced last sector}(c) imply that $\pazocal{C}$ operates entirely as a computation of a single one-letter machine $\textbf{One}_{X,y}$.

It then follows from Lemmas \ref{one-letter unreduced locked sector}(a) and \ref{one-letter unreduced last sector}(a) that $|W_i|_a\leq2\max(|W_0|_a,|W_t|_a)$.  

Moreover, Lemmas \ref{one-letter unreduced locked sector}(b) and \ref{one-letter unreduced last sector}(b) imply that $|H|_{mv}\leq1$.  Since $I(B)\leq2$, it then suffices to show that $\max(|W_0|_a,|W_t|_a)\geq1$.  But the unreduced base implies the tape word in this sector must be non-empty in each admissible word, so that $|W_i|_a\geq1$ for all $i$.

\end{proof}

\begin{lemma} \label{move_1 comp move}

Identifying the input alphabet $X_1$ with $X$, $\textbf{Move}_{1,X}$ is a $(F(X),n^2,0)$-Move machine satisfying the computational $2$-move condition.

\end{lemma}

\begin{proof}

By \Cref{move_1 move}, $\textbf{Move}_{1,X}$ satisfies (Mv3) and (Mv4).  Further, the machine satisfies (Mv5) by \Cref{move_1 Mv5} and (Mv6) by \Cref{move_1 Mv6}.

Let $\pazocal{C}:W_0\to\dots\to W_t$ be a reduced computation of $\textbf{Move}_{1,X}$ with full base $B$ and history $H$.  It then suffices to show that $\pazocal{C}$ is a $2$-narrow computation such that $|W_i|_a\leq4\max(|W_0|_a,|W_t|_a)$ and $|W_i|_{NI}\leq4\max(|W_0|_{NI},|W_t|_{NI})$ for all $i$.

As rules of the form $\theta_{sf}^{\pm1}$ change no tape word and can only appear as the first or last letter of $H$, it suffices to assume that $H$ contains no letter of this form.  

Suppose $B$ contains a subword of the form $Q_1Q_1^{-1}$, $Q_2^{-1}Q_2$, or $Q_2Q_2^{-1}$.  Then \Cref{move_1 basic narrow} implies $\pazocal{C}$ is $2$-narrow and satisfies $|W_i|_a\leq2\max(|W_0|_a,|W_t|_a)$.  

Moreover, if $B$ also contains a subword of the form $Q_1^{-1}Q_1$ or $(Q_0Q_1)^{-1}$, then the maximal subword of $B$ containing no such subword is either $Q_1Q_1^{-1}$ or $Q_1Q_2Q_2^{-1}Q_1^{-1}$.  But then \Cref{move_1 basic narrow}(ii) or (iii) implies $|W_i|_{NI}\leq2\max(|W_0|_{NI},|W_t|_{NI})$.

Hence, it suffices to assume that $B$ contains no subword of the form $Q_1Q_1^{-1}$, $Q_2^{-1}Q_2$, or $Q_2Q_2^{-1}$.

By the makeup of full bases, if $B$ contains a subword of the form $Q_3^{-1}Q_3$, then $B\equiv Q_3^{-1}Q_3$.  The definition of the rules then implies $|W_0|_a=\dots=|W_t|_a$, while $\pazocal{C}$ is $2$-narrow (indeed $0$-narrow) since $I(B)=0$.

Similarly, it suffices to assume that $B$ contains no subword of the form $Q_0Q_0^{-1}$.

Thus we may assume $B$ is either the standard base or $Q_3^{-1}Q_2^{-1}Q_1^{-1}Q_1Q_2Q_3$.  In particular, $B$ contains a subword $B'$ of the form $Q_1Q_2Q_3$.  Let $\pazocal{C}':W_0'\to\dots\to W_t'$ be the restriction of $\pazocal{C}$ to $B'$.

\textbf{1.} Suppose $\pazocal{C}$ corresponds to a computation of a single one-letter machine.  Then the corresponding computation satisfies the hypotheses of \Cref{one-letter standard}, so that (3) and (4) imply $|H|_{mv}\leq1$ and $|W_i'|_a\leq\max(|W_0'|_a,|W_t'|_a)$ for all $i$.

A similar argument applies to the restriction of $\pazocal{C}$ to any subword of $B$ of the form $(Q_1Q_2Q_3)^{-1}$.  So, as these two subwords contain all non-input tape words, $|W_i|_{NI}\leq2\max(|W_0|_{NI},|W_t|_{NI})$.

Further, $|H|_{mv}\leq1$ implies $|W_i|_I\leq\max(|W_0|_I,|W_t|_I)$ for all $i$, so that $|W_i|_a\leq2\max(|W_0|_a,|W_t|_a)$ for all $i$.
%As all other sectors have fixed tape word, it then follows that $|W_i|_a\leq\max(|W_0|_a,|W_t|_a)$ for all $i$.

Since the reduced computation corresponding to $\pazocal{C}'$ satisfies the hypotheses of \Cref{one-letter standard mv}, it follows that $\max(|W_0'|_a,|W_t'|_a)\geq |H|_{mv}$.  Note that the analogous argument applies to the restriction to any subword of the form $(Q_1Q_2Q_3)^{-1}$.  So, as the number of such subwords of $B$ is the same as $I(B)$, it follows that $\max(|W_0|_{NI},|W_t|_{NI})\geq I(B)|H|_{mv}$.
%If $|H|_{mv}=1$, then for $j\in\{1,\dots,t\}$ such that $W_{j-1}\to W_j$ is the transition corresponding to the move rule, $\max(|W_{j-1}'|_a,|W_j'|_a)\geq1$.  In particular, 

Hence, $\pazocal{C}$ is a $2$-narrow (indeed a $1$-narrow) computation.

\textbf{2.} Suppose $H\equiv H_1H_2$ where the nonempty subcomputation $\pazocal{C}_i$ with history $H_i$ corresponds to a computation of a single one-letter machine.

Let $\bar{\pazocal{C}}_1:W_s\to\dots\to W_0$ be the `inverse' subcomputation with history $H_1^{-1}$.  The restriction of $\bar{\pazocal{C}}_1$ to $B'$ then corresponds to a reduced computation $\pazocal{D}_1:V_s\to\dots\to V_0$ of a one-letter machine $\textbf{One}_{X,y_1}$ with base $Q_1Q_2Q_3$ and history $H_1^{-1}$.  By construction, the state letters of $V_s$ must either be start or end letters.

Similarly, the restriction of $\pazocal{C}_2$ to $B'$ corresponds to a reduced computation $\pazocal{D}_2:U_s\to\dots\to U_t$ of a one-letter machine $\textbf{One}_{X,y_2}$ with base $Q_1Q_2Q_3$ and history $H_2$.  Again, the state letters of $U_s$ must either be start or end letters.

Suppose $\pazocal{D}_1$ satisfies the hypotheses of \Cref{one-letter standard 2}.  Then $|V_i|_a\leq|V_0|_a$ for all $0\leq i\leq s$ and $|V_0|_a\geq|V_s|_a+|H_1|_{mv}$.  Applying \Cref{one-letter standard}(4) to $\pazocal{D}_2$ further implies $|U_i|_a\leq\max(|U_s|_a,|U_t|_a)$ for all $s\leq i\leq t$, so that $|W_i'|_a\leq\max(|W_0'|_a,|W_t'|_a)$ for all $i$.  \Cref{one-letter standard mv} further implies $\max(|U_s|_a,|U_t|_a)\geq|H_2|_{mv}$, so that:
\begin{align*}
|W_0'|_a+|W_t'|_a&=|V_0|_a+|U_t|_a\ge|V_s|_a+|U_t|_a+|H_1|_{mv} \\
&=|U_s|_a+|U_t|_a+|H_1|_{mv}\geq|H_2|_{mv}+|H_1|_{mv} \\
&=|H|_{mv}
\end{align*}
Otherwise, assume $\pazocal{D}_1$ does not satisfy the hypotheses of \Cref{one-letter standard 2}.  Let $w$ be the reduced word over $X\cup X^{-1}$ whose copy is written in the $Q_2Q_3$-sector of $V_s$.  If the state letters comprising $V_s$ are start letters, then $y_1^{-1}$ is a prefix of $w$; if they are end letters, then $y_1$ is a prefix of $w$.

Note that if $y_2=y_1$, then the state letters of $U_s$ must be the opposite of those of $V_s$.  But $U_s$ also has the natural copy of $w$ written in its $Q_2Q_3$-sector, so that $\pazocal{D}_2$ satisfies the hypotheses of \Cref{one-letter standard 2}.  The analogous argument to that above then implies $|W_i'|_a\leq\max(|W_0'|_a,|W_t'|_a)$ for all $i$ and $|W_0'|_{a}+|W_t'|_{a}\geq|H|_{mv}$.

The same argument then applies to the restriction $\pazocal{C}''$ of $\pazocal{C}$ to any subword of $B$ of the form $(Q_1Q_2Q_3)^{-1}$.  Hence, $|W_i|_{NI}\leq2\max(|W_0|_{NI},|W_t|_{NI})$ for all $i$ and $|W_0|_{NI}+|W_t|_{NI}\geq I(B)|H|_{mv}$, so that $\pazocal{C}$ is a $2$-narrow computation.

As the tape word of any other sector corresponds to an input tape alphabet and is only altered by move computations, it then follows that 
\begin{align*}
|W_i|_I&\leq\max(|W_0|_I,|W_t|_I)+I(B)|H|_{mv}\leq(|W_0|_I+|W_t|_I)+(|W_0|_{NI}+|W_t|_{NI}) \\
&=|W_0|_a+|W_t|_a\leq2\max(|W_0|_a,|W_t|_a)
\end{align*}
for all $i$.  Hence, $|W_i|_a=|W_i|_I+|W_i|_{NI}\leq4\max(|W_0|_a,|W_t|_a)$ for all $i$.

\textbf{3.} By \Cref{move_1 end factorization}, it suffices to assume there exists a factorization $H\equiv H_1H_{y_1,w_1}^{\eps_1}\dots H_{y_k,w_k}^{\eps_k}H_2$ such that:

\begin{itemize}

\item $H_1$ and $H_2$ are histories of (perhaps empty) computations of a single one-letter machine.

\item $H_{y_i,w_i}$ is the natural copy of the history of the unique reduced computation of $\textbf{One}_{X,y_i}$ corresponding to $w_i\in F(X)$ given in \Cref{one-letter standard}(1).

\item $\eps_i\in\{\pm1\}$ and $k\geq1$.

\end{itemize}

Setting $w_{k+1}=y_k^{\eps_k}w_k$, it follows that $w_{i+1}=y_i^{\eps_i}w_i$ for each $i$.

Let $j\in\{1,\dots,k+1\}$ be the index for which $\|w_j\|$ is minimal.  Then, $\|w_{i+1}\|=\|w_i\|+1$ for all $i\geq j$ and $\|w_{i-1}\|=\|w_i\|+1$ for all $i\leq j$.  In particular, $\|w_0\|+\|w_{k+1}\|\geq k$ and $\|w_i\|\leq\max(\|w_0\|,\|w_{k+1}\|)$ for all $i$.

Let $\pazocal{C}_i:W_{r(i)}\to\dots\to W_{r(i+1)}$ be the subcomputation of $\pazocal{C}$ with history $H_{y_i,w_i}^{\eps_i}$.  The restriction of $\pazocal{C}_i$ to $B'$ (or to any subword of $B$ of the form $(Q_1Q_2Q_3)^{\pm1}$) satisfies the hypotheses of \Cref{one-letter standard}, so that (4) implies the $a$-length of any admissible word in this computation is at most $\max(\|w_i\|,\|w_{i+1}\|)$.  In particular, $|W_i|_{NI}\leq\max(|W_{r(1)}|_{NI},|W_{r(k+1)}|_{NI})$ for all $r(1)\leq i\leq r(k+1)$.

Further, as \Cref{one-letter standard}(3) implies $|H_{y_i,w_i}|_{mv}=1$ and \Cref{move_1 end factorization} implies $y_1^{\eps_1}\dots y_k^{\eps_k}$ is a reduced word, it follows that $|W_i|_I\leq\max(|W_{r(1)}|_I,|W_{r(k+1)}|_I)+I(B)(k/2)$ for all $r(1)\leq i\leq r(k+1)$.  

But $I(B)(k/2)\leq I(B)\max(\|w_0\|,\|w_{k+1}\|)\leq\max(|W_{r(1)}|_{NI},|W_{r(k+1)}|_{NI})$, so that:
\begin{align*}
|W_i|_a&=|W_i|_I+|W_i|_{NI}\leq2\max(|W_{r(1)}|_{NI},|W_{r(k+1)}|_{NI})+\max(|W_{r(1)}|_I,|W_{r(k+1)}|_I) \\
&\leq4\max(|W_{r(1)}|_a,|W_{r(k+1)}|_a)
\end{align*}
for all $r(1)\leq i\leq r(k+1)$.

%If $\|H_1\|=\|H_2\|=0$, then this implies $|W_i|_a\leq2\max(|W_0|_a,|W_t|_a)$ for all $i$ and $2\max(|W_0|_{NI},|W_t|_{NI})\geq I(B)k=I(B)|H|_{mv}$.  So, it suffices to assume 

For $0\leq i\leq r(1)$, let $H_{1,i}$ be the prefix of $H_1$ with $\|H_{1,i}\|=i$.  

If the computation of the one-letter machine corresponding to the inverse subcomputation of $\pazocal{C}'$ with history $H_1^{-1}$ satisfies the hypotheses of \Cref{one-letter standard 2}, then $|W_i|_{NI}\leq|W_0|_{NI}-I(B)|H_{1,i}|_{mv}$ for all $i\leq r(1)$.  But also $|W_i|_I\leq|W_0|_I+I(B)|H_{1,i}|_{mv}$, so that $|W_i|_a\leq|W_0|_a$ for all $i\leq r(1)$ and $|W_0|_{NI}\geq|W_{r(1)}|_{NI}+I(B)|H_1|_{mv}$.  Note that these inequalities are satisfied trivially if $\|H_1\|=0$.

Similarly, if either $\|H_2\|=0$ or the computation corresponding to the subcomputation of $\pazocal{C}'$ with history $H_2$ satisfies the hypotheses of \Cref{one-letter standard 2}, then $|W_i|_a\leq |W_t|_a$ for all $i\geq r(k+1)$ and $|W_t|_{NI}\geq|W_{r(k+1)}|_{NI}+I(B)|H_2|_{mv}$.

Hence, if both computations are either empty or satisfy the hypotheses of \Cref{one-letter standard 2}, then $|W_i|_{NI}\leq\max(|W_0|_{NI},|W_t|_{NI})$ and $|W_i|_a\leq4\max(|W_0|_a,|W_t|_a)$ for all $i$, while also:
\begin{align*}
2\max(|W_0|_{NI},|W_t|_{NI})&\geq|W_0|_{NI}+|W_t|_{NI} \\
&\geq|W_{r(1)}|_{NI}+|W_{r(k+1)}|_{NI}+I(B)(|H_1|_{mv}+|H_2|_{mv}) \\
%&\geq I(B)(\|w_0\|+\|w_{k+1}\|+|H_1|_{mv}+|H_2|_{mv}) \\
&\geq I(B)(k+|H_1|_{mv}+|H_2|_{mv}) \\
&=I(B)|H|_{mv}
\end{align*}

Hence, perhaps passing to the inverse computation, it suffices to assume that the computation $\pazocal{D}:V_0\to\dots\to V_s$ of the one-letter machine corresponding to the inverse subcomputation of $\pazocal{C}'$ with history $H_1^{-1}$ is nonempty and does not satisfy the hypotheses of \Cref{one-letter standard 2}.  In particular, letting $y\in X$ such that $\pazocal{D}$ is a reduced computation of $\textbf{One}_{X,y}$, then $w_1$ has prefix $y^\delta$ where:

\begin{itemize}

\item $\delta=-1$ if the state letters comprising $V_0$ are start letters, or

\item $\delta=1$ if the state letters comprising $V_0$ are end letters.

\end{itemize}

Suppose $y_1=y$.  Then, since only one rule of $\textbf{One}_{X,y}$ applies to an admissible word consisting of start (respectively end) letters, that $H$ is reduced implies $\eps_1=\delta$.  

In particular, in any case $w_1$ does not have prefix $y_1^{-\eps_1}$.

As a result, $\|w_{i+1}\|=\|w_i\|+1$ for all $i$.  This then implies that $|W_{r(k+1)}|_{NI}=|W_{r(1)}|_{NI}+I(B)k$, so that $|W_{r(k+1)}|_a\geq|W_{r(1)}|_a$.  Hence, $|W_i|_a\leq|W_{r(k+1)}|_a$ for all $r(1)\leq i\leq r(k+1)$.

Moreover, as $H$ is reduced, the computation of the one-letter machine corresponding to the subcomputation of $\pazocal{C}'$ with history $H_2$ is either empty or satisfies the hypotheses of \Cref{one-letter standard 2}.  As above, this implies $|W_i|_a\leq|W_t|_a$ for all $r(k+1)\leq i\leq t$ and $|W_t|_{NI}\geq|W_{r(k+1)}|_{NI}+I(B)|H_2|_{mv}$.  In particular, $|W_i|_a\leq|W_t|_a$ for all $r(1)\leq i\leq t$ and $|W_t|_{NI}\geq I(B)(k+|H_2|_{mv})$.

Finally, as $H_1$ is the history of a single one-letter machine with such a base, the same argument as that in Case 1 implies $|W_i|_{NI}\leq2\max(|W_0|_{NI},|W_{r(1)}|_{NI})$ and $|W_i|_a\leq2\max(|W_0|_a,|W_{r(1)}|_a)$ for all $i\leq r(1)$.  

Thus, $|W_i|_{NI}\leq2\max(|W_0|_{NI},|W_t|_{NI})$ and $|W_i|_a\leq\max(|W_0|_a,|W_t|_a)$ for all $i$, while the inequalities $k\geq1\geq|H_1|_{mv}$ imply:
$$2\max(|W_0|_{NI},|W_t|_{NI})\geq2|W_t|_{NI}\geq2I(B)(k+|H_2|_{mv})\geq I(B)|H|_{mv}$$

\end{proof}

\medskip
	
%%%%%%%%%%%%%%%%%%%%%%%%%%%%%%%%%%%%%%%%%%%%%%%%%%

\section{Parameters} \label{sec-parameters} \

The arguments spanning the rest of this paper are reliant on the \textit{highest parameter principle}, the dual to the lowest parameter principle described in \cite{O}. In particular, we introduce the relation $<<$ on parameters defined as follows:

If $\a_1,\a_2,\dots,\a_n$ are parameters with $\a_1<<\a_2<<\dots<<\a_n$, then for all $2\leq i\leq n$, it is understood that $\a_1,\dots,\a_{i-1}$ are assigned prior to the assignment of $\a_i$ and that the assignment of $\a_i$ is dependent on the assignment of its predecessors. The resulting inequalities are then understood as `$\a_i\geq$(any expression used henceforth involving $\a_1,\dots,\a_{i-1}$)'.

Specifically, the assignment of parameters we use here is:
\begin{align*}
M<&<C<<\lambda^{-1}<<k<<N<<c_0<<c_1<<c_2<<c_3<<c_4<<c_5<<L_0<< \\
&L<<K_0<<K<<J<<\delta^{-1}<<C_1<<C_2<<C_3<<C_4<<N_1<<N_2<<N_3<<N_4
\end{align*}

\medskip
	
%%%%%%%%%%%%%%%%%%%%%%%%%%%%%%%%%%%%%%%%%%%%%%%%%%

\section{Auxiliary Machines} \label{sec-auxiliary}

In this section, several $S$-machines are constructed, carefully tailored so that the finitely presented groups associated to the `main machine' $\textbf{M}$ whose makeup is sufficient for the proof of \Cref{main-theorem}. 

Per the hypotheses of \Cref{main-theorem}, this construction is outlined for a fixed finitely generated group $G$ with finite generating set $X$ with a pair of sets of defining relations $\pazocal{R}_1$ and $\pazocal{R}_2$.  Define $\Omega$ to be the set of non-trivial cyclically reduced words over $X\cup X^{-1}$ which represent the identity in $G$ and fix nondecreasing functions $f_1$ and $f_2$ on the nonnegative reals with $f_1(n)\geq n$ and $f_2(n)\geq1$ for all $n\in\N$.

\subsection{The machine $\textbf{Move}$} \

The first machines of the construction are those promised by the hypotheses of \Cref{main-theorem}.

The machine $\textbf{Move}$ is the $(\pazocal{R}_1,f_1,C)$-Move machine referenced in part (1) of the hypotheses.  While it is also a critical assumption that $\textbf{Move}$ satisfies the $(n^2f_2(n),G)$-move condition, for the outline of the construction of the main machine this condition can be overlooked, so that $\textbf{Move}$ is simply treated as a general Move machine.

So, letting $Q_0\dots Q_n$ be the standard base of $\textbf{Move}$, the tape alphabet of the input $Q_0Q_1$-sector is $X$ and the tape alphabet of the $Q_{n-1}Q_n$-sector is $Y\cup Z$, where $Y$ is an alphabet in fixed correspondence with $X$.  As a Move machine, $\textbf{Move}$ then satisfies conditions (Mv1)-(Mv8).  

In particular, conditions (Mv4), (Mv7), and (Mv8) imply the following statements:

\begin{lemma} \label{Move move}

For each $w\in\pazocal{R}_1$, there exists a reduced computation $\pazocal{C}_w:W_0\to\dots\to W_t$ of $\textbf{Move}$ in the standard base such that:

\begin{itemize}

\item $W_0\equiv q_0^s ~ w ~ q_1^s \dots q_{n-1}^s q_n^s$, where $q_i^s$ is the start letter of $Q_i$

\item $W_t\equiv q_0^f q_1^f \dots q_{n-1}^f ~ v ~ q_n^f$, where $q_i^f$ is the end letter of $Q_i$

\item $\bar{v}\equiv w$ and $|v|_Z\leq C\|w\|$

\item $t\leq c_0f_1(c_0\|w\|)+c_0$

\end{itemize}

\end{lemma}

\begin{lemma} \label{Move faulty}

For every reduced computation $\pazocal{D}:V_0\to\dots\to V_s$ of $\textbf{Move}$ with full base, $|V_i|_a\leq c_0\max(|V_0|_a,|V_s|_a)$ and $|V_i|_{NI}\leq c_0\max(|V_0|_{NI},|V_s|_{NI})$ for all $i$.

\end{lemma}

\medskip

%%%%%%%%%%%%%%%%%%%%%%%%%%%%%%%%%%%%%%%%%%%%%%%%%%

\subsection{Clean machines} \ 

The next step in our construction is the introduction of a simple class of $S$-machines called \textit{Clean machines}.  The function of these machines is to essentially `clean' out the `noise' of the letters of $Z$ from the words over $Y\cup Z$.

Given the finite alphabets $Y$ and $Z$, the Clean machine $\textbf{Clean}(Y,Z)$ has standard base $PQR$ with each part consisting of a single letter, {\frenchspacing i.e. $P=\{p\}$, $Q=\{q\}$, and $R=\{r\}$}.  The tape alphabet of the $PQ$-sector is the alphabet $Y$ and that of the $QR$-sector is a copy $Y'$ of $Y$.

The positive rules of $\textbf{Clean}(Y,Z)$ are of the following forms (where the domain of an unlocked sector is implicitly taken to be the full tape alphabet):

\begin{itemize}

\item For all $y\in Y$, the rule $\tau_1(y)=[p\to p, \ q\to y^{-1} q y', \ r \to r]$, where $y'$ is the letter of $Y'$ corresponding to $y$.

\item For all $z\in Z$, the rule $\tau_1(z)=[p\to p, \ q\to z^{-1} q , \ r \to r]$.

%\item The rule $\zeta=[p_s\xrightarrow{\ell}p_f, \ q_s\to q_f, \ r_s\to r_f]$.

%\item For all $y\in Y$, the rule $\tau_2(y)=[q_2\to q_2, \ p_2\to y^{-1} p_2 y', \ r_2 \to r_2]$, where $y'$ is the letter of $Y'$ corresponding to $y$.

\end{itemize}

Let $\pazocal{C}:W_0\to\dots\to W_t$ be a reduced computation of $\textbf{Clean}(Y,Z)$ in the standard base such that:

\begin{itemize}

\item $W_0\equiv p~v~qr$ for some $v\in F(Y\cup Z)$.

\item $W_t\equiv pq~u~r$ for some $u\in F(Y')$
\end{itemize}

Then $\pazocal{C}$ is called a \textit{cleaning of $v$} and $u$ is called the \textit{clean word of $\pazocal{C}$}.

%Given a reduced computation of $\textbf{Clean}(Y,Z)$, note that by construction, a letter of the form $\zeta^{\pm1}$ can appear only as the first or last letter of the history.  The maximal subcomputation whose history contains no such letter is called the \textit{functional subcomputation}.

Note that any reduced computation with base $QP$ satisfies the hypotheses of \Cref{multiply one letter}.  Accordingly, the next statements follow immediately:% and is proved in just the same way \Cref{primitive computations}(3):

%Note the resemblance between this machine and primitive machine $\textbf{LR}(Y)$.  Indeed, the only substantial difference between the machines is the existence of the rules of the form $\tau_1(z)$, which function to essentially erase the `noise' letters of $Z$ when running the state letter $p_1$ to the left.

\begin{lemma} \label{clean computations}

For any $v\in F(Y\cup Z)$, there exists a unique cleaning $\pazocal{C}(v)$ of $v$.  In this case, the clean word of $\pazocal{C}(v)$ is freely equal to the copy of $\bar{v}$ over $Y'$ and the length of $\pazocal{C}(v)$ is $\|v\|$.

%Let $\pazocal{C}:W_0\to\dots\to W_t$ be a reduced computation of $\textbf{Clean}(Y,Z)$ in the standard base.  Suppose $W_0\equiv p_s ~ u ~ q_sr_s$ and $W_t\equiv p_fq_f ~ v ~ r_f$ for some $u\in F(Y\cup Z)$ and $v\in F(Y')$.  Then:

%\begin{enumerate}
%
%%\item $v\in F(Y)$ 
%
%\item The copy of $v$ over $Y\cup Y^{-1}$ is freely equal to $\bar{u}$
%
%\item $t=\|u\|+1$
%
%%\item Suppose $W_0\equiv q_j ~ v_1 ~ p_jr_j$ and $W_t\equiv q_j ~ v_2 ~ p_jr_j$ for some $j\in\{1,2\}$ and $v_1,v_2\in F(Y\cup Z)$.  Then $t=0$.
%
%\end{enumerate}

\end{lemma}

%A computation of $\textbf{Clean}(Y,Z)$ of the form described in \Cref{clean computations} is called a \textit{cleaning of $u$}.  

%Note that in the setting of \Cref{clean computations}, $\pazocal{C}$ is uniquely determined by the word $u$.  Such a computation is then called a \textit{cleaning of $u$}.

%\begin{lemma} \label{clean computations exist}
%
%For any $u\in F(Y\cup Z)$, there exists a (unique) cleaning of $u$, denoted $\pazocal{C}(u)$.
%
%\end{lemma}

%Similarly, the next statement is the analogue of \Cref{primitive computations}(4):

\begin{lemma} \label{clean end}

Let $\pazocal{C}:W_0\to\dots\to W_t$ be a reduced computation of $\textbf{Clean}(Y,Z)$ in the standard base.  Suppose $W_0\equiv p ~ u ~ qr$ and $W_t\equiv p ~ v ~ qr$ for some $u,v\in F(Y\cup Z)$.  Then $t=0$.

\end{lemma}

\medskip

%%%%%%%%%%%%%%%%%%%%%%%%%%%%%%%%%%%%%%%%%%%%%%%%%%

\subsection{Accept machine} \

For our next machine, we combine the Clean machine constructed in the previous section with the $S$-machine whose existence is given in the hypotheses of \Cref{main-theorem} to form the \textit{Accept machine} necessary for our construction.

Let $\textbf{S}$ be the $S$ machine recognizing $\pazocal{R}_1$ with $\TM_{\textbf{S}}\preceq f_1$ given by hypothesis (2) of \Cref{main-theorem}.  In particular, we assume that for every $w\in\pazocal{R}_1$, there exists a computation of $\textbf{S}$ accepting $w$ of length at most $c_0f_1(c_0\|w\|)+c_0$.

We also assume that for any $w\in\pazocal{R}_1$, every accepting computation of $\textbf{S}$ has length at least $k$.  This can be achieved by adding $k-1$ new state letters to each part and $k$ new positive rules that progressively switch the original end state letters through these new letters, with the final such letters taken as the end letters; as such, an accepting computation of an input configuration in this new machine corresponds to the concatenation of an accepting computation of the original machine followed by applications of each of these $k$ new rules.  Note that the parameter choice $c_0>>k$ allows us to achieve both this lower bound and the preceding upper bound simultaneously.

Denote the standard base of $\textbf{S}$ by $Q_0Q_1\dots Q_m$, with $Q_i=\{q_i(1),\dots,q_i(\ell)\}$ for all $i$, where $q_i(1)$ and $q_i(\ell)$ are the start and end state letters, respectively, of this part.  %Note that we may assume without loss of generality that $\ell>1$, {\frenchspacing i.e. that} the start and end states of the machine are distinct.  Assuming the input sector to be the $Q_0Q_1$-sector, note that by hypothesis the tape alphabet of this sector is $X$.

The standard base of the machine $\textbf{Acc}$ is then $PQ_0Q_1Q_2\dots Q_m$, with the $PQ_0$-sector taken as the input sector.  Each part of the state letters consists of $\ell+1$ distinct letters, identified with the letter comprising the corresponding part of the Clean machine $\textbf{Clean}(Y,Z)$ and the $\ell$ letters comprising that of $\textbf{S}$.  The letters corresponding to $\textbf{Clean}(Y,Z)$ are taken to be the start letters of $\textbf{Acc}$, while those corresponding to the end letters of $\textbf{S}$ are taken as the end letters.

The tape alphabets of the machine are given as follows:

\begin{itemize}

\item The tape alphabet of the $PQ_0$-sector is $Y\cup Z$.

\item The tape alphabet of the $Q_0Q_1$-sector is the copy $Y'$ of $Y$.

\item For $2\leq i\leq m$, the tape alphabet of the $Q_{i-1}Q_i$-sector is the same as the tape alphabet of $\textbf{S}$ in the sector of the same name.

%\item For $1\leq j\leq n$, the tape alphabet of the $Q_{m+j-1}Q_{m+j}$-sector is empty.

\end{itemize}

Note that, per the makeup of the hardware, every rule of $\textbf{Acc}$ must lock the $Q_{m+j-1}Q_{m+j}$-sector for all $1\leq j\leq n$.  With that in mind, the software of the machine can be understood by the following description:

\begin{itemize}

\item For every positive rule of $\textbf{Clean}(Y,Z)$, there is a corresponding positive rule of $\textbf{Acc}$ which operates in the analogous way on the subword $PQ_0Q_1$ of the standard base and locks every other sector.

\item The positive rule $\sigma$ switches the state letters from those of $\textbf{Clean}(Y,Z)$ to the start letters of $\textbf{S}$, locking every sector besides the $Q_0Q_1$-sector, in which its domain is $Y'$.

\item For every positive rule of $\textbf{S}$, there is a corresponding positive rule of $\textbf{Acc}$ which, identifying $Y'$ with $X$, operates in the analogous way on the subword $Q_0Q_1Q_2\dots Q_m$ of the standard base and locks the $PQ_0$-sector.

\end{itemize}

Note that, per this construction, the machine $\textbf{Acc}$ can be interpreted as the \textit{composition} of two \textit{submachines} identified with $\textbf{Clean}(Y,Z)$ and $\textbf{S}$ across through the \textit{transition rule} $\sigma$.

Let $\tilde{\pazocal{R}}_1$ be the set of reduced words $v\in F(Y\cup Z)$ for which $\bar{v}$ is freely equal to some element, denoted $r_1(v)$, of $\pazocal{R}_1$. 

\begin{lemma} \label{Acc language}

The language of words accepted by $\textbf{Acc}$ is $\tilde{\pazocal{R}}_1$.  Moreover, for every $v\in\tilde{\pazocal{R}}_1$:

\begin{enumerate}

\item There exists an accepting computation of length at most $\|v\|+c_0f_1(c_0\|r_1(v)\|)+c_0+1$.

\item For any acceping computation, there exists a subcomputation which can be identified with a reduced computation of $\textbf{S}$ that accepts an input configuration.

\end{enumerate}

\end{lemma}

\begin{proof}

Let $W$ be the input configuration of $\textbf{Acc}$ corresponding to the input $v\in\tilde{\pazocal{R}}_1$.  

By \Cref{clean computations}, there exists a reduced computation $W\equiv W_0\to\dots\to W_s$ corresponding to the cleaning $\pazocal{C}(v)$.  Per the statement, $W_s$ has the copy of $r_1(v)$ over $Y'$ written in its $Q_0Q_1$-sector with all other sectors empty and $s=\|v\|$.

In particular, $W_s$ is $\sigma$-admissible, with $W_{s+1}\equiv W_s\cdot\sigma$ the analogue of the input configuration of $\textbf{S}$ with input $r_1(v)$.

As $r_1(v)\in\pazocal{R}_1$, there exists a computation $W_{s+1}\to\dots\to W_t$ corresponding to the computation of $\textbf{S}$ accepting $r_1(v)$, so that $t-s-1\leq c_0f_1(c_0\|r_1(v)\|)+c_0$.

Concatenating these then yields a computation $\pazocal{C}:W\equiv W_0\to\dots\to W_t$ accepting $v$ such that $t\leq\|v\|+c_0f_1(c_0\|r_1(v)\|)+c_0+1$.  %The bound in the statement then follows by taking $c_0\geq1$.

Conversely, let $\pazocal{D}$ be a reduced computation of $\textbf{Acc}$ accepting the input configuration $W'$ corresponding to the input $u\in F(Y\cup Z)$.  

Letting $H$ be the history of $\pazocal{D}$, by construction there exists a factorization $H\equiv H_1\sigma H'$ such that the subcomputation $\pazocal{D}_1$ of $\pazocal{D}$ with history $H_1$ operates entirely as $\textbf{Clean}(Y,Z)$.  By \Cref{clean computations}, $W''\equiv W'\cdot H_1\sigma$ is the analogue of the input configuration of $\textbf{S}$ corresponding to the reduced word $w$ over $X\cup X^{-1}$ that is freely equal to $\bar{u}$.

Suppose the subcomputation $\pazocal{D}'$ of $\pazocal{D}$ with history $H'$ does not 
operate entirely as $\textbf{S}$.  This implies there exists a factorization $H'\equiv H_2\sigma^{-1}H_3\sigma H''$, where the subcomputation $\pazocal{D}_2$ (respectively $\pazocal{D}_3$) of $\pazocal{D}$ with history $H_2$ (respectively $H_3$) operates entirely as $\textbf{S}$ (respectively as $\textbf{Clean}(Y,Z)$).  As the initial and terminal configurations of $\pazocal{D}_3$ must both be $\sigma$-admissible, $\pazocal{D}_3$ corresponds to a computation of the Clean machine which satisfies the hypotheses of \Cref{clean end}.  But then $\|H_3\|=0$, contradicting the hypothesis that $H$ is reduced.

Hence, $\pazocal{D}'$ operates entirely as $\textbf{S}$, so that $w$ must be an accepted input of $\textbf{S}$.  But then $w\in\pazocal{R}_1$, so that $u\in\tilde{\pazocal{R}}_1$.

\end{proof}

\medskip

%%%%%%%%%%%%%%%%%%%%%%%%%%%%%%%%%%%%%%%%%%%%%%%%%%

\subsection{The machine $\textbf{M}_1$} \

The next machine, $\textbf{M}_1$, combines the Move machine $\textbf{Move}$ with the Accept machine $\textbf{Acc}$ to obtain an $S$-machine that recognizes the language $\pazocal{R}_1$.  As such, $\textbf{M}_1$ is viewed as the composition of three \textit{submachines}, two of which correspond to the submachines of $\textbf{Acc}$.

Letting $s=n+m$, the standard base of $\textbf{M}_1$ is $Q_0\dots Q_{n-1}Q_0'Q_1'\dots Q_s'Q_0''\dots\ Q_s''$, with the $Q_0Q_1$-sector taken as the input sector.  Per the construction, each part of the state letters is the disjoint union of three sets, each of which corresponds to the states of the submachines.

The tape alphabets are given as follows:

\begin{itemize}

\item For $1\leq i\leq n-1$, the tape alphabet of the $Q_{i-1}Q_i$-sector is identified with that of $\textbf{Move}$ in the same named sector.

\item For $1\leq i\leq m$, the tape alphabets of the $Q_{i-1}'Q_i'$-sectors are identified with the corresponding sector of $\textbf{Acc}$.  

\item The tape alphabet of the $Q_{n-1}Q_0'$-sector is taken to be $Y\cup Z$, {\frenchspacing i.e. the} same alphabet as the $Q_{n-1}Q_n$-sector of $\textbf{Move}$ and the $PQ_0$-sector of $\textbf{Acc}$.

\item Every other sector (including every $Q_{i-1}''Q_i''$-sector) has empty tape alphabet.

\end{itemize}

The submachines of $\textbf{M}_1$, denoted $\textbf{M}_1(1)$, $\textbf{M}_1(2)$, and $\textbf{M}_1(3)$, are understood as follows:

\begin{itemize}

\item $\textbf{M}_1(1)$ operates as $\textbf{Move}$ on the subword $Q_0\dots Q_{n-1}Q_0'$ of the standard base, locking every other sector.

\item $\textbf{M}_1(2)$ operates as $\textbf{Clean}(Y,Z)$ on the subword $Q_{n-1}Q_0'Q_1'$ of the standard base, locking every other sector.

\item $\textbf{M}_1(3)$ operates as $\textbf{S}$ on the subword $Q_0'\dots Q_m'$ of the standard base, locking every other sector.

\end{itemize}

These submachines are concatenated by transition rules $\sigma(i,i+1)$ given as follows:

\begin{itemize}

\item $\sigma(12)$ switches the state letters from the end letters of $\textbf{M}_1(1)$ to the start (the only) letters of $\textbf{M}_1(2)$, locking all sectors except for the $PQ_0'$-sector.

\item $\sigma(23)$ switches the state letters from the end (the only) letters of $\textbf{M}_1(2)$ to the start letters of $\textbf{M}_1(3)$, locking all sectors except for the $Q_0'Q_1'$-sector.

\end{itemize}

%As with the definition of $\textbf{Acc}$ in the previous section, the structure of $\textbf{M}_1$ can be fully understood by interpreting the machine as a concatenation of two submachines.  Specifically, the submachine $\textbf{M}_1(1)$ operates on the subword $Q_0\dots Q_{n-1}Q_0'$ of the standard base as $\textbf{Move}$, while the submachine $\textbf{M}_1(2)$ operates on the subword $Q_{n-1}Q_0'Q_1'\dots Q_s'$ as $\textbf{Acc}$, with each rule of both machines locking any sector on which it is not explicitly operating.

%These submachines are concatenated by the positive rule $\sigma$, which locks every sector other than the $Q_{n-1}Q_0'$-sector and switches the state letters from the end letters of $\textbf{M}_1(1)$ to the start letters of $\textbf{M}_1(2)$.  As this rule plays an important role in our construction going forward, we distinguish $\sigma$ by naming it a \textit{transition rule} (rather than a connecting rule) of the machine.

Note that the submachines $\textbf{M}_1(2)$ and $\textbf{M}_1(3)$ along with the transition rule $\sigma(23)$ correspond to a submachine $\textbf{M}_1(23)$ which operates as $\textbf{Acc}$ on the subword $Q_{n-1}Q_0'\dots Q_m'$ of the standard base.  As such, \Cref{Acc language} may be applied to arguments involving this machine.

Note that any configuration $W$ of $\textbf{M}_1$ which is $\sigma(12)$-admissible corresponds to an end configuration of $\textbf{Move}$ in which every sector is empty except for the $Q_{n-1}Q_0'$-sector, {\frenchspacing i.e. the configuration} $A_{\textbf{Move}}(v)$ for some $v\in F(Y\cup Z)$.  In this case, we write $W\equiv A_1(v)$.

Similarly, a configuration $W$ of $\textbf{M}_1$ which is $\sigma(12)^{-1}$-admissible corresponds to an input configuration of $\textbf{Acc}$.  If this input configuration corresponds to the input $v\in F(Y\cup Z)$, then we write $W\equiv I_2(v)$.

Note that $A_1(v)\cdot \sigma(12)\equiv I_2(v)$.

\begin{lemma} \label{M_1(1) sigma}

Suppose there exists a reduced computation $\pazocal{C}:W_0\to\dots\to W_t$ of $\textbf{M}_1(1)$ such that $W_0\equiv A_1(u)$ and $W_t\equiv A_1(v)$ for some $u,v\in F(Y\cup Z)$.  Then $u\in\tilde{\pazocal{R}}_1$ if and only if $v\in\tilde{\pazocal{R}}_1$.

\end{lemma}

\begin{proof}

By construction, $\pazocal{C}$ corresponds to a reduced computation of $\textbf{Move}$ between the end configurations $A_{\textbf{Move}}(u)$ and $A_{\textbf{Move}}(v)$.  By condition (Mv6), it follows that $\bar{u}$ and $\bar{v}$ are freely equal to one another.

\end{proof}

\begin{remark}

Note that the proof of \Cref{M_1(1) sigma} indicates the purpose of property (Mv6) in the definition of Move machines.

\end{remark}

\begin{lemma} \label{M_1(2) sigma}

Suppose there exists a reduced computation $\pazocal{C}:W_0\to\dots\to W_t$ of $\textbf{M}_1(23)$ such that $W_0\equiv I_2(u)$ and $W_t\equiv I_2(v)$ for some $u,v\in F(Y\cup Z)$.  Then $u\in\tilde{\pazocal{R}}_1$ if and only if $v\in\tilde{\pazocal{R}}_1$.

\end{lemma}

\begin{proof}

By construction $\pazocal{C}$ corresponds to a reduced computation $\pazocal{D}$ of $\textbf{Acc}$ between the input configurations corresponding to $u$ and $v$.

Suppose $v\in\tilde{\pazocal{R}}_1$.  Then \Cref{Acc language} implies there exists a reduced computation $\pazocal{E}(v)$ of $\textbf{Acc}$ that accepts the input $v$.  Concatenating $\pazocal{D}$ and $\pazocal{E}(v)$ then produces a computation of $\textbf{Acc}$ that accepts the input $u$, so that \Cref{Acc language} implies $u\in\tilde{\pazocal{R}}_1$.

If $u\in\tilde{\pazocal{R}}_1$, then the analogous argument applies, concatenating the inverse computation $\bar{\pazocal{D}}$ with a computation that accepts the input $u$.

\end{proof}

\begin{lemma} \label{M_1 language}

The language of words accepted by $\textbf{M}_1$ is $\pazocal{R}_1$.  Moreover, for every $w\in\pazocal{R}_1$:

\begin{enumerate}

\item There exists an accepting computation of length at most $c_1f_1(c_1\|w\|)+c_1$.

\item For any accepting computation, there exists a subcomputation which is a reduced computation of $\textbf{M}_1(3)$ of length at least $k$.

\end{enumerate}

\end{lemma}

\begin{proof}

Let $W$ be the input configuration of $\textbf{M}_1$ corresponding to the input $w\in\pazocal{R}_1$.  By \Cref{Move move}, there exists a reduced computation $W\equiv W_0\to\dots\to W_r$ of $\textbf{M}_1(1)$ such that $W_r\equiv A_1(v)$ for some $v\in F(Y\cup Z)$ with $\bar{v}\equiv w$, $|v|_Z\leq C\|w\|$, and $r\leq c_0f_1(c_0\|w\|)+c_0$.

Let $W_{r+1}\equiv W_r\cdot\sigma(12)\equiv I_2(v)$.  Then, since $\bar{v}\equiv w$, it follows that $v\in\tilde{\pazocal{R}}_1$, so that \Cref{Acc language} implies the existence of a reduced computation $W_{r+1}\to\dots\to W_t$ of $\textbf{M}_1(2)$ accepting $W_{r+1}$ with 
$$t-r-1\leq\|v\|+c_0f_1(c_0\|r_1(v)\|)+c_0+1$$
But since $r_1(v)=w$ and $\|v\|=|v|_Z+\|\bar{v}\|\leq(C+1)\|w\|$, concatenating these computations produces a reduced computation accepting $W$ of length $$t\leq(C+1)\|w\|+2c_0f_1(c_0\|w\|)+2c_0+2$$
The bound then follows from the hypothesis that $f_1(n)\geq n$ for all $n$ and the parameter choices $c_1>>c_0>>C$.

Conversely, let $\pazocal{D}$ be a reduced computation of $\textbf{M}_1$ accepting the input configuration $W'$ corresponding to the input $w'\in F(X)$.  Then the history $H$ of $\pazocal{D}$ must have:

\begin{itemize}

\item a prefix of the form $H_1\sigma(12)$, where $H_1$ is the history of a (maximal) subcomputation $\pazocal{D}_1$ of $\pazocal{D}$ which operates entirely as $\textbf{M}_1(1)$.

\item a suffix of the form $\sigma(23)H_{23}$, where $H_{23}$ is the history of a (maximal) subcomputation $\pazocal{D}_{23}$ of $\pazocal{D}$ which operates entirely as $\textbf{M}_1(23)$.

\end{itemize}

Note that by \Cref{Acc language}(2), there exists a subword $H_3$ of $H_{23}$ with $\|H_3\|\geq k$ which is the history of a reduced computation of $\textbf{M}_1(3)$.

Letting $W''$ be the accept configuration of $\textbf{M}_1$, it then follows that there exist $u,v\in F(Y\cup Z)$ such that $W'\cdot H_1\equiv A_1(v)$ and $W''\cdot H_2^{-1}\equiv I_2(u)$.  

By construction, $\pazocal{D}_1$ corresponds to a setup computation of $\textbf{Move}$.  Condition (Mv3) then implies that this setup computation is a move computation of $w'$, so that $\bar{v}$ is freely equal to $w'$.  %Hence, it suffices to show that $v\in\bar{\pazocal{R}}_1$.

Similarly, $\pazocal{D}_{23}$ corresponds to a reduced computation of $\textbf{Acc}$ accepting the input $u$.  As such, \Cref{Acc language} implies $u\in\tilde{\pazocal{R}}_1$.  

But applications of Lemmas \ref{M_1(1) sigma} and \ref{M_1(2) sigma} then imply $v\in\bar{\pazocal{R}}_1$ as well, so that $w'\in\pazocal{R}_1$.

\end{proof}

\begin{remark}

The use of \Cref{Move move} in the proof of \Cref{M_1 language} (and the implicit use in all subsequent statements that use \Cref{M_1 language}) indicates the purpose of (Mv4) as a defining property of Move machines.  Similarly, the purpose of (Mv3) is evident in the proof of (2).

\end{remark}

\medskip

%%%%%%%%%%%%%%%%%%%%%%%%%%%%%%%%%%%%%%%%%%%%%%%%%%

\subsection{The machine $\textbf{M}_2$} \

The next machine in our construction operates in much the same way as $\textbf{M}_1$, but adds several new `historical' sectors which keep track of the commands applied throughout a computation.  This is done in much the same way as in \cite{O18}, \cite{OS19}, and \cite{WCubic}, but with several important caveats.

First, partition the positive rules of the machine $\textbf{M}_1$ as  $\Phi_1^+\sqcup\Phi_2^+\sqcup\Phi_3^+\sqcup\{\sigma(12),\sigma(23)\}$, where $\Phi_j^+$ consists of the positive rules corresponding to the submachine $\textbf{M}_1(j)$.  Denote $\Phi^+=\sqcup_{j=1}^3\Phi_j^+$.

The standard base of $\textbf{M}_2$ is $\bigg(Q_0\dots Q_{n-1}\bigg)\bigg(Q_{0,\ell}'Q_{0,r}'\dots Q_{s,\ell}'Q_{s,r}'\bigg)\bigg(Q_{0,\ell}''Q_{0,r}''\dots Q_{s,\ell}''Q_{s,r}''\bigg)$, where:

\begin{itemize}

\item For $0\leq i\leq n-1$, the part $Q_i$ is identified with the part of the same name in $\textbf{M}_1$.

%\item $Q_{n,\ell}$ and $Q_{n,r}$ are disjoint copies of the part $Q_n$ of $\textbf{M}_1$.

\item For $0\leq i\leq s$, $Q_{i,\ell}'$ and $Q_{i,r}'$ are copies of the part $Q_i'$ of $\textbf{M}_1$.

\item For $0\leq i\leq s$, $Q_{i,\ell}''$ and $Q_{i,r}''$ are copies of the part $Q_i''$ of $\textbf{M}_1$.

\end{itemize}

Naturally, the state letters corresponding to the start (respectively the end) letters of $\textbf{M}_1$ are taken to be the start (respectively the end) letters of the machine.

Further, the tape alphabets are given as follows:

\begin{itemize}

\item For $0\leq i\leq n-1$, the tape alphabet of the $Q_{i-1}Q_i$-sector is the same as that of the corresponding sector of $\textbf{M}_1$.

\item The tape alphabet of the $Q_{n-1}Q_{0,\ell}'$-sector is identified with that of the $Q_{n-1}Q_n$-sector of $\textbf{M}_1$ ({\frenchspacing i.e. $Y\cup Z$}).

%\item The tape alphabet of the $Q_{n,r}Q_{0,\ell}'$-sector is identified with that of the $Q_nQ_0'$-sector of the $\textbf{M}_1$ ({\frenchspacing i.e. $Y'$}).

\item For $1\leq i\leq s$, the tape alphabet of the $Q_{i-1,r}'Q_{i,\ell}'$- and $Q_{i-1,r}''Q_{i,\ell}''$-sectors are identified with that of the $Q_{i-1}'Q_i'$- and $Q_{i-1}''Q_i''$-sectors, respectively, of $\textbf{M}_1$.

\item For $0\leq i\leq s$, the tape alphabet of each of the $Q_{i,\ell}'Q_{i,r}'$- and $Q_{i,\ell}''Q_{i,r}''$-sectors consist of two disjoint copies of $\Phi^+$.  These two copies are called the \textit{left and right historical alphabets} of the sector.

\end{itemize}

The $Q_0Q_1$-sector, the $Q_{i,\ell}'Q_{i,r}'$-sectors, and the $Q_{i,\ell}''Q_{i,r}''$-sectors are taken to be the input sectors of the machine.

Any sector whose tape alphabet consists of two historical alphabets ({\frenchspacing e.g. the $Q_{i,\ell}'Q_{i,r}'$- and $Q_{i,\ell}''Q_{i,r}''$-sectors}) is called a \textit{historical sector}.  Every other sector corresponding to a two-letter subword of the standard base is called a \textit{working sector}.  Note that the working sectors correspond to the sectors forming the standard base of $\textbf{M}_1$.

The positive rules of $\textbf{M}_2$ are then in correspondence with those of $\textbf{M}_1$, with each rule operating in the working sectors as its corresponding rule.  

In the historical sectors, the rule corresponding to the rule $\theta\in\Phi^+$ operates as follows:

\begin{itemize}

\item The rule multiplies the $Q_{i,\ell}'Q_{i,r}'$-sector on the left by the inverse of the copy of $\theta$ in the left historical alphabet and on the right by the copy in the right historical alphabet.

\item The rule multiplies the $Q_{i,\ell}''Q_{i,r}''$-sector on the left by the inverse of the copy of $\theta$ in the right historical alphabet and on the right by the copy in the left historical alphabet.

\end{itemize}

The transition rules $\sigma(12)$ and $\sigma(23)$, on the other hand, do not alter the tape words of the historical sectors, but instead have restricted domains in these sectors:

\begin{itemize}

\item For each such sector, the domain of $\sigma(12)$ is the copy of $\Phi_1^+$ in the right historical alphabet and the copy of $\Phi_2^+\sqcup\Phi_3^+$ in the left historical alphabet.

\item For each such sector, the domain of $\sigma(23)$ is the copy of $\Phi_1^+\sqcup\Phi_2^+$ in the right historical alphabet and the copy of $\Phi_3^+$ in the left historical alphabet.

\end{itemize}

By construction, the machine $\textbf{M}_2$ can again be viewed as the composition of three submachines, denoted $\textbf{M}_2(j)$, through the transition rules.

The history $H$ of a reduced computation of $\textbf{M}_2$ can be factored in such a way that each factor is either a transition rule or a maximal nonempty product of rules of one of the three defining submachines $\textbf{M}_2(j)$. The \textit{step history} of a reduced computation is then defined so as to capture the order of the types of these factors. To do this, denote the transition rule $\sigma(ij)$ by the pair $(ij)$ and a factor that is a product of rules in $\textbf{M}_2(i)$ simply by $(i)$.  

For example, if $H\equiv H'H''H'''$ where $H'$ is a product of rules from $\textbf{M}_2(2)$, $H''\equiv\sigma(23)$, and $H'''$ is a product of rules from $\textbf{M}_2(3)$, then the step history of a computation with history $H$ is $(2)(23)(3)$. So, the step history of a computation is some concatenation of the letters
$$\{(1),(2),(3),(12),(23),(21),(32)\}$$ 
One can omit reference to a transition rule when its existence is clear from its necessity. For example, given a reduced computation with step history $(2)(23)(3)$, one can instead write the step history as $(2)(3)$, as the rule $\sigma(23)$ must occur in order for the subcomputation of $\textbf{M}_2(3)$ to be possible.

If the step history of a computation is $(i-1,i)(i,i+1)$, it is also permitted for the step history to be written as $(i-1,i)(i)(i,i+1)$ even though the `maximal subcomputation' with step history $(i)$ is empty.

%A \textit{one-step computation} is a reduced computation with step history of one of the following forms:
%
%\begin{itemize}
%
%\item $(i)$
%
%\item $(i)(i,i\pm1)$
%
%\item $(i\pm1,i)(i)$
%
%\item $(i\pm1,i)(i)(i,i\pm1)$
%
%\end{itemize}

Certain subwords cannot appear in the step history of a reduced computation. For example, it is clear that it is impossible for the step history of a reduced computation to contain the subword $(1)(3)$. 

The next statement displays the impossibility of some less obvious potential subwords.

\begin{lemma} \label{M_2 step history}

Let $\pazocal{C}$ be a reduced computation of $\textbf{M}_2$ with base $B$.  %Fix indices $i\in\{0,\dots,s_0\}$ and $j\in\{s_0+1,\dots,s\}$.
%Suppose $B$ contains a subword $UV$ of one of the following forms:

%\begin{enumerate}
%
%\item $(Q_{i,\ell}Q_{i,r})^{\pm1}$ for any $i\in\{0,\dots,s\}$
%
%\item $Q_{i,\ell}Q_{i,\ell}^{-1}$ for $i\in\{0,\dots,s_0\}$
%
%\item $Q_{i,r}^{-1}Q_{i,r}$ for $i\in\{s_0+1,\dots,s\}$.
%
%\end{enumerate}
%
%Then the step history of $\pazocal{C}$ is neither $(21)(1)(12)$ nor $(32)(2)(23)$.

\begin{enumerate}[label=(\alph*)]

\item If $B$ contains a subword $UV$ of the form $(Q_{i,\ell}'Q_{i,r}')^{\pm1}$, of the form $(Q_{i,\ell}''Q_{i,r}'')^{\pm1}$, of the form $Q_{i,\ell}'(Q_{i,\ell}')^{-1}$, or of the form $(Q_{i,r}'')^{-1}Q_{i,r}''$, then the step history of $\pazocal{C}$ is neither $(21)(1)(12)$ nor $(32)(2)(23)$.

\item If $B$ contains a subword $UV$ of the form $(Q_{i,\ell}'Q_{i,r}')^{\pm1}$, of the form $(Q_{i,\ell}''Q_{i,r}'')^{\pm1}$, of the form $(Q_{i,r}')^{-1}Q_{i,r}'$, or of the form $Q_{i,\ell}''(Q_{i,\ell}'')^{-1}$, then the step history of $\pazocal{C}$ is neither $(12)(2)(21)$ nor $(23)(3)(32)$.

\end{enumerate}

\end{lemma}

\begin{proof}

(a) Suppose to the contrary that the step history of $\pazocal{C}$ is $(j+1,j)(j)(j,j+1)$ for $j\in\{1,2\}$.

Denoting $\pazocal{C}:W_0\to\dots\to W_t$, let $\pazocal{C}':W_1\to\dots\to W_{t-1}$ be the corresponding subcomputation.  As the history of $\pazocal{C}$ is reduced, $\pazocal{C}'$ is necessarily a nonempty computation.

Further, since $\pazocal{C}'$ is a reduced computation of $\textbf{M}_2(j)$, there exists a word $H\in F(\Phi_j^+)$ such that the history of $\pazocal{C}'$ is the natural copy of $H$ in the software of $\textbf{M}_2(j)$.  Let $H_{i,\ell}$ and $H_{i,r}$ be the copies of $H$ over the corresponding left and right historical alphabets, respectively, of the $UV$-sector.

Let $w_1$ and $w_{t-1}$ be the tape words of $W_1$ and $W_{t-1}$, respectively, in the $UV$-sector.  As $W_1$ and $W_{t-1}$ are both $\sigma(j,j+1)$-admissible, no letter comprising either $w_1$ or $w_{t-1}$ is from the copy of $\Phi_j$ in the corresponding left historical alphabet.

Suppose $UV=Q_{i,\ell}'Q_{i,r}'$ or $UV=(Q_{i,\ell}''Q_{i,r}'')^{-1}$.  Then $w_{t-1}=H_{i,\ell}^{-1}w_1H_{i,r}$.

As $w_1$ contains no letters from the left historical alphabet, no letter of $H_{i,\ell}^{-1}$ cancels in the product $H_{i,\ell}^{-1}w_1H_{i,r}$.  But then $w_{t-1}$ also contains no letters from the left historical alphabet, so that $H_{i,\ell}$ must be trivial.  However, this implies $H$ is also trivial, contradicting the hypothesis that $\pazocal{C}'$ is nonempty.

The same argument applies if $UV=(Q_{i,\ell}'Q_{i,r}')^{-1}$ or $Q_{i,\ell}''Q_{i,r}''$, noting that $w_{t-1}=H_{i,r}^{-1}w_1H_{i,\ell}$ in this case.

Finally, if $UV=Q_{i,\ell}'(Q_{i,\ell}')^{-1}$ or $(Q_{i,r}'')^{-1}Q_{i,r}''$, then $w_{t-1}=H_{i,\ell}^{-1}w_1H_{i,\ell}$.  But $w_1$ must be non-trivial, so that again no letter of $H_{i,\ell}^{\pm1}$ can cancel in the product.

(b) follows from just the same argument, noting that an admissible word that is $\sigma(j,j-1)$-admissible can have no letter from the copy of $\Phi_j^+$ in any right historical alphabet.

\end{proof}

%The next statement follows in just the same way:
%
%\begin{lemma} \label{M_2 step history 2}
%
%Let $\pazocal{C}$ be a reduced computation of $\textbf{M}_2$ with base $B$.  %Fix indices $i\in\{0,\dots,s_0\}$ and $j\in\{s_0+1,\dots,s\}$.
%Suppose $B$ contains a subword $UV$ of one of the following forms:
%
%\begin{enumerate}
%
%\item $(Q_{i,\ell}Q_{i,r})^{\pm1}$ for any $i\in\{0,\dots,s\}$
%
%\item $Q_{i,r}^{-1}Q_{i,r}$ for $i\in\{0,\dots,s_0\}$
%
%\item $Q_{i,\ell}Q_{i,\ell}^{-1}$ for $i\in\{s_0+1,\dots,s\}$.
%
%\end{enumerate}
%
%Then the step history of $\pazocal{C}$ is neither $(12)(2)(21)$ nor $(23)(3)(32)$.
%
%\end{lemma}

For $w\in F(X)$ and $H\in F(\Phi^+)$, define $I_2(w,H)$ to be the input configuration with: 

\begin{itemize}

\item $w$ written in the $Q_0Q_1$-sector

\item the copy of $H$ over the corresponding left historical alphabet written in the $Q_{i,\ell}'Q_{i,r}'$-sector

\item the copy of $H^{-1}$ over the corresponding left historical alphabet written in the $Q_{i,\ell}''Q_{i,r}''$-sector.

\end{itemize}

Similarly, given $H\in F(\Phi^+)$, let $A_2(H)$ be the end configuration with:

\begin{itemize} 

\item empty working sectors

\item the copy of $H$ over the corresponding right historical alphabet written in the $Q_{i,\ell}'Q_{i,r}'$-sector

\item the copy of $H^{-1}$ over the corresponding right historical alphabet written in the $Q_{i,\ell}''Q_{i,r}''$-sector.

\end{itemize}

%Given the history $H$ of a reduced computation of $\textbf{M}_3$, define the \textit{$\sigma$-length} of the computation $|H|_\sigma$ to be the number of occurrences of $\sigma(12)^{\pm1}$ and $\sigma(23)^{\pm1}$ in $H$. 

\begin{lemma} \label{M_2 language} 

Let $w\in F(X)$.

\begin{enumerate}

\item If $w\in\pazocal{R}_1$, then there exists $H_w\in F(\Phi^+)$ with $\|H_w\|\leq c_1f_1(c_1\|w\|)+c_1$ such that there exists a reduced computation $I_2(w,H_w)\to\dots\to A_2(H_w)$ of $\textbf{M}_2$.

\item If there exists a reduced computation $\pazocal{C}:I_2(w,H_1)\to\dots\to A_2(H_2)$ of $\textbf{M}_2$ for some $H_1,H_2\in F(\Phi^+)$, then $w\in\pazocal{R}_1$, $H_1\equiv H_2$, and $\|H_1\|\geq k$.  Moreover, in this case $\pazocal{C}$ is the unique reduced computation between these configurations, the length of $\pazocal{C}$ is $\|H_1\|+2$, and $w$ is uniquely determined by $H_1$.

\end{enumerate}

\end{lemma}

\begin{proof}

(1) By \Cref{M_1 language}, there exists a reduced computation $\pazocal{C}_1(w)$ of $\textbf{M}_1$ accepting the input $w$ whose history $H$ has length at most $c_1f_1(c_1\|w\|)+c_1$.  By construction, $H\equiv H_1\sigma(12)H_2\sigma(23)H_3$ where $H_j\in F(\Phi_j^+)$.

Let $\bar{H}_j$ be the copy of $H_j$ over the positive rules of $\textbf{M}_2(j)$ and set $H_w\equiv H_1H_2H_3$.  Noting that $H_w$ is necessarily a reduced word, it follows from the construction of the machine that $\bar{H}_1\sigma(12)\bar{H}_2\sigma(23)\bar{H}_3$ is the history of a reduced computation satisfying the statement.

(2) As $\textbf{M}_2$ operates as $\textbf{M}_1$ in the working sectors, \Cref{M_1 language} implies $w\in\pazocal{L}$.

By \Cref{M_2 step history}, the step history of $\pazocal{C}$ must be $(1)(2)(3)$, {\frenchspacing i.e. the} history of $\pazocal{C}$ can be factored $H\equiv h_1\sigma(12) h_2\sigma(23)h_3$.  Again, the word $\tilde{H}\equiv h_1h_2h_3$ must be reduced.  Moreover, \Cref{M_1 language}(2) implies $\|h_3\|\geq k$.

Note that $\tilde{H}$ uniquely determines the history of a reduced computation of $\textbf{M}_1$ that accepts the input $w$.  In particular, $\tilde{H}$ uniquely determines $w$.

Then, the operation of $\pazocal{C}$ in each $Q_{i,\ell}'Q_{i,r}'$-sector multiplies on the left by the copy of $\tilde{H}^{-1}$ over the corresponding left historical alphabet and on the right by the copy of $\tilde{H}$ over the corresponding right historical alphabet.  

Note that the tape word of $I_2(w,H_1)$ (respectively $A_2(H_2)$) in the $Q_{i,\ell}'Q_{i,r}'$-sector is the copy $H_{1,\ell}$ (respectively $H_{2,r}$) of $H_1$ (respectively $H_2$) over the corresponding left (respectively right) historical alphabet.  Letting $\tilde{H}_\ell$ and $\tilde{H}_r$ be the copies of $\tilde{H}$ over the left and right historical alphabets, respectively, of the $Q_{i,\ell}'Q_{i,r}'$-sector, then $H_{2,r}=\tilde{H}_\ell^{-1}H_{1,\ell}\tilde{H}_r$.  But since $H_{2,r}$ contains no letters from the left historical alphabet, it follows immediately that $H_{1,\ell}\equiv\tilde{H}_\ell$ and so $H_{2,r}\equiv\tilde{H}_r$.

Hence, $H_1=\tilde{H}=H_2$.

\end{proof}

A benefit of adding historical sectors is in providing a linear estimate for the lengths of reduced computations in terms of the $a$-lengths of the initial and terminal admissible words of the computation.

\begin{lemma}[Compare with Lemma 3.9 of \cite{O18}] \label{one alphabet historical words}

Given $i\in\{0,\dots,s\}$ and $j\in\{1,2,3\}$, let $W_0\to\dots\to W_t$ be a reduced computation of $\textbf{M}_2(j)$ with base $Q_{i,\ell}'Q_{i,r}'$ or $Q_{i,\ell}''Q_{i,r}''$ and history $H$.  Suppose the $a$-letters of $W_0$ which are copies of $\Phi_j^+$ are all from the corresponding left or all from the corresponding right historical alphabet.  Then $\|H\|\leq|W_t|_a$ and $|W_0|_a\leq|W_t|_a$.  Moreover, if $t\geq1$, then the tape word of $W_t$ contains letters which are copies of $\Phi_j^+$ from the right (respectively left) historical alphabet.

\end{lemma}

The next statement then follows in just the same way:

\begin{lemma} \label{one alphabet historical words unreduced}

Given $j\in\{1,2,3\}$, let $W_0\to\dots\to W_t$ be a reduced computation of $\textbf{M}_2(j)$ with base $Q_{i,\ell}'(Q_{i,\ell}')^{-1}$ or $(Q_{i,r}'')^{-1}Q_{i,r}''$ (respectively $(Q_{i,r}')^{-1}Q_{i,r}'$ or $Q_{i,\ell}''(Q_{i,\ell}'')^{-1}$) and history $H$.  Suppose the $a$-letters of $W_0$ which are copies of $\Phi_j^+$ are all from the corresponding right (respectively left) historical alphabet.  Then $|W_0|_a=|W_t|_a-2\|H\|$.  Moreover, if $t\geq1$, then the tape word of $W_t$ contains letters which are copies of $\Phi_j$ from the left (respectively right) historical alphabet.

\end{lemma}

\begin{lemma}[Compare with Lemma 3.12 of \cite{O18}] \label{M_2 bound}

Let $\pazocal{C}:W_0\to\dots\to W_t$ be a reduced computation of $\textbf{M}_2(j)$ with base $B$.  Suppose $B$ contains no letter of the form $Q_i^{\pm1}$ and that $(Q_{0,\ell}')^{-1}Q_{0,\ell}'$ is not a subword of $B$.  If $B$ has length at least 3, then $|W_i|_a\leq9(|W_0|_a+|W_t|_a)$ for all $0\leq i\leq t$.

\end{lemma}

\begin{proof}

As $(Q_{0,\ell}')^{-1}Q_{0,\ell}'$ is not a subword of $B$, $\pazocal{C}$ can be viewed as a reduced computation of the $S$-machine with standard base $Q_{0,\ell}'Q_{0,r}'\dots Q_{s,\ell}'Q_{s,r}'Q_{0,\ell}''Q_{0,r}''\dots Q_{s,\ell}''Q_{s,r}''$ which operates as $\textbf{M}_2(j)$ in these sectors.  With this view, the argument outlined in \cite{O18} implies the statement.

\end{proof}

\medskip

%%%%%%%%%%%%%%%%%%%%%%%%%%%%%%%%%%%%%%%%%%%%%%%%%%

\subsection{The machine $\overline{\textbf{M}}_2$} \label{sec-overline-M_2} \

The next machine in our construction, $\textbf{M}_3$, is the composition of the machine $\textbf{M}_2$ with primitive machines operating in the historical sectors. To aid in this construction, we first introduce (similar to the exposition of \cite{OS19} and \cite{WCubic}) the intermediate recognizing $S$-machine $\overline{\textbf{M}}_2$.

The standard base of $\overline{\textbf{M}}_2$ is obtained by inserting two new state letters in between pairs forming historical sectors.  Specifically, the standard base of $\overline{\textbf{M}}_2$ is:
$$\bigg(Q_0\dots Q_{n-1}\bigg)~\bigg(Q_{0,\ell}' P_0' R_0' Q_{0,r}' \dots Q_{s,\ell}'P_s' R_s' Q_{s,r}'\bigg) ~ \bigg(Q_{0,\ell}''P_0''R_0''Q_{0,r}''\dots Q_{s,\ell}''P_s''R_s''Q_{s,\ell}''\bigg)$$
The parts of the state letters of the same name are identified with those of $\textbf{M}_2$, while the new parts are copies of those of the parts they are between ({\frenchspacing e.g. $P_i'$ and $R_i'$ are copies of $Q_{i,\ell}'$ and $Q_{i,r}'$, respectively)}.

The tape alphabets of the $Q_{i,\ell}'P_i'$-, $P_i'R_i'$-, and $R_i'Q_{i,r}'$-sectors are each copies of the tape alphabet of the $Q_{i,\ell}Q_{i,r}$-sector of $\textbf{M}_2$ (and so themselves consist of left and right historical alphabets).  The analogue is true for the tape alphabets of the $Q_{i,\ell}''P_i''$-, $P_i''R_i''$-, and $R_i''Q_{i,r}''$-sectors.

The state letter of each part corresponding to a start (respectively end) letter of $\textbf{M}_2$ is the start (respectively end) letter of the part.  The input sectors of the machine are the $Q_0Q_1$-sector, all $P_i'R_i'$-sectors, and all $P_i''R_i''$-sectors.

The positive rules of $\overline{\textbf{M}}_2$ are in correspondence with, and so identified with, those of $\textbf{M}_2$.  For each rule $\theta$ of $\textbf{M}_2$, the corresponding rule of $\overline{\textbf{M}}_2$ locks the $Q_{i,\ell}'P_i'$-, $R_i'Q_{i,r}'$- $Q_{i,\ell}''P_i''$-, and $R_i''Q_{i,r}''$-sectors and, identifying the $P_i'R_i'$- and $P_i''R_i''$-sectors with the $Q_{i,\ell}'Q_{i,r}'$- and $Q_{i,\ell}''Q_{i,r}''$-sectors of $\textbf{M}_2$, respectively, operates on the rest of the standard base as $\theta$.  As such, $\overline{\textbf{M}}_2$ can be viewed as the composition of the three submachines $\overline{\textbf{M}}_2(1)$, $\overline{\textbf{M}}_2(2)$, and $\overline{\textbf{M}}_2(3)$ in a manner similar to the definition of $\textbf{M}_2$.

Naturally, there are analogues of the statements pertaining to $\textbf{M}_2$ in the setting of $\overline{\textbf{M}}_2$.  For example, the following statement follows immediately from \Cref{M_2 step history}:

\begin{lemma} \label{M_2 bar step history}

Let $\pazocal{C}$ be a reduced computation of $\overline{\textbf{M}}_2$ with base $B$.  

\begin{enumerate}[label=(\alph*)]

\item If $B$ contains a subword of the form $(P_i'R_i')^{\pm1}$, $(P_i''R_i'')^{\pm1}$, $P_i'(P_i')^{-1}$, or $(R_i'')^{-1}R_i''$, then the step history of $\pazocal{C}$ is neither $(21)(1)(12)$ nor $(32)(2)(23)$.

\item If $B$ contains a subword of the form $(P_i'R_i')^{\pm1}$, $(P_i''R_i'')^{\pm1}$, $(R_i')^{-1}R_i'$, or $P_i''(P_i'')^{-1}$, then the step history of $\pazocal{C}$ is neither $(12)(2)(21)$ nor $(23)(3)(32)$.

\end{enumerate}

\end{lemma}

Given $w\in F(X)$ and $H\in F(\Phi^+)$, let $\bar{I}_2(w,H)$ and $\bar{A}_2(H)$ be the input and end configurations of $\overline{\textbf{M}}_2$ analogous to $I_2(w,H)$ and $A_2(H)$, respectively, obtained by writing the corresponding historical words in the $P_i'R_i'$- and $P_i''R_i''$-sectors and keeping the other historical sectors empty.  The next statement then follows immediately from \Cref{M_2 language}:

\begin{lemma} \label{M_2 bar language} 

Let $w\in F(X)$.

\begin{enumerate}

\item If $w\in\pazocal{R}_1$, then there exists $H_w\in F(\Phi^+)$ with $\|H_w\|\leq c_1 f_1(c_1\|w\|)+c_1$ such that there exists a reduced computation $\bar{I}_2(w,H_w)\to\dots\to \bar{A}_2(H_w)$ of  $\overline{\textbf{M}}_2$.

\item If there exists a reduced computation $\pazocal{C}:\bar{I}_2(w,H_1)\to\dots\to \bar{A}_2(H_2)$ of $\overline{\textbf{M}}_2$ for some $H_1,H_2\in F(\Phi^+)$, then $w\in\pazocal{R}_1$, $H_1\equiv H_2$, and $\|H_1\|\geq k$.  Moreover, in this case $\pazocal{C}$ is the unique reduced computation between these configurations, the length of $\pazocal{C}$ is $\|H_1\|+2$, and $w$ is uniquely determined by $H_1$.

\end{enumerate}

\end{lemma}

\medskip

%%%%%%%%%%%%%%%%%%%%%%%%%%%%%%%%%%%%%%%%%%%%%%%%%%

\subsection{The machine $\textbf{M}_3$} \label{sec-M3} \

The standard base of the recognizing $S$-machine $\textbf{M}_3$ has the same form as that of $\overline{\textbf{M}}_2$, with the tape alphabets of each sector the same as that of the corresponding sector of $\overline{\textbf{M}}_2$.  However, each part of the state letters contains more letters than its counterpart in $\overline{\textbf{M}}_2$.

In particular, the parts consist of the disjoint union of a copy of the corresponding part of $\overline{\textbf{M}}_2$ and $2k$ letters corresponding to the states of the primitive machine $\textbf{LR}_k$ or $\textbf{RL}_k$.  The letters corresponding to the start letters of $\overline{\textbf{M}}_2$ function as the start letters of $\textbf{M}_3$, while those corresponding to the end letters of the primitive machine function as the end letters.

As in the definition of machines in previous sections, the software of $\textbf{M}_3$ can be understood by viewing the machine as a composition of submachines.  There are four such submachines, denoted $\textbf{M}_3(1),\dots,\textbf{M}_3(4)$; however, the first three of these correspond to the machine $\overline{\textbf{M}}_2$, as it consists of three submachines itself.  

The final submachine, $\textbf{M}_3(4)$, corresponds to a parallel composition of the primitive machines $\textbf{LR}_k(\Phi^+)$ and $\textbf{RL}_k(\Phi^+)$.  Each rule of this submachine operates in parallel as $\textbf{LR}_k$ on the subwords $P_i'R_i'Q_{i,r}'$ of the standard base and as $\textbf{RL}_k$ on the subwords of $Q_{i,\ell}''P_i''R_i''$, taking the corresponding right historical alphabets as the corresponding copies of $\Phi^+$.  Naturally, each of these rules also locks all other sectors.

To be clear, given a rule of $\textbf{M}_3(4)$ not corresponding to a connecting rule, there exists $\theta\in\Phi$ such that the rule:

\begin{itemize}

\item Multiplies each of the $P_i'R_i'$- and $Q_{i,\ell}''P_i''$-sectors on the right by the copy of $\theta^{-1}$ in the corresponding right historical alphabet.

\item Multiplies the $R_i'Q_{i,r}'$- and $P_i''R_i''$-sectors on the left by the copy of $\theta$ in the corresponding right historical alphabet.

\end{itemize}

These submachines are concatenated by the transition rules $\sigma(12)$, $\sigma(23)$, and $\sigma(34)$.  As positive rules of $\overline{\textbf{M}}_2$, the makeup of $\sigma(12)$ and $\sigma(23)$ are understood.  The rule $\sigma(34)$ switches the state letters from the end letters of $\overline{\textbf{M}}_2$ to the start letters of the primitive machine and locks all sectors other than the $P_i'R_i'$- and $P_i''R_i''$-sectors, in which its domain is the corresponding right historical alphabet.

The input configurations are assigned as in $\overline{\textbf{M}}_2$, {\frenchspacing i.e. the} $Q_0Q_1$-sector, the $P_i'R_i'$-sectors, and the $P_i''R_i''$-sectors.

The notion of step history is adapted naturally from the definitions of the previous sections.  As the first three submachines of $\textbf{M}_3$ can be identified with $\overline{\textbf{M}}_2$, an analogue of \Cref{M_2 bar step history} holds in this setting.  The next two statements add to the list of banned subwords:

\begin{lemma} \label{M_3 first step history}

Let $\pazocal{C}$ be a reduced computation of $\textbf{M}_3$ with base $B$.  If $B$ contains a subword of the form $(P_i'R_i')^{\pm1}$, $(P_i''R_i'')^{\pm1}$, $P_i'(P_i')^{-1}$, or $(R_i'')^{-1}R_i''$, then the step history of $\pazocal{C}$ is not $(43)(3)(34)$.

\end{lemma}

\begin{proof}

This follows by the exact same argument as that of \Cref{M_2 step history}(a).

\end{proof}

\begin{lemma} \label{M_3 second step history}

Let $\pazocal{C}$ be a reduced computation of $\textbf{M}_3$ with base $B$.  If $B$ contains a subword of the form $(P_i'R_i'Q_{i,r}')^{\pm1}$, $(Q_{i,\ell}''P_i''R_i'')^{\pm1}$, $(Q_{i,r}')^{-1}(R_i')^{-1}R_i'Q_{i,r}'$, or $Q_{i,\ell}''P_i''(P_i'')^{-1}(Q_{i,\ell}'')^{-1}$, then the step history of $\pazocal{C}$ is not $(34)(4)(43)$.

\end{lemma}

\begin{proof}

This follows from similar arguments as the proof of \Cref{M_3 first step history}, invoking \Cref{primitive computations}(4) if the subword is reduced and \Cref{primitive unreduced} if it is unreduced.

\end{proof}

Given $w\in F(X)$ and $H\in F(\Phi^+)$, the configuration $I_3(w,H)$ of $\textbf{M}_3$ is the input configuration corresponding to the configuration $\bar{I}_2(w,H)$ of $\overline{\textbf{M}}_2$.  Likewise, the configuration $A_3(H)$ is the end configuration with the copy of $H$ over the corresponding right historical alphabet written in each $P_i'R_i'$-sector, the copy of $H^{-1}$ over the corresponding right historical alphabet written in each $P_i''R_i''$-sector, and all other sectors empty.

An input (respectively end) configuration of $\textbf{M}_3$ is called \textit{tame} if the letters comprising its tape words in the $P_i'R_i'$- and $P_i''R_i''$-sectors are all from the left (respectively right) historical alphabets. Note that for all $w\in F(X)$ and $H\in F(\Phi^+)$, the configurations $I_3(w,H)$ and $A_3(H)$ are tame.

\begin{lemma} \label{M_3 language} \

\begin{enumerate}

\item For any $w\in \pazocal{R}_1$, there exists $H_w\in F(\Phi^+)$ with $\|H_w\|\leq c_1f_1(c_1\|w\|)+c_1$ such that there exists a reduced computation $I_3(w,H_w)\to\dots\to A_3(H_w)$ of $\textbf{M}_3$.

\item If there exists a reduced computation $\pazocal{C}:W_0\to\dots\to W_t$ of $\textbf{M}_3$ such that $W_0$ and $W_t$ are tame input and end configurations, respectively, then there exists $w\in \pazocal{R}_1$ and $H\in F(\Phi^+)$ with $\|H\|\geq k$ such that $W_0\equiv I_3(w,H)$ and $W_t\equiv A_3(H)$.  Moreover, in this case $\pazocal{C}$ is the unique reduced computation between these configurations, $t=(2k+1)\|H\|+2k+2$, and $w$ is uniquely determined by $H$.

\end{enumerate}

\end{lemma}

\begin{proof}

(1) Letting $H_w$ be the word associated to $w$ given in \Cref{M_2 bar language}(1), there exists a reduced computation $\pazocal{C}_1:I_3(w,H_w)\to\dots\to W$ with step history $(1)(2)(3)$ such that $W$ is the copy of $\bar{A}_2(H_w)$ in the hardware of $\textbf{M}_3$.  As such, $W$ is $\sigma(34)$-admissible.

It suffices then to note that the reduced computation of $\textbf{LR}_k(\Phi^+)$ given in \Cref{LR_k analogue} corresponding to $H_w$ and its analogue of $\textbf{RL}_k(\Phi^+)$ corresponding to $H_w^{-1}$ are performed simultaneously by the rules, providing a reduced computation $\pazocal{C}_2:W\cdot\sigma(34)\to\dots\to A_3(H_w)$ of $\textbf{M}_3(4)$.

Concatenating $\pazocal{C}_1$ and $\pazocal{C}_2$ then provides a reduced computation $I_3(w,H_w)\to\dots\to A_3(H_w)$.% of length $(2k+1)\|H_w\|+2k+1$.  The statement then follows by taking $k\geq1$.

(2) Lemmas \ref{M_2 bar step history}, \ref{M_3 first step history}, and \ref{M_3 second step history} imply the step history of $\pazocal{C}$ must be $(1)(2)(3)(4)$.

Let $\pazocal{C}':W_0\to\dots\to W_s$ be the maximal subcomputation with step history $(1)(2)(3)$.  Then $W_s$ is $\sigma(34)$-admissible, so all of its tape words in the historical sectors are written in the corresponding right historical alphabets.

Then, there exists a word $H\in F(\Phi^+)$ such that $\pazocal{C}'$ multiplies each $P_i'R_i'$-sector on the left by the copy of $H^{-1}$ over the corresponding left historical alphabet and on the right by the copy of $H$ over the corresponding right historical alphabet.  But then the tape word of $W_0$ in the $P_i'R_i'$-sector must be the copy of $H$ over the corresponding left historical alphabet.

An analogous argument for the $P_i''R_i''$-sector then implies that the tape word of $W_0$ in this sector is the copy of $H^{-1}$ over the left historical alphabet, {\frenchspacing i.e. $W_0\equiv I_3(w,H)$} for some $w\in F(X)$.

\Cref{M_2 bar language}(2) then implies $w\in\pazocal{R}_1$, $H$ uniquely determines $w$, $W_s$ is the copy of $\bar{A}_2(H)$ in the hardware of $\textbf{M}_3$, $\|H\|\geq k$, and $s=\|H\|+2$.

The restriction of the subcomputation $W_{s+1}\to\dots\to W_t$ to the subword $P_i'R_i'Q_{i,r}'$ of the standard base then satisfies \Cref{LR_k analogue}, implying the uniqueness of the computation, that $W_t\equiv A_3(H)$, and that $t-s-1=2k\|H\|+2k-1$.

\end{proof}

\medskip

%%%%%%%%%%%%%%%%%%%%%%%%%%%%%%%%%%%%%%%%%%%%%%%%%%

\subsection{The machine $\textbf{M}_4$} \

The next machine of our construction, $\textbf{M}_4$, is the composition of $\textbf{M}_3$ with one more simple machine whose function is to delete the letters from the configurations of the form $A_3(H)$.

As in the previous section, the tape alphabets and standard base of $\textbf{M}_4$ are the same as its predecessor, but with each part of the state letters possessing more state letters.  Indeed, as a composition, the makeup of the state letters and software of the machine can be understood through discussion of its submachines.

The machine $\textbf{M}_4$ consists of five submachines, denoted $\textbf{M}_4(1),\dots,\textbf{M}_4(5)$, which are concatenated through the transition rules $\sigma(i,i+1)$.  Naturally, the start letters of $\textbf{M}_4(1)$ function as the start letters of the machine, while the end (the only) letters of $\textbf{M}_4(5)$ function as the end letters.  The input sectors are assigned in the same way as in $\textbf{M}_3$ and $\overline{\textbf{M}}_2$.

The first four submachines are identical to the submachines of $\textbf{M}_3$.  Indeed, these submachines along with the transition rules $\sigma(12)$, $\sigma(23)$, and $\sigma(34)$ form a copy of the machine $\textbf{M}_3$.

The machine $\textbf{M}_4(5)$ consists of a single state, with every part of the state letters containing one letter corresponding to this state.

The transition rule $\sigma(45)$ switches the state letters from the end letters of $\textbf{M}_4(4)$ ({\frenchspacing i.e. the} copies of the end letters of $\textbf{M}_3$) to the letters of $\textbf{M}_4(5)$, locking every sector other than those of the form $P_i'R_i'$ and $P_i''R_i''$, in which the domain is the corresponding right historical alphabet.

The positive rules of $\textbf{M}_4(5)$ are in bijection with $\Phi^+$, with the rule corresponding to $\theta\in\Phi^+$ multiplying each $P_i'R_i'$-sector on the right by the copy of $\theta$ over the corresponding right historical alphabet, multiplying each $P_i''R_i''$-sector on the left by the copy of $\theta^{-1}$ over the corresponding right historical alphabet, and locking all other sectors.  %The only other positive rule of $\textbf{M}_4(5)$ is the rule $\omega$, which switches the state letters from the start letters to the end letters and locks every sector.

The notion of step history defined in the previous sections adapts naturally to the setting of $\textbf{M}_4$.  Specifically, the step history of a computation of $\textbf{M}_4$ is some concatenation of the letters
$$\{(1),(2),(3),(4),(5),(12),(23),(34),(45),(21),(32),(43),(54)\}$$
As the first four submachines (along with the corresponding transition rules) operate identically as $\textbf{M}_3$, the analogues of Lemmas \ref{M_2 bar step history}, \ref{M_3 first step history}, and \ref{M_3 second step history} follow immediately.  The next statement is in the same spirit and follows in just the same way as \Cref{M_3 second step history}:

\begin{lemma} \label{M_4 step history}

Let $\pazocal{C}$ be a reduced computation of $\textbf{M}_4$ with base $B$.  If $B$ contains a subword of the form $(P_i'R_i'Q_{i,r}')^{\pm1}$, $(Q_{i,\ell}''P_i''R_i'')^{\pm1}$, $(Q_{i,r}')^{-1}(R_i')^{-1}R_i'Q_{i,r}'$, or $Q_{i,\ell}''P_i''(P_i'')^{-1}(Q_{i,\ell}'')^{-1}$, then the step history of $\pazocal{C}$ is not $(54)(4)(45)$.

\end{lemma}

The next statement is useful for the estimates in future sections and follows immediately from Lemmas \ref{simplify rules}, \ref{one alphabet historical words}, and \ref{one alphabet historical words unreduced}:

\begin{lemma}[Compare with Lemma 3.13 of \cite{OS19}] \label{M_4 restriction}

Given $j\in\{1,2,3\}$, let $W_0\to\dots\to W_t$ be a reduced computation of $\textbf{M}_4(j)$ with base $B$.  Suppose either:

\begin{enumerate}[label=(\alph*)]

\item $B$ contains a subword $UV$ of the form $(P_i'R_i')^{\pm1}$, $(P_i''R_i'')^{\pm1}$, $P_i'(P_i')^{-1}$, or $(R_i'')^{-1}R_i''$ such that the tape word of $W_0$ in the $UV$-sector has no letters from the left historical copy of $\Phi_j$, or

\item $B$ contains a subword $UV$ of the form $(P_i'R_i')^{\pm1}$, $(P_i''R_i'')^{\pm1}$, $(R_i')^{-1}R_i'$, or $P_i''(P_i'')^{-1}$ such that the tape word of $W_0$ in the $UV$-sector has no letters from the right historical copy of $\Phi_j$.

\end{enumerate}

Then $|W_0|_a\leq(2\|B\|-3)|W_t|_a$.

\end{lemma}

\begin{proof}

It follows from the hypotheses that the restriction $\pazocal{C}_{UV}:W_{0,UV}\to\dots\to W_{t,UV}$ of $\pazocal{C}$ to the $UV$-sector satisfies the hypotheses of either \Cref{one alphabet historical words} or \Cref{one alphabet historical words unreduced}.  As such, $t\leq |W_{t,UV}|_a$ and $|W_{0,UV}|_a\leq|W_{t,UV}|_a$.

Letting $\pazocal{C}':W_0'\to\dots\to W_t'$ be the restriction of $\pazocal{C}$ to any other sector, \Cref{simplify rules} implies $|W_0'|_a\leq|W_t'|_s+2t\leq|W_t'|_a+2|W_{t,UV}|_a$.  So, since there are $\|B\|-1$ total sectors in an admissible word with base $B$:
\begin{align*}
|W_0|_a&=|W_{0,UV}|_a+\sum|W_0'|_a\leq|W_{t,UV}|_a+\sum(|W_t'|_a+2|W_{t,UV}|_a) \\
&\leq(2\|B\|-3)|W_{t,UV}|_a+\sum|W_t'|_a\leq(2\|B\|-3)|W_t|_a
\end{align*}

\end{proof}

\begin{lemma} \label{right historical end}

Let $\pazocal{C}:W_0\to\dots\to W_t$ be a reduced computation of $\textbf{M}_5$ with base $B$ and history $H$.  Suppose there exists $j\in\{1,2,3\}$ such that:

\begin{itemize}

\item $W_0$ is $\sigma(j,j+1)$-admissible

\item The first letter of $H$ is not $\sigma(j,j+1)$.

\item $B$ contains a subword $UV$ of the form $(P_i'R_i')^{\pm1}$, $(P_i''R_i'')^{\pm1}$, $P_i'(P_i')^{-1}$, or $(R_i'')^{-1}R_i''$

\end{itemize}

Then the step history of $\pazocal{C}$ is a subword of $(3)(32)(2)(21)(1)$.

%Then:
%
%\begin{enumerate}[label=(\alph*)]
%
%\item The step history of $\pazocal{C}$ is a subword of $(3)(32)(2)(21)(1)$
%
%\item $|W_t|_a\leq(2\|B\|-3)^{j}|W_0|_a$.
%
%\end{enumerate}

\end{lemma}

\begin{proof}

Let $\pazocal{C}':W_0'\to\dots\to W_t'$ be the restriction of $\pazocal{C}$ to the $UV$-sector and $\pazocal{C}_j':W_0'\to\dots\to W_r'$ be the maximal (perhaps empty) subcomputation of with step history $(j)$.  

Suppose $r>0$.  Then since $W_0$ is $\sigma(j,j+1)$-admissible, $\pazocal{C}_j'$ satisfies the hypotheses of either \Cref{one alphabet historical words} or \Cref{one alphabet historical words unreduced}.  In particular, $W_r'$ must contain letters corresponding to the copy of $\Phi_j$ in the left historical alphabet, so that $W_r$ is not $\sigma(j,j+1)$-admissible.

Hence, if $t>r$, then since the first letter of $H$ is not $\sigma(j,j+1)$ it follows that the transition $W_r\to W_{r+1}$ is given by $\sigma(j,j-1)$.

Because any subsequent rule of $H$ cannot be $\sigma(j-1,j)$, the statement follows by iterating the above argument.

%What's more, (b) follows by applying \Cref{M_4 restriction} to the one-step subcomputations.

\end{proof}

Given $w\in F(X)$ and $H\in F(\Phi^+)$, the (tame) input configuration $I_4(w,H)$ is the copy of the input configuration $I_3(w,H)$ of $\textbf{M}_3$ in this hardware.

\begin{lemma} \label{M_4 language} \

\begin{enumerate}

\item For any $w\in\pazocal{R}_1$, there exists $H_w$ with $\|H_w\|\leq c_1f_1(c_1\|w\|)+c_1$ such that $I_4(w,H_w)$ is accepted by $\textbf{M}_4$.

\item The tame input configuration $W$ of $\textbf{M}_4$ is accepted if and only if there exist $w\in\pazocal{R}_1$ and $H\in F(\Phi^+)$ such that $W\equiv I_4(w,H)$.  Moreover, in this case $\|H\|\geq k$, $w$ is uniquely determined by $H$, and there exists a unique accepting computation of $W$, which has length $(2k+2)\|H\|+2k+3$.

\end{enumerate}

\end{lemma}

\begin{proof}

(1) Letting $H_w$ be the word given in \Cref{M_3 language}(1), there exists a reduced computation of $\textbf{M}_4$ with step history $(1)(2)(3)(4)$ beginning with $I_4(w,H_w)$ and ending with the copy of $A_3(H_w)$.

The configuration $A_3(H_w)$ is $\sigma(45)$-admissible, with $A_3(H_w)\cdot\sigma(45)$ containing copies of $H_w$ over the right historical alphabets written in its $P_i'R_i'$-sectors and copies of $H_w^{-1}$ over the right historical alphabets written in its $P_i''R_i''$-sectors.

Letting $\bar{H}_w$ be the copy of $H_w$ over the positive rules of $\textbf{M}_4(5)$, then $A_3(H_w)\cdot\sigma(45)\bar{H}_w^{-1}$ has empty tape words, and so is the accept configuration.

(2) Let $\pazocal{C}:W\equiv W_0\to\dots\to W_t$ be a reduced computation accepting $W$.

By Lemmas \ref{M_2 bar step history}, \ref{M_3 first step history}, \ref{M_3 second step history}, and \ref{M_4 step history}, the step history of $\pazocal{C}$ must be $(1)(2)(3)(4)(5)$.

Let $W_0\to\dots\to W_s$ be the maximal subcomputation with step history $(1)(2)(3)(4)$.  As $W_s$ is $\sigma(45)$-admissible, it corresponds to a tame end configuration of $\textbf{M}_3$.  As such, \Cref{M_3 language}(2) implies the existence of $w\in\pazocal{R}_1$ and $H\in F(\Phi^+)$ such that $\|H\|\geq k$, $H$ uniquely determines $w$, $W\equiv I_4(w,H)$, and $W_s\equiv A_3(H)$.  Moreover, $\pazocal{C}'$ must be the unique reduced computation with this step history between these configurations, with $s=(2k+1)\|H\|+2k+2$.

Let $H_5$ be the history of the subcomputation $W_s\cdot\sigma(45)\equiv W_{s+1}\to\dots\to W_t$.  Applying \Cref{multiply one letter} to the restriction of this subcomputation to any $P_i'R_i'$-sector (or to any $P_i''R_i''$-sector) then implies $H_5$ is the copy of $H^{-1}$ over the positive rules of $\textbf{M}_5(5)$.  In particular, $H_5$ is uniquely determined by $H$ with $\|H_5\|=\|H\|$.

\end{proof}

\medskip

%%%%%%%%%%%%%%%%%%%%%%%%%%%%%%%%%%%%%%%%%%%%%%%%%%

\subsection{The machine $\overline{\textbf{M}}_4$} \

As in \Cref{sec-overline-M_2}, the next step in our construction is the introduction of an `intermediate' machine which functions in much the same way as the previous machine.  In particular, the machine $\overline{\textbf{M}}_4$ is the `circular' analogue of a simple tweak of the machine $\textbf{M}_4$.

The standard base of $\overline{\textbf{M}}_4$ adds just one part to that of $\textbf{M}_4$. In particular, setting $B_4$ as the standard base of $\textbf{M}_4$, the standard base of $\overline{\textbf{M}}_4$ is $\{t\}B_4$, where $\{t\}$ consists of a single letter (which, of course, acts as both the start and end letter of its part). The tape alphabet of the new sector in the standard base, {\frenchspacing i.e. the} $\{t\}Q_0$-sector, is empty. All other tape alphabets are carried over from $\textbf{M}_4$. 

The length of this standard base is henceforth taken to be the parameter $N$. Note that this length may be adjusted and taken sufficiently large by adding any number of locked sectors to the standard base of $\textbf{Move}$ or of $\textbf{S}$ without affecting the functionality of the machine or any previous statement.

However, a tape alphabet is also assigned to the space after the final letter of $B_4$, corresponding to the $Q_{s,r}''\{t\}$-sector.  As such, it is possible for an admissible word of $\overline{\textbf{M}}_4$ to have base 
$$Q_{s,\ell}''P_s''R_s''Q_{s,r}''\{t\}Q_0Q_1Q_1^{-1}Q_0^{-1}\{t\}^{-1}(Q_{s,r}'')^{-1}(R_s'')^{-1}(P_s'')^{-1}(Q_{s,\ell}'')^{-1}$$
{\frenchspacing i.e. to} essentially `wrap around' the standard base. An $S$-machine with this property is called a \textit{cyclic machine}, as one can think of the standard base as being written on a circle.  As will be seen in \Cref{sec-associated-groups}, this is a natural consideration given the structure of the associated groups.

In this setting, the tape alphabet assigned to the $Q_{s,r}''\{t\}$-sector is empty.

The positive rules of $\overline{\textbf{M}}_4$ correspond to those of $\textbf{M}_4$, operating on the copy the hardware of $\textbf{M}_4$ in the same way and locking the new sectors.

The input sectors are the same as that of $\textbf{M}_4$. The corresponding definitions (for example, submachines, historical sector, working sector, etc) then extend in the clear way, as do all statements pertaining to the machine.

The definition of a faulty base of a cyclic machine is the same as that for a regular machine (see \Cref{sec-S-machines}).  However, note that there are more possible faulty bases in the setting of cyclic machines, {\frenchspacing i.e. those that wrap around the standard base}.

More generally, a base of a cyclic $S$-machine is said to be \textit{revolving} if:

\begin{enumerate}

\item It starts and end with the same base letter

\item None of its proper subwords satisfies (1).

\end{enumerate}

Note that a base is faulty if and only if it is revolving and unreduced.

\begin{lemma} \label{barM_4(j) faulty}

For every reduced computation $\pazocal{C}:W_0\to\dots\to W_t$ of $\overline{\textbf{M}}_4(j)$ with faulty base, $|W_i|_a\leq c_0(|W_0|_a+|W_t|_a)$ for all $0\leq i\leq t$.

\end{lemma}

\begin{proof}

Let $H$ be the history and $B$ be the base of $\pazocal{C}$.

%It suffices to find a sequence of subwords $B_1,\dots,B_m$ of a cyclic permutation of $B$ such that:
%
%\begin{enumerate}
%
%\item Letting $B_\ell'$ be the word obtained from $B_\ell$ by removing the first letter, $B_1B_2'\dots B_m'$ is a cyclic permutation of $B$.
%
%%\item There is a bijection between the sectors of an admissible word with base $B$ and the sectors of an $m$-tuple of admissible words with bases $B_1,\dots,B_m$.
%
%\item For all $\ell\in\{0,\dots,m\}$, the restriction $\pazocal{C}_\ell:W_{0,\ell}\to\dots\to W_{t,\ell}$ of (a cyclic permutation of) $\pazocal{C}$ to the subword $B_\ell$ satisfies $|W_{i,\ell}|_a\leq c_0(|W_{0,\ell}|_a+|W_{t,\ell}|_a)$ for all $0\leq i\leq t$.
%
%\end{enumerate}

%Note that condition (1) can be interpreted as the ability to `decompose' an admissible word with base $B$ into an $m$-tuple of admissible words with bases $B_1,\dots,B_m$, where the overlap of the words is simply the first or last $q$-letters so as every two-letter subword of $B$ belongs to exactly one $B_\ell$.

If there exists $1\leq i\leq t-1$ such that $|W_i|_a\leq|W_0|_a$, then it suffices to prove the statement for the subcomputations $W_0\to\dots\to W_i$ and $W_i\to\dots\to W_t$.  Hence, it suffices to assume $|W_i|_a>\max(|W_0|_a,|W_t|_a)$ for all $1\leq i\leq t-1$.

\textbf{1.} Suppose $j=5$.

Then the restriction $\pazocal{C}':W_0'\to\dots\to W_t'$ of $\pazocal{C}$ to any two-letter subword of $B$ either has fixed tape word, satisfies the hypotheses of \Cref{multiply one letter}, or satisfies the hypotheses of \Cref{unreduced base}.  As a result:
$$|W_i|_a=\sum|W_i'|_a\leq\sum(|W_0'|_a+|W_t'|_a)=|W_0|_a+|W_t|_a$$
Hence, the statement follows for $c_0\geq1$.

\textbf{2.} Suppose $j=4$.

If the history contains no letter corresponding to a connecting rule (or its inverse) of the primitive machine, then the same argument as Case 1 would imply the statement.  So, it suffices to assume that $H\equiv H_1\chi^{\pm1} H_2$ where $\chi$ is a connecting rule and $H_2$ contains no connecting rule.

Let $\pazocal{C}_2:W_s\to\dots\to W_t$ be the subcomputation with history $H_2$.  Note that since connecting rules do not alter tape words, we may assume that $H_2$ is non-trivial, so that $0<s<t$.

Without loss of generality, suppose $\chi$ locks the $R_i'Q_{i,r}'$-sectors, and so also locks the $Q_{i,\ell}''P_i''$-sectors.  Then by \Cref{locked sectors}, any unreduced two-letter subword of $B$ must be of the form $P_i'(P_i')^{-1}$, $(R_i')^{-1}R_i'$, $P_i''(P_i'')^{-1}$, or $(R_i'')^{-1}R_i''$.

So, any occurrence of a letter of the form $(R_i')^{\pm1}$ or $(P_i'')^{\pm1}$ in $B$ must be part of an `active' subword of a cyclic permutation of $B$, {\frenchspacing i.e. one} of the form $(P_i'R_i'Q_{i,r}')^{\pm1}$, $(Q_{i,\ell}''P_i''R_i'')^{\pm1}$, $(Q_{i,r}')^{-1}(R_i')^{-1}R_i'Q_{i,r}'$, or $Q_{i,\ell}''P_i''(P_i'')^{-1}(Q_{i,\ell}'')^{-1}$.  If $B$ does contain such a letter, then the cyclic permutation whose prefix is the corresponding active subword must contain all active subwords.

Let $\pazocal{C}_2':W_s'\to\dots\to W_t'$ be the restriction of $\pazocal{C}_2$ to an active subword $B'$.  If $B'$ is reduced, then \Cref{primitive computations}(5) implies $|W_s'|_a\leq|W_t'|_a$; otherwise, \Cref{primitive unreduced} implies $|W_s'|_a\leq|W_t'|_a$. 

But every sector which does not appear as part of an active subword has fixed tape word throughout the computation, so that $|W_s|_a\leq|W_t|_a$.

\textbf{3.} Suppose $j\in\{1,2,3\}$ and $B$ contains neither a letter of the form $Q_i^{\pm1}$ nor a subword of the form $(Q_{0,\ell}')^{-1}Q_{0,\ell}'$.

Since every rule of $\overline{\textbf{M}}_4$ locks the $\{t\}Q_0$-sector, $B$ cannot contain any letter of the form $t^{\pm1}$.

Then, as each rule of $\overline{\textbf{M}}_4(j)$ locks the $Q_{i,\ell}'P_i'$-, $R_i'Q_{i,r}'$-, $Q_{i,\ell}''P_i''$-, and $R_i''Q_{i,r}''$-sectors, it follows that $\pazocal{C}$ can be identified with a reduced computation $\pazocal{D}$ of $\textbf{M}_2(j)$ with a faulty base that does not contain any letter of the form $Q_i^{\pm1}$ or subword of the form $(Q_{0,\ell}')^{-1}Q_{0,\ell}'$.  

Faulty bases must have length at least 3, though, so that $\pazocal{D}$ satisfies the hypotheses of \Cref{M_2 bound}.  

Hence, the statement follows for $c_0\geq9$.

\textbf{4.} Suppose $j\in\{1,2,3\}$ and $B$ consists entirely of letters of the form $Q_i^{\pm1}$ or of the form $t^{\pm1}$.

Similar to the previous case, since every rule locks the $Q_{s,r}''\{t\}$-sector, $B$ cannot contain any letter of the form $t^{\pm1}$.  As such, $B$ consists entirely of letters of the form $Q_i^{\pm1}$.

Further, since every rule of $\overline{\textbf{M}}_4(2)$ and every rule of $\overline{\textbf{M}}_4(3)$ locks the $Q_{i-1}Q_i$-sectors, it suffices to assume $j=1$.

But then $\pazocal{C}$ can be identified with a reduced computation of $\textbf{Move}$ with faulty base, so that the statement follows from \Cref{Move faulty}.

\textbf{5.} Thus, it suffices to assume $j\in\{1,2,3\}$ and $B$ contains both a subword of the form considered in Case 3 and one of the form considered in Case 4.  As such, there exists $m\geq1$ for which there exists a factorization $B'\equiv B_1C_1\dots B_mC_mB_{m+1}$ of a cyclic permutation of $B$ such that:

\begin{itemize}

\item each $B_x$ is a nonempty word of length at least 2 which does not contain a letter of the form $Q_i^{\pm1}$ or a subword of the form $(Q_{0,\ell}')^{-1}Q_{0,\ell}'$

\item each $C_x$ is a possibly empty word consisting entirely of letters of the form $Q_i^{\pm1}$ or $t^{\pm1}$

\item If $C_x$ is empty, then the last letter of $B_x$ is $(Q_{0,\ell}')^{-1}$ and the first letter of $B_{x+1}$ is $Q_{0,\ell}'$

\item $B_{m+1}$ consists of a single letter

\end{itemize}

For all $1\leq x\leq m$, let $C_x'$ be the word $R_xC_xP_{x+1}$, where $R_x$ is the last letter of $B_x$ and $P_{x+1}$ is the first letter of $B_{x+1}$.   Then, define:

\begin{itemize}

\item $\pazocal{C}_x:W_{0,x}\to\dots\to W_{t,x}$ to be the restriction of (a cyclic permutation of) $\pazocal{C}$ to $B_x$.

\item $\pazocal{C}_x':W_{0,x}'\to\dots\to W_{t,x}'$ to be the restriction of (a cyclic permutation of) $\pazocal{C}$ to $C_x'$.

\end{itemize}

Note that $|W_i|_a=\sum(|W_{i,x}|_a+|W_{i,x}'|_a)$ for all $i$.

As $B$ is faulty, the first letter of $B_1$ must be the same as the last (only) letter of $B_{m+1}$.  So, it follows that the first letter of each $B_x$ must be either $Q_{0,\ell}'$ or $(Q_{s,r}'')^{-1}$, while the last letter must be either $(Q_{0,\ell}')^{-1}$ or $Q_{s,r}''$.

As such, each $\pazocal{C}_x$ can be identified with a reduced computation of $\textbf{M}_2$ whose base begins and ends with the analogous letters.  If the length of this base is at least 3, then \Cref{M_2 bound} yields $|W_{i,x}|_a\leq9(|W_{0,x}|_a+|W_{t,x}|_a)$ for all $i$; otherwise, the base is either $Q_{0,\ell}'(Q_{0,\ell}')^{-1}$ or $(Q_{s,r}'')^{-1}Q_{s,r}''$, so that \Cref{unreduced base} implies $|W_{i,x}|_a\leq|W_{0,x}|_a+|W_{t,x}|_a$.

Hence, taking $c_0\geq9$, it suffices to show that $|W_{i,x}'|_a\leq c_0(|W_{0,x}'|_a+|W_{t,x}'|_a)$ for all $i$ and $x$.

If $j=3$, then this bound is immediate since every rule of $\overline{\textbf{M}}_4(3)$ locks the sectors that can appear as subwords of $C_x'$.

If $j=1$, then $\pazocal{C}_x'$ can be identified with a reduced computation of $\textbf{Move}$ with full base, so that the bound follows from \Cref{Move faulty}.

Thus, it suffices to assume $j=2$.  

For any rule of $\overline{\textbf{M}}_4(2)$, the only possible subwords of $C_x'$ that the rule does not lock are those of the form (a) $(Q_{n-1}Q_{0,\ell}')^{\pm1}$, (b) $Q_{n-1}Q_{n-1}^{-1}$, or (c) $(Q_{0,\ell}')^{-1}Q_{0,\ell}'$.  

But given the structure of the rules arising from the Clean machine, the restriction of $\pazocal{C}_x'$ to a subword of the form (a) satisfies the hypotheses of \Cref{multiply one letter}, the restriction to those of the form (b) has fixed tape word, and the restriction to those of the form (c) satisfies the hypotheses of \Cref{unreduced base}.  As such, the bound again follows.

\end{proof}

\begin{remark}

The use of \Cref{Move faulty} in Case 4 of the proof of \Cref{barM_4(j) faulty} indicates the purpose of condition (Mv7) in the definition of Move machines.

\end{remark}

\begin{lemma} \label{barM_4 faulty}

If $\pazocal{C}:W_0\to\dots\to W_t$ is a reduced computation of $\overline{\textbf{M}}_4$ whose base $B$ is faulty, then $|W_i|_a\leq c_1\max(|W_0|_a,|W_t|_a)$ for all $0\leq i\leq t$.

\end{lemma}

\begin{proof}

By \Cref{barM_4(j) faulty} and the parameter choice $c_1>>c_0$, it suffices to assume that the history $H$ of $\pazocal{C}$ contains a letter corresponding to a transition rule.  

Further, as in the proof of \Cref{barM_4(j) faulty}, we may assume that $|W_i|_a>\max(|W_0|_a,|W_t|_a)$ for all $1\leq i\leq t-1$.  In particular, neither the first nor last letter of $H$ corresponds to a transition rule.

\textbf{1.} Suppose $H$ contains a letter of the form $\sigma(45)^{\pm1}$ and $B$ contains a letter that is either of the form $(R_i')^{\pm1}$ or of the form $(P_i'')^{\pm1}$.

Without loss of generality (perhaps passing to the inverse computation), suppose $H$ contains a subword of the form $H_5\sigma(54)H_4$, where $H_4$ and $H_5$ are the histories of (nonempty) maximal subcomputations which have step history $(4)$ and $(5)$, respectively.  Let $\pazocal{C}_4:W_r\to\dots\to W_s$ be the subcomputation with history $H_4$.

As $\sigma(45)$ locks every sector of the standard base except for the $P_i'R_i'$- and $P_i''R_i''$-sectors, \Cref{locked sectors} implies any unreduced two-letter subword of $B$ must be of one of the following forms:

\begin{enumerate}[label=(\alph*)]

\item $P_i'(P_i')^{-1}$ or $P_i''(P_i'')^{-1}$ 

\item $(R_i')^{-1}R_i'$ or $(R_i'')^{-1}R_i''$.  

\end{enumerate}

So, as it is faulty, $B$ must contain exactly one two-letter subword of one of the forms given in (a) and exactly one as given in (b).

Note that if $H_4$ contains a connecting rule, then the first connecting rule would necessarily lock all $P_i'R_i'$- and $P_i''R_i''$-sectors.  But given the unreduced two-letter subwords of $B$, the presence of such a connecting rule along contradicts \Cref{locked sectors}.  Hence, $H_4$ cannot contain any connecting rules, and so $W_s$ is not $\sigma(43)$-admissible.

Moreover, as $B$ contains a letter of the form $(R_i')^{\pm1}$ or $(P_i'')^{\pm1}$, \Cref{locked sectors} implies $B$ must also contain a subword of the form $(R_i'Q_{i,r}')^{\pm1}$ or of the form $(Q_{i,\ell}''P_i'')^{\pm1}$.  Since $H_4$ contains no connecting rule, the restriction of $\pazocal{C}_4$ to this subword then satisfies the hypotheses of \Cref{multiply one letter}.  But as $\|H_4\|>0$, it follows that $W_s$ is not $\sigma(45)$-admissible.

Hence, $H_4$ is a suffix of $H$, {\frenchspacing i.e. $s=t$}.

Now, as above \Cref{locked sectors} implies any letter of $B$ of the form $(R_i')^{\pm1}$ must belong to a subword of a cyclic permutation of $B$ of the form $(P_i'R_i'Q_{i,r}')^{\pm1}$ or $(Q_{i,r}')^{-1}(R_i')^{-1}R_i'Q_{i,r}'$.  Similarly, since $\sigma(45)$ locks the $Q_{i,\ell}''P_i''$-sectors, any letter of $B$ of the form $(P_i'')^{\pm1}$ must belong to a subword of a cyclic permutation of $B$ of the form $(Q_{i,\ell}''P_i''R_i'')^{\pm1}$ or $Q_{i,\ell}''P_i''(P_i'')^{-1}(Q_{i,\ell}'')^{-1}$.

The restriction $\pazocal{C}_4':W_r'\to\dots\to W_t'$ of $\pazocal{C}_4$ to any such subword above either satisfies the hypotheses of \Cref{primitive computations} (if the subword is reduced) or satisfies the hypotheses of \Cref{primitive unreduced} (if the subword is unreduced), so that $|W_r'|_a\leq|W_t'|_a$.

But the restriction of $\pazocal{C}_4$ to any other sector has fixed tape word, yielding $|W_r|_a\leq|W_t|_a$.  As $0<r<t$, this contradicts our initial assumption.

\textbf{2.} Suppose $H$ contains a letter of the form $\sigma(45)^{\pm1}$.

By Case 1, it suffices to assume that $B$ contains neither a letter of the form $(R_i')^{\pm1}$ nor one of the form $(P_i'')^{\pm1}$.  However, $B$ must still contain exactly one subword of type (a) and one subword of type (b), {\frenchspacing i.e. it must} contain a subword of the form $P_i'(P_i')^{-1}$ and one of the form $(R_i'')^{-1}R_i''$.

Suppose $B$ contains the subword $P_i'(P_i')^{-1}$ for $i>0$.  Then by \Cref{locked sectors} the subsequent letters must be $(Q_{i,\ell}')^{-1}(Q_{i-1,r}')^{-1}(R_{i-1}')^{-1}$, contradicting our assumption that $B$ contains no letter of the form $(R_i')^{\pm1}$.

The analogous argument implies $B$ does not contain the subword $(R_i'')^{-1}R_i''$ for $i<s$.

Hence, $B$ is a cyclic permutation of:
$$B'\equiv P_0'(P_0')^{-1}(Q_{0,\ell}')^{-1}Q_{n-1}^{-1}\dots Q_0^{-1}(Q_{s,r}'')^{-1}(R_s'')^{-1}R_s''Q_{s,r}''Q_0\dots Q_{n-1}Q_{0,\ell}'P_0'$$
Now, by the same argument as presented in Case 1, $H$ cannot contain any connecting rule of $\overline{\textbf{M}}_4(4)$, and so cannot contain the letter $\sigma(43)$.  In particular, the step history of $\pazocal{C}$ is a concatenation of $(4)$ and $(5)$.

But by construction no rule of $\overline{\textbf{M}}_4(4)$ or $\overline{\textbf{M}}_4(5)$ alters the tape words of any admissible word with base $B'$, so that $|W_0|_a=\dots=|W_t|_a$.

\textbf{3.} Suppose $H$ contains a letter of the form $\sigma(34)^{\pm1}$ and $B$ contains a letter that is either of the form $(R_i')^{\pm1}$ or of the form $(P_i'')^{\pm1}$.

As in Case 1, we may assume without loss of generality that $H$ contains a subword of the form $H_3\sigma(34)H_4$ where $H_j$ is the history of a maximal subcomputation with step history $(j)$.  Let $\pazocal{C}_4:W_r\to\dots\to W_s$ be the subcomputation with history $H_4$.

Then, as $\sigma(34)$ and $\sigma(45)$ lock the same sectors, the identical arguments as those presented in Case 1 imply $s=t$ and $|W_r|_a\leq|W_t|_a$, again yielding a contradiction.

\textbf{4.} Suppose $H$ contains a letter of the form $\sigma(34)^{\pm1}$.

As in Case 2, it then suffices to assume that $B$ contains neither a letter of the form $(R_i')^{\pm1}$ nor of the form $(P_i'')^{\pm1}$, so that $B$ is again a cyclic permutation of the base $B'$.

Let $W_r\to\dots\to W_s$ be a maximal subcomputation with step history $(4)$.  Noting that no rule of $\overline{\textbf{M}}_4(4)$ alters the tape words of any admissible word with base $B'$, $|W_r|_a=\dots=|W_s|_a$.  Hence, we may assume that both $t>s$ and $r>0$.

Let $H_3$ be the history of the subcomputation $\pazocal{C}_3:W_s\to\dots\to W_t$.  Then the first letter of $H_3$ must be $\sigma(43)$, so that the subcomputation $W_{s+1}\to\dots\to W_t$ satisfies the hypotheses of \Cref{right historical end}.  Hence, the step history of $\pazocal{C}_3$ must be a prefix of $(43)(3)(2)(1)$.

Letting $\pazocal{C}_{3,3}:W_{s+1}\to\dots\to W_{x_3}$ be the maximal subcomputation with step history $(3)$, then \Cref{M_4 restriction} implies $|W_{s+1}|_a\leq(2\|B\|-3)|W_{x_3}|_a$.  As faulty bases must have length at most $2N+1$, the parameter choice $c_0>>N$ then implies $|W_{s}|_a\leq c_0|W_{x_3}|_a$.  Applying \Cref{barM_4(j) faulty} to $\pazocal{C}_{3,3}$ then implies $|W_i|_a\leq2c_0^2|W_{x_3}|_a$ for all $s\leq i\leq x_3$.

If $t>x_3$, then letting $\pazocal{C}_{3,2}:W_{x_3+1}\to\dots\to W_{x_2}$ be the maximal subcomputation with step history $(2)$, \Cref{M_4 restriction} implies $|W_{x_3+1}|_a\leq c_0|W_{x_2}|_a$, and so $|W_s|_a\leq c_0^2|W_{x_2}|_a$.  \Cref{barM_4(j) faulty} further implies $|W_i|_a\leq 2c_0^3|W_{x_3}|_a$ for all $s\leq i\leq x_2$.

Finally, if $t>x_2$, then the same argument implies $|W_s|_a\leq c_0^3|W_t|_a$ and $|W_i|_a\leq 2c_0^4|W_t|_a$ for all $s\leq i\leq t$.

Hence, taking $c_0\geq1$ in any case $|W_i|_a\leq 2c_0^4|W_t|_a$ for all $s\leq i\leq t$.

The same arguments may then be applied to the inverse subcomputation $W_r\to\dots\to W_0$, yielding $|W_i|_a\leq 2c_0^4|W_0|_a$ for all $0\leq i\leq r$, so that the statement follows from the parameter choice $c_1>>c_0$.

\textbf{5.} Suppose $H$ contains a letter of the form $\sigma(23)^{\pm1}$ and $B$ contains a letter either of the form $(P_i')^{\pm1}$ or of the form $(R_i'')^{\pm1}$.

Then $B$ must contain a subword of the form $(P_i'R_i')^{\pm1}$, $P_i'(P_i')^{-1}$, $(P_i''R_i'')^{-1}$, or $(R_i'')^{-1}R_i''$.

Let $\pazocal{C}_3:W_r\to\dots\to W_s$ be a maximal subcomputation with step history $(3)$.  Assuming without loss of generality that $t>s$, by Case 4 it suffices to assume that the history of the subcomputation $\pazocal{C}_2:W_s\to\dots\to W_t$ has prefix $\sigma(32)$.  As such, the subcomputation $W_{s+1}\to\dots\to W_t$ satisfies the hypotheses of \Cref{right historical end}, so that the step history of $\pazocal{C}_2$ is a prefix of $(32)(2)(1)$.

Repeating the argument presented in Case 4 then yields $|W_s|_a\leq c_0^2|W_t|_a$ and $|W_i|_a\leq2c_0^3|W_t|_a$ for all $s\leq i\leq t$.  Moreover, if $r>0$, then the analogous argument implies $|W_r|_a\leq c_0^2|W_0|_a$ and $|W_i|_a\leq2c_0^3|W_0|_a$ for all $0\leq i\leq r$.

Finally, applying \Cref{barM_4(j) faulty} to $\pazocal{C}_3$ yields $|W_i|_a\leq2c_0\max(|W_r|_a,|W_s|_a)\leq2c_0^3\max(|W_0|_a,|W_t|_a)$, so that the statement follows by the parameter choice $c_1>>c_0$.

\textbf{6.} Suppose $H$ contains a letter of the form $\sigma(23)^{\pm1}$.

By Case 5, it suffices to assume that $B$ contains neither a letter of the form $(P_i')^{\pm1}$ nor one of the form $(R_i'')^{\pm1}$.

Note that the only sectors of the standard base that $\sigma(23)$ does not lock are the $Q_{0,r}'Q_{1,\ell}'$-sector, the $P_i'R_i'$-sectors, and the $P_i''R_i''$-sectors.  So, $B$ must be a cyclic permutation of either:

\begin{enumerate}[label=(\alph*)]

\item $(R_0')^{-1}R_0'Q_{0,r}'(Q_{0,r}')^{-1}(R_0')^{-1}$

\item $(R_s')^{-1}R_s'Q_{s,r}'Q_{0,\ell}''P_0''(P_0'')^{-1}(Q_{0,\ell}'')^{-1}(Q_{s,r}')^{-1}(R_s')^{-1}$

\end{enumerate}

Define the subwords $U_1V_1$ and $U_2V_2$ of $B$ as follows:

\begin{itemize}

\item In case (a), let $U_1V_1=(R_0')^{-1}R_0'$ and $U_2V_2=Q_{0,r}'(Q_{0,r}')^{-1}$

\item In case (b), let $U_1V_1=(R_s')^{-1}R_s'$ and $U_2V_2=P_0''(P_0'')^{-1}$

\end{itemize}

Now, assume without loss of generality that $H$ contains a subword of the form $H_2\sigma(23)H_3$ where $H_3$ is the history of a maximal subcomputation $\pazocal{C}_3:W_r\to\dots\to W_s$ with step history $(3)$.  Then, let $\pazocal{C}_3':W_r'\to\dots\to W_s'$ and $\pazocal{C}_3'':W_r''\to\dots\to W_s''$ be the restrictions of $\pazocal{C}_3$ to the subwords $U_1V_1$ and $U_2V_2$, respectively.  

As $W_r$ is $\sigma(32)$-admissible, $\pazocal{C}_3'$ satisfies the hypotheses of \Cref{one alphabet historical words unreduced}.  As a result, it follows that $|W_r'|_a=|W_s'|_a-2(s-r)$ and the tape word of $W_s'$ contains letters which are copies of $\Phi_3$ from the right historical alphabet.  

But this implies $W_s$ cannot be $\sigma(32)$-admissible, and so by Case 4 we may assume that $s=t$, {\frenchspacing i.e. $|W_r'|_a=|W_t'|_a-2(t-r)$.}

Moreover, \Cref{simplify rules} implies $|W_r''|_a\leq|W_t''|_a+2(t-r)$.  (Note that this bound may be improved in case (b), as in that instance \Cref{one alphabet historical words unreduced} can again be applied to $\pazocal{C}_3''$).  But as every rule of $\overline{\textbf{M}}_4(3)$ locks the $Q_{s,r}'Q_{0,\ell}''$-sector, the $R_i'Q_{i,r}'$-sectors, and the $Q_{i,\ell}''P_i''$-sectors, it then follows that:
$$|W_r|_a=|W_r'|_a+|W_r''|_a\leq|W_t'|_a+|W_t''|_a=|W_t|_a$$

\textbf{7.} Suppose $B$ contains a letter either of the form $(P_i')^{\pm1}$ or of the form $(R_i'')^{\pm1}$.

Then as in Case 5, $B$ must contain a subword of the form $(P_i'R_i')^{\pm1}$, $P_i'(P_i')^{-1}$, $(P_i''R_i'')^{\pm1}$, or $(R_i'')^{-1}R_i''$.

Let $\pazocal{C}_2:W_r\to\dots\to W_s$ be a maximal subcomputation with step history $(2)$.  Assuming without loss of generality that $t>s$, by Case 6 the first letter of the history of the subcomputation $W_s\to\dots\to W_t$ must be $\sigma(21)$.

So, since $W_{s+1}$ is $\sigma(12)$-admissible, $W_{s+1}\to\dots\to W_t$ satisfies the hypotheses of \Cref{right historical end} and so has step history $(1)$.

\Cref{M_4 restriction} then implies $|W_s|_a\leq c_0|W_t|_a$ by the parameter choice $c_0>>N$, so that \Cref{barM_4(j) faulty} implies $|W_i|_a\leq 2c_0^2|W_t|_a$ for all $s\leq i\leq t$.

If $r>0$, then the same argument applies to the inverse computation $W_r\to\dots\to W_0$, so that $|W_r|_a\leq c_0|W_0|_a$ and $|W_i|_a\leq2c_0^2|W_t|_a$ for all $0\leq i\leq r$.

Finally, applying \Cref{barM_4(j) faulty} to $\pazocal{C}_2$ yields $|W_i|_a\leq2c_0\max(|W_r|_a,|W_s|_a)\leq2c_0^2\max(|W_0|_a,|W_t|_a)$ for all $r\leq i\leq s$, so that the statement follows by the parameter choice $c_1>>c_0$.

\textbf{8.} Thus, it suffices to assume every transition rule of $H$ is of the form $\sigma(12)^{\pm1}$ and $B$ contains neither a letter of the form $(P_i')^{\pm1}$ nor one of the form $(R_i'')^{\pm1}$.

As the only sectors of the standard base that $\sigma(12)$ does not lock are the $Q_{n-1}Q_{0,\ell}'$-sector, the $P_i'R_i'$-sectors, and the $P_i''R_i''$-sectors, it follows that $B$ must be a cyclic permutation of: $$(R_s')^{-1}R_s'Q_{s,r}'Q_{0,\ell}''P_0''(P_0'')^{-1}(Q_{0,\ell}'')^{-1}(Q_{s,r}')^{-1}(R_s')^{-1}$$

Assume without loss of generality that $H$ contains a subword of the form $H_1\sigma(12)H_2$ where $H_2$ is the history of a maximal subcomputation $\pazocal{C}_2:W_r\to\dots\to W_s$ with step history (2).  As in Case 6, let $\pazocal{C}_2':W_r'\to\dots\to W_s'$ and $\pazocal{C}_2'':W_r''\to\dots\to W_s''$ be the restrictions of $\pazocal{C}_2$ to the subwords $(R_s')^{-1}R_s'$ and $P_0''(P_0'')^{-1}$, respectively.

Since $W_r$ is $\sigma(21)$-admissible, \Cref{one alphabet historical words unreduced} applies to each of $\pazocal{C}_2'$ and $\pazocal{C}_2''$.  This implies $|W_r'|_a\leq|W_s'|_a$ and $|W_r''|_a\leq|W_s''|_a$.  Moreover, the tape words of each of $W_s'$ and $W_s''$ contain letters from the right historical alphabet corresponding to letters from $\Phi_2$, so that $W_s$ is not $\sigma(21)$-admissible.  Hence $t=s$.

But again, every rule of $\textbf{M}_5(2)$ locks the $R_i'Q_{i,r}'$-, $Q_{s,r}'Q_{0,\ell}''$-, and $Q_{i,\ell}''P_i''$-sectors, so that $$|W_r|_a=|W_r'|_a+|W_r''|_a\leq|W_t'|_a+|W_t''|_a=|W_t|_a$$

\end{proof}

\medskip

%%%%%%%%%%%%%%%%%%%%%%%%%%%%%%%%%%%%%%%%%%%%%%%%%%

\subsection{The machine $\textbf{M}_5$} \

The cyclic machine $\textbf{M}_5$ is the composition of $\overline{\textbf{M}}_4$ with a primitive machine that operates much like the submachine $\overline{\textbf{M}}_4(4)$.  However, this composition is a `pre-composition', in that it is carried out so that the primitive machine is essentially the `first submachine' of $\textbf{M}_5$.

To be precise, $\textbf{M}_5$ has standard base of the same form as that of $\overline{\textbf{M}}_4$ and consists of the six submachines $\textbf{M}_5(0),\textbf{M}_5(1),\dots,\textbf{M}_5(5)$ concatenated by the transition rules $\sigma(01),\dots,\sigma(45)$.  Naturally, the submachines $\textbf{M}_5(1),\dots,\textbf{M}_5(5)$ along with the corresponding transition rules produce a copy of $\overline{\textbf{M}}_4$ in $\textbf{M}_5$.

The submachine $\textbf{M}_5(0)$ operates in parallel as $\textbf{RL}(\Phi^+)$ on the subwords $Q_{i,\ell}'P_i'R_i'$ of the standard base and as $\textbf{LR}(\Phi^+)$ on the subwords $P_i''R_i''Q_{i,r}''$.  This is realized in much the same way as the definition of $\textbf{M}_3(4)$, except that the $Q_0Q_1$-sector is left unlocked by all rules (with the domain of each rule in this sector being the entire tape alphabet).

The transition rule $\sigma(01)$ then locks all sectors other than the $Q_0Q_1$-, $P_i'R_i'$-, and $P_i''R_i''$-sectors.

Naturally, the start letters of $\textbf{M}_5(0)$ are the start letters of the machine, while the end letters of $\textbf{M}_5(5)$ are the end letters of the machine.  The input sectors are assigned in the same way as for $\overline{\textbf{M}}_4$, while the part $\{t\}$ again consists of a single state letter.

Many definitions are extended or carried over from previous machines.  As an example, given $w\in F(X)$ and $H\in F(\Phi^+)$, define $I_5(w,H)$ to be the input configuration with $w$ written in its $Q_0Q_1$-sector, the copy of $H$ over the corresponding left historical alphabet written in each $P_i'R_i'$-sector, and the copy of $H^{-1}$ over the corresponding left historical alphabet written in each $P_i''R_i''$-sector.  What's more, as the tape words in the historical alphabets forming this word are formed over the corresponding left tape alphabets, $I_5(w,H)$ is a \textit{tame input configuration} of $\textbf{M}_5$.

\medskip

%%%%%%%%%%%%%%%%%%%%%%%%%%%%%%%%%%%%%%%%%%%%%%%%%%

\subsection{Standard computations of $\textbf{M}_5$} \label{sec M_5 standard} \

In this section, we study the reduced computations of $\textbf{M}_5$ in the standard base.

First, we adopt the definition of the step history of a computation developed in the previous sections to the setting $\textbf{M}_5$.  As a copy of $\overline{\textbf{M}}_4$ exists in $\textbf{M}_5$, there are natural analogues of Lemmas \ref{M_2 step history}, \ref{M_3 first step history}, \ref{M_3 second step history}, and \ref{M_4 step history} in this machine.  The next statement is in the same spirit:

\begin{lemma} \label{M_5 step history 1}

Let $\pazocal{C}$ be a reduced computation of $\textbf{M}_5$ with base $B$.

\begin{enumerate}[label=(\alph*)]

\item If $B$ contains a subword of the form $(P_i'R_i')^{\pm1}$, $(R_i')^{-1}R_i'$, $(P_i''R_i'')^{\pm1}$, or $P_i''(P_i'')^{-1}$, then the step history of $\pazocal{C}$ is not $(01)(1)(10)$.

\item If $B$ contains a subword of the form $(Q_{i,\ell}'P_i'R_i')^{\pm1}$, $Q_{i,\ell}'P_i'(P_i')^{-1}(Q_{i,\ell}')^{-1}$, $(P_i''R_i''Q_{i,r}'')^{\pm1}$, or $(Q_{i,r}'')^{-1}(R_i'')^{-1}R_i''Q_{i,r}''$, then the step history of $\pazocal{C}$ is not $(10)(0)(01)$.

\end{enumerate}

\end{lemma}

\begin{proof}

Statement (a) follows in the same way as \Cref{M_2 step history}(b), statement (b) as \Cref{M_4 step history}.

\end{proof}

We now collect these statements pertaining to step history into a single statement about the reduced computations of the machine in the standard base:

\begin{lemma} \label{M_5 step history}

The step history of a reduced computation of $\textbf{M}_5$ is a subword of
$$(0)(01)(1)(12)(2)(23)(3)(34)(4)(45)(5)(54)(4)(43)(3)(32)(2)(21)(1)(10)(0)$$

\end{lemma}

This implies the following analogue of \Cref{M_4 language}:

\begin{lemma} \label{M_5 language} \

\begin{enumerate}

\item For any $w\in\pazocal{R}_1$, there exists $H_w$ with $\|H_w\|\leq c_1f_1(c_1\|w\|)+c_1$ such that $I_5(w,H_w)$ is accepted by $\textbf{M}_5$.

\item The tame input configuration $W$ of $\textbf{M}_5$ is accepted if and only if there exist $w\in\pazocal{R}_1$ and $H\in F(\Phi^+)$ such that $W\equiv I_5(w,H)$.  Moreover, in this case $\|H\|\geq k$, $w$ is uniquely determined by $H$, and there exists a unique accepting computation of $W$, which has length $(2k+4)\|H\|+2k+5$.

\end{enumerate}

\end{lemma}

\begin{proof}

(1) Let $H_w$ be the word given in \Cref{M_4 language}(1).  As in the proof of \Cref{M_3 language}(1), the reduced computation of $\textbf{RL}$ corresponding to $H_w$ given in \Cref{primitive computations}(3) and that of $\textbf{LR}$ corresponding to $H_w^{-1}$ can be run in parallel in the rules of $\textbf{M}_5(0)$, producing a reduced computation $I_5(w,H_w)\to\dots\to W_1$ such that $W_1$ is $\sigma(01)$-admissible.  $W_1\cdot\sigma(01)$ is then the concatenation of $t$ with the natural copy of $I_4(w,H_w)$ in the hardware of $\textbf{M}_5$.  Thus, the statement follows by \Cref{M_4 language}(1).

(2) Let $\pazocal{C}:W\equiv W_0\to\dots\to W_t$ be a reduced computation accepting the tame input configuration $W$.  \Cref{M_5 step history} then implies the step history of $\pazocal{C}$ is $(0)(1)(2)(3)(4)(5)$.

Let $W_0\to\dots\to W_r$ be the maximal subcomputation of $\pazocal{C}$ with step history $(0)$.  Then the subcomputation $W_{r+1}\to\dots\to W_t$ with step history $(1)(2)(3)(4)(5)$ can be identified with a reduced computation of $\textbf{M}_4$ accepting a tame input configuration.  Hence, \Cref{M_4 language}(2) implies there exists $w\in\pazocal{R}_1$ and $H\in F(\Phi^+)$ such that:

\begin{itemize}

\item $\|H\|\geq k$

\item $w$ is uniquely determined by $H$

\item $t-r-1=(2k+2)\|H\|+2k+3$

\item $W_{r+1}$ is the concatenation of $t$ with the natural copy of $I_4(w,H_w)$ in the hardware of $\textbf{M}_5$.

\end{itemize}

As $W_r\equiv W_{r+1}\cdot\sigma(10)$, it then follows from \Cref{primitive computations}(3) that $W\equiv I(w,H)$ and $r=2\|H\|+1$.

\end{proof}

Our next goal is to bound the time and space of a reduced computation in the standard base.  We first achieve this for one-step computations:

\begin{lemma} \label{M_5(j) standard time}

Let $\pazocal{C}:W_0\to\dots\to W_t$ be a reduced computation of $\textbf{M}_5(j)$ in the standard base.

\begin{enumerate}[label=(\alph*)]

\item $t\leq\max(|W_0|_a,|W_t|_a)+2k$.  

\item If $j\neq5$ and either:

\begin{itemize}

\item $W_0$ is $\sigma(j,j\pm1)$-admissible, or

\item $W_0$ is a tame input configuration,

\end{itemize} 
then $|W_0|_a\leq2|W_t|_a$

\end{enumerate}

\end{lemma}

\begin{proof}

We prove both statements simultaneously for each $j$:

\textbf{1.} Suppose $j\in\{1,2,3\}$.  

For $0\leq x\leq s$, let $\pazocal{C}_x':W_{0,x}'\to\dots\to W_{t,x}'$ be the restriction of $\pazocal{C}$ to the $P_x'R_x'$-sector.  Then $\pazocal{C}_x'$ satisfies the hypotheses of \Cref{multiply two letters}, so that $t\leq\frac{1}{2}(|W_{0,x}'|_a+|W_{t,x}'|_a)\leq\max(|W_0|_a,|W_t|_a)$.

Moreover, if $W_0$ is $\sigma(j,j\pm1)$-admissible, then $\pazocal{C}_x'$ satisfies the hypotheses of \Cref{one alphabet historical words}, so that $t\leq|W_{t,x}'|_a$ and $|W_{0,x}'|_a\leq|W_{t,x}'|_a$.

Similarly, letting $\pazocal{C}_x'':W_{0,x}''\to\dots\to W_{t,x}''$ be the restriction of $\pazocal{C}$ to the $P_x''R_x''$-sector, if $W_0$ is $\sigma(j,j\pm1)$-admissible then $t\leq|W_{t,x}''|_a$ and $|W_{0,x}''|_a\leq|W_{t,x}''|_a$.

Suppose $UV$ is another two-letter subword of the standard base such that not every rule of $\textbf{M}_5(j)$ locks the $UV$-sector.  Let $\pazocal{C}_{UV}:W_{0,UV}\to\dots\to W_{t,UV}$ be the restriction of $\pazocal{C}$ to $UV$.  By \Cref{simplify rules}, it follows that $|W_{i,UV}|_a\leq|W_{i+1,UV}|_a+2$, so that $|W_{0,UV}|_a\leq|W_{t,UV}|_a+2t$.

Note that the number of such two-letter subwords $UV$ is at most:

\begin{itemize}

\item $n$ if $j=1$

\item $2$ if $j=2$

\item $m$ if $j=3$

\end{itemize}

As $s=n+m$, though, it follows that there are at least as many sectors of the form $P_x'R_x'$ and at least as many of the form $P_x''R_x''$.  Hence,
\begin{align*}
|W_0|_a&=\sum_{x=0}^s(|W_{0,x}'|_a+|W_{0,x}''|_a)+\sum_{UV}|W_{0,UV}|_a\leq\sum_{x=0}^s(|W_{t,x}'|_a+|W_{t,x}''|_a)+\sum_{UV}(|W_{t,UV}|_a+2t) \\
&\leq2\sum_{x=0}^s(|W_{t,x}'|_a+|W_{t,x}''|_a)+\sum_{UV}|W_{t,UV}|_a\leq2|W_t|_a
\end{align*}

\textbf{2.} Suppose $j=4$.

For all $0\leq x\leq s$, let $\pazocal{C}_x':W_{0,x}'\to\dots\to W_{t,x}'$ be the restriction of $\pazocal{C}$ to the subword $P_x'R_x'Q_{x,r}'$ of the standard base. 

As $\pazocal{C}_x'$ can be identified with a reduced computation of a primitive machine in the standard base, \Cref{primitive computations}(1) implies there exists a maximal subcomputation $\pazocal{D}_x':W_{\ell,x}'\to\dots\to W_{r,x}'$ such that:

\begin{itemize}

\item $|W_{\ell,x}'|_a=|W_{i,x}'|_a=|W_{r,x}'|_a$ for all $\ell\leq i\leq r$

\item $|W_{i-1,x}'|_a=|W_{i,x}'|_a+2$ for all $0\leq i\leq \ell$

\item $|W_{i+1,x}'|_a=|W_{i,x}'|_a+2$ for all $r\leq i \leq t$

\end{itemize}

As such, $\ell\leq\frac{1}{2}|W_{0,x}'|_a$, $t-r\leq\frac{1}{2}|W_{t,x}'|_a$, and $|W_{\ell,x}'|_a=|W_{r,x}'|_a\leq\max(|W_{0,x}'|_a,|W_{t,x}'|_a)$.

Moreover, \Cref{primitive computations} implies $r-\ell\leq 2k|W_{r,x}'|_a+2k-1$, so that $$t\leq (2k+1)\max(|W_{0,x}'|_a,|W_{t,x}'|_a)+2k-1$$
Note that the same argument may be applied to the restriction $\pazocal{C}_x'':W_{0,x}''\to\dots\to W_{t,x}''$ of $\pazocal{C}$ to the subword $Q_{x,\ell}''P_x''R_x''$ of the standard base.  Thus:
\begin{align*}
2(s+1)t&\leq\sum_{x=0}^s(2k+1)\left(\max(|W_{0,x}'|_a,|W_{t,x}'|_a)+\max(|W_{0,x}''|_a,|W_{t,x}''|_a\right)+2(2k-1) \\
&\leq4k(s+1)+\sum_{x=0}^s(2k+1)(|W_{0,x}'|_a+|W_{t,x}'|_a+|W_{0,x}''|_a+|W_{t,x}''|_a) \\
&\leq4k(s+1)+(2k+1)(|W_0|_a+|W_t|_a) \\
&\leq(4k+2)\max(|W_0|_a,|W_t|_a)+4k(s+1)
\end{align*}
Hence, $t\leq\frac{2k+1}{s+1}\max(|W_0|_a,|W_t|_a)+2k$.

Noting that $N=8s+n+8\leq 9(s+1)$, it then follows that $t\leq\frac{18k+9}{N}\max(|W_0|_a,|W_t|_a)+2k$.  But then (a) follows from the parameter choice $N>>k$.

Moreover, if $W_0$ is $\sigma(43)$- or $\sigma(45)$-admissible, then each of the computations $\pazocal{C}_x'$ and $\pazocal{C}_x''$ satisfies the hypotheses of \Cref{primitive computations}(5), implying $|W_0|_a\leq|W_t|_a$.

\textbf{3.} Suppose $j=0$.

The same arguments as Case 2 apply with $k$ replaced by $1$, yielding $t\leq\frac{3}{s+1}\max(|W_0|_a,|W_t|_a)+2$.  The statement thus follows by parameter choices for $N$ and $k$.

Moreover, if $W_0$ is $\sigma(01)$-admissible or a tame input configuration, then \Cref{primitive computations}(5) again implies $|W_0|_a\leq|W_t|_a$.

\textbf{4.} Hence, it suffices to assume $j=5$.

Then the restriction $\pazocal{C}_x':W_{0,x}'\to\dots\to W_{t,x}'$ of $\pazocal{C}$ to the $P_x'R_x'$-sector satisfies the hypotheses of \Cref{multiply one letter}, implying $t\leq|W_{0,x}'|_a+|W_{t,x}'|_a$.

Similarly, letting $\pazocal{C}_x'':W_{0,x}''\to\dots\to W_{t,x}''$ be the restriction of $\pazocal{C}$ to the $P_x''R_x''$-sector, \Cref{multiply one letter}(b) implies $t\leq|W_{0,x}''|_a+|W_{t,x}''|_a$. Hence, 
$$2(s+1)t\leq\sum_{x=0}^s(|W_{0,x}'|_a+|W_{0,x}''|_a+|W_{t,x}'|_a+|W_{t,x}''|_a)\leq|W_0|_a+|W_t|_a$$
so that $t\leq\frac{1}{s+1}\max(|W_0|_a,|W_t|_a)\leq\max(|W_0|_a,|W_t|_a)$.

\end{proof}

%\begin{lemma} \label{M_5 standard time transition}
%
%Let $\pazocal{C}:W_0\to\dots\to W_t$ be a reduced computation of $\textbf{M}_5(j)$ for some $j\in\{1,2,3,4\}$. If $W_0$ is $\sigma(j,j-1)$- or $\sigma(j,j+1)$-admissible, then $t\leq|W_t|_a+2k$
%
%%For $j\in\{2,3,4\}$, any reduced computation $\pazocal{C}:W_0\to\dots\to W_t$ of $\textbf{M}_5$ in the standard base with step history $(j-1,j)(j)(j,j+1)$ satisfies $t\leq\min(|W_0|_a,|W_t|_a)+2k+1$. 
%
%\end{lemma}
%
%\begin{proof}
%
%If $j\in\{2,3\}$, then let $\pazocal{C}':W_0'\to\dots\to W_t'$ be the restriction of $\pazocal{C}$ to some $Q_{i,\ell}P_i$-sector.  Then both the subcomputation $W_1'\to\dots\to W_{t-1}'$ and its inverse satisfy the hypotheses of \Cref{one alphabet historical words}, so that $t-2\leq\min(|W_1'|_a,|W_{t-1}'|_a)$.  But since transition rules change no tape words, it follows that $t\leq\min(|W_0'|_a,|W_t'|_a)+2$, implying the statement.
%
%Hence, it suffices to suppose $j=4$.  In this case, for all $0\leq x\leq s$, let $\pazocal{C}_x:W_{0,x}\to\dots\to W_{t,x}$ be the restriction of $\pazocal{C}$ to the subword $Q_{x,\ell}P_xQ_{x,r}$ of the standard base.  Then the subcomputation $W_{1,x}\to\dots\to W_{t-1,x}$ satisfies the hypotheses of \Cref{LR_k analogue}, so that $|W_{1,x}|_a=|W_{t-1,x}|_a$ and $t-2=2k|W_{1,x}|_a+2k-1$.
%
%In particular, fixing $x\in\{0,\dots,s\}$, $|W_0|_a=|W_t|_a=(s+1)|W_{0,x}|_a$ and $t=2k|W_{0,x}|_a+2k+1$.  But as in the proof of \Cref{M_5(j) standard time}, $s+1\geq N/4$, so that the parameter choice $N>>k$ implies the statement.
%
%\end{proof}

\begin{lemma} \label{M_5 standard time no (5)}

Let $\pazocal{C}:W_0\to\dots\to W_t$ be a reduced computation of $\textbf{M}_5$ in the standard base.  Suppose the step history of $\pazocal{C}$ does not contain any letter of the form $(5)$, $(45)$, or $(54)$.  Then:

\begin{enumerate}[label=(\alph*)]

\item $t\leq 128\max(|W_0|_a,|W_t|_a)+10k+4$

\item If either:

\begin{itemize} 

\item $W_0$ is admissible for some transition rule, or

\item $W_0$ is a tame input configuration,

\end{itemize}
then $|W_0|_a\leq 32|W_t|_a$.

\end{enumerate}

\end{lemma}

\begin{proof}

%Suppose the history $H$ of $\pazocal{C}$ contains no transition rule.  Then \Cref{M_5(j) standard time}(a) implies that $t\leq2\max(|W_0|_a,|W_t|_a)+2k$.  %So, since $\|W\|=|W|_a+N$ for any configuration $W$, the parameter choice $N>>k$ implies $t\leq 2\max(\|W_0\|,\|W_t\|)$.

By \Cref{M_5(j) standard time}(a), it suffices to assume that $H$ contains a transition rule.  Further, applying \Cref{M_5 step history} and perhaps passing to the inverse computation, it suffices to assume the step history of $\pazocal{C}$ is a subword of $(0)(01)(1)(12)(2)(23)(3)(34)(4)$.

As such, there exists $\ell\in\{1,2,3,4\}$ such that $H\equiv H_1\sigma_1\dots H_\ell\sigma_\ell H_{\ell+1}$, where each $\sigma_i$ is a transition rule and each (perhaps empty) factor $H_i$ contains no transition rule.

For $j\in\{1,\dots,\ell+1\}$, define $0\leq r_j\leq s_j\leq t$ such that $W_{r_j}\to\dots\to W_{s_j}$ is the subcomputation of $\pazocal{C}$ with history $H_j$.

By \Cref{M_5(j) standard time}, $s_j-r_j\leq\max(|W_{r_j}|_a,|W_{s_j}|_a)+2k$ for all $j$ and $|W_{r_j}|_a\leq2|W_{s_j}|_a$ for $j>1$.

As $|W_{r_{j+1}}|_a=|W_{s_j}|_a$, it then follows that $|W_{s_j}|_a\leq2^{\ell+1-j}|W_{s_{\ell+1}}|_a$.  So, since $s_{\ell+1}=t$ and $r_{j+1}=s_j+1$, $\sum_{j=1}^{\ell+1}|W_{s_j}|_a\leq2^{\ell+1}|W_t|_a\leq32|W_t|_a$.

Thus, as $r_1=0$,
\begin{align*}
t&\leq\sum_{j=1}^{\ell+1}(\max(|W_{r_j}|_a,|W_{s_j}|_a)+2k)+\ell\leq|W_0|_a+|W_{s_1}|_a+2\sum_{j=2}^{\ell+1}|W_{s_j}|_a+10k+4 \\
&\leq|W_0|_a-|W_{s_1}|_a+64|W_t|_a+10k+4
\end{align*}
Hence, (a) follows.
%Hence, the parameter choice $N>>k$ implies $t\leq486\max(\|W_0\|,\|W_t\|)$.

Moreover, if $W_0$ satisfies the hypotheses of (b), then the subcomputation with history $H_1$ satisfies the hypotheses of \Cref{M_5(j) standard time}(b), implying $|W_0|_a\leq2|W_{s_1}|_a\leq2^{\ell+1}|W_t|_a\leq32|W_t|_a$.

\end{proof}

\begin{lemma} \label{M_5 standard time}

For any reduced computation $\pazocal{C}:W_0\to\dots\to W_t$ of $\textbf{M}_5$ in the standard base, $t\leq c_0\max(\|W_0\|,\|W_t\|)$

\end{lemma}

\begin{proof}

If $\pazocal{C}$ satisfies the hypotheses of \Cref{M_5 standard time no (5)}, then part (a) of that statement yields the bound $t\leq128\max(|W_0|_a,|W_t|_a)+10k+4$.  The parameter choice $N>>k$ then allows us to assume $10k+4\leq128N$, so that $t\leq128\max(\|W_0\|,\|W_t\|)$.  Thus, the statement follows for $c_0\geq128$.

Hence, it suffices to assume that the step history of $\pazocal{C}$ contains $(5)$, $(45)$, or $(54)$.  \Cref{M_5 step history} then implies the existence of a nonempty maximal subcomputation $\pazocal{C}':W_r\to\dots\to W_s$ whose step history is a subword of $(45)(5)(54)$ and such that the step histories of the (possibly empty) subcomputations $\pazocal{C}_1:W_0\to\dots\to W_r$ and $\pazocal{C}_2:W_s\to\dots\to W_t$ contain no such letters.

Let $\pazocal{C}_5':V_0\to\dots\to V_\ell$ be the maximal subcomputation of $\pazocal{C}'$ with step history $(5)$.  Note that since transition rules do not alter tape words, $|V_0|_a=|W_r|_a$ and $|V_\ell|_a=|W_s|_a$.  Further, it is immediate that $\ell\geq s-r+2$.

The restriction of $\pazocal{C}_5'$ to any $P_i'R_i'$-sector (or any $P_i''R_i''$-sector) then satisfies the hypotheses of \Cref{multiply one letter}, so that $\ell\leq\max(|W_r|_a,|W_s|_a)$.  Hence, $s-r\leq\max(|W_r|_a,|W_s|_a)+2$.

%The restriction of $\pazocal{C}_5'$ to any sector of the standard base that is not locked satisfies the hypotheses of \Cref{multiply one letter}, so that $\ell\leq|V_0|_a+|V_\ell|_a$.  Hence, taking $N\geq4$, $s-r\leq\|W_r\|+\|W_s\|$.

We may then apply \Cref{M_5 standard time no (5)} to $\pazocal{C}_2$ and the inverse computation of $\pazocal{C}_1$, so that statement (b) implies $|W_r|_a\leq32|W_0|_a$ and $|W_s|_a\leq32|W_t|_a$ and as a result statement (a) implies that $r\leq4096|W_0|_a+10k+4$ and $t-s\leq4096|W_t|_a+10k+4$.

Thus, as above the statement follows by the parameter choices $c_0>>N>>k$.

\end{proof}

The history $H$ of a reduced computation of $\textbf{M}_5$ is called \textit{controlled} if $H^{\pm1}$ is of the form $\zeta_{2i}H'\zeta_{2i+1}H''\zeta_{2i+2}$ for some $0\leq i\leq k-1$ (where we take $\zeta_0=\sigma(34)$ and $\zeta_{2k}=\sigma(45)$), where $H'$ and $H''$ contain no connecting rules.  In other words, $H$ is controlled if a reduced computation with history $H^{\pm1}$ has step history $(4)$ and operates (in parallel) as a primitive machine.

Noting that all sectors are locked by at least one of the two connecting rules defining a controlled history, the next statement follows immediately from \Cref{locked sectors} and \Cref{primitive computations}:

\begin{lemma} [Lemma 4.18 of \cite{W}] \label{M_5 controlled}

Let $\pazocal{C}:W_0\to\dots\to W_t$ be a reduced computation of $\textbf{M}_5$ with controlled history $H$. Then the base of the computation is a reduced word and the admissible words $W_i$ are uniquely defined by the history $H$ and the base of $\pazocal{C}$.\newline 
Moreover, if $\pazocal{C}$ is a computation in the standard base, then $|W_0|_a=\dots=|W_t|_a$, $W_0$ contains the copy of the same word $H_1\in F(\Phi^+)$ over the right historical alphabet in each $P_i'R_i'$-sector, the copy of $H_1^{-1}$ over the right historical alphabet in each $P_i''R_i''$-sector, $\|H\|=2\|H_1\|+3$, and each $W_i$ is an accepted configuration.

\end{lemma}

\begin{lemma} \label{M_5 PR}

Let $\pazocal{C}:W_0\to\dots\to W_t$ be a nonempty reduced computation of $\textbf{M}_5$ whose base $B$ is either:

\begin{enumerate}[label=(\roman*)]

\item of the form $Q_{i,\ell}'P_i'R_i'Q_{i,r}'$, or

\item of the form $Q_{i,\ell}''P_i''R_i''Q_{i,r}''$

\end{enumerate}

Suppose the history $H$ of $\pazocal{C}$ does not contain a controlled subword and the step history contains a letter of the form $(j)$ for $j\leq 4$.  Then $t\leq 7\max(|W_0|_a,|W_t|_a)+6$ and no $W_i$ is an admissible subword of an end configuration.  Moreover, if $W_0$ is an admissible subword of a tame input configuration, then $|W_0|_a\leq|W_t|_a$ and $W_t$ is not an admissible subword of a tame input configuration.

\end{lemma}

\begin{proof}

By Lemmas \ref{M_2 bar step history}, \ref{M_3 second step history}, and \ref{M_5 step history 1}, perhaps passing to the inverse computation the step history of $\pazocal{C}$ is a subword $(0)(01)(1)(12)(2)(23)(3)(34)(4)(45)(5)$.

Suppose there exists a maximal subcomputation $\pazocal{D}:W_r\to\dots\to W_s$ of $\pazocal{C}$ with step history $(4)$ whose history contains at least 2 connecting rules.  If $B$ is of the form detailed in (i), let $B_\pazocal{D}$ be the subword $P_i'R_i'Q_{i,r}'$; otherwise, let $B_\pazocal{D}$ be the subword $Q_{i,\ell}''P_i''R_i''$.  In either case, the restriction of $\pazocal{D}$ to the subword $B_\pazocal{D}$ can be identified with a reduced computation of a primitive machine $\textbf{LR}_k$ or $\textbf{RL}_k$ in the standard base.  But then \Cref{primitive computations} implies $H$ must contain a controlled subword.

Hence, the step history of $\pazocal{C}$ cannot contain the subword $(34)(4)(45)$.  In particular, since the step history of $\pazocal{C}$ must contain a letter of the form $(j)$ for $j\leq4$, the step history must be a subword of $(0)(01)(1)(12)(2)(23)(3)(34)(4)$.

Now, let $\pazocal{C}_j:W_x\to\dots\to W_y$ be a maximal nonempty subcomputation of $\pazocal{C}$ with step history $(j)$.

If $j=4$, then as above the restriction of $\pazocal{C}_j$ to the subword $B_\pazocal{D}$ can be identified with a reduced computation of either the primitive machine $\textbf{LR}$ or $\textbf{RL}$ in the standard base.  \Cref{primitive computations} then implies $y-x\leq2\max(|W_x|_a,|W_y|_a)+1$.  Moreover, if $W_x$ is $\sigma(43)$-admissible, then $W_y$ is not $\sigma(43)$-admissible and $|W_x|_a\leq|W_y|_a$.

If $j\in\{1,2,3\}$, then the restriction of $\pazocal{C}_j$ to the subword $P_i'R_i'$ or $P_i''R_i''$ satisfies the hypotheses of \Cref{multiply two letters}, so that $y-x\leq\frac{1}{2}(|W_x|_a+|W_y|_a)$.  Moreover, if $W_x$ is $\sigma(j,j-1)$-admissible, then \Cref{one alphabet historical words} implies $|W_x|_a\leq|W_y|_a$; and similarly if $W_y$ is $\sigma(j,j+1)$-admissible, then $|W_y|_a\leq|W_x|_a$.

Finally, suppose $j=0$.  If $B$ is of the form (i), let $B'$ be the subword $Q_{i,\ell}'P_i'R_i'$; otherwise, let $B'$ be the subword $P_i''R_i''Q_{i,r}''$.  Then, the restriction of $\pazocal{C}_j$ to the subword $B'$ can be identified with a reduced computation of $\textbf{RL}$ or $\textbf{LR}$ in the standard base, so that \Cref{primitive computations} again implies:

\begin{enumerate}[label=(\alph*)]

\item $y-x\leq2\max(|W_x|_a,|W_y|_a)+1$

\item If $W_y$ is $\sigma(01)$-admissible, then $W_x$ is not $\sigma(01)$-admissible and $|W_y|_a\leq|W_x|_a$

\item If $W_x$ is an admissible subword of a tame input configuration, then $W_y$ is not a subword of such a configuration and $|W_x|_a\leq|W_y|_a$

\end{enumerate}

In each case, the bound is satisfied, and hence it suffices to assume that $H$ contains a transition rule.  Let $W_0\to\dots\to W_\ell$ and $W_z\to\dots\to W_t$ be the maximal (perhaps empty) subcomputations not containing transition rules.

By the arguments above, $|W_\ell|_a=\dots=|W_z|_a$ and $z-\ell\leq3|W_\ell|_a+4$.  Moreover, as $W_\ell$ and $W_z$ must be admissible for transition rules, $|W_\ell|_a\leq|W_0|_a$ and $|W_z|_a\leq|W_t|_a$, so that $\ell\leq2|W_0|_a+1$ and $t-z\leq2|W_t|_a+1$.  Thus, $t\leq7\max(|W_0|_a,|W_t|_a)+6$.

Moreover, if $W_0$ is an admissible subword of a tame input configuration, then $W_t$ cannot be by (c) above, while $|W_0|_a\leq|W_\ell|_a$ implies $|W_0|_a\leq|W_t|_a$.

\end{proof}

\begin{lemma} \label{M_5 long history}

Let $W$ be a tame input configuration.  Suppose $\pazocal{C}:W\equiv W_0\to\dots\to W_t$ is a nonempty reduced computation of $\textbf{M}_5$ such that either:

\begin{enumerate}[label=(\roman*)]

\item $W_t$ is the accept configuration,

\item $W_t$ is another tame input configuration, or

\item $|W|_a>c_0|W_t|_a$.

\end{enumerate}

Then:

\begin{enumerate}[label=(\alph*)]

\item $W$ is an accepted configuration.

\item $t>k$.

%\item There exists $w\in\pazocal{R}_1$ and $H\in F(\Phi^+)$ with $\|H\|\geq k$ such that $W\equiv I_5(w,H)$.  Moreover, $w$ is uniquely determined by $H$.

\item The sum of the lengths of all subcomputations of $\pazocal{C}$ with step history $(34)(4)(45)$ or $(54)(4)(43)$ is at least $\left(1-\frac{1}{C}\right)t$.

\end{enumerate}

\end{lemma}

\begin{proof}

If the step history of $\pazocal{C}$ is $(0)$, then \Cref{M_5(j) standard time}(b) implies $|W|_a\leq2|W_t|_a$.  Moreover, applying \Cref{primitive computations} to the restriction of $\pazocal{C}$ to any subword of the form $P_i'R_i'Q_{i,r}'$ implies $W_t$ is not a tame input configuration.

Further, if the step history of $\pazocal{C}$ does not contain any letter of the form $(5)$, $(45)$, or $(54)$, then \Cref{M_5 standard time no (5)}(b) implies $|W|_a\leq32|W_t|_a$.

Hence, taking $c_0\geq32$, \Cref{M_5 step history} implies the step history of $\pazocal{C}$ has prefix $(0)(1)(2)(3)(4)(5)$.  

Define the following subcomputations of $\pazocal{C}$:

\begin{itemize}

\item Let $\pazocal{C}_0:W_0\to\dots\to W_j$ be the maximal subcomputation with step history $(0)$.

\item Let $\pazocal{C}_2:W_0\to\dots\to W_\ell$ be the maximal subcomputation with step history $(0)(1)(2)(3)$.

\item Let $\pazocal{C}_3:W_0\to\dots\to W_r$ be the maximal subcomputation with step history $(0)(1)(2)(3)(4)$.

\end{itemize}

\Cref{M_5 standard time no (5)} then immediately implies $|W|_a\leq32|W_r|_a$.

Since $W_{j+1}$ and $W_r$ must be $\sigma(10)$- and $\sigma(45)$-admissible, respectively, the maximal subcomputation $W_{j+1}\to\dots\to W_r$ with step history $(1)(2)(3)(4)$ can be identified with a reduced computation of $\textbf{M}_3$ satisfying the hypotheses of \Cref{M_3 language}(2).  In particular, there exist $w\in\pazocal{R}_1$ and $H\in F(\Phi^+)$ with $\|H\|\geq k$ such that $W_{j+1}$ is the concatenation of $t$ with the natural copy of $I_4(w,H)$ in the hardware of $\textbf{M}_5$ and $r-j-1=(2k+1)\|H\|+2k+2$.

Then, since $W_j$ is $\sigma(01)$-admissible, an application of \Cref{primitive computations}(3) to $\pazocal{C}_0$ implies $W_0\equiv I_5(w,H)$, so that (a) follows from \Cref{M_5 language}(2).  Moreover, $j=2\|H\|+1$, so that $r=(2k+3)\|H\|+2k+4$.  Taking $k\geq1$ then implies $r\leq(2k+9)\|H\|$.

%Note that taking $k\geq1$, then $r\leq(2k+5)\|H\|$.

Next, note that since $W_\ell$ is $\sigma(34)$-admissible, the maximal subcomputation $W_{j+1}\to\dots\to W_\ell$ with step history $(1)(2)(3)$ can be identified with a reduced computation of $\textbf{M}_2$ satisfying the hypotheses of \Cref{M_2 language}(2).  Hence, $\ell-j-1=\|H\|+2$, so that $\ell=3\|H\|+4$.

Thus, the subcomputation $\pazocal{D}:W_\ell\to\dots\to W_{r+1}$ with step history $(34)(4)(45)$ has length $z=2k\|H\|+2k+1$.  (Note that this also follows from applying \Cref{LR_k analogue} to the restriction of the maximal subcomputation with step history (4) to any subword of the form $P_i'R_i'Q_{i,r}'$ or of the form $Q_{i,\ell}''P_i''R_i''$).

(i) Suppose $W_t$ is the accept configuration.

Taking $k\geq1$, \Cref{M_5 language}(2) implies $t=(2k+4)\|H\|+2k+5\leq(2k+11)\|H\|$, while \Cref{M_5 step history} implies $\pazocal{D}$ is the only subcomputation of $\pazocal{C}$ with step history $(34)(4)(45)$ or $(54)(4)(43)$.

Note that $\frac{2k}{2k+11}t\leq2k\|H\|<z$, so that $z>\left(1-\frac{11}{2k+11}\right)t$.  Thus, (c) follows from the parameter choice $k>>C$.

(ii) Suppose $W_t$ is another tame input configuration.

Then \Cref{M_5 step history} implies the step history of $\pazocal{C}$ is $(0)(1)(2)(3)(4)(5)(4)(3)(2)(1)(0)$.

Let $\pazocal{C}_1':W_s\to\dots\to W_t$ be the maximal subcomputation with step history $(4)(3)(2)(1)(0)$.  Then applying the same reasoning as above to the inverse computation of $\pazocal{C}_1'$ implies: 

\begin{itemize}

\item There exists $w'\in\pazocal{R}$ and $H'\in F(\Phi^+)$ such that $\|H'\|\geq k$ and $W_t\equiv I_5(w',H')$.  

\item $t-s=(2k+3)\|H'\|+2k+4\leq(2k+9)\|H'\|$.

\item The subcomputation $\pazocal{D}'$ with step history $(54)(4)(43)$ has length $z'=2k\|H'\|+2k+1$.

\end{itemize}

Further, the restriction of the subcomputation $W_{r+1}\to\dots\to W_{s-1}$ with step history $(5)$ to any $P_i'R_i'$-sector satisfies the hypotheses of \Cref{multiply one letter}, so that $s-r-2\leq\max(\|H\|,\|H'\|)$.

Hence, $t\leq(2k+10)(\|H\|+\|H'\|)+2\leq(2k+11)(\|H\|+\|H'\|)$, while $\pazocal{D}$ and $\pazocal{D}'$ are the only subcomputations of $\pazocal{C}$ with step history $(34)(4)(45)$ or $(54)(4)(43)$.
%the sum of the lengths of the subcomputations $\pazocal{D}$ and $\pazocal{D}'$ is $z+z'=2k(\|H+\|H'\|)+4k+2$.

Note then that $\frac{2k}{2k+11}t\leq2k(\|H\|+\|H'\|)<z+z'$, so that $z+z'>\left(1-\frac{11}{2k+11}\right)t$.  Thus, (c) again follows from the parameter choice $k>>C$.

(iii) Suppose $|W|_a>c_0|W_t|_a$.

As condition (Mv1) and \Cref{simplify rules} imply that each rule alters the input sector by at most one letter, then necessarily $\|w\|\leq\|H\|$ (note that this bound is certainly not sharp).  Hence, $|W|_a=2(s+1)\|H\|+\|w\|\leq(2s+3)\|H\|\leq N\|H\|$.%, and so the parameter choice $c_0>>N$ implies $\|H\|>N|W_t|_a$.

Let $\pazocal{C}_5:W_{r+1}\to\dots\to W_x$ be the maximal subcomputation with step history $(5)$.  Then the restriction of $\pazocal{C}_5$ to any $P_i'R_i'$- or $P_i''R_i''$-sector satisfies the hypotheses of \Cref{multiply one letter}.

As the tape word of $W_{r+1}$ in each $P_i'R_i'$-sector is the copy of $H$ over the corresponding right tape alphabet and the rules of $\textbf{M}_5(5)$ operate in parallel on these sectors, there exists a word $H'\in F(\Phi^+)$ such that the tape word of $W_x$ in each $P_i'R_i'$-sector is the copy of $H'$ over the corresponding right historical alphabet.  Analogously, the tape word of $W_x$ in each $P_i''R_i''$-sector is the copy of $(H')^{-1}$ over the corresponding right historical alphabet.

Applying \Cref{multiply one letter} then implies $x-r-1\leq\|H\|+\|H'\|$.

If $t>x$, then \Cref{M_5 step history} implies the step history of the subcomputation $W_x\to\dots\to W_t$ is a prefix of $(54)(4)(3)(2)(1)(0)$.  In this case $|W_{x+1}|_a=|W_x|_a$, so that \Cref{M_5 standard time no (5)}(b) implies $|W_x|_a\leq32|W_t|_a$ and, as a result, \Cref{M_5 standard time no (5)}(a) that $t-x\leq4096|W_t|_a+10k+5$.

Now, this implies $N\|H\|\geq|W|_a>c_0|W_t|_a\geq\frac{c_0}{32}|W_x|_a\geq\frac{c_0}{32}\|H'\|$, so that the parameter choice $c_0>>N$ implies $\|H\|>\|H'\|$.  So, $x-r\leq\|H\|+\|H'\|+1\leq2\|H\|$, implying $x\leq(2k+11)\|H\|$.

Moreover, taking $k\geq1$ and $c_0\geq4096N$ through the parameter choice $c_0>>N$, then
$$t-x\leq \frac{c_0|W_t|_a}{N}+10k+5<\frac{|W|_a}{N}+10k+5\leq\|H\|+10k+5\leq16\|H\|$$
Hence, $t\leq(2k+27)\|H\|$, so that $\frac{2k}{2k+27}t\leq2k\|H\|<z$.  Thus, (c) again follows from the parameter choice $k>>C$.

\end{proof}

Finally, the following statement details the structure of a particular type of one-step computation in the standard base.

\begin{lemma} [Compare to Lemma 4.14 of \cite{OS19} and Lemma 4.12 of \cite{W}] \label{M_5 one-step}

Let $\pazocal{C}:W_0\to\dots\to W_t$ be a reduced computation of $\textbf{M}_5(j)$ in the standard base. Suppose there exists $z\in\{1,\dots,t\}$ such that $|W_z|_a>2|W_0|_a$. Then there exist two-letter subwords $U_\ell V_\ell$ and $U_rV_r$ of the standard base such that:

\begin{enumerate}

\item Letting $\pazocal{C}_\ell:W_{0,\ell}\to\dots\to W_{t,\ell}$ and $\pazocal{C}_r:W_{0,r}\to\dots\to W_{t,r}$ be the restrictions of $\pazocal{C}$ to the $U_\ell V_\ell$- and $U_rV_r$-sectors, respectively, $|W_{z,\ell}|_a<\dots<|W_{t,\ell}|_a$ and $|W_{z,r}|_a<\dots<|W_{t,r}|_a$.

\item Every rule of the subcomputation $W_z\to\dots\to W_t$ multiplies the $U_\ell V_\ell$-sector by one letter on the left and the $U_rV_r$-sector by one letter on the right.

\end{enumerate}

\end{lemma}

\begin{proof}

\textbf{1.} Suppose $j\in\{1,2,3\}$.

For $0\leq x\leq s$, let $\pazocal{C}_x':W_{0,x}'\to\dots\to W_{t,x}'$ be the restriction of $\pazocal{C}$ to the $P_x'R_x'$-sector.  Then $\pazocal{C}_x'$ satisfies the hypotheses of \Cref{multiply two letters}, so that (c) implies $z\leq\max(|W_{0,x}'|_a,|W_{z,x}'|_a)$.

Similarly, letting $\pazocal{C}_x'':W_{0,x}''\to\dots\to W_{t,x}''$ be the restriction of $\pazocal{C}$ to the $P_x''R_x''$-sector, \Cref{multiply two letters}(c) implies $z\leq\max(|W_{0,x}''|_a,|W_{z,x}''|_a)$.

Suppose $UV$ is another two-letter subword of the standard base such that not every rule of $\textbf{M}_5(j)$ locks the $UV$-sector.  As in the proof of \Cref{M_5(j) standard time}, the number of such subwords is at most $s$.

Letting $\pazocal{C}_{UV}:W_{0,UV}\to\dots\to W_{t,UV}$ be the restriction of $\pazocal{C}$ to the $UV$-sector, \Cref{simplify rules} implies $|W_{z,UV}|_a\leq|W_{0,UV}|_a+2z$.

Suppose $|W_{0,x}'|_a\geq|W_{z,x}'|_a$ and $|W_{0,x}''|_a\geq|W_{z,x}''|_a$ for all $0\leq x\leq s$.  Then for every $UV$ as above, $|W_{z,UV}|_a\leq|W_{0,UV}|_a+|W_{0,x}'|_a+|W_{0,x}''|_a$.  As a result, $|W_z|_a\leq2|W_0|_a$.

Hence, by hypothesis there exists $x\in\{0,\dots,s\}$ such that $|W_{0,x}'|_a<|W_{z,x}'|_a$ or $|W_{0,x}''|_a<|W_{z,x}''|_a$.  

Suppose there exists $0\leq y\leq z-1$ such that $|W_{y+1,x}'|_a>|W_{y,x}'|_a$.  Then \Cref{multiply two letters}(a) implies $|W_{y,x}'|_a<\dots<|W_{t,x}'|_a$, so that the statement holds for $U_\ell V_\ell=U_rV_r=P_x'R_x'$.

Otherwise, there exists $0\leq y\leq z-1$ such that $|W_{y+1,x}''|_a>|W_{y,x}''|_a$.  But then again \Cref{multiply two letters}(a) implies the statement holds for $U_\ell V_\ell=U_rV_r=P_x''R_x''$.

\textbf{2.} Suppose $j=4$.  

For $0\leq x\leq s$, let $\pazocal{C}_x':W_{0,x}'\to\dots\to W_{t,x}'$ be the restriction of $\pazocal{C}$ to the subword $P_x'R_x'Q_{x,r}'$ of the standard base.  Then $\pazocal{C}_x'$ can be identified with a reduced computation of $\textbf{LR}_k$ in the standard base, and so satisfies the hypotheses of \Cref{primitive computations}.

Similarly, the restriction $\pazocal{C}_x'':W_{0,x}''\to\dots\to W_{t,x}''$ of $\pazocal{C}$ to the subword $Q_{x,\ell}''P_x''R_x''$ can be identified with a reduced computation of $\textbf{RL}_k$ in the standard base.

Note that $|W_0|_a=\sum_{x=0}^s(|W_{0,x}'|_a+|W_{0,x}''|_a)$ and $|W_z|_a=\sum_{x=0}^s(|W_{z,x}'|_a+|W_{z,x}''|_a)$.  By hypothesis, there then exists $x\in\{0,\dots,s\}$ such that $|W_{0,x}'|_a<|W_{z,x}'|_a$ or $|W_{0,x}''|_a<|W_{z,x}''|_a$.  

Suppose there exists $0\leq y\leq z-1$ such that $|W_{y+1,x}'|_a>|W_{y,x}'|_a$.  Then \Cref{primitive computations}(1) implies $|W_{y,x}'|_a<\dots<|W_{t,x}'|_a$, so that the statement holds for $U_\ell V_\ell=R_x'Q_{x,r}'$ and $U_rV_r=P_x'R_x'$.

Otherwise, there exists $0\leq y\leq z-1$ such that $|W_{y+1,x}''|_a>|W_{y,x}''|_a$.  But then again \Cref{primitive computations}(1) implies the statement for $U_\ell V_\ell=P_x''R_x''$ and $U_rV_r=Q_{x,\ell}''P_x''$.

\textbf{3.} Suppose $j=0$.

For $0\leq x\leq s$, let $\pazocal{C}_x':W_{0,x}'\to\dots\to W_{t,x}'$ and $\pazocal{C}_x'':W_{0,x}''\to\dots\to W_{t,x}''$ be the restrictions of $\pazocal{C}$ to the subwords $Q_{x,\ell}'P_x'R_x'$ and $P_x''R_x''Q_{x,r}''$, respectively.  Analogous to the previous case, both of these computations can be identified with a reduced computation of a primitive machine in the standard base, and so satisfies the hypotheses of \Cref{primitive computations}.

Let $w\in F(X)$ such that $W_0$ has $w$ written in its $Q_0Q_1$-sector.  Then since every rule of $\textbf{M}_5(0)$ leaves the tape word of the $Q_0Q_1$-fixed and locks any other sector not addressed, for all $i$ we have $|W_i|_a=\|w\|+\sum_{x=0}^s(|W_{i,x}'|_a+|W_{i,x}''|_a)$.  

As in the previous case, the hypothesis then implies there must then exist $x\in\{0,\dots,s\}$ such that $|W_{0,x}'|_a<|W_{z,x}'|_a$ or $|W_{0,x}''|_a<|W_{z,x}''|_a$.

If there exists $0\leq y\leq z-1$ such that $|W_{y+1,x}'|_a>|W_{y,x}'|_a$, then by \Cref{primitive computations}(1) the statement holds for $U_\ell V_\ell=P_x'R_x'$ and $U_rV_r=Q_{x,\ell}'P_x'$.

Otherwise, there exists $0\leq y\leq z-1$ such that $|W_{y+1,x}''|_a>|W_{y,x}''|_a$, so that the statement holds for $U_\ell V_\ell=R_x''Q_{x,r}''$ and $U_rV_r=P_x''R_x''$.

\textbf{4.} Suppose $j=5$.

For $0\leq x\leq s$, let $\pazocal{C}_x':W_{0,x}'\to\dots\to W_{t,x}'$ be the restriction of $\pazocal{C}$ to the $P_x'R_x'$-sector.  Then $\pazocal{C}_x'$ satisfies the hypotheses of \Cref{multiply one letter}.

Similarly, the restriction $\pazocal{C}_x'':W_{0,x}''\to\dots\to W_{t,x}''$ of $\pazocal{C}$ to the $P_x''R_x''$-sector satisfies the hypotheses of \Cref{multiply one letter}.

Suppose $|W_{y+1,x}'|_a\leq|W_{y,x}'|_a$ for all $0\leq y\leq z-1$ and all $0\leq x\leq s$.  Then by construction $|W_{z,x}'|_a=|W_{0,x}'|_a-z$ for all $x$.  But at the same time $|W_{z,x}''|_a\leq|W_{0,x}''|_a+z$ for all $x$, so that $|W_z|_a\leq|W_0|_a$.

Hence, by hypothesis there must exist $x_1\in\{0,\dots,s\}$ such that $|W_{y_1+1,x_1}'|_a>|W_{y_1,x_1}'|_a$ for some $0\leq y_1\leq z-1$.  \Cref{multiply one letter}(c) then implies $|W_{y_1,x_1}'|_a<\dots<|W_{t,x_1}'|_a$.

By the symmetric argument, there exists $x_2\in\{0,\dots,s\}$ such that $|W_{y_2+1,x_2}''|_a>|W_{y_2,x_2}''|_a$ for some $0\leq y_2\leq z-1$, so that $|W_{y_2,x_2}''|_a<\dots<|W_{t,x_2}|_a$.

Hence, the statement follows for $U_\ell V_\ell=P_{x_2}''R_{x_2}''$ and $U_rV_r=P_{x_1}'R_{x_1}'$.

\end{proof}

\medskip

%%%%%%%%%%%%%%%%%%%%%%%%%%%%%%%%%%%%%%%%%%%%%%%%%%

\subsection{Faulty computations of $\textbf{M}_5$} \

In this section, we establish a bound on the space of a reduced computation of $\textbf{M}_5$ with faulty base.  Our first step toward this goal is to achieve this bound for one-step computations, specifically for the only step left outstanding by \Cref{barM_4(j) faulty}.

\begin{lemma} \label{M_5(j) faulty}

For every reduced computation $\pazocal{C}:W_0\to\dots\to W_t$ of $\textbf{M}_5(0)$ with faulty base, $|W_i|_a\leq c_0(|W_0|_a+|W_t|_a)$ for all $0\leq i\leq t$.

\end{lemma}

\begin{proof}

Let $H$ be the history and $B$ be the base of $\pazocal{C}$.

As in the proofs of \Cref{barM_4(j) faulty} and \Cref{barM_4 faulty}, it suffices to assume that $|W_i|_a>\max(|W_0|_a,|W_t|_a)$ for all $0<i<t$.

As in the first two cases of the proof of \Cref{barM_4(j) faulty}, if $H$ does not contain a connecting rule, then the restriction of $\pazocal{C}$ to any two-letter subword of $B$ either has fixed tape word, satisfies the hypotheses of \Cref{multiply one letter}, or satisfies the hypotheses of \Cref{unreduced base}, so that the statement holds for $c_0\geq1$.

Hence, it suffices to assume $H\equiv H_1\chi^{\pm1}H_2$ where $\chi$ is a connecting rule and $H_2$ contains no connecting rule.  Let $\pazocal{C}_2:W_s\to\dots\to W_t$ be the subcomputation with history $H_2$.

Suppose $\chi$ locks the $Q_{i,\ell}'P_i'$- and $R_i''Q_{i,r}''$-sectors.  

Then any occurrence of a letter of the form $(P_i')^{\pm1}$ in $B$ must be part of a subword of a cyclic permutation of $B$ of the form $(Q_{i,\ell}'P_i'R_i')^{\pm1}$ or $Q_{i,\ell}'P_i'(P_i')^{-1}(Q_{i,\ell}')^{-1}$.  For $W_s'\to\dots\to W_t'$ the restriction of $\pazocal{C}_2$ to such a subword satisfies either the hypotheses of \Cref{primitive computations}(5) or those of \Cref{unreduced base}, so that $|W_s'|_a\leq|W_t'|_a$.

Similarly, any occurrence of a letter of the form $(R_i'')^{\pm1}$ in $B$ must be a part of such a subword of the form $(P_i''R_i''Q_{i,r}'')^{\pm1}$ or $(Q_{i,r}'')^{-1}(R_i'')^{-1}R_i''Q_{i,r}''$, in which case we may apply the same statements to the restriction of $\pazocal{C}_2$.

But any other two-letter subword of $B$ must correspond to a sector whose tape word is fixed by all rules of $\textbf{M}_5(0)$, so that $|W_s|_a\leq|W_t|_a$.

Thus, we may assume that $\chi$ locks the $P_i'R_i'$- and $P_i''R_i''$-sectors.  But then the analogous argument again yields the contradiction $|W_s|_a\leq|W_t|_a$.

\end{proof}

\begin{lemma} \label{M_5 faulty}

If $\pazocal{C}:W_0\to\dots\to W_t$ is a reduced computation of $\textbf{M}_5$ whose base $B$ is faulty, then $|W_i|_a\leq c_1\max(|W_0|_a,|W_t|_a)$ for all $0\leq i\leq t$.

\end{lemma}

\begin{proof}

As in the proofs of the previous statements of this nature, it suffices to assume that $|W_i|_a>\max(|W_0|_a,|W_t|_a)$ for all $0<i<t$.

By \Cref{barM_4 faulty} and \Cref{M_5(j) faulty}, it suffices to assume that the history $H$ of $\pazocal{C}$ contains a letter of the form $\sigma(01)^{\pm1}$.  Perhaps passing to the inverse computation, we may then assume the existence of a nonempty maximal subcomputation $\pazocal{C}_0:W_r\to\dots\to W_s$ with step history $(0)$ such that the transition $W_{r-1}\to W_r$ is given by $\sigma(01)$.

Suppose $B$ contains a letter of the form $(P_i')^{\pm1}$ or one of the form $(R_i'')^{\pm1}$.  

Then a cyclic permutation of $B$ contains a subword of the form $(Q_{i,\ell}'P_i'R_i')^{\pm1}$, $Q_{i,\ell}'P_i'(P_i')^{-1}(Q_{i,\ell}')^{-1}$, $(P_i''R_i''Q_{i,r}'')^{\pm1}$, or $(Q_{i,r}'')^{-1}(R_i'')^{-1}R_i''Q_{i,r}''$.  The restriction $W_r'\to\dots\to W_s'$ of $\pazocal{C}_0$ to this subword then either satisfies the hypotheses of \Cref{primitive computations} or \Cref{primitive unreduced}.  In either case, this implies $|W_r'|_a\leq|W_s'|_a$ and $W_s$ is not $\sigma(01)$-admissible, so that $s=t$.

But then applying the same argument to any other occurrences of letters of the form $(P_i')^{\pm1}$ or $(R_i'')^{\pm1}$ in $B$ and noting that every rule of $\textbf{M}_5(0)$ fixes the tape word of any sector not containing such a letter, we then deduce that $|W_r|_a\leq|W_t|_a$.

Hence, it suffices to assume that $B$ contains neither a letter of the form $(P_i')^{\pm1}$ nor one of the form $(R_i'')^{\pm1}$, and so is a cyclic permutation of:
$$(R_s')^{-1}R_s'Q_{s,r}'Q_{0,\ell}''P_0''(P_0'')^{-1}(Q_{0,\ell}'')^{-1}(Q_{s,r}')^{-1}(R_s')^{-1}$$
By the reasoning above, it follows that every rule of $\textbf{M}_5(0)$ fixes the tape words of an admissible word with base $B$, so that $|W_r|_a=\dots=|W_s|_a$.  It then follows from the hypothesis that $t>s$, so that there exists a maximal subcomputation $\pazocal{C}_1:W_{s+1}\to\dots\to W_x$ with step history $(1)$.

Let $\pazocal{C}_1':W_{s+1}'\to\dots\to W_x'$ and $\pazocal{C}_1'':W_{s+1}''\to\dots\to W_x''$ be the restrictions of $\pazocal{C}_1$ to the $(R_s')^{-1}R_s'$- and $P_0''(P_0'')^{-1}$-sectors, respectively.  Since every rule of $\textbf{M}_5(1)$ locks the $R_i'Q_{i,r}'$-, $Q_{s,r}'Q_{0,\ell}''$-, and $Q_{i,\ell}''P_i''$-sectors, $|W_i|_a=|W_i'|_a+|W_i''|_a$ for all $s+1\leq i\leq x$.

For each $s+1\leq i\leq x$, let $u_i'$ and $u_i''$ be the tape words of $W_i'$ and $W_i''$, respectively.  Note that since $W_{s+1}$ is $\sigma(10)$-admissible, $u_{s+1}'$ and $u_{s+1}''$ must both be nonempty words over the corresponding left historical alphabets.

As a result, if $\pazocal{C}_1$ is nonempty, then both $\pazocal{C}_1'$ and $\pazocal{C}_1''$ satisfy the hypotheses of \Cref{M_4 restriction}(b), so that $W_x$ is not $\sigma(10)$-admissible, $|W_{s+1}'|_a<|W_x'|_a$, and $|W_{s+1}''|_a<|W_x''|_a$.  But then $|W_r|_a<|W_x|_a$, so that $t>x$ and $W_x$ must necessarily be $\sigma(12)$-admissible.

Conversely, if $\pazocal{C}_1$ is empty, then since transition rules do not alter tape words by hypothesis $t>x$ and, as $H$ is reduced, again $W_x$ must necessarily be $\sigma(12)$-admissible.

In either case, there exists a (perhaps empty) word $v_1$ over the copy of $\Phi_1^+$ in the corresponding right historical alphabet such that $u_x'=v_1^{-1}u_{s+1}'v_1$.  Since $u_{s+1}'$ is formed over the corresponding left historical alphabet, none of its letters can cancel in this product.  So, since $W_x$ is $\sigma(12)$-admissible, no letter of $u_{s+1}'$ can be a copy of a letter of $\Phi_1$ in the corresponding left alphabet.  

The analogous argument may be applied to $u_{s+1}''$.

Now, as above there must exist a maximal subcomputation $\pazocal{C}_2:W_{x+1}\to\dots\to W_y$ with step history (2).  Applying the same arguments to the restrictions of this computation then implies $t>y$ and $W_y$ must necessarily be $\sigma(23)$-admissible.  

Moreover, letting $u_y'$ be the tape word of $W_y$ in the $(R_s')^{-1}R_s'$-sector, then there exists a (perhaps empty) word $v_2$ over the copy of $\Phi_2^+$ in the corresponding right historical alphabet such that $u_y'=(v_1v_2)^{-1}u_{s+1}'v_1v_2$.  Again, no letter of $u_{s+1}'$ can cancel in this product. So, since $W_y$ is $\sigma(23)$-admissible, every letter of $u_{s+1}'$ must be a copy of a letter of $\Phi_3$ in the corresponding left alphabet, {\frenchspacing i.e. $W_{y+1}\equiv W_y\cdot\sigma(23)$ cannot be $\sigma(34)$-admissible.}

This implies there must be a nonempty maximal subcomputation $\pazocal{C}_3:W_{y+1}\to\dots\to W_z$ with step history $(3)$.  

The restriction $\pazocal{C}_3':W_{y+1}'\to\dots\to W_z'$ to the $(R_s')^{-1}R_s'$-sector then satisfies the hypotheses of \Cref{M_4 restriction}(b), so that $W_z$ is not $\sigma(32)$-admissible.  Moreover, letting $u_z'$ be the tape word of $W_z'$, there exists a (nonempty) word $v_3$ over the copy of $\Phi_3^+$ in the corresponding right alphabet such that $u_z'=(v_1v_2v_3)^{-1}u_{s+1}'v_1v_2v_3$.  As no letter of $u_{s+1}'$ cancels in this product, $W_z$ cannot be $\sigma(34)$-admissible.  Hence, $t=z$.

But then the same argument that was applied to $\pazocal{C}_1$ above implies $|W_{y+1}|_a<|W_t|_a$.

\end{proof}

\medskip

%%%%%%%%%%%%%%%%%%%%%%%%%%%%%%%%%%%%%%%%%%%%%%%%%%

\subsection{The machines $\textbf{M}_{6,1}$ and $\textbf{M}_{6,2}$} \

As in \cite{WCubic}, \cite{W}, \cite{WMal}, the main machine of our construction is built from two concatenations of many copies of the machine $\textbf{M}_5$.  As in those settings, this construction is facilitated by the introduction of a final pair of auxiliary machines, denoted $\textbf{M}_{6,1}$ and $\textbf{M}_{6,2}$.

Recall that the standard base of $\textbf{M}_5$ is $\{t\}B_4$, where $B_4$ is the standard base of $\textbf{M}_4$, i.e.
$$B_4\equiv Q_0\dots Q_{n-1}~Q_{0,\ell}'P_0'R_0'Q_{0,r}'\dots Q_{s,\ell}'P_s'R_s'Q_{s,r}'~Q_{0,\ell}''P_0''R_0''Q_{0,r}''\dots Q_{s,\ell}''P_s''R_s''Q_{s,r}''$$
For $1\leq i\leq L$, let $\{t(i)\}B_4(i)$ be a copy of the standard base of $\textbf{M}_5$, where each letter of $B_4(i)$ is denoted with the letter $i$ {\frenchspacing(i.e. $B_4(i)=Q_0(i)Q_1(i)\dots$)}.

The standard base of each of the machines $\textbf{M}_{6,j}$ is then:
$$\{t(1)\}B_4(1)\{t(2)\}B_4(2)\dots\{t(L)\}B_4(L)$$
For any letter of $(\{t(i)\}B_4(i))^{\pm1}$, the index $i$ is called its \textit{coordinate}.

Each part of the state letters is in bijection with the corresponding of $\textbf{M}_5$, with the start and end letters assigned in the natural way.  

The tape alphabet of any sector formed by a one-letter part $\{t(i)\}$ of the standard base (including the $Q_{s,r}''(L)\{t(1)\}$-sector) is defined to be empty. The tape alphabets of all other sectors arise from $\textbf{M}_5$ in the natural way. 

Every sector corresponding to an input sector of $\textbf{M}_5$ is taken as an input sector of the machine, {\frenchspacing i.e. the} $Q_0(i)Q_1(i)$-sectors, the $P_j'(i)R_j'(i)$-sectors, and the $P_j''(i)R_j''(i)$-sectors.  An input configuration of $\textbf{M}_{6,j}$ thus corresponds to the concatenation of $L$ input configurations of $\textbf{M}_5$.  Accordingly, such an input configuration is called \textit{tame} if all of the input configurations of $\textbf{M}_5$ comprising it are tame.

While the hardware of these two machines are analogous, though, the software has a slight but fundamental difference.

The rules of $\textbf{M}_{6,1}$ are in correspondence with those of $\textbf{M}_5$, with each rule operating in parallel on each of the copies of the standard base of $\textbf{M}_5$ in the same way as its corresponding rule.

Naturally, there arise submachines $\textbf{M}_{6,1}(0),\dots,\textbf{M}_{6,1}(5)$ corresponding to the submachines of $\textbf{M}_5$. As such, the definition of step history and controlled history extend to computations of $\textbf{M}_{6,1}$.

The parallel nature of this construction implies that many of the statements of the previous sections have natural analogues to computations in $\textbf{M}_{6,1}$. For example, given $w\in F(X)$ and $H\in F(\Phi^+)$, letting $I_6(w,H)$ be the input configuration obtained by concatenating $L$ copies of $I_5(w,H)$, the following is the obvious analogue of Lemma \ref{M_5 language}:

\begin{lemma} \label{M_{6,1} language} \

\begin{enumerate}

\item For any $w\in\pazocal{R}_1$, the configuration $I_6(w,H_w)$ is accepted by $\textbf{M}_{6,1}$, where $H_w\in F(\Phi^+)$ is the word given by \Cref{M_5 language}(1).

\item The tame input configuration $W$ of $\textbf{M}_{6,1}$ is accepted if and only if there exists $w\in\pazocal{R}_1$ and $H\in F(\Phi^+)$ such that $W\equiv I_6(w,H)$.  Moreover, in this case $\|H\|\geq k$, $w$ is uniquely determined by $H$, and there exists a unique accepting computation of $W$, which has length $(2k+4)\|H\|+2k+5$.

\end{enumerate}

\end{lemma}

The rules of $\textbf{M}_{6,2}$, on the other hand, are given in just the same way, but with on foundational difference: Each rule locks the first input sector, i.e the $Q_0(1)Q_1(1)$-sector.

Despite this difference, though, the definitions of submachines and step histories again adapt naturally, as do many of the statements of previous sections.  For example, letting $J_6(w,H)$ be the input configuration obtained from $I_6(w,H)$ by erasing the copy of $w$ in the $Q_0(1)Q_1(1)$-sector, the following analogue of \Cref{M_{6,1} language} arises in just the same way:

\begin{lemma} \label{M_{6,2} language} \

\begin{enumerate}

\item For any $w\in\pazocal{R}_1$, the configuration $J_6(w,H_w)$ is accepted by $\textbf{M}_{6,2}$, where $H_w\in F(\Phi^+)$ is the word given by \Cref{M_5 language}(1).

\item The tame input configuration $W$ of $\textbf{M}_{6,2}$ is accepted if and only if there exists $w\in\pazocal{R}_1$ and $H\in F(\Phi^+)$ such that $W\equiv J_6(w,H)$.  Moreover, in this case $\|H\|\geq k$, $w$ is uniquely determined by $H$, and there exists a unique accepting computation of $W$, which has length $(2k+4)\|H\|+2k+5$.

\end{enumerate}

\end{lemma}

\medskip

%%%%%%%%%%%%%%%%%%%%%%%%%%%%%%%%%%%%%%%%%%%%%%%%%%

\section{The machine $\textbf{M}$} \label{sec-main-machine}

\subsection{Definition of the machine} \

We now use the auxiliary machines $\textbf{M}_{6,1}$ and $\textbf{M}_{6,2}$ to construct the main cyclic machine $\textbf{M}$ that is sufficient for the proof of \Cref{main-theorem}.  This construction is analogous to those of \cite{WCubic}, \cite{W}, and \cite{WMal}.

Similar to $\textbf{M}_{6,1}$ and $\textbf{M}_{6,2}$, the standard base of $\textbf{M}$ is of the form $\{t(1)\}B_4(1)\dots\{t(L)\}B_4(L)$, with the tape alphabets and input designation assigned in the same way.  However, each of the parts making up $B_4(i)$ consists of more state letters than its counterparts in $\textbf{M}_{6,1}$ and $\textbf{M}_{6,2}$.

Specifically, any part of the standard base that is not a one-letter part $\{t(i)\}$ consists of a copy of the corresponding part of the standard base of $\textbf{M}_{6,1}$, a (disjoint) copy of the corresponding part of the standard base of $\textbf{M}_{6,2}$, and two new letters which function as the part's start and end letters. The accept configuration of $\textbf{M}$ is denoted $W_{ac}$.

The set of rules $\Theta$ of $\textbf{M}$ is partitioned into two symmetric sets, $\Theta_1$ and $\Theta_2$.  The positive rules of each consist of a set of `working' rules and two more transition rules.  Unlike in the constructions outlined in previous sections, though, these two sets are not concatenated in order to force them to run sequentially, but rather in order to force them to operate `one or the other'.

The rules of $\Theta_1^+$ are defined as follows:

\begin{itemize}

\item The transition rule $\sigma(s)_1$ locks all sectors other than the input sectors and switches the state letters from the start state of $\textbf{M}$ to the copy of the start state of $\textbf{M}_{6,1}$.  Its domain in the historical input sectors is the corresponding left historical alphabet.

\item The `working' rules of $\Theta_1^+$ are copies of the positive rules of the machine $\textbf{M}_{6,1}$.

\item The transition rule $\sigma(a)_1$ locks all sectors and switches the state letters from the copies of the end state letters of $\textbf{M}_{6,1}$ to the end state letters of $\textbf{M}$.

\end{itemize}

Similarly, the rules of $\Theta_2^+$ are defined as follows:

\begin{itemize}

\item The transition rule $\sigma(s)_2$ locks each of the sectors locked by $\sigma(s)_1$, but also locks the $Q_0(1)Q_1(1)$-sector.  It switches the state letters from the start state of $\textbf{M}$ to the copy of the start state of $\textbf{M}_{6,2}$.  Again, its domain in the historical input sectors is the corresponding left historical alphabet.

\item The `working' rules of $\Theta_2^+$ are copies of the positive rules of the machine $\textbf{M}_{6,2}$.

\item The transition rule $\sigma(a)_2$ locks all sectors and switches the state letters from the copies of the end state letters of $\textbf{M}_{6,2}$ to the end state letters of $\textbf{M}$.

\end{itemize}

By the definition of the rules, one might infer that the first input sector $Q_0(1)Q_1(1)$ is of particular significance. Hence, we refer to it as the \textit{`special' input sector}.

Given $w\in F(X)$ and $H\in F(\Phi^+)$, note that the natural copy of $I_6(w,H)$ (respectively $J_6(w,H)$) in the hardware of this machine is $\sigma(s)_1^{-1}$-admissible (respectively $\sigma(s)_2^{-1}$-admissible).  The configuration $I(w,H)$ (respectively $J(w,H)$) is then defined to be the input configuration satisfying $I(w,H)\equiv I_6(w,H)\cdot\sigma(s)_1^{-1}$ (respectively $J(w,H)\equiv J_6(w,H)\cdot\sigma(s)_2^{-1}$).  Note that both $I(w,H)$ and $J(w,H)$ are $\sigma(s)_1$-admissible, while $I(w,H)$ is not $\sigma(s)_2$-admissible if $w\neq1$.

\medskip

%%%%%%%%%%%%%%%%%%%%%%%%%%%%%%%%%%%%%%%%%%%%%%%%%%

\subsection{Standard computations of $\textbf{M}$} \

Next, the definition of the step history to computations of $\textbf{M}$. To this end, let the letters $(s)_i^{\pm1}$ and $(a)_i^{\pm1}$ represent the transition rules $\sigma(s)_i^{\pm1}$ and $\sigma(a)_i^{\pm1}$ of $\Theta_i$, respectively, and add the subscript $i$ to each letter of the step history of a subcomputation acting as $\textbf{M}_{6,i}$.

So, an example of a step history of a reduced computation of $\textbf{M}$ is $(s)_1(0)_1(01)_1(1)_1$, while a general step history is some concatenation of the letters 
$$\begin{Bmatrix*}[l]
	&(0)_j, \ (1)_j, \ (2)_j, \ (3)_j, \ (4)_j, \ (5)_j, \\ 
	&(01)_j, \ (12)_j, \ (23)_j, \ (34)_j, \ (45)_j \\ 
	&(10)_j, \ (21)_j, \ (32)_j, \ (43)_j, \ (54)_j, \\ 
	&(s)_j^{\pm1}, \ (a)_j^{\pm1}; \ j=1,2
\end{Bmatrix*}$$
A reduced computation is called a \textit{one-machine computation} if every letter of its step history has the same index. If this index is $i$, then the computation is called a \textit{one-machine computation of the $i$-th machine}. 

For example, a computation with step history $(s)_1(0)_1(01)_1(1)_1$ is a one-machine computation of the first machine, while a computation with step history $(0)_1(s)_1^{-1}(s)_2(0)_2$ is not a one-machine computation, {\frenchspacing i.e. it is a \textit{multi-machine} computation.}

As in \Cref{sec M_5 standard}, some subwords clearly cannot appear in the step history of a reduced computation, while other impossibilities are less obvious. However, there are immediate analogues of Lemmas \ref{M_2 step history}, \ref{M_3 first step history}, \ref{M_3 second step history}, \ref{M_4 step history}, \ref{M_5 step history 1}, and \ref{M_5 step history} (after adding the same index to each letter of the step histories), as $\textbf{M}$ operates on the standard base as parallel copies of $\textbf{M}_5$ in any one-machine computation whose step history does not contain $(s)_i^{\pm1}$, $(a)_i^{\pm1}$, or $(0)_2$, or $(1)_2$.

%The following is the analogue of \Cref{M_2 step history}(a) and is proved in exactly the same way.
%
%\begin{lemma} \label{M step history (1)}
%
%Suppose the base $B$ of a reduced computation $\pazocal{C}$ of $\textbf{M}$ contains a subword $UV$ of the form $(Q_1(i)R_1(i))^{\pm1}$ or a mirror copy. Then the step history of $\pazocal{C}$ is not $(21)_1(1)_1(12)_1$. Moreover, if $i\neq 1$ or $UV$ is a mirror copy, then the step history of $\pazocal{C}$ is not $(21)_2(1)_2(12)_2$.
%
%\end{lemma}

The following statement is in the same nature as these analogues:

\begin{lemma} \label{M step history (0)}

Let $\pazocal{C}$ be a reduced computation of $\textbf{M}$ with base $B$.

\begin{enumerate}[label=({\alph*})]

\item If $B$ contains a subword of the form $(Q_{i,\ell}'P_i'R_i')^{\pm1}$, $Q_{i,\ell}'P_i'(P_i')^{-1}(Q_{i,\ell}')^{-1}$, $(P_i''R_i''Q_{i,r}'')^{\pm1}$, or $(Q_{i,r}'')^{-1}(R_i'')^{-1}R_i''Q_{i,r}''$, then the step history of $\pazocal{C}$ is not $(s)_j(0)_j(s)_j^{-1}$ for $j=1,2$.

\item If $B$ contains a subword of the form $(P_i'R_i')^{\pm1}$ or $(P_i''R_i'')^{\pm1}$, then the step history of $\pazocal{C}$ is not $(a)_j^{-1}(5)_j(a)_j$ for $j=1,2$.

\end{enumerate}

\end{lemma}

\begin{proof}

(a) Let $B'$ be the corresponding subword of $B$.

Assume toward contradiction that the step history of $\pazocal{C}:W_0\to\dots\to W_t$ is $(s)_j(0)_j(s)_j^{-1}$ and let $\pazocal{C}_0:W_1\to\dots\to W_{t-1}$ be the maximal subcomputation with step history $(0)_j$.  As $W_1$ and $W_{t-1}$ are $\sigma(s)_j^{-1}$-admissible, the restriction of $\pazocal{C}_0$ to the subword $B'$ satisfies the hypotheses of \Cref{primitive computations}(4) if $B'$ is reduced and the hypotheses of \Cref{primitive unreduced} otherwise.  But then $\pazocal{C}_0$ must be empty, so that $\pazocal{C}$ is unreduced.

(b) follows in just the same way, applying \Cref{multiply one letter}.

\end{proof}

This statement along with those pertaining to the step histories in previous sections imply the following analogue of \Cref{M_5 step history}:

\begin{lemma} \label{M step history}

Suppose $\pazocal{C}$ is a one-machine computation of the $i$-th machine whose base $B$ contains a subword of the form $(\{t(j)\}B_4(j))^{\pm1}$ for some $j\in\{1,\dots,L\}$.  Then the step history of $\pazocal{C}$ is a subword of either:

\begin{enumerate}[label=(\roman*)]

\item $(s)_i(0)_i(1)_i(2)_i(3)_i(4)_i(5)_i(a)_i$,

\item $(a)_i^{-1}(5)_i(4)_i(3)_i(2)_i(1)_i(0)_i(s)_i^{-1}$, or

\item $(s)_i(0)_i(1)_i(2)_i(3)_i(4)_i(5)_i(4)_i(3)_i(2)_i(1)_i(0)_i(s)_i^{-1}$.

\end{enumerate}

\end{lemma}

\begin{lemma}[Compare with Lemma 5.3 of \cite{W}] \label{long step history}

Let $\pazocal{C}$ be a reduced computation with base $\{t(i)\}B_4(i)$ for some $i\in\{2,\dots,L\}$.  Suppose $\pazocal{C}$ contains at least $9$ distinct maximal one-step computations.  Then $\pazocal{C}$ contains a subword of the form $(34)_j(4)_j(45)_j$ or $(54)_j(4)_j(43)_j$.

\end{lemma}

\begin{proof}

Assuming the step history of $\pazocal{C}$ contains no such subword, it follows from \Cref{M step history} that the step history of $\pazocal{C}$ or its inverse is a subword of either:

\begin{itemize}

\item $(4)_1(5)_1(a)_1(a)_2^{-1}(5)_2(4)_2$

\item $(4)_1(3)_1(2)_1(1)_1(s)_1^{-1}(s)_2(1)_2(2)_2(3)_2(4)_2$

\end{itemize}

But in either case the step history has length at most $8$.

\end{proof}

\begin{lemma} \label{M return to start}

Let $\pazocal{C}:W_0\to\dots\to W_t$ be a one-machine computation of the $i$-th machine in the standard base with step history of the form $(s)_ih_i(s)_i^{-1}$.  Then there exists $w\in\pazocal{R}_1$ and $H\in F(\Phi^+)$ with $\|H\|\geq k$ such that $W_0\equiv I(w,H)$ if $i=1$ or $W_0\equiv J(w,H)$ if $i=2$.  Moreover, $W_0$ is accepted by a one-machine computation of the $i$-th machine.

\end{lemma}

\begin{proof}

By hypothesis, the step history of $\pazocal{C}$ must be of form (iii) in \Cref{M step history}.  In particular, $h_i$ must have prefix $(0)_i(01)_i$.

Let $\pazocal{C}':W_1\to\dots\to W_{t-1}$ be the maximal subcomputation with step history $h_i$ and define $\pazocal{C}_0':W_1\to\dots\to W_r$ to be its maximal subcomputation with step history $(0)_i$.  

For every $j\in\{1,\dots,L\}$ and $x\in\{0,\dots,s\}$, the restriction of $\pazocal{C}_0'$ to the subword $P_x'(j)R_x'(j)Q_{x,r}'(j)$ of the standard base can be identified with a reduced computation of $\textbf{LR}$ satisfying \Cref{primitive computations}(3).  Similarly, the restriction of $\pazocal{C}_0'$ to the subword $Q_{x,\ell}''(j)P_x''(j)R_x''(j)$ can be identified with a reduced computation of $\textbf{RL}$ satisfying \Cref{primitive computations}(3).

As these primitive computations work in parallel, there exists $H\in F(\Phi^+)$ such that the tape word of $W_1$ in each $P_x'(j)R_x'(j)$-sector is the copy of $H$ over the corresponding left historical alphabet, while the tape word in each $P_x''(j)R_x''(j)$ is the copy of $H^{-1}$ over the left historical alphabet.

%As $W_0$ and $W_t$ are $\sigma(s)_i$-admissible by hypothesis, they are start configurations with all sectors empty except perhaps for the input sectors, with the words written in the historical input sectors formed over the corresponding left alphabets.

Now, for all $j\in\{2,\dots,L\}$, let $\pazocal{C}'(j):W_1(j)\to\dots\to W_{t-1}(j)$ be the restriction of $\pazocal{C}'$ to the subword $\{t(j)\}B_4(j)$ of the standard base.  Then each $\pazocal{C}'(j)$ can be identified with a reduced computation $\pazocal{D}_j$ of $\textbf{M}_5$.  Since $W_1$ and $W_{t-1}$ are both $\sigma(s)_i^{-1}$-admissible, $\pazocal{D}_j$ satisfies the hypotheses of \Cref{M_5 long history}(ii).  

Hence, for each $j$ there exists $w_j\in\pazocal{R}_1$ and $H_j\in F(\Phi^+)$ such that $W_1(j)$ is the natural copy of $I_5(w_j,H_j)$.  But since $W_1$ has the copy of $H$ written in each $P_x'(j)R_x'(j)$-sector, $H_j=H$ for all $j$.  Moreover, since \Cref{M_5 long history} implies $w_j$ is uniquely determined by $H_j$, then there exists $w\in\pazocal{R}_1$ such that each $w_j=w$ for all $j$.

If $i=1$, then the same argument then implies $W_1\equiv I_6(w,H)$, so that $W_0\equiv I(w,H)$.

If $i=2$, then the analogous argument implies the same except for an empty `special' input sector, i.e $W_0\equiv J(w,H)$.

Moreover, \Cref{M_5 long history}(a) implies $I_5(w,H)$ is accepted by $\textbf{M}_5$, so that the parallel nature of the rules produces a one-machine computation of the $i$-th machine accepting $W_0$.

\end{proof}

\begin{lemma} \label{M language} \

\begin{enumerate}

\item For any $w\in\pazocal{R}_1$, there exists $H_w\in F(\Phi^+)$ with $\|H_w\|\leq c_1 f_1(c_1\|w\|)+c_1$ such that $I(w,H_w)$ (respectively $J(w,H_w)$) is accepted by a one-machine computation of the first (respectively second) machine.

\item A start configuration $W$ of $\textbf{M}$ is accepted if and only if it is accepted by a one-machine computation of the $i$-th machine, so that there exist $w\in\pazocal{R}_1$ and $H\in F(\Phi^+)$ such that $W\equiv I(w,H)$ if $i=1$ or $W\equiv J(w,H)$ if $i=2$.

\end{enumerate}

\end{lemma}

\begin{proof}

(1) Let $H_w$ be the word given by \Cref{M_5 language}(1).  

\Cref{M_{6,1} language}(1) yields a one-machine computation of the first machine between $I_6(w,H_w)$ and a configuration that is $\sigma(a)_1$-admissible.  Letting $H_1$ be the history of this computation, it follows that $\sigma(s)_1H_1\sigma(a)_1$ is the history of a one-machine computation of the first machine accepting $I(w,H_w)$.

Using \Cref{M_{6,2} language}(1), the analogous argument produces a one-machine computation of the second machine accepting $J(w,H_w)$.

(2) Let $\pazocal{C}:W\equiv W_0\to\dots\to W_t$ be a reduced computation accepting $W$.

If $\pazocal{C}$ is a one-machine computation of the first (respectively second) machine, then the statement follows from \Cref{M_{6,1} language}(2) (respectively \Cref{M_{6,2} language}(2)).

Otherwise, let $\pazocal{C}_i:W_0\to\dots\to W_r$ be a maximal nonempty subcomputation which is a one-machine computation, say of the $i$-th machine.  Let $H_i$ be the history of $\pazocal{C}_i$.

As $W$ is a start configuration, the first rule of $H_i$ must be $\sigma(s)_i$.  Moreover, since $t>r$ and $\pazocal{C}_i$ is maximal, $W_r$ must be admissible for a rule from both the first and the second machine.  In particular, the last letter of $H_i$ must either be $\sigma(a)_i$ or $\sigma(s)_i^{-1}$.

If the last letter of $H_i$ is $\sigma(a)_i$, then $\pazocal{C}_i$ is a one-machine computation accepting $W$, so that the statement again follows from \Cref{M_{6,1} language}(2) or \Cref{M_{6,2} language}(2).

Otherwise, if the last letter of $H_i$ is $\sigma(s)_i^{-1}$, then the statement follows from \Cref{M return to start}.

\end{proof}

For any nonempty reduced computation $\pazocal{C}$ of $\textbf{M}$, define $\ell(\pazocal{C})$ to be the number of maximal nonempty one-machine subcomputations of $\pazocal{C}$.  Further, for any accepted configuration $W$ of $\textbf{M}$, define $A(W)$ to be the set of accepting computations of $W$.

For $W\neq W_{ac}$, define $\ell(W)=\{\ell(\pazocal{C}):\pazocal{C}\in A(W)\}$.  For completeness, set $\ell(W_{ac})=0$.

\begin{lemma} \label{ell at most 2}

For any accepted configuration $W$ of $\textbf{M}$, $\ell(W)\leq2$.

Moreover, if $\ell(W)=2$, then $W$ is not a start configuration and for any $\pazocal{C}\in A(W)$ with $\ell(\pazocal{C})=2$, there exists a factorization $H\equiv H_1H_2$ of the history of $\pazocal{C}$ such that:

\begin{enumerate}[label=(\alph*)]

\item $H_i$ is the history of a one-machine computation of the $i$-th machine.

\item $W\cdot H_1\equiv J(w,H')$ for some $w\in\pazocal{R}_1$ and $H'\in F(\Phi^+)$.

\end{enumerate}

\end{lemma}

\begin{proof}

Without loss of generality suppose $W\neq W_{ac}$.  Letting $\ell(W)=\ell$, fix an accepting computation $\pazocal{C}\in A(W)$ such that $\ell(\pazocal{C})=\ell$.  Factor the history $H\equiv H_1\dots H_\ell$ of $\pazocal{C}$ such that each $H_i$ is the history of a maximal one-machine subcomputation.

Assume $\ell\geq2$.  Then for each $j\in\{1,\dots,\ell-1\}$, $W_j\equiv W\cdot H_1\dots H_j$ is admissible for both the first rule of $H_{j+1}$ and the inverse of the last rule of $H_j$.  The maximality of the factorization then implies that $W_j$ is admissible for rules for both machines, and so must be either a start or an end configuration.

If $W_j$ is an end configuration, then necessarily $W_j\equiv W_{ac}$.  But then $H_1\dots H_j$ is the history of a computation $\pazocal{C}'\in A(W)$ with $\ell(\pazocal{C}')=j<\ell$, contradicting our choice of $\ell$.

So, each $W_j$ must be an accepted start configuration, so that \Cref{M language}(2) $W_j$ is accepted by a one-machine computation.  But if $j\leq\ell-2$, then the history of this one-machine accepting computation could replace $H_{j+1}\dots H_\ell$ to produce an accepting computation of $W$ which contradicts the definition of $\ell$.

Hence, $\ell=2$ and there exists $w\in\pazocal{R}_1$ and $H'\in F(\Phi^+)$ such that $W_1\equiv I(w,H')$ or $W_1\equiv J(w,H')$.

Moreover, since $1\notin\pazocal{R}_1$, $I(w,H')$ is not admissible for a rule of the second machine, so that $W_1\equiv J(w,H')$.  The statement then follows.

\end{proof}

The history $H$ of a reduced computation $\pazocal{C}$ of $\textbf{M}$ is called \textit{controlled} if $\pazocal{C}$ is a one-machine computation and $H$ corresponds to a controlled computation of $\textbf{M}_5$. As such, the next statement follows immediately from Lemma \ref{M_5 controlled}.

\begin{lemma} \label{M controlled}

Let $\pazocal{C}:W_0\to\dots\to W_t$ be a reduced computation of $\textbf{M}$ with controlled history $H$. Then the base of the computation is a reduced word and all the admissible words $W_i$ are uniquely defined by the history $H$ and the base of $\pazocal{C}$.\newline Moreover, if $\pazocal{C}$ is a computation in the standard base, then $|W_0|_a=\dots=|W_t|_a$, $W_0$ contains the copy of the same word $H_1\in F(\Phi^+)$ over the right historical alphabet in each $P_i'(j)R_i'(j)$-sector, the copy of $H_1^{-1}$ over the right historical alphabet in each $P_i''(j)R_i''(j)$-sector, $\|H\|=2\|H_1\|+3$, and each $W_i$ is accepted by a one-machine computation.

\end{lemma}

\medskip

%%%%%%%%%%%%%%%%%%%%%%%%%%%%%%%%%%%%%%%%%%%%%%%%%%

\subsection{Components and extending computations} \

For a configuration $W$ and $1\leq i\leq L$, the \emph{$i$-th component of $W$}, $W(i)$, is defined to be the admissible subword of $W$ with base $\{t(i)\}B_4(i)$. Since the tape alphabet of the $Q_{s,r}''(i)\{t(i+1)\}$-sector is empty for each $i$, $W\equiv W(1)\dots W(L)$ for any configuration $W$. It is useful to note that if a rule $\theta$ is applicable to some configuration $W$, then $\theta$ operates on each $W(j)$ identically for each $j\geq2$ (but may not operate on $W(1)$ in the analogous way).

Particularly, for $1\leq i\leq L$, we denote the components $A(i)\equiv W_{ac}(i)$, $I(w,H,i)\equiv (I(w,H))(i)$, and $J(w,H,i)\equiv (J(w,H))(i)$ for all $w\in F(X)$ and $H\in F(\Phi^+)$.

Let $V$ be an admissible word with base $B$ and suppose there exists $i\in\{1,\dots,L\}$ such that every letter of $B$ has coordinate $i$. Then, a \textit{coordinate shift} of $V$ is an admissible word $V'$ obtained by simply changing all the state letters' coordinates from $i$ to $j$ for some $j\in\{1,\dots,L\}$ and taking the natural copies of the tape words. For example, if $W$ is an accepted configuration, then $W(i)$ and $W(j)$ are coordinate shifts of one another for $i,j\geq2$, while $J(w,H,1)$ is not a coordinate shift of $J(w,H,2)$ for $w\neq1$.

\begin{lemma} \label{one-machine accepted is admissible}

Suppose $W$ is a configuration of $\textbf{M}$ accepted by a (perhaps empty) one-machine computation of the $i$-th machine.  If $W(j)$ is $\theta$-admissible for some $\theta\in\Theta_i$ and $j\in\{2,\dots,L\}$, then $W$ is $\theta$-admissible.
%
%\begin{enumerate}[label=(\alph*)]
%
%\item 
%
%\item If $W(j)$ is $\theta$-admissible for some $\theta\notin\Theta_i$, then 
%
%\end{enumerate}

\end{lemma}

\begin{proof}

First, suppose $i=1$.  Then since every rule of $\Theta_i$ operates in parallel on the components of the standard base, it follows that $W(\ell)$ a coordinate shift of $W(j)$ for each $\ell\in\{1,\dots,L\}$.  But $\theta$ also operates in parallel on the components, so that each $W(\ell)$ must be $\theta$-admissible.

Conversely, if $i=2$, then $W(\ell)$ is a coordinate shift of $W(j)$ for each $\ell\in\{2,\dots,L\}$ while $W(1)$ is the admissible word obtained from the coordinate shift of $W(j)$ by emptying the tape word of the `special' input sector.  But since $\theta$ must lock the `special' input sector anyway, the statement follows.

\end{proof}

\begin{lemma} \label{projected not admissible}

Let $W$ be an accepted configuration of $\textbf{M}$ with $\ell(W)\leq1$.  Suppose $W(j)$ is $\theta$-admissible for $\theta\in\Theta$ and $j\in\{2,\dots,L\}$.

\begin{enumerate}[label=(\alph*)]

\item If $W$ is not $\theta$-admissible, then $\theta=\sigma(s)_2$ and $W\equiv I(w,H)$ for some $w\in\pazocal{R}_1$ and some $H\in F(\Phi^+)$.

\item If $W$ is $\theta$-admissible but $\ell(W\cdot\theta)=2$, then $\theta=\sigma(s)_1$ and $W\equiv J(w,H)$ for some $w\in\pazocal{R}_1$ and $H\in F(\Phi^+)$.

\end{enumerate}

\end{lemma}

\begin{proof}

As $\ell(W)\leq1$, there exists a (perhaps empty) one-machine computation of the $i$-th machine accepting $W$.  In either case, \Cref{one-machine accepted is admissible} implies $\theta\notin\Theta_i$.

In particular, $W(j)$ must be $\theta$-admissible for rules from both machines, so that $W$ must be either an accepted start or an accepted end configuration.

If $W$ is an accepted end configuration, then it follows immediately that $W\equiv W_{ac}$.  But then $A(j)$ is $\theta$-admissible implies $\theta=\sigma(a)_\ell^{-1}$, so that $W$ is $\theta$-admissible and $\ell(W\cdot\theta)=1$.

The statement then follows from \Cref{M language}(2).

\end{proof}

The next statement is an immediate consequence of \Cref{M step history}:

\begin{lemma} \label{M return to end}

For $i\in\{2,\dots,L\}$, let $\pazocal{C}:A(i)\to\dots\to A(i)$ be a nonempty reduced computation of $\textbf{M}$. Then $\pazocal{C}$ is not a one-machine computation.

\end{lemma}

The next statement follows from the parallel nature of the rules and is proved in much the same way as in \cite{W}:

\begin{lemma}[Lemma 5.9 of \cite{W}]  \label{extend one-machine}

Let $V_0\to\dots\to V_t$ be a one-machine computation of the $i$-th machine with history $H$ and base $\{t(j)\}B_4(j)$ for some $j\in\{2,\dots,L\}$. Then there exists a one-machine computation of the $i$-th machine $W_0\to\dots\to W_t$ in the standard base with history $H$ such that $W_\ell(j)\equiv V_\ell$ for all $\ell\in\{0,\dots,t\}$.  Moreover:

\begin{enumerate} [label=(\alph*)]

\item If $V_\ell\equiv A(j)$, then $W_\ell\equiv W_{ac}$

\item If $V_\ell\equiv I(w,H,j)$ for some $w\in F(X)$ and $H\in F(\Phi^+)$, then

\begin{itemize}

\item $W_\ell\equiv I(w,H)$ if $i=1$, or

\item $W_\ell\equiv J(w,H)$ if $i=2$.

\end{itemize}

\end{enumerate}

\end{lemma}

The next statement can be readily seen in the proof of \Cref{M return to start}, but the simple proof exemplifies the use of \Cref{extend one-machine}.

\begin{lemma} \label{I vs J}

Given $w\in\pazocal{R}_1$ and $H\in F(\Phi^+)$, the input configuration $I(w,H)$ is accepted by a one-machine computation of the first machine if and only if the input configuration $J(w,H)$ is accepted by a one-machine computation of the second machine.

\end{lemma}

\begin{proof}

Suppose $I(w,H)$ is accepted by a one-machine computation $\pazocal{C}$ of the first machine.  Then the restriction of $\pazocal{C}$ to the subword $\{t(2)\}B_4(2)$ is a one-machine computation of the first machine of the form $I(w,H,2)\to\dots\to A(2)$.  As the rules of the second machine operate on admissible words with base $\{t(2)\}B_4(2)$ identically to the rules of the first machine, there then exists a one-machine computation $\pazocal{D}$ of the second machine of the form $I(w,H,2)\to\dots\to A(2)$.  But then applying \Cref{extend one-machine} to $\pazocal{D}$ produces a one-machine computation of the second machine accepting $J(w,H)$.

The analogous argument then implies the converse.

\end{proof}

\begin{lemma} \label{projected start to end}

For $j\in\{2,\dots,L\}$, let $\pazocal{C}:V_0\to\dots\to V_t$ be a one-machine computation of the $i$-th machine with base $\{t(j)\}B_4(j)$ and step history of the form $(s)_ih_i(a)_i$.  Then there exists an accepted input configuration $I(w,H)$ such that $V_0\equiv I(w,H,j)$.

\end{lemma}

\begin{proof}

As it must be $\sigma(a)_i^{-1}$-admissible, $V_t$ must be equivalent to $A(j)$.

Applying \Cref{extend one-machine} to $\pazocal{C}$ yields a one-machine computation accepting an input configuration $W_0$ such that $V_0\equiv W_0(j)$.  The statement then follows from \Cref{M language}(2) and \Cref{I vs J}.

\end{proof}

\begin{lemma} \label{projected return to start}

For $j\in\{2,\dots,L\}$, let $\pazocal{C}:V_0\to\dots\to V_t$ be a one-machine computation of the $i$-th machine with base $\{t(j)\}B_4(j)$ and step history of the form $(s)_ih_i(s)_i^{-1}$.  Then there exists an accepted input configuration $I(w,H)$ such that $V_0\equiv I(w,H,j)$.

\end{lemma}

\begin{proof}

Applying \Cref{extend one-machine} to $\pazocal{C}$ produces a one-machine computation in the standard base satisfying the hypotheses of \Cref{M return to start}, so that the statement follows immediately.

\end{proof}

The next statement then follows from Lemmas \ref{M return to end}, \ref{projected start to end}, and \ref{projected return to start}:

\begin{lemma} \label{projected end to end}

Let $j\in\{2,\dots,L\}$ and suppose $\pazocal{C}:A(j)\to\dots\to A(j)$ is a reduced computation. Let $H\equiv H_1\dots H_k$ be the factorization of the history of $\pazocal{C}$ such that for all $i\in\{1,\dots,k\}$, $H_i$ is the history of a maximal one-machine subcomputation of the $z_i$-th machine $\pazocal{C}_i:U_i\to\dots\to V_i$ of $\pazocal{C}$. Then for all $i$, either:

\begin{enumerate} [label=({\alph*})]

\item $V_i\equiv A(j)$ or

\item $V_i\equiv I(w_i,K_i,j)$ for some $w_i\in\pazocal{R}_1$ and $K_i\in F(\Phi^+)$.

\end{enumerate}

In case (a), set $W_i^{(1)}\equiv W_i^{(2)}\equiv W_{ac}$; in case (b), set $W_i^{(1)}\equiv I(w_i,K_i)$ and $W_i^{(2)}\equiv J(w_i,K_i)$. Further, set $W_0^{(1)}\equiv W_0^{(2)}\equiv W_{ac}$.
\newline
Then for each $i\in\{1,\dots,k\}$, there exists a reduced computation $\pazocal{C}_i':W_{i-1}^{(z_i)}\to\dots\to W_i^{(z_i)}$ in the standard base with history $H_i$.

\end{lemma}

In other words, Lemma \ref{projected end to end} says that $\pazocal{C}$ can be `\textit{almost-extended}' to a reduced computation $W_{ac}\to\dots\to W_{ac}$, in that such a computation exists if one were to allow the insertion/deletion of elements of $\pazocal{R}_1$ in the `special' input sector between maximal one-machine subcomputations.

\medskip

%%%%%%%%%%%%%%%%%%%%%%%%%%%%%%%%%%%%%%%%%%%%%%%%%%

\subsection{Computations of $\textbf{M}$ with long history} \

%Our goal in this section is to study reduced computations of $\textbf{M}$ with base $\{t(i)\}B_4(i)$ whose length is much longer than the size of the admissible words it starts and ends with.  The first step toward this study is to show that such a computation must be multi-machine:

\begin{lemma} \label{M projected long history one-machine}

Let $\pazocal{C}:V_0\to\dots\to V_t$ be a one-machine computation of the $j$-th machine with base $\{t(i)\}B_4(i)$ for some $i\in\{2,\dots,L\}$.  Then $t\leq c_1\max(\|V_0\|,\|V_t\|)$.

\end{lemma}

\begin{proof}

By the construction of the software, any occurrence of $\sigma(s)_j^{\pm1}$ or $\sigma(a)_j^{\pm1}$ in the history $H$ of $\pazocal{C}$ is as the first or last letter.  Hence, letting $\pazocal{D}:U_0\to\dots\to U_s$ be the maximal subcomputation of $\pazocal{C}$ containing no letters of the form $\sigma(s)_j^{\pm1}$ or $\sigma(a)_j^{\pm1}$, then $s\geq t-2$.  Moreover, since transition rules do not alter the tape words of the admissible words, $\|U_0\|=\|V_0\|$ and $\|U_s\|=\|V_t\|$.

Now, $\pazocal{D}$ can be identified with a reduced computation of $\textbf{M}_5$ in the standard base, so that \Cref{M_5 standard time} implies $s\leq c_0\max(\|U_0\|,\|U_s\|)$.  As a result, $t\leq c_0\max(\|V_0\|,\|V_t\|)+2$, so that the statement follows by the parameter choice $c_1>>c_0$.

\end{proof}

\begin{lemma} \label{M projected long history}

Let $\pazocal{C}:V_0\to\dots\to V_t$ be a reduced computation with base $\{t(i)\}B_4(i)$ for some $i\in\{2,\dots,L\}$. Suppose $t>c_2\max(\|V_0\|,\|V_t\|)$.  Then:

\begin{enumerate}[label=({\alph*})]

\item For every $\ell\in\{0,\dots,t\}$, there exists an accepted configuration $W_\ell$ such that $W_\ell(i)\equiv V_\ell$.

\item The sum of the lengths of all subcomputations of $\pazocal{C}$ with step histories $(34)_j(4)_j(45)_j$ and $(54)_j(4)_j(43)_j$ for $j\in\{1,2\}$ is at least $0.99t$.

\end{enumerate}

\end{lemma}

\begin{proof}

By \Cref{M projected long history one-machine} and the parameter choice $c_2>>c_1$, $\pazocal{C}$ must be a multi-machine computation.

First, suppose there exists a factorization $H\equiv H_1H_2$ of the history of $\pazocal{C}$ such that each $H_j$ is the history of a one-machine subcomputation.  Let $\pazocal{C}_1:V_0\to\dots\to V_s$ and $\pazocal{C}_2:V_s\to\dots\to V_t$ be the subcomputations with histories $H_1$ and $H_2$, respectively.

By \Cref{M projected long history one-machine}, $s\leq c_1\max(\|V_0\|,\|V_s\|)$ and $t-s\leq c_1\max(\|V_s\|,\|V_t\|)$.  Moreover, since $V_s$ is admissible for both the first rule of $H_2$ and the inverse of the last rule of $H_1$, it must be an admissible subword of a start or an end configuration.

Suppose $\|V_s\|\leq c_0\max(\|V_0\|,\|V_t\|)$.  Then $t\leq 2c_0c_1\max(\|V_0\|,\|V_t\|)$, so that the statement follows from the parameter choices $c_2>>c_1>>c_0$.

If $V_s$ is an admissible subword of an end configuration, then it must be $\sigma(a)_j$-admissible.  But then $|V_s|_a=0$, so that $\|V_s\|\leq\min(\|V_0\|,\|V_t\|)$.  Hence, we may assume $V_s$ is an admissible subword of a start configuration.

As in the proof of \Cref{M projected long history one-machine}, letting $\pazocal{D}_2:U_0\to\dots\to U_r$ be the maximal subcomputation of $\pazocal{C}_2$ whose history does not contain a letter of the form $\sigma(s)_j^{\pm1}$ or $\sigma(a)_j^{\pm1}$, $r\geq t-s-2$, $\|U_0\|=\|V_s\|$, and $\|U_r\|=\|V_t\|$.  

As $U_0$ is $\sigma(s)_j^{-1}$-admissible, $\pazocal{D}_2$ can be identified with a reduced computation of $\textbf{M}_5$ in the standard base where $U_0$ is a tame input configuration.  So, since $\|V_s\|>c_0\|V_t\|$ implies $|V_s|_a>c_0|V_t|_a$, the reduced computation of $\textbf{M}_5$ corresponding to $\pazocal{D}_2$ satisfies the hypotheses of \Cref{M_5 long history}(iii).  Hence, $U_0$ can be identified with an accepted tame input configuration of $\textbf{M}_5$, $t-s\geq k$, and the sum of the lengths of the subcomputations of $\pazocal{D}_2$ (and so of $\pazocal{C}_2$) with step history $(34)_j(4)_j(45)_j$ or $(54)_j(4)_j(43)_j$ is at least $(1-\frac{1}{C})r\geq(1-\frac{1}{C})(t-s-2)$.

As $U_0$ can be identified with an accepted tame input configuration of $\textbf{M}_5$, there exists a one-machine computation $\pazocal{E}:U_0\equiv U_0'\to\dots\to U_x'$ such that $U_x'$ is $\sigma(a)_j$-admissible.  If the history of $\pazocal{E}$ is $H'$, then $\sigma(s)_jH'\sigma(a)_j$ is the history of a one-machine computation $\pazocal{E}'$ between $V_s$ and $A(i)$.  Applying \Cref{extend one-machine} to $\pazocal{E}'$ then produces an accepted configuration $W_s$ with $W_s(i)\equiv V_s$.

Applying \Cref{extend one-machine} to $\pazocal{C}_2$ implies that for $\ell\geq s$, there exists an accepted configuration $W_\ell$ such that $W_\ell(i)\equiv V_\ell$.

Further, as $t-s\geq k$, a parameter choice for $k$ implies the sum  of the lengths of the subcomputations of $\pazocal{C}_2$ with step history $(34)_j(4)_j(45)_j$ or $(54)_j(4)_j(43)_j$ is at least $(1-\frac{1}{C})r\geq(1-\frac{1}{C})^2(t-s)$.

The symmetric argument further implies statement (a) and that the sum of the lengths of the subcomputations of $\pazocal{C}_1$ with step history $(34)_j(4)_j(45)_j$ or $(54)_j(4)_j(43)_j$ is at least $(1-\frac{1}{C})^2s$.

Thus, the sum of the lengths of the subcomputations of $\pazocal{C}_1$ with step history $(34)_j(4)_j(45)_j$ or $(54)_j(4)_j(43)_j$ is at least $(1-\frac{1}{C})^2t$.  Hence, statement (b) follows by taking $C\geq200$.

Now, suppose $H\equiv H_1\dots H_m$ for $m\geq3$ where each $H_j$ is the history of a one-machine subcomputation.  

For $2\leq j\leq m-1$, let $\pazocal{C}_j:V_{\ell(j)}\to\dots\to V_{r(j)}$ be the subcomputation of $\pazocal{C}$ with history $H_j$.  By \Cref{M return to end}, at least one of $V_{\ell(j)}$ or $V_{r(j)}$ is an admissible subword of a start configuration.  As such, the maximal subcomputation of $\pazocal{C}_j$ whose history contains no letter of the form $\sigma(s)_j^{\pm1}$ or $\sigma(a)_j^{\pm1}$ (or its inverse) can be identified with a reduced computation of $\textbf{M}_5$ satisfying the hypotheses of either \Cref{M_5 long history}(i) or (ii).  As above, this implies there exists an accepted configuration $W_{\ell(j)}$ such that $W_{\ell(j)}(i)\equiv V_{\ell(j)}$ and the sum of the lengths of the subcomputations of $\pazocal{C}_j$ with step history $(34)_j(4)_j(45)_j$ or $(54)_j(4)_j(43)_j$ is at least $(1-\frac{1}{C})^2(r(j)-\ell(j))$.

Applying \Cref{extend one-machine} to each subcomputation of $\pazocal{C}$ with history $H_j$ then implies statement (a).

Now, let $\pazocal{C}_m:V_{r(m-1)}\to\dots\to V_t$ be the subcomputation with history $H_m$.  If $|V_{r(m-1)}|_a>c_0|V_t|_a$, then the same argument as above implies the sum of the lengths of the computations of $\pazocal{C}_m$ with step history $(34)_j(4)_j(45)_j$ or $(54)_j(4)_j(43)_j$ is at least $(1-\frac{1}{C})^2(t-r(m-1))$.  Otherwise, \Cref{M projected long history one-machine} implies $t-r(m-1)\leq c_0c_1\|V_t\|$, so that the parameter choices $c_2>>c_1>>c_0>>C$ imply $t-r(m-1)\leq\frac{1}{2C}t$.

Applying the same reasoning to the subcomputation $V_0\to\dots\to V_{\ell(2)}$ with step history $H_1$ then implies the sum of the subcomputations of $\pazocal{C}$ with step history $(34)_j(4)_j(45)_j$ or $(54)_j(4)_j(43)_j$ is at least $(1-\frac{1}{C})^3t$, so that statement (b) follows by taking $C\geq300$.

\end{proof}

\begin{lemma}[Compare with Lemma 5.20 of \cite{W}] \label{M projected long history controlled}

Let $\pazocal{C}:V_0\to\dots\to V_t$ be a reduced computation of $\textbf{M}$ with base $\{t(i)\}B_4(i)$ for some $i\in\{2,\dots,L\}$.  If $t>c_2\max(\|V_0\|,\|V_t\|)$, then the history of any subcomputation $\pazocal{D}:V_r\to\dots\to V_s$ of $\pazocal{C}$ (or the inverse of $\pazocal{D}$) of length at least $0.4t$ contains a controlled subword.

\end{lemma}

Finally, the next statement provides an adaptation of \Cref{M_5 one-step} for the machine $\textbf{M}$:

\begin{lemma}[Lemma 5.21 of \cite{W}] \label{M one-step}

Let $\pazocal{C}:V_0\to\dots\to V_t$ be a reduced computation of $\textbf{M}$ with base $\{t(i)\}B_4(i)$ for some $i\in\{2,\dots,L\}$. Suppose the step history of $\pazocal{C}$ has length 1 and $|V_z|_a>2|V_0|_a$ for some $z\in\{1,\dots,t\}$.  Then there exist two-letter subwords $U_\ell V_\ell$ and $U_rV_r$ of $\{t(i)\}B_4(i)$ such that:

\begin{enumerate}

\item Letting $\pazocal{C}_\ell:V_{0,\ell}\to\dots\to V_{t,\ell}$ and $\pazocal{C}_r:W_{0,r}\to\dots\to V_{t,r}$ be the restriction of $\pazocal{C}$ to the $U_\ell V_\ell$- and $U_rV_r$-sectors, respectively, $|V_{z,\ell}|_a<\dots<|V_{t,\ell}|_a$ and $|V_{z,r}|_a<\dots<|V_{t,r}|_a$.

\item Every rule of the subcomputation $V_z\to\dots\to V_t$ multiplies the $U_\ell V_\ell$-sector by one letter on the left and the $U_rV_r$-sector by one letter on the right.

\end{enumerate}

\end{lemma}

\medskip

%%%%%%%%%%%%%%%%%%%%%%%%%%%%%%%%%%%%%%%%%%%%%%%%%%

\subsection{Reverted Bases} \

Let $B$ be the base of an admissible word $W$ of $\textbf{M}$. The \textit{reversion} of $B$, denoted $\pi(B)$, is the word obtained from $B$ by `forgetting' the coordinates of its letters. In this case, $\pi(B)$ is called the \textit{reverted base} of $W$. 

For example, the reverted base of any configuration is the concatenation of $L$ copies of the standard base of $\textbf{M}_5$. Similarly, if $B=Q_0(2)^{-1}\{t(2)\}^{-1}(Q_{s,r}''(1))^{-1}(R_s''(1))^{-1}R_s''(1)(R_s''(1))^{-1}$, then $$\pi(B)=Q_0^{-1}\{t\}^{-1}(Q_{s,r}'')^{-1}(R_s'')^{-1}R_s''(R_s'')^{-1}$$

\begin{lemma} [Lemma 5.22 of \cite{W}] \label{lifted base}

Let $B$ be the base of an admissible word $W$ of $\textbf{M}$. Then there exists an admissible word $W'$ of $\textbf{M}_5$ with base $\pi(B)$ and such that $|W'|_a=|W|_a$. \newline
Moreover, if none of the state letters of $W$ are start or end letters (or their inverses), then $W'$ can be chosen to be the natural copy of $W$ in the hardware of $\textbf{M}_5$.

\end{lemma}

\begin{lemma} [Lemma 5.23 of \cite{W}] \label{lifted rule}

Suppose $\pazocal{C}:W_0\to W_1$ is a one-rule computation of $\textbf{M}$ with history $\theta\in\Theta$, where $\theta\notin\{\theta(s)_i,\theta(a)_i:i=1,2\}^{\pm1}$. Further, suppose that either:

\begin{enumerate}[label=({\alph*})]

\item the step history of $\pazocal{C}$ is not $(0)_2$, $(01)_2$, $(10)_2$, or $(1)_2$, or

\item the base of $\pazocal{C}$ does not contain a subword of the form $(Q_0(1)Q_1(1))^{\pm1}$.

\end{enumerate}

Then there exists a one-rule computation $\pazocal{C}':W_0'\to W_1'$ of $\textbf{M}_5$ with history $\theta'$, where $\theta'$ is the natural copy of $\theta$ in $\Theta(\textbf{M}_5)$ and $W_0'$ and $W_1'$ are the natural copies of $W_0$ and $W_1$, respectively, in the hardware of $\textbf{M}_5$.

\end{lemma}

%\begin{proof}
%
%As $\theta\notin\{\theta(s)_i,\theta(a)_i:i=1,2\}^{\pm1}$, none of the state letters of $W_0$ or $W_1$ are start or end letters (or their inverses). So, applying Lemma \ref{lifted base}, we can find $W_0'$ and $W_1'$ that are the natural copies of $W_0$ and $W_1$, respectively, in the hardware of $\textbf{M}_5$.
%
%Let $\theta'$ be the natural copy of $\theta$ in $\Theta(\textbf{M}_5)$. If $(a)$ holds, then $\theta$ operates on each sector of the standard base of $\textbf{M}$ in the same way as $\theta'$ operates on the copy of the corresponding sector of the standard base of $\textbf{M}_5$. 
%
%Conversely, if $\theta$ is a rule of $\textbf{M}_{6,2}(1)$, then all sectors of the standard base of $\textbf{M}$ other than the `special' input sector are again operated on by $\theta$ in the same way as $\theta'$ operates on their copy. As $\theta$ locks the `special' input sector, Lemma \ref{locked sectors} implies that this sector is not present in $W_0$ if the base of $\pazocal{C}$ satisfies $(b)$.
%
%\end{proof}

The base $B$ of an admissible word of $\textbf{M}$ is called \textit{hyperfaulty} (or \textit{pararevolving}) if its reversion $\pi(B)$ is faulty (or revolving) as the base of an admissible word of $\textbf{M}_5$. Note that a base is hyperfaulty if and only if it is pararevolving and unreduced.

A hyperfaulty base is necessarily faulty, while a faulty base need not be hyperfaulty. For example, if
$B\equiv Q_1(3)Q_1(3)^{-1}Q_0(3)^{-1}\{t(3)\}^{-1}\dots (Q_{s,r}''(1))^{-1}(R_s''(1))^{-1}R_s''(1)Q_{s,r}''(1)\dots \{t(3)\}Q_0(3)Q_1(3)$, where gaps correspond to strings of letters that follow the order of the standard base of $\textbf{M}$ or its inverse, then $B$ is faulty but not hyperfaulty.

Conversely, a pararevolving base that is not hyperfaulty (for example, $\{t(1)\}\dots\{t(2)\}$) is not revolving, while a revolving base that is not faulty (for example, $\{t(1)\}\dots\{t(1)\}$) is not pararevolving.  A pararevolving base has length at most $2N+1$, while a revolving base of $\textbf{M}$ has length at most $2LN+1$.

\begin{lemma} \label{one-step hyperfaulty}

Let $\pazocal{C}:W_0\to\dots\to W_t$ be a reduced computation of $\textbf{M}$ with hyperfaulty base $B$.  Suppose the step history of $\pazocal{C}$ has length 1.  Then $|W_i|_a\leq c_0(|W_0|_a+|W_t|_a)$ for all $0\leq i\leq t$.

\end{lemma}

\begin{proof}

If every transition $W_{i-1}\to W_i$ satisfies the hypotheses of \Cref{lifted rule}, then by hypothesis $\pazocal{C}$ can be identified with a one-step computation of $\textbf{M}_5$ with faulty base.  Hence, the statement follows by \Cref{M_5(j) faulty}.

Moreover, since every rule of $\textbf{M}_5(0)$ fixes the tape word of the $Q_0Q_1$-sector, if the step history of $\pazocal{C}$ is $(0)_2$ then an analogous proof to that of \Cref{M_5(j) faulty} implies the statement.

Hence, it suffices to assume that the step history of $\pazocal{C}$ is $(1)_2$ and that $B$ contains a subword of the form $(Q_0(1)Q_1(1))^{\pm1}$.

Now, suppose every letter of $B$ is of the form $Q_j(1)^{\pm1}$.  Then, since we may assume that $\pazocal{C}$ is nonempty, $W_0$ can be identified with an admissible word $V_0$ of $\textbf{Move}$ whose tape words in the $(Q_0Q_1)^{\pm1}$-sectors are empty.  Identifying the history $H$ of $\pazocal{C}$ with the corresponding rules of $\textbf{Move}$, then (Mv2) implies there exists a reduced computation $\pazocal{D}:V_0\to\dots\to V_t$ of $\textbf{Move}$ with history $H$ such that $|W_i|_a=|V_i|_{NI}$ for all $i$.  But then the base of $\pazocal{D}$ is faulty, so that \Cref{Move faulty} implies the statement.

Otherwise, as in Case 5 of the proof of \Cref{barM_4(j) faulty}, we may assume there exists a factorization $B'\equiv B_1C_1\dots B_mC_mB_{m+1}$ of a cyclic permutation of $B$ such that:

\begin{itemize}

\item each $B_x$ is a nonempty word of length at least 2 which does not contain a letter of the form $Q_i(j)^{\pm1}$ or a subword of the form $(Q_{0,\ell}'(j))^{-1}Q_{0,\ell}'(j)$

\item each $C_x$ is a possibly empty word consisting entirely of letters of the form $Q_i(j)^{\pm1}$ or $t(j)^{\pm1}$

\item If $C_x$ is empty, then the last letter of $B_x$ is $(Q_{0,\ell}'(j))^{-1}$ and the first letter of $B_{x+1}$ is $Q_{0,\ell}'(j)$

\item $B_{m+1}$ consists of a single letter

\end{itemize}

As in that setting, let $C_x'$ be the base word obtained from $C_x$ by adjoining the last letter of $B_x$ to its front and the first letter of $B_{x+1}$ to its end.  For $1\leq x\leq m$, let $\pazocal{C}_x:W_{0,x}\to\dots\to W_{t,x}$ and $\pazocal{C}_x':W_{0,x}'\to\dots\to W_{t,x}'$ be the restrictions of $\pazocal{C}$ to the subwords $B_x$ and $C_x'$, respectively.

As $\pazocal{C}$ operates on the subwords $B_x$ as $\overline{\textbf{M}}_4(1)$, the identical argument as \Cref{barM_4(j) faulty} implies $|W_{i,x}|_a\leq c_0(|W_{0,x}|_a+|W_{t,x}|_a)$ for all $i$.

Moreover, if $C_x'$ is of the form $(Q_{0,\ell}'(j)^{-1}Q_{0,\ell}'(j)$, then $\pazocal{C}_x'$ satisfies the hypotheses of \Cref{unreduced base}, so that $|W_{i,x}'|_a\leq|W_{0,x}'|_a+|W_{t,x}'|_a$; otherwise, as every rule locks the $Q_{s,r}''(j-1)\{t(j)\}$- and $\{t(j)\}Q_0(j)$-sectors, the argument detailed above implies $|W_{i,x}'|_a\leq c_0(|W_{0,x}'|_a+|W_{t,x}'|_a)$ for all $i$.

Thus, $|W_i|_a=\sum(|W_{i,x}|_a+|W_{i,x}'|_a)\leq c_0\sum(|W_{0,x}|_a+|W_{t,x}|_a+|W_{0,x}'|_a+|W_{t,x}'|_a)=c_0(|W_0|_a+|W_t|_a)$.

\end{proof}

\begin{remark}

Note that the arguments outlined in the proof of \Cref{one-step hyperfaulty} indicate the purpose of properties (Mv2) and (Mv8) in the definition of Move machines.

\end{remark}

\begin{lemma} \label{one-machine hyperfaulty}

Let $\pazocal{C}:W_0\to\dots\to W_t$ be a reduced computation of $\textbf{M}$ with hyperfaulty base $B$.  Suppose no letter of the step history of $\pazocal{C}$ is $(s)_i^{\pm1}$ or $(a)_i^{\pm1}$.  Then $|W_i|_a\leq c_2\max(|W_0|_a,|W_t|_a)$ for all $0\leq i\leq t$.

\end{lemma}

\begin{proof}

By \Cref{one-step hyperfaulty} and the parameter choice $c_2>>c_1$, it suffices to assume that the history $H$ of $\pazocal{C}$ contains a transition rule.  

Moreover, as in the proofs of \Cref{barM_4(j) faulty}, \Cref{barM_4 faulty}, \Cref{M_5(j) faulty}, and \Cref{M_5 faulty}, it suffices to assume that $|W_i|_a>\max(|W_0|_a,|W_t|_a)$ for all $1\leq i\leq t-1$.

\textbf{1.} Suppose $H$ contains a letter of the form $\sigma(01)_2^{\pm1}$ and $B$ contains a subword of the form $(Q_0(1)Q_1(1))^{\pm1}$.

Then perhaps passing to the inverse computation, we may assume that there exists a maximal nonempty subcomputation $\pazocal{C}_0:W_r\to\dots\to W_s$ with step history $(0)_2$ such that $W_r$ is $\sigma(01)_2$-admissible.

Now, note that any letter of $B$ of the form $P_i'(j)^{\pm1}$ or of the form $R_i''(j)$ must be part of a subword of a cyclic permutation of $B$ of the form $(Q_{i,\ell}'(j)P_i'(j)R_i'(j))^{\pm1}$, $Q_{i,\ell}'(j)P_i'(j)P_i'(j)^{-1}Q_{i,\ell}'(j)^{-1}$, $(P_i''(j)R_i''(j)Q_{i,r}''(j))^{\pm1}$, or $Q_{i,r}''(j)^{-1}R_i''(j)^{-1}R_i''(j)Q_{i,r}''(j)$.  The restriction $\pazocal{C}_0':W_r'\to\dots\to W_s'$ of $\pazocal{C}_0$ to such a subword satisfies either the hypotheses of \Cref{primitive computations}(5) or \Cref{primitive unreduced}, so that $|W_r'|_a\leq|W_s'|_a$.  As any other sector must have fixed tape word throughout $\pazocal{C}_0$, it then follows that $|W_r|_a\leq|W_s|_a$.

Moreover, by \Cref{M_5 step history 1}(b) the presence of such a subword would implies $W_s$ is not $\sigma(01)_2$-admissible, so that $s=t$.

But since $B$ contains a subword of the form $(Q_0(1)Q_1(1))^{\pm1}$, then \Cref{locked sectors} implies it must also contain a letter of the form $R_s''(L)^{\pm1}$ (and one of the form $P_0'(1)^{\pm1}$).

\textbf{2.} Suppose $H$ contains a letter of the form $\sigma(12)_2$ and $B$ contains a subword of the form $(Q_0(1)Q_1(1))^{\pm1}$.

As above, by perhaps passing to the inverse computation we may assume that there exists a maximal nonempty subcomputation $\pazocal{C}_1:W_r\to\dots\to W_s$ with step history $(1)_2$ such that $W_r$ is $\sigma(12)_2$-admissible.

By \Cref{locked sectors}, $B$ must contain a letter of the form $R_s''(L)^{\pm1}$.  As such, $B$ must contain a subword of the form $(P_s''(L)R_s''(L))^{\pm1}$ or $R_s''(L)^{-1}R_s''(L)$.  The obvious analogue of \Cref{M_4 restriction}(a) then applies to $\pazocal{C}_1$, implying that $W_s$ is not $\sigma(12)_2$-admissible and $|W_r|_a\leq(2\|B\|-3)|W_s|_a$.  By Case 1, it then suffices to assume that $s=t$, so that the parameter choice $c_0>>N$ implies $|W_r|_a\leq c_0|W_t|_a$.  \Cref{one-step hyperfaulty} and the parameter choice $c_1>>c_0$ then implies $|W_i|_a\leq c_1|W_t|_a$ for all $i\geq r$.

Let $\pazocal{C}':W_\ell\to\dots\to W_r$ be the maximal subcomputation containing no nonempty subcomputation with step history $(1)_2$.  Then the same arguments above imply $|W_\ell|_a\leq c_0|W_0|_a$ and $|W_i|_a\leq c_1|W_0|_a$ for all $i\leq\ell$.

But then every transition of $\pazocal{C}'$ satisfies the hypotheses of \Cref{lifted rule}, so that $\pazocal{C}'$ can be identified with a reduced computation of $\textbf{M}_5$ in a faulty base, so that \Cref{M_5 faulty} implies $|W_i|_a\leq c_1\max(|W_\ell|_a,|W_r|_a)$ for all $\ell\leq i \leq r$.  Hence, the statement follows from the parameter choice $c_2>>c_1>>c_0$.

\textbf{3.} Thus, it suffices to assume that $H$ contains no transition rule of the form $\sigma(01)_2^{\pm1}$ and no transition rule of the form $\sigma(12)_2^{\pm1}$ or $B$ contains no subword of the form $(Q_0(1)Q_1(1))^{\pm1}$.

But then every transition $W_{i-1}\to W_i$ satisfies the hypotheses of \Cref{lifted rule}, so that $\pazocal{C}$ can be identified with a reduced computation of $\textbf{M}_5$ in a faulty base.  The statement then follows by \Cref{M_5 faulty}.

\end{proof}

\begin{lemma} \label{hyperfaulty}

Let $\pazocal{C}:W_0\to\dots\to W_t$ be a reduced computation of $\textbf{M}$ with hyperfaulty base $B$.  Then $|W_i|_a\leq c_2\max(|W_0|_a,|W_t|_a)$ for all $0\leq i\leq t$.

\end{lemma}

\begin{proof}

As in the proofs of previous statements of this form, it suffices to assume that $|W_i|_a>\max(|W_0|_a,|W_t|_a)$ for all $1\leq i\leq t$.

As $\sigma(a)_i$ locks all sectors of the standard base and $B$ is unreduced, \Cref{locked sectors} implies the history $H$ of $\pazocal{C}$ contains no letter of the form $\sigma(a)_i^{\pm1}$.  So, by \Cref{one-machine hyperfaulty}, we may assume that $H$ contains a letter of the form $\sigma(s)_i^{\pm1}$.  What's more, as transition rules do not alter the tape words of an admissible word, it follows that $H$ contains a subword of the form $(\sigma(s)_1^{-1}\sigma(s)_2)^{\pm1}$.

Perhaps passing to the inverse computation, there then exists a maximal nonempty subcomputation $\pazocal{C}_0:W_r\to\dots\to W_s$ with step history $(0)_1$ such that $W_r$ is $\sigma(s)_1^{-1}$-admissible.  

As $\sigma(s)_1$ locks every sector of the standard base other than the input sectors, it follows that $B$ must contain a letter of the form $P_i'(j)^{\pm1}$ or $R_i''(j)^{\pm1}$.  \Cref{locked sectors} then implies $B$ contains a subword of the form $(Q_{i,\ell}'(j)P_i'(j)R_i'(j))^{\pm1}$, $Q_{i,\ell}'(j)P_i'(j)P_i'(j)^{-1}Q_{i,\ell}'(j)^{-1}$, $(P_i''(j)R_i''(j)Q_{i,r}''(j))^{\pm1}$, or $Q_{i,r}''(j)^{-1}R_i''(j)^{-1}R_i''(j)Q_{i,r}''(j)$.  So, \Cref{M step history (0)}(a) implies $W_s$ is not $\sigma(s)_1^{-1}$-admissible.

Moreover, for any subword as above, the restriction $\pazocal{C}_0':W_r'\to\dots\to W_s'$ of $\pazocal{C}_0$ to this subword satisfies either \Cref{primitive computations}(5) or \Cref{primitive unreduced}, so that $|W_r'|_a\leq|W_s'|_a$.  So, since all other sectors have fixed tape word in $\pazocal{C}_0'$, $|W_r|_a\leq|W_s|_a$.

Hence, we may assume $t>s$, so that $W_s$ is $\sigma(01)_1$-admissible.  This implies $B$ cannot contain a subword of the form $Q_{i,\ell}'(j)P_i'(j)P_i'(j)^{-1}Q_{i,\ell}'(j)^{-1}$ or of the form $Q_{i,r}''(j)^{-1}R_i''(j)^{-1}R_i''(j)Q_{i,r}''(j)$, so that every unreduced two-letter subword of $B$ is of the form $Q_0(j)Q_0(j)^{-1}$ or $Q_1(j)^{-1}Q_1(j)$.  As such, $B$ must contain a subword $B'$ of the form $(P_i'(j)R_i'(j))^{\pm1}$.  \Cref{primitive computations}(3) further implies $|W_r|_a=\dots=|W_s|_a$.

By \Cref{M_5 step history}(a), the step history of $\pazocal{C}$ cannot contain the subword $(01)_1(1)_1(10)_1$.  Further, as $B$ must contain an unreduced two-letter subword of the form $Q_0(j)Q_0(j)^{-1}$ or $Q_1(j)^{-1}Q_1(j)$, \Cref{locked sectors} implies $H$ cannot contain a letter of the form $\sigma(12)_i^{\pm1}$.  So, the subcomputation $\pazocal{C}_1:W_{s+1}\to\dots\to W_t$ must have step history $(1)_1$.

The existence of $B$ then allows us to apply \Cref{M_4 restriction}(b) to $\pazocal{C}_1$, so that the parameter choice $c_0>>N$ implies $|W_s|_a\leq c_0|W_t|_a$.  Finally, by the parameter choices $c_2>>c_1>>c_0$, applying \Cref{one-step hyperfaulty} to $\pazocal{C}_1$ implies $|W_i|_a\leq c_2|W_t|_a$ for all $r\leq i\leq t$.

Hence, it suffices to assume $r>0$, in which case by construction the two-rule subcomputation $W_{r-2}\to W_{r-1}\to W_r$ must have history $\sigma(s)_2^{-1}\sigma(s)_1$.  But then the same argument above applies to the subcomputation $W_0\to\dots\to W_{r-2}$, so that $|W_i|_a\leq c_2|W_0|_a$ for all $i\leq r$.

\end{proof}

\begin{lemma} \label{one-step pararevolving}

Let $\pazocal{C}:W_0\to\dots\to W_t$ be a reduced computation of $\textbf{M}$ with base $\{t(1)\}B_4(1)$.  If the step history of $\pazocal{C}$ is of length $1$, then $\|W_i\|\leq c_1\max(\|W_0\|,\|W_t\|)$ for all $0\leq i\leq t$.

\end{lemma}

\begin{proof}

We proceed in cases:

\textbf{1.} Suppose the step history of $\pazocal{C}$ is neither $(0)_2$ or $(1)_2$.

Then all transitions of $\pazocal{C}$ satisfy the hypotheses of \Cref{lifted rule}, so that $\pazocal{C}$ can be identified with a reduced computation of $\textbf{M}_5$ in the standard base.   \Cref{M_5 standard time} then implies $t\leq c_0\max(\|W_0\|,\|W_t\|)$.

As each transition changes any tape word by at most 2, for each $0\leq i\leq t$:
$$\|W_i\|\leq\max(\\W_0\|,\|W_t\|)+\frac{t}{2}\leq\max(\|W_0\|,\|W_t\|)+Nt\leq(c_0N+1)\max(\|W_0\|,\|W_t\|)$$
Hence, the statement follows from the parameter choices $c_1>>c_0>>N$.

\textbf{2.} Suppose the step history of $\pazocal{C}$ is $(0)_2$.

Then the restriction of $\pazocal{C}$ to any subword of the form $Q_{j,\ell}'(1)P_j'(1)R_j'(1)$ or $P_j''(1)R_j''(1)Q_{j,r}''(1)$ satisfies the hypotheses of \Cref{primitive computations}.  As the tape word of all other sectors are fixed by all rules of this step history, the statement follows for $c_1\geq2$.

\textbf{3.} Hence, it suffices to assume the step history of $\pazocal{C}$ is $(1)_2$.

By construction, the admissible subword of $W_i$ with base $Q_0(1)Q_1(1)\dots Q_{n-1}(1)Q_{0,\ell}'(1)$ can be identified with a configuration $V_i$ of $\textbf{Move}$ with empty input sector, while the history $H$ of $\pazocal{C}$ can be identified with a word in the rules of $\textbf{Move}$.  

By (Mv2), there then exists a reduced computation $\pazocal{D}:V_0\equiv U_0\to\dots\to U_t$ of $\textbf{Move}$ with history $H$ such that $|U_i|_{NI}=|V_i|_a$ for all $i$.  As a result, \Cref{Move faulty} implies $|V_i|_a\leq c_0\max(|V_0|_a,|V_t|_a)$ for all $i$.

Since the restriction of $\pazocal{C}$ to any subword of the form $P_j'(1)R_j'(1)$ or $P_j''(1)R_j''(1)$ satisfies the hypotheses of \Cref{multiply two letters} and all other sectors are locked by the rules of $\pazocal{C}$, the bound then follows from the parameter choices $c_1>>c_0\geq2$.

\end{proof}

\begin{lemma} \label{one-step standard bound}

Let $\pazocal{C}:W_0\to\dots\to W_t$ be a reduced computation of $\textbf{M}$ in the standard base.  If the step history of $\pazocal{C}$ is of length 1, then $\|W_i\|\leq c_1(\|W_0\|+\|W_t\|)$ for all $0\leq i\leq t$.

\end{lemma}

\begin{proof}

For all $1\leq j\leq L$, let $\pazocal{C}(j):W_0(j)\to\dots\to W_t(j)$ be the restriction of $\pazocal{C}$ to the base $\{t(j)\}B_4(j)$.

By \Cref{one-step pararevolving}, $\|W_i(1)\|\leq c_1\max(\|W_0(1)\|,\|W_t(1)\|)$ for all $0\leq i\leq t$.

For $j\neq1$, every transition satisfies the hypotheses of \Cref{lifted rule}, so that $\pazocal{C}(j)$ can be identified with a reduced computation of $\textbf{M}_5$ in the standard base.  As in Case 1 of the proof of \Cref{one-step pararevolving}, it then follows from \Cref{M_5 standard time} that $t\leq c_0\max(\|W_0(j)\|,\|W_t(j)\|)$, and so $\|W_i(j)\|\leq c_1\max(\|W_0(j)\|,\|W_t(j)\|)$ for all $i$.

%As each transition changes any tape word by at most $2$, for each $2\leq j\leq t$ and each $0\leq i\leq t$: 
%\begin{align*}
%\|W_i(j)\|&\leq\max(\|W_0(j)\|,\|W_t(j)\|)+\frac{t}{2}\leq\max(\|W_0(j)\|,\|W_t(j)\|)+Nt \\
%&\leq(c_0N+1)\max(\|W_0(j)\|,\|W_t(j)\|)\leq c_1\max(\|W_0(j)\|,\|W_t(j)\|)
%\end{align*}
%by the parameter choices $c_1>>c_0>>N$.  
Thus, the statement follows by noting that $\|W\|=\sum_{j=1}^L\|W(j)\|$ for any configuration $W$.

\end{proof}

\begin{lemma} \label{one-machine pararevolving}

For any one-machine computation $\pazocal{C}:W_0\to\dots\to W_t$ with base $\{t(1)\}B_4(1)$, $\|W_i\|\leq c_2\max(\|W_0\|,\|W_t\|)$ for all $0\leq i\leq t$.

\end{lemma}

\begin{proof}

By \Cref{one-step pararevolving} and the parameter choice $c_2>>c_1$, it suffices to assume that $\pazocal{C}$ contains a transition rule.

As in the proofs of \Cref{hyperfaulty} and statements of a similar nature, it suffices to assume that $\|W_i\|>\max(\|W_0\|,\|W_t\|)$ for all $1\leq i\leq t-1$.  In particular, neither the first nor the last letter of the history $H$ of $\pazocal{C}$ is a transition rule.

\textbf{1.} Suppose the step history of $\pazocal{C}$ does not contain the letters $(0)_2$ or $(1)_2$.

Note that since transition rules cannot be the first nor the last letter of $H$, the step history also does not contain the letters $(01)_2$, $(10)_2$, $(s)_j^{\pm1}$, or $(a)_j^{\pm1}$.

So, every transition of $\pazocal{C}$ satisfies the hypotheses of \Cref{lifted rule}.  But then as in the proof of \Cref{one-step pararevolving}, $\pazocal{C}$ can be identified with a reduced computation of $\textbf{M}_5$ in the standard base, so that \Cref{M_5 standard time} implies $\|W_i\|\leq c_1\max(\|W_0\|,\|W_t\|)$ for all $i$.

\textbf{2.} Suppose the step history of $\pazocal{C}$ contains the letter $(01)_2$.

Perhaps passing to the inverse computation, we may assume that $\pazocal{C}$ contains a nonempty maximal subcomputation $\pazocal{C}_0:W_r\to\dots\to W_s$ with step history $(0)_2$ such that $W_r$ is $\sigma(01)_2$-admissible.

Then, \Cref{M_5 step history 1}(b) implies $W_s$ is not $\sigma(01)_2$-admissible, so that $s=t$.

Moreover, the restriction of $\pazocal{C}_0$ to any subword of the form $Q_{i,\ell}'(1)P_i'(1)R_i'(1)$ or of the form $P_i''(1)R_i''(1)Q_{i,r}''(1)$ satisfies the hypotheses of \Cref{primitive computations}(5).  But since the tape word of all other sectors is fixed by every rule of $\pazocal{C}_0(1)$, $|W_r|_a\leq|W_t|_a$.

\textbf{3.} Hence, it suffices to assume that the step history of $\pazocal{C}$ contains the letter $(12)_2$ but no letter of the form $(0)_2$.

Perhaps passing to the inverse computation, we may assume $\pazocal{C}$ contains a nonempty maximal subcomputation $\pazocal{C}_1:W_r\to\dots\to W_s$ with step history $(1)_2$ such that $W_r$ is $\sigma(12)_2$-admissible.  By \Cref{M_2 bar step history}(a), $W_s$ cannot be $\sigma(12)_2$-admissible, so that $s=t$.  

Further, the restriction of $\pazocal{C}_1$ to any $P_j'(1)R_j'(1)$- or $P_j''(1)R_j''(1)$-sector satisfies the hypotheses of \Cref{one alphabet historical words}.  So, as in the proof of \Cref{M_5(j) standard time}, $|W_r|_a\leq2|W_t|_a$.  \Cref{one-step standard bound} and the parameter choice $c_2>>c_1$ then implies $|W_i|_a\leq c_2|W_t|_a$ for all $r\leq i\leq t$.

Let $\pazocal{D}:W_\ell\to\dots\to W_r$ be the maximal subcomputation whose step history does not contain the letter $(1)_2$.  Then \Cref{M step history} implies $W_0\to\dots\to W_\ell$ has step history $(1)_2$ if it is nonempty, so that an identical argument implies $|W_\ell|_a\leq 2|W_0|_a$ and $|W_i|_a\leq c_2|W_0|_a$ for all $0\leq i\leq\ell$.

But then $\pazocal{D}$ satisfies the hypotheses of Case 1, so that $$|W_i|_a\leq c_1\max(|W_\ell|_a,|W_r|_a)\leq c_0c_1\max(|W_0|_a,|W_t|_a)$$ for all $\ell\leq i\leq r$.  Thus, the statement follows by the parameter choices $c_2>>c_1>>c_0$.

\end{proof}

\begin{lemma} \label{one-machine standard bound}

For any one-machine computation $\pazocal{C}:W_0\to\dots\to W_t$ in the standard base, $\|W_i\|\leq c_2(\|W_0\|+\|W_t\|)$ for all $0\leq i\leq t$.

\end{lemma}

\begin{proof}

As in the proof of \Cref{one-machine pararevolving}, by \Cref{one-step standard bound} and the parameter choice $c_2>>c_1$, it suffices to assume that $\pazocal{C}$ contains a transition rule.

Further, we may again assume that $\|W_i\|>\max(\|W_0\|,\|W_t\|)$ for all $1\leq i\leq t-1$, so that neither the first nor the last letter of the history $H$ of $\pazocal{C}$ is a transition rule.

By the construction of the rules, no rule of $H$ can then be of the form $\sigma(s)_j^{\pm1}$ or $\sigma(a)_j^{\pm1}$.  So, letting $\pazocal{C}(j):W_0(j)\to\dots\to W_t(j)$ be the restriction of $\pazocal{C}$ to the base $\{t(j)\}B_4(j)$, then for $j\neq1$ every transition of $\pazocal{C}(j)$ satisfies the hypotheses of \Cref{lifted rule}.  Hence, this again implies $\|W_i(j)\|\leq c_1\max(\|W_0(j)\|,\|W_t(j)\|)$ for all $i$.

But \Cref{one-step pararevolving} also implies $\|W_i(1)\|\leq c_2\max(\|W_0(1)\|,\|W_t(1)\|)$ for all $i$, so that the statement again follows by noting $\|W\|=\sum_{j=1}^L\|W_j\|$ for all configurations $W$.

\end{proof}

\begin{lemma} \label{pararevolving special}

Let $\pazocal{C}:W_0\to\dots\to W_t$ be a nonempty one-machine computation of the first machine with base $\{t(1)\}B_4(1)$.  If $W_0$ is $\sigma(s)_2$-admissible, then $|W_0|_a\leq8|W_t|_a$ and $W_t$ is not the admissible subword of either a start or an end configuration.

\end{lemma}

\begin{proof}

Note that since $W_0$ is $\sigma(s)_2$-admissible, it must be the admissible subword of a start configuration with empty tape word in the `special' input sector.  Moreover, the first letter of the history $H$ of $\pazocal{C}$ must be $\sigma(s)_1$.

Let $\pazocal{C}_0:W_1\to\dots\to W_r$ be the maximal subcomputation with step history $(0)_1$.  Then \Cref{M step history (0)}(a) implies $W_r$ cannot be $\sigma(s)_1^{-1}$-admissible.  Further, the restriction of $\pazocal{C}_0$ to any subword of the form $Q_{i,\ell}'(1)P_i'(1)R_i'(1)$ or $P_i''(1)R_i''(1)Q_{i,r}''(1)$ satisfies the hypotheses of \Cref{primitive computations}(5).   But every other sector has fixed tape word throughout $\pazocal{C}_0$, so that $|W_0|_a\leq|W_r|_a$.

Hence, we may assume $t>r$, so that the restriction of $\pazocal{C}_0$ to any subword as above satisfies the hypotheses of \Cref{primitive computations}(3).  As such, $|W_0|_a=\dots=|W_r|_a$ and there exists $H\in F(\Phi^+)$ such that $W_0\equiv I(1,H,1)$.

Moreover, there exists a maximal subcomputation $\pazocal{C}_1:W_{r+1}\to\dots\to W_s$ with step history $(1)_1$.  As $W_{r+1}$ is $\sigma(10)_1$-admissible, \Cref{M_2 bar step history}(a) then implies $W_s$ is not $\sigma(10)_1$-admissible.  As the restriction of $\pazocal{C}_1$ to any $P_i'(1)R_i'(1)$- or $P_i''(1)R_i''(1)$-sector satisfies the hypotheses of \Cref{one alphabet historical words}, then as in the proof of \Cref{M_5(j) standard time} $|W_r|_a\leq2|W_s|_a$.

Hence, we may assume $t>s$.  As such, the restriction of $\pazocal{C}_1$ to the subword $Q_0(1)\dots Q_{n-1}(1)Q_{0,\ell}'(1)$ can be identified with a setup computation of $\textbf{Move}$, so that (Mv3) implies it is a move computation of the trivial word.

Let $\pazocal{C}_2:W_{s+1}\to\dots\to W_x$ be the maximal subcomputation with step history $(2)_1$.  Then an analogous argument to that applied to $\pazocal{C}_1$ implies $W_x$ is not $\sigma(21)_1$-admissible and $|W_s|_a\leq2|W_x|_a$, so that $|W_0|_a\leq4|W_x|_a$.

Hence, we may assume $t>x$.  So, as $\pazocal{C}_2$ operates as a Clean machine in the $Q_{n-1}(1)Q_{0,\ell}'(1)$- and $Q_{0,r}'(1)Q_{1,\ell}'(1)$-sectors, it follows that the tape words of $W_x$ are all empty except for the historical $P_i'(1)R_i'(1)$- and $P_i''(1)R_i''(1)$-sectors.

Now, there exists a maximal subcomputation $\pazocal{C}_3:W_{x+1}\to\dots\to W_y$ with step history $(3)_1$.  Again, the analogous argument implies $W_y$ is not $\sigma(32)_1$-admissible and $|W_0|_a\leq8|W_y|_a$.  But since $\pazocal{C}_3$ operates as $\textbf{S}$ in the working sectors and $1\notin\pazocal{R}_1$, $W_y$ cannot be $\sigma(34)_1$-admissible either, so that $y=t$.

\end{proof}

\begin{lemma} \label{M standard bound special}

Let $\pazocal{C}:W_0\to\dots\to W_t$ be a nonempty one-machine computation of the first machine in the standard base.  If $W_0$ is $\sigma(s)_2$-admissible, then $|W_0|_a\leq 8|W_t|_a$ and $W_t$ is neither a start nor an end configuration.

\end{lemma}

\begin{proof}

The restriction of $\pazocal{C}$ to the subword $\{t(1)\}B_4(1)$ of the standard base satisfies the hypotheses of \Cref{pararevolving special}, so that $|W_0(1)|_a\leq8|W_t(1)|_a$ and $W_t$ is neither a start nor an end configuration.  Moreover, for each $2\leq j\leq L$, an identical argument to the proof of \Cref{pararevolving special} implies $|W_0(j)|_a\leq8|W_t(j)|_a$.  Hence, $|W_0|_a=\sum_{j=1}^L|W_0(j)|_a\leq8\sum_{j=1}^L|W_t(j)|_a=8|W_t|_a$.

\end{proof}

\begin{lemma} \label{pararevolving}

For any reduced computation $\pazocal{C}:W_0\to\dots\to W_t$ of $\textbf{M}$ with base $\{t(1)\}B_4(1)$, $\|W_i\|\leq c_3\max(\|W_0\|,\|W_t\|)$ for all $0\leq i\leq t$.

\end{lemma}

\begin{proof}

By \Cref{one-machine pararevolving} and the parameter choice $c_3>>c_2$, it suffices to assume that $\pazocal{C}$ is a multi-machine compuation.

Further, we may again assume $\|W_i\|>\max(\|W_0\|,\|W_t\|)$ for all $1\leq i\leq t$, so that $W_i\neq A(1)$ for any $i$.

Perhaps passing to the inverse computation, it may be assumed there exists a maximal subcomputation $\pazocal{C}_1:W_r\to\dots\to W_s$ with $r>0$ which is a one-machine computation of the first machine.  As $W_r\neq A(1)$, it must be a subword of an input configuration, and so must be $\sigma(s)_2$-admissible.  \Cref{pararevolving special} then implies $s=t$ and $|W_r|_a\leq8|W_t|_a$, so that \Cref{one-machine pararevolving} and the parameter choice $c_3>>c_2$ yield $\|W_i\|\leq c_3\|W_t\|$ for all $r\leq i\leq t$.

Let $\pazocal{C}_2:W_\ell\to\dots\to W_r$ be the maximal subcomputation which is a one-machine computation of the second machine.  If $\ell>0$, then the analogous argument applies to the inverse computation $W_\ell\to\dots\to W_0$, so that $|W_\ell|_a\leq8|W_0|_a$ and $\|W_i\|\leq c_3\|W_0\|$ for all $0\leq i\leq\ell$.

Finally, applying \Cref{one-machine pararevolving} to $\pazocal{C}_2$ implies $\|W_i\|\leq c_2\max(\|W_\ell\|,\|W_r\|)\leq8c_2\max(\|W_0\|,\|W_t\|)$ for all $\ell\leq i\leq r$, so that the statement follows by the parameter choice $c_3>>c_2$.

\end{proof}

\begin{lemma} \label{M standard bound}

For any reduced computation $\pazocal{C}:W_0\to\dots\to W_t$ of $\textbf{M}$ in the standard base, $\|W_i\|\leq c_3\max(\|W_0\|,\|W_t\|)$ for all $0\leq i\leq t$.

\end{lemma}

\begin{proof}

This follows by an identical proof to \Cref{pararevolving}, employing Lemmas \ref{M standard bound special} and \ref{one-machine standard bound} in place of Lemmas \ref{pararevolving special} and \ref{one-machine pararevolving}, respectively.

%By \Cref{one-machine standard bound} and the parameter choice $c_3>>c_2$, it suffices to assume that $\pazocal{C}$ is a multi-machine computation.  
%
%As in proofs of similar statements, it suffices to assume that $\|W_i\|>\max(\|W_0\|,\|W_t\|)$ for all $1\leq i\leq t$.  As such, it suffices to assume that $W_i\neq W_{ac}$ for any $i$.
%
%Perhaps passing to the inverse computation, let $\pazocal{D}:W_r\to\dots\to W_s$ be a nonempty maximal one-machine computation of $\pazocal{C}$ of the first machine such that $r>0$.  As it cannot be the accept configuration, $W_r$ must be an input configuration, and so must be $\sigma(s)_2$-admissible.  \Cref{M standard bound special} then implies $s=t$ and $|W_r|_a\leq8|W_t|_a$, so that \Cref{one-machine standard bound} and the parameter choice $c_3>>c_2$ yield $\|W_i\|\leq c_3\|W_t\|$ for all $r\leq i\leq t$.
%
%Let $\pazocal{C}_2:W_\ell\to\dots\to W_r$ be the maximal subcomputation which is a one-machine computation of the second machine.  If $\ell>0$, then the analogous argument applies to the inverse computation $W_\ell\to\dots\to W_0$, so that $|W_\ell|_a\leq8|W_0|_a$ and $\|W_i\|\leq c_3\|W_0\|$ for all $0\leq i\leq \ell$.
%
%Finally, applying \Cref{one-machine standard bound} to $\pazocal{C}_2$ implies $\|W_i\|\leq c_2(\|W_\ell\|+\|W_r\|)\leq8c_2(\|W_0\|+\|W_t\|)$ for all $\ell\leq i\leq r$, so that the statement follows by the parameter choice $c_3>>c_2$.

\end{proof}

\begin{lemma} \label{revolving bound}

For any reduced computation $\pazocal{C}:W_0\to\dots\to W_t$ of $\textbf{M}$ with revolving base $B$, $\|W_i\|\leq c_4\max(\|W_0\|,\|W_t\|)$ for all $0\leq i\leq t$.

\end{lemma}

\begin{proof}

By \Cref{hyperfaulty}, \Cref{M standard bound}, and the parameter choices $c_4>>c_3>>c_2$, it suffices to assume $B$ is faulty but not hyperfaulty.

Passing to a cyclic permutation, we may then assume $B\equiv ED_1\dots D_mD_m^{-1}\dots D_1^{-1}P_0^{-1}$ such that:

\begin{itemize}

\item Each $D_i$ is a (cyclic) subword of the standard base.

\item $P_0$ is the last letter of $E$ and the word $D_1'\equiv P_0D_1$ is pararevolving.

\item For each $i\in\{2,\dots,m\}$, letting $P_{i-1}$ be the last letter of $D_{i-1}$, the word $D_i'\equiv P_{i-1}D_i$ is pararevolving.

\item $E$ is unreduced and no letter appears twice in $\pi(E)$.

%\item The first letter of $\pi(D_m^{-1})$ does not appear in $\pi(E)$.

\end{itemize}

Let $\pazocal{C}':W_0'\to\dots\to W_t'$ be the restriction of $\pazocal{C}$ to $E$.  Further, letting $P_m$ be the last letter of $D_m$, let $\pazocal{C}'':W_0''\to\dots\to W_t''$ be the restriction of $\pazocal{C}$ to the subword $P_mP_m^{-1}$.

Note that by construction, $\pi(P_0)=\pi(P_1)=\dots=\pi(P_m)$ is a positive letter from the standard base of $\textbf{M}_5$.  Further, as $B$ is faulty, the first letter of $E$ is $P_0^{-1}$.

Suppose $\pi(P_0)\neq Q_0$.  Then there exists a coordinate shift $V_0''\to\dots\to V_t''$ of $\pazocal{C}''$ whose base is $P_0P_0^{-1}$.  For each $i$, pasting $W_i'$ to $V_i''$ across the state letter of $P_0$ then produces an admissible word $U_i$ so that there exists a reduced computation $\pazocal{E}:U_0\to\dots\to U_t$ with base $EP_0^{-1}$.

If $P_0\neq Q_0(1)$ and $P_m\neq Q_0(1)$, then a similar coordinate shift produces a a reduced computation $\pazocal{E}$ with base $EP_0^{-1}$.

Now, suppose $P_m=Q_0(1)$.  Then since $B$ contains the subword $P_mP_m^{-1}$, no rule of the history $H$ of $\pazocal{C}$ can lock the `special' input sector (or any $Q_0(j)Q_1(j)$-sector).  Hence, we may construct the coordinate shift $\pazocal{E}$ in just the same way.

Finally, suppose $P_0=Q_0(1)$.  Then by construction $E$ cannot contain a subword of the form $(Q_0(1)Q_1(1))^{\pm1}$, $Q_0(1)Q_0(1)^{-1}$, or $Q_1(1)^{-1}Q_1(1)$.  So, for each $i$ there exists a coordinate shift $V_0'\to\dots\to V_t'$ of $\pazocal{C}'$ whose base $\tilde{E}$ ends with the letter $P_m$.  In this case, we paste $V_i'$ to $W_i''$ across the state letter of $P_m$ to produce admissible words $U_i$ so that there exists a reduced computation $\pazocal{E}:U_0\to\dots\to U_t$ with base $\tilde{E}P_m^{-1}$.

In either case, $\pazocal{E}$ is a reduced computation with hyperfaulty base, so that \Cref{hyperfaulty} implies $|U_i|_a\leq c_2\max(|U_0|_a,|U_t|_a)$ for all $0\leq i\leq t$.  But by construction $|U_i|_a=|W_i'|_a+|W_i''|_a$ for all $i$.

Hence, letting $\pazocal{D}_j:W_{0,j}\to\dots\to W_{t,j}$ be the restriction of $\pazocal{C}$ to $D_j'$ for each $j\in\{1,\dots,m\}$, it suffices to show that $\|W_{i,j}\|\leq c_3\max(\|W_{0,j}\|,\|W_{t,j}\|)$ for all $i$.

For each $1\leq j\leq m$, fix $\ell_j\in\{1,\dots,L\}$ such that $D_j'$ contains the subword $Q_0(\ell_j)Q_1(\ell_j)$.  As every rule operates in parallel on other types of subwords and locks the $Q_{s,r}''(\ell)\{t(\ell+1)\}$-sectors, taking cyclic shifts produces a reduced computation $\pazocal{D}_j':V_0\to\dots\to V_t$ with base $\{t(\ell_j)\}B_4(\ell_j)$ such that $\|V_i\|=\|W_{i,j}\|-1$ for each $i$.

If $\ell_j=1$, then \Cref{pararevolving} implies $\|V_i\|\leq c_3\max(\|V_0\|,\|V_t\|)$ for all $i$.

Otherwise, since $B$ is unreduced, \Cref{M controlled} implies $H$ contains no controlled subword.  So, \Cref{M projected long history controlled} implies $t\leq c_2\max(\|V_0\|,\|V_t\|)$.  So, since the length of the base of $\pazocal{D}_j'$ is $N$, the parameter choices $c_3>>c_2>>N$ again imply $\|V_i\|\leq c_3\max(\|V_0\|,\|V_t\|)$ for al $i$.

Hence, taking $c_3\geq1$, for all $0\leq i\leq t$ we have:
\begin{align*}
\|W_{i,j}\|&=\|V_i\|+1\leq c_3\max(\|V_0\|,\|V_t\|)+1=c_3\max(\|W_{0,j}\|-1,\|W_{t,j}\|-1)+1 \\
&=c_3\max(\|W_{0,j}\|,\|W_{t,j}\|)-c_3+1\leq c_3\max(\|W_{0,j}\|,\|W_{t,j}\|)
\end{align*}

\end{proof}

%An $S$-graph representation of $\textbf{M}$ is given below in \Cref{fig-MainMachine}.  As in previous such representations, the vertices $(6,i)_s$ and $(6,i)_f$ correspond to the start and end states, respectively, of the machine $\textbf{M}_{6,i}$, while the edges labelled $\textbf{M}_{6,i}$ correspond to computations of that machine.  The purpose of the other vertices and edges are evident from their labels.
%
%\begin{figure}[hbt]
%	\noindent
%	\include{MainMachine}
%	\caption{The machine $\textbf{M}$} 
%	\label{fig-MainMachine}       
%\end{figure} 

\bigskip

%%%%%%%%%%%%%%%%%%%%%%%%%%%%%%%%%%%%%%%%%%%%%%%%%%

\section{Groups Associated to an $S$-machine and their Diagrams}

\subsection{The groups} \label{sec-associated-groups} \

As in previous literature (for example \cite{O18}, \cite{OS19}, \cite{WCubic}, \cite{W}), we now associate two finitely presented groups to a cyclic $S$-machine $\textbf{S}$. These groups are denoted $M(\textbf{S})$ and $G(\textbf{S})$ and `simulate' the work of $\textbf{S}$ in the precise sense described in \Cref{sec-trapezia}.

Let $\textbf{S}$ be an arbitrary cyclic recognizing $S$-machine with hardware $(Y,Q)$, where $Q=\sqcup_{i=0}^s Q_i$ and $Y=\sqcup_{i=1}^{s+1} Y_i$, and software the set of rules $\Theta=\Theta^+\sqcup\Theta^-$. For notational purposes, set $Q_0=Q_{s+1}$ and denote the accept word of $\textbf{S}$ by $W_{ac}$.

For $\theta\in\Theta^+$, Lemma \ref{simplify rules} allows us to assume that $\theta$ takes the form $$\theta=[q_0\to v_{s+1}q_0'u_1, \ q_1\to v_1q_1'u_2, \ \dots, \ q_{s-1}\to v_{s-1}q_{s-1}'u_s, \ q_s\to v_sq_s'u_{s+1}]$$ where $q_i,q_i'\in Q_i$, $u_i$ and $v_i$ are either empty or single letters in $Y_i^{\pm1}$, and some of the arrows may take the form $\xrightarrow{\ell}$. Note that if $\theta$ locks the $i$-th sector, then both $u_i$ and $v_i$ are necessarily empty.

Define $R=\{\theta_i: \theta\in\Theta^+,0\leq i\leq s\}$. For notational convenience, set $\theta_{s+1}=\theta_0$ for all $\theta\in\Theta^+$.

The group $M(\textbf{S})$ is then defined by taking the (finite) generating set $\pazocal{Y}=Q\cup Y\cup R$ and subjecting it to the (finite number of) relations:

\begin{addmargin}[1em]{0em}

$\bullet$ $q_i\theta_{i+1}=\theta_i v_iq_i'u_{i+1}$ for all $\theta\in\Theta^+$ and $0\leq i\leq s$,

$\bullet$ $\theta_ia=a\theta_i$ for all $0\leq i\leq s$ and $a\in Y_i(\theta)$.

\end{addmargin}

As in the language of computations of $S$-machines, letters from $Q^{\pm1}$ are called \textit{$q$-letters} and those from $Y^{\pm1}$ are called \textit{$a$-letters}. Additionally, those from $R^{\pm1}$ are called \textit{$\theta$-letters}. The relations of the form $q_i\theta_{i+1}=\theta_iv_iq_i'u_{i+1}$ are called \textit{$(\theta,q)$-relations}, while those of the form $\theta_ia=a\theta_i$ are called \textit{$(\theta,a)$-relations}.

Note that the number of $a$-letters in any part of $\theta$, and so in any defining relation of $M(\textbf{S})$, is at most two.

To simplify these relations, it is convenient to omit reference to the indices of the letters of $R$. This notational quirk may make it appear as though $\theta$ commutes with the letters of $Y_i(\theta)$ and conjugates $q_i$ to $v_iq_i'u_{i+1}$ for each $i$; it should be noted that these statements are not strictly true. Further, it is useful to note that if $\theta$ locks the $i$-th sector, then $Y_i(\theta)=\emptyset$ so that $\theta$ has no relation with the elements of $Y_i$.

However, this group evidently lacks any reference to the accept configuration. To amend this, the group $G(\textbf{S})$ is constructed by adding one more relation to the presentation of $M(\textbf{S})$, namely the \textit{hub-relation} $W_{ac}=1$. In other words, $G(\textbf{S})\cong M(\textbf{S})/\gen{\gen{W_{ac}}}$.

Moreover, as in \cite{W} and \cite{WMal}, it is useful for the purposes of our construction to consider extra relations, called \textit{$a$-relations}, within the language of tape letters. If $\Omega$ is the set of relators defining these $a$-relations, then we denote the groups arising from the addition of $a$-relations by $M_\Omega(\textbf{S})$ and $G_\Omega(\textbf{S})$. Note that $M_\Omega(\textbf{S})\cong M(\textbf{S})/\gen{\gen{\Omega}}$ and $G_\Omega(\textbf{S})\cong G(\textbf{S})/\gen{\gen{\Omega}}$.

It is henceforth taken as an assumption that any $a$-relation adjoined to the groups associated to the machine $\textbf{M}$ correspond to words over the alphabet $X$ of the `special' input sector.  Moreover, as in the previous settings, it is assumed that $\Omega$ is the set of all non-trivial cyclically reduced words over $X$ that represent the trivial element in $G$.

%It is henceforth taken as an assumption that any $a$-relation adjoined to the groups associated to the machine $\textbf{M}$ correspond to words over the alphabet $X$ of the `special' input sector.  Moreover, it is assumed that $\Omega$ is a normal subgroup of $F(X)$ containing the set of all reduced words $\pazocal{S}$ that represent the trivial element in $G$.
%
%For the purposes of the proofs of Theorems \ref{main theorem} and \ref{distortion} presented in Sections 11 and 12, it is sufficient to take $\Omega=\pazocal{S}$. However, in the proof of Theorem \ref{CEP} presented in Section 13, the $a$-relators must be taken as a larger set of words.  Thus, $\Omega$ is treated generally in Sections 6-10.

Note that though the presentations defining $M_\Omega(\textbf{S})$ and $G_\Omega(\textbf{S})$ still have finite generating sets, the number of relations become infinite.  As such, the presentations we are studying are no longer finite.  However, as we will establish, in the relevant cases $M_\Omega(\textbf{M})$ and $G_\Omega(\textbf{M})$ define groups isomorphic to the groups $M(\textbf{M})$ and $G(\textbf{M})$, respectively, and so remain finitely presented (or at least finitely presentable).

\medskip

%%%%%%%%%%%%%%%%%%%%%%%%%%%%%%%%%%%%%%%%%%%%%%%%%%

\subsection{Bands and annuli} \label{sec-bands-annuli} \

The arguments presented in the forthcoming sections rely on van Kampen diagrams over the presentations of the groups introduced in \Cref{sec-associated-groups}.  It is assumed that the author is intimately acquainted with this notion; for reference, see \cite{Lyndon-Schupp} and \cite{O}.

To present these arguments efficiently, we first differentiate between the types of edges and cells that arise in such diagrams in a way similar to that employed in \cite{W}.  For simplicity, we will adopt the convention that the contour of any diagram, subdiagram, or cell is traced in the counterclockwise direction.

An edge labelled by a $q$-letter is called a \textit{$q$-edge}. Similarly, an edge labelled by an $a$-letter is called an \textit{$a$-edge} and one labelled by a $\theta$-letter is a \textit{$\theta$-edge}. 

For a path \textbf{p} in $\Delta$, the (combinatorial) length of $\textbf{p}$ is denoted $\|\textbf{p}\|$. Further, the path's \textit{$a$-length} $|\textbf{p}|_a$ is the number of $a$-edges in the path. The path's \textit{$\theta$-length} and \textit{$q$-length}, denoted $|\textbf{p}|_{\theta}$ and $|\textbf{p}|_q$, respectively, are defined similarly.

A cell whose contour label corresponds to a $(\theta,q)$-relation is called a \textit{$(\theta,q)$-cell}. Similarly, there are \textit{$(\theta,a)$-cells}, \textit{$a$-cells}, and \textit{hubs}.

In the general setting of a reduced diagram $\Delta$ over an arbitrary presentation, let $\pazocal{Z}$ be a subset of the generating set. For $m\geq1$, a sequence of (distinct) cells $\pazocal{B}=(\Pi_1,\dots,\Pi_m)$ in $\Delta$ is called a \textit{$\pazocal{Z}$-band} of length $m$ if:

\begin{itemize}

\item every two consecutive cells $\Pi_i$ and $\Pi_{i+1}$ have a common boundary edge $\textbf{e}_i$ labeled by a letter from $\pazocal{Z}^{\pm1}$ and

\item for every $i$, $\partial\Pi_i$ has exactly two edges labelled by a letter from $\pazocal{Z}^{\pm1}$, $\textbf{e}_{i-1}^{-1}$ and $\textbf{e}_i$, and $\text{Lab}(\textbf{e}_{i-1})$ and $\text{Lab}(\textbf{e}_i)$ are either both positive or both negative.

\end{itemize}

For convenience, we extend this definition by saying that any edge labelled by a letter of $\pazocal{Z}^{\pm1}$ is a $\pazocal{Z}$-band of length zero.

A $\pazocal{Z}$-band is \textit{maximal} if it is not contained in any other $\pazocal{Z}$-band. Note that every edge labelled by a letter of  $\pazocal{Z}^{\pm1}$ is contained in a maximal $\pazocal{Z}$-band.

In a $\pazocal{Z}$-band $\pazocal{B}$ of length $m\geq1$ made up of the cells $(\Pi_1,\dots,\Pi_m)$, using only edges from the contours of $\Pi_1,\dots,\Pi_m$, there exists a closed path $\textbf{e}_0^{-1}\textbf{q}_1\textbf{e}_m\textbf{q}_2^{-1}$ such that $\textbf{q}_1$ and $\textbf{q}_2$ are simple (perhaps closed) paths. In this case, $\textbf{q}_1$ is called the \text{bottom} of $\pazocal{B}$, denoted $\textbf{bot}(\pazocal{B})$, while $\textbf{q}_2$ is called the \textit{top} of $\pazocal{B}$ and denoted $\textbf{top}(\pazocal{B})$. When $\textbf{q}_1$ and $\textbf{q}_2$ need not be distinguished, they are called the \textit{sides} of the band.

\begin{figure}[H]
\centering
\begin{subfigure}[b]{0.48\textwidth}
\centering
\raisebox{0.5in}{\includegraphics[scale=1.35]{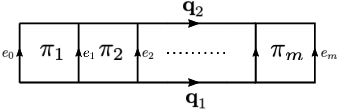}}
\caption{Non-annular $\pazocal{Z}$-band of length $m$}
\end{subfigure}\hfill
\begin{subfigure}[b]{0.48\textwidth}
\centering
\includegraphics[scale=1.35]{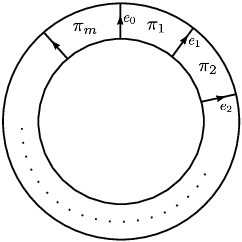}
\caption{Annular $\pazocal{Z}$-band of length $m$}
\end{subfigure}
\caption{ \ }
\end{figure}

If $\textbf{e}_0=\textbf{e}_m$ in a $\pazocal{Z}$-band $\pazocal{B}$ of length $m\geq1$, then $\pazocal{B}$ is called a \textit{$\pazocal{Z}$-annulus}. If $\pazocal{B}$ is a non-annular $\pazocal{Z}$-band of length $m\geq1$, then $\textbf{e}_0^{-1}\textbf{q}_1\textbf{e}_m\textbf{q}_2^{-1}$ is called the \textit{standard factorization} of the contour of $\pazocal{B}$. If either $(\textbf{e}_0^{-1}\textbf{q}_1\textbf{e}_m)^{\pm1}$ or $(\textbf{e}_m\textbf{q}_2^{-1}\textbf{e}_0^{-1})^{\pm1}$ is a subpath of $\partial\Delta$, then $\pazocal{B}$ is called a \textit{rim $\pazocal{Z}$-band}.

A $\pazocal{Z}_1$-band and a $\pazocal{Z}_2$-band \textit{cross} if they have a common cell and $\pazocal{Z}_1\cap\pazocal{Z}_2=\emptyset$.

In particular, in a reduced diagram over the canonical presentations of the groups of interest, there exist \textit{$q$-bands} corresponding to bands arising from $\pazocal{Z}=Q_i^{\pm1}$ for some $i$, where every cell is a $(\theta,q)$-cell. Similarly, there exist \textit{$\theta$-bands} for $\theta\in\Theta^+$ and \textit{$a$-bands} for $a\in Y$. However, it is useful to restrict the definition of an $a$-band so that they consist only of $(\theta,a)$-cells.

Note that by definition, distinct maximal $q$-bands ($\theta$-bands, $a$-bands) cannot intersect.

Given an $a$-band $\pazocal{B}$, the makeup of the groups' relations dictates that the defining $a$-edges $\textbf{e}_0,\dots,\textbf{e}_m$ are labelled identically. Similarly, the $\theta$-edges of a $\theta$-band correspond to the same rule; however, the (suppressed) index of two such $\theta$-edges may differ.

If a maximal $a$-band contains a cell with an $a$-edge that is also on the contour of a $(\theta,q)$-cell, then the $a$-band is said to \textit{end} (or \textit{start}) on that $(\theta,q)$-cell and the corresponding $a$-edge is said to be the \textit{end} (or \textit{start}) of the band. This definition extends similarly, so that: 

\begin{itemize}

\item a maximal $a$-band can end on a $(\theta,q)$-cell, on an $a$-cell, or on the diagram's contour,

\item a maximal $\theta$-band can end only on the diagram's contour, and

\item a maximal $q$-band can end on a hub or on the diagram's contour.

\end{itemize}

Note that if a maximal $\theta$-band ($a$-band, $q$-band) ends as above in one part of the diagram, then it must also end in another part of the diagram as it cannot be a $\theta$-annulus ($a$-annulus, $q$-annulus).

The natural projection of the label of the top (or bottom) of a $q$-band onto $F(\Theta^+)$ is called the \textit{history} of the band; the \textit{step history} of the band is then defined in the obvious way. The natural projection (without reduction) of the top (or bottom) of a $\theta$-band onto the alphabet $\{Q_0,\dots,Q_s\}$ is called the \textit{base} of the band.

Let $\pazocal{T}$ be a maximal $\theta$-band in a reduced diagram $\Delta$ over $G_\Omega(\textbf{M})$ with two ends on $\partial\Delta$. Suppose that any cell between one side of $\pazocal{T}$ and $\partial\Delta$ is an $a$-cell. Then $\pazocal{T}$ is called a \textit{quasi-rim $\theta$-band}. Note that a rim $\theta$-band is a quasi-rim $\theta$-band.

Suppose the sequence of cells $(\pi_0,\pi_1,\dots,\pi_m)$ comprises a $\theta$-band and $(\gamma_0,\gamma_1,\dots,\gamma_\ell)$ a $q$-band such that $\pi_0=\gamma_0$, $\pi_m=\gamma_\ell$, and no other cells are shared. Suppose further that $\partial\pi_0$ and $\partial\pi_m$ both contain edges on the outer countour of the annulus bounded by the two bands. Then the union of these two bands is called a \textit{$(\theta,q)$-annulus} and $\pi_0$ and $\pi_m$ are called its \textit{corner} cells. A \textit{$(\theta,a)$-annulus} is defined similarly.

\begin{figure}[H]
\centering
\includegraphics[scale=1.2]{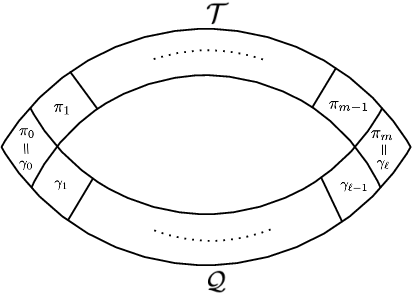}
\caption{$(\theta,q)$-annulus with defining $\theta$-band $\pazocal{T}$ and $q$-band $\pazocal{Q}$}
\end{figure}

The following statement is proved in a more general setting in \cite{O97}:

\begin{lemma}[Lemma 6.1 of \cite{O97}] \label{M(S) annuli}

A reduced diagram over $M(\textbf{S})$ contains no:

\begin{enumerate}[label=({\arabic*})]

\item $(\theta,q)$-annuli

\item $(\theta,a)$-annuli

\item $a$-annuli

\item $q$-annuli

\item $\theta$-annuli

\end{enumerate}

\end{lemma}

As a result, in a reduced diagram $\Delta$ over $M(\textbf{S})$, if a maximal $\theta$-band and a maximal $q$-band (respectively $a$-band) cross, then their intersection is exactly one $(\theta,q)$-cell (respectively $(\theta,a)$-cell). Further, every maximal $\theta$-band and maximal $q$-band ends on $\partial\Delta$ in two places.

\medskip

%%%%%%%%%%%%%%%%%%%%%%%%%%%%%%%%%%%%%%%%%%%%%%%%%%

\subsection{Trapezia} \label{sec-trapezia} \

Let $\Delta$ be a reduced diagram over the canonical presentation of $M(\textbf{S})$ whose contour is of the form $\textbf{p}_1^{-1}\textbf{q}_1\textbf{p}_2\textbf{q}_2^{-1}$, where $\textbf{p}_1$ and $\textbf{p}_2$ are sides of $q$-bands and $\textbf{q}_1$ and $\textbf{q}_2$ are maximal parts of the sides of $\theta$-bands whose labels start and end with $q$-letters. Then $\Delta$ is called a \textit{trapezium}.

In this case, $\textbf{q}_1$ and $\textbf{q}_2$ are called the \textit{bottom} and \textit{top} of the trapezium, respectively, while $\textbf{p}_1$ and $\textbf{p}_2$ are the \textit{left} and \textit{right} sides. Further, $\textbf{p}_1^{-1}\textbf{q}_1\textbf{p}_2\textbf{q}_2^{-1}$ is called the \textit{standard factorization} of the contour.

%\begin{figure}[H]
%\centering
%\includegraphics[scale=1.25]{trapezium.eps}
%\caption{Trapezium with side $q$-bands $\pazocal{Q}_1$ and $\pazocal{Q}_2$}
%\end{figure}

The \textit{(step) history} of the trapezium is the (step) history of the rim $q$-band with $\textbf{p}_2$ as one of its sides and the length of this history is the trapezium's \textit{height}. The base of $\text{Lab}(\textbf{q}_1)$ is called the \textit{base} of the trapezium.

It's easy to see from this definition that a $\theta$-band $\pazocal{T}$ whose first and last cells are $(\theta,q)$-cells can be viewed as a trapezium of height 1 as long as its top and bottom start and end with $q$-edges. We extend this to all such $\theta$-bands by merely disregarding any $a$-edges of the top and bottom that precede the first $q$-edge or follow the final $q$-edge. The paths formed by disregarding these edges are called the \textit{trimmed} top and bottom of the band and are denoted $\textbf{ttop}(\pazocal{T})$ and $\textbf{tbot}(\pazocal{T})$.

\begin{figure}[H]
\centering
\includegraphics[scale=1.75]{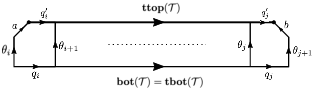}
\caption{$\theta$-band $\pazocal{T}$ with trimmed top}
\end{figure}

By Lemma \ref{M(S) annuli}, any trapezium $\Delta$ of height $h\geq1$ can be decomposed into $\theta$-bands $\pazocal{T}_1,\dots,\pazocal{T}_h$ connecting the left and right sides of the trapezium, with $\textbf{bot}(\pazocal{T}_1)$ and $\textbf{top}(\pazocal{T}_h)$ making up the bottom and top of $\Delta$, respectively.  Moreover, the first and last cells of each $\pazocal{T}_i$ are $(\theta,q)$-cells, so that $\textbf{ttop}(\pazocal{T}_i)=\textbf{tbot}(\pazocal{T}_{i+1})$ for all $1\leq i\leq h-1$.  In this case, the bands $\pazocal{T}_1,\dots,\pazocal{T}_h$ are said to be \textit{enumerated from bottom to top}.

The following two statements are proved in more generality in \cite{WMal} and exemplify how the group $M(\textbf{S})$ simulates the work of the $S$-machine:

\begin{lemma} \label{trapezia are computations}

Let $\Delta$ be a trapezium with history $H\equiv\theta_1\dots\theta_h$ for $h\geq1$ and maximal $\theta$-bands $\pazocal{T}_1,\dots,\pazocal{T}_h$ enumerated from bottom to top. Letting $W_{j-1}\equiv\lab(\textbf{tbot}(\pazocal{T}_j))$ for $1\leq j\leq h$ and letting $W_h\equiv\lab(\textbf{ttop}(\pazocal{T}_h))$, then there exists a reduced computation $W_0\to\dots\to W_h$ of $\textbf{S}$ with history $H$.

\end{lemma}

\begin{lemma} \label{computations are trapezia}

For any non-empty reduced computation $W_0\to\dots\to W_t$ of $\textbf{S}$ with history $H$, there exists a trapezium $\Delta$ such that: 

\begin{enumerate} [label=(\alph*)]

\item $\lab(\textbf{tbot}(\Delta))\equiv W_0$

\item $\lab(\textbf{ttop}(\Delta))\equiv W_t$

\item The history of $\Delta$ is  $H$

\item $\text{Area}(\Delta)\leq t\max(\|W_0\|,\dots,\|W_t\|)$

\end{enumerate}

\end{lemma}

\bigskip

%%%%%%%%%%%%%%%%%%%%%%%%%%%%%%%%%%%%%%%%%%%%%%%%%%

\section{Modified length and area functions}

\subsection{Modified length function} \

To assist with the proofs to come, we now modify the length function on words over the groups associated to an $S$-machine and paths in diagrams over their presentations. This is done in the same way as in \cite{O18}, \cite{OS19}, and \cite{W}. The standard length of a word/path will henceforth be referred to as its \textit{combinatorial length} and the modified length simply as its \textit{length}.

Define a word consisting of no $q$-letters, one $\theta$-letter, and one $a$-letter as a \textit{$(\theta,a)$-syllable}. Then, define the length of:

\begin{itemize}

\item any $q$-letter as 1

\item any $\theta$-letter as 1

\item any $a$-letter as the parameter $\delta$ (as indicated in \Cref{sec-parameters}, this should be thought of as a very small positive number)

\item any $(\theta,a)$-syllable as 1

\end{itemize}

For a word $w$ over the generators of the canonical presentation of $G_\Omega(\textbf{S})$ (or any group associated to $\textbf{S}$), define a \textit{decomposition} of $w$ as a factorization of $w$ into a product of letters and $(\theta,a)$-syllables. The length of a decomposition of $w$ is then defined to be the sum of the lengths of the factors. 

Finally, the length of $w$, denoted $|w|$, is defined to be the minimum of the lengths of its decompositions.

The length of a path in a diagram over the presentations of the groups associated to $\textbf{S}$ is defined to be the length of its label.

The following gives some basic properties of the length function. Its proof is an immediate consequence of Lemma \ref{simplify rules}.

\begin{lemma}[Lemma 6.2 of \cite{OS19}] \label{lengths}

Let \textbf{s} be a path in a diagram $\Delta$ over the canonical presentation of $G_\Omega(\textbf{S})$ (or any of the groups associated to $\textbf{S}$) consisting of $c$ $\theta$-edges and $d$ $a$-edges. Then:

\begin{enumerate}[label=({\alph*})]

\item $|\textbf{s}|\geq\max(c,c+(d-c)\delta)$

\item $|\textbf{s}|=c$ if $\textbf{s}$ is the top or a bottom of a $q$-band

\item For any product $\textbf{s}=\textbf{s}_1\textbf{s}_2$ of two paths in a diagram,
$$|\textbf{s}_1|+|\textbf{s}_2|-\delta\leq|\textbf{s}|\leq|\textbf{s}_1|+|\textbf{s}_2|$$

\item Let $\pazocal{T}$ be a $\theta$-band with base of length $l_b$. If $\textbf{top}(\pazocal{T})$ (or $\textbf{bot}(\pazocal{T})$) has $l_a$ $a$-edges, then the number of cells in $\pazocal{T}$ is between $l_a-l_b$ and $l_a+3l_b$.

\end{enumerate}

\end{lemma}

\medskip

%%%%%%%%%%%%%%%%%%%%%%%%%%%%%%%%%%%%%%%%%%%%%%%%%%

\subsection{Disks and weights} \label{sec-disks-and-weights} \

Next, we add extra relations to the groups $G(\textbf{M})$ and $G_\Omega(\textbf{M})$ that will aid with later estimates. This is done in same way as in \cite{WMal}, adapted from the methods of \cite{O18}, \cite{OS19}, and \cite{W}.

These relations, called \textit{disk relations}, are of the form $W=1$ for any accepted configuration $W$ of $\textbf{M}$ with $\ell(W)\leq1$.

The following statement sheds some light on the reasoning for these extra relations:

\begin{lemma}[Lemma 7.2 of \cite{W}] \label{disks are relations}

If the configuration $W$ is accepted by the machine $\textbf{S}$ and $Y_{s+1}=\emptyset$, then the word $W$ is trivial over the groups $G(\textbf{S})$ and $G_\Omega(\textbf{S})$.

\end{lemma}

%\begin{proof}
%
%Let $\pazocal{C}$ be an accepting computation of $W$ and $H$ be its history. By Lemma \ref{computations are trapezia}, there exists a trapezium $\Delta$ corresponding to $\pazocal{C}$ with trimmed bottom label $W$ and trimmed top label $W_{ac}$. 
%
%As this is a computation of the standard base and every rule locks the $Q_sQ_0$-sector, one can further assume that no trimming was necessary in $\Delta$, i.e the labels of the bottom and top of $\Delta$ are $W$ and $W_{ac}$, respectively. Finally, it follows that the sides of the trapezium are labelled identically; specifically, they are labelled by the copy of $H$ obtained by adding the index $0$ to each letter. 
%
%So, $W$ and $W_{ac}$ are conjugate in $M(\textbf{S})$. Taking into account the hub relation in both $G(\textbf{S})$ and $G_\Omega(\textbf{S})$ then implies the relation $W=1$.
%
%\end{proof}

A sketch of the proof of \Cref{disks are relations} can aid in some arguments to follow: That $W$ is accepted implies the existence of an accepting computation $W\to\dots\to W_{ac}$; \Cref{computations are trapezia} then produces a trapezium $\Delta$ with top and bottom labels $W$ and $W_{ac}$; that $Y_{s+1}=\emptyset$ implies no trimming was necessary in forming this trapezium, {\frenchspacing i.e. the} side labels of $\Delta$ are labelled identically; this implies $W$ and $W_{ac}$ are conjugate in $M(\textbf{S})$ (and in $M_\Omega(\textbf{S})$); but $W_{ac}$ is trivial in $G(\textbf{S})$ (and in $G_\Omega(\textbf{S})$).

As a result of Lemma \ref{disks are relations}, the presentation obtained by adding the disk relations to the canonical presentation of $G(\textbf{M})$ (respectively $G_\Omega(\textbf{M})$) defines a group isomorphic to $G(\textbf{M})$ (respectively $G_\Omega(\textbf{M})$). The presentation containing disk relations will be referred to in what follows as the \textit{disk presentation} of the group $G(\textbf{M})$ (respectively $G_\Omega(\textbf{M})$). A cell of a diagram over the disk presentation corresponding to a disk relation (or its inverse) is referred to simply as a \textit{disk}.

One should note the following when considering diagrams over a disk presentation rather than diagrams over a canonical presentation:

\begin{enumerate}

\item The disk presentation of $G(\textbf{M})$ or of $G_\Omega(\textbf{M})$ need not be finite. In particular, there may be (and generally are) infinitely many disk relations defining this presentation. %In particular, the disk presentations of $G(\textbf{M})$ and of $G_\Omega(\textbf{M})$ are not finitely presented.

\item As in \Cref{sec-bands-annuli}, we insist that an $a$-band in a diagram over the disk presentation of $G_\Omega(\textbf{M})$ consist only of $(\theta,a)$-cells. As a consequence, a maximal $a$-band may end on a disk in addition to the other possibilities outlined in \Cref{sec-bands-annuli}.

\item For a word $w\in F(\pazocal{Y})$ that represents the trivial element of $G(\textbf{M})$, the minimal area of diagrams over the disk presentation with contour label $w$ can be drastically different than that of diagrams over the canonical presentation of $G(\textbf{M})$.

\end{enumerate}

To begin to address point (2), we note the following statements about reduced diagrams over these new presentations:

\begin{lemma}[Lemma 8.1 of \cite{W}] \label{M_a no annuli 1}

A reduced diagram $\Delta$ over the disk presentation of $G_\Omega(\textbf{M})$ contains no:

\begin{enumerate}[label=(\alph*)]

\item $(\theta,q)$-annuli

\item $(\theta,a)$-annuli

\item $a$-annuli

\item $q$-annuli

\end{enumerate}

\end{lemma}

%\renewcommand{\thesubfigure}{\alph{subfigure}}
%\begin{figure}[H]
%\centering
%\captionsetup[subfigure]{labelformat=parens}
%\begin{subfigure}[b]{0.48\textwidth}
%\centering
%\raisebox{0.4in}{\includegraphics[scale=0.85]{Mannuli1.eps}}
%\caption{$\Delta_S$ for $S$ a $(\theta,q)$-annulus}
%\end{subfigure}\hfill
%\begin{subfigure}[b]{0.48\textwidth}
%\centering
%\includegraphics[scale=0.75]{Mannuli2.eps}
%\caption{$\Delta_S$ for $S$ an $a$-annulus}
%\end{subfigure}
%\caption{ \ }
%\end{figure}

\begin{lemma} \label{a-band on same a-cell}

For any $a$-cell $\pi$ in a reduced diagram $\Delta$ over the disk presentation of $G_\Omega(\textbf{M})$, no $a$-band can have two ends on $\pi$.

\end{lemma}

\begin{figure}[H]
\centering
\includegraphics[scale=1.25]{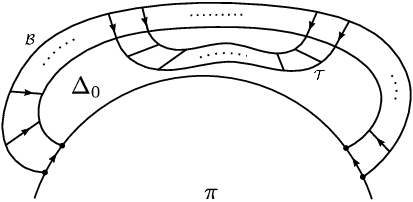}
\caption{$a$-band ending twice on an $a$-cell}
\end{figure}

On the other hand, to begin to address points (1) and (3) above, we alter the definition of the area of a diagram over the disk presentations of $G(\textbf{M})$ and $G_\Omega(\textbf{M})$ (in a manner analogous to the treatment in \cite{W}). 

We do this first by introducing a weight function on the cells of such diagrams, $\text{wt}$, defined by:

$
\begin{array}{ll}
      \bullet \ \text{wt}(\Pi)=1 & \ \text{if $\Pi$ is a $(\theta,q)$-cell or a $(\theta,a)$-cell} \\
      \bullet \ \text{wt}(\Pi)=C_1|\partial\Pi|^2 & \ \text{if $\Pi$ is a disk} \\
      \bullet \ \text{wt}(\Pi)=C_1\|\partial\Pi\|^2f_2(C_1\|\partial\Pi\|) & \ \text{if $\Pi$ is an $a$-cell}
   \end{array}
$

Naturally, we extend this to define the weight of a reduced diagram $\Delta$, $\text{wt}(\Delta)$, as the sum of the weights of its cells.

Given a positive constant $C$, define the function $\phi_C$ on the nonnegative reals by $\phi_C(x)=Cx^2f_2(Cx)$, so that for any $a$-cell $\pi$ we have $\text{wt}(\pi)=\phi_{C_1}(\|\partial\pi\|)$.  Note then that $\phi_C$ is a non-decreasing function which, when restricted to the naturals, is equivalent to the function $n^2f_2(n)$.  Moreover, for any $x,y$, 
\begin{align*}
\phi_C(x+y)&=C(x+y)^2f_2(C(x+y))\geq C(x^2+y^2)f_2(C(x+y)) \\
&=Cx^2f_2(C(x+y))+Cy^2f_2(C(x+y)) \\
&\geq Cx^2f_2(Cx)+Cy^2f_2(Cy)=\phi_C(x)+\phi_C(y)
\end{align*}
{\frenchspacing i.e. $\phi_C$ is} superadditive.

It is also useful to observe that for any $x\geq y$, 
\begin{align*}
\phi_C(x-y)&=C(x-y)^2f_2(Cx-Cy)\leq C(x^2-xy)f_2(Cx) \\
&\leq \phi_C(x)-Cxyf_2(Cx)
\end{align*}

\medskip

%%%%%%%%%%%%%%%%%%%%%%%%%%%%%%%%%%%%%%%%%%%%%%%%%%

\subsection{Mixtures} \

We now recall an invariant of reduced diagrams over the relevant presentations, first introduced in \cite{OS12}, that will prove invaluable in the numerical estimates that follow.

Let $O$ be a circle containing a finite two-colored set of points, with the two colors taken to be black and white. The circle $O$ is called a \textit{necklace} while the corresponding points are called \textit{white beads} and \textit{black beads}.

Let $P_j$ be the set of ordered pairs of distinct white beads, $(o_1,o_2)$, such that the counterclockwise simple arc on $O$ from $o_1$ to $o_2$ contains at least $j$ black beads.

Define $\mu_J(O)=\sum\limits_{j=1}^J \#P_j$ as the \textit{$J$-mixture} of $O$, where $J$ is the parameter specified in Section 3.3.

\begin{lemma}[Lemma 6.1 of \cite{OS12}] \label{mixtures}

Let $O$ be a necklace with $x$ white beads and $y$ black beads.

\begin{enumerate}

\item $\mu_J(O)\leq J(x^2-x)$

\item If $O'$ is a necklace obtained from $O$ through the removal of one white bead, then for every $j$, $\# P_j-2x<\#P_j'\leq\#P_j$, and so $\mu_J(O)-2Jx<\mu_J(O')\leq\mu_J(O)$

\item If $O'$ is a necklace obtained from $O$ through the removal of one black bead, then for every $j$, $\#P_j'\leq\#P_j$, and so $\mu_J(O')\leq\mu_J(O)$

\item Suppose $v_1,v_2,v_3$ are three black beads on $O$ such that the counterclockwise arc from $v_1$ to $v_3$, $v_1 - v_3$, has at most $J$ black beads (excluding $v_1$ and $v_3$). Let $y_1$ and $y_2$ be the number of white beads on the counterclockwise arcs $v_1 - v_2$ and $v_2 - v_3$, respectively. If $O'$ is the necklace obtained from $O$ through the removal of $v_2$, then $\mu_J(O')\leq\mu_J(O)-y_1y_2$.

\end{enumerate}

\end{lemma}

%\begin{proof}
%
%$(a)$ The number of pairs of white beads that are distinct is $x(x-1)=x^2-x$. As each pair contributes at most $J$ to $\mu_J(O)$, the inequality is clear.
%
%$(b)$ is clear since any white bead can be in at most $2(x-1)$ ordered pairs in $P_j$.
%
%$(c)$ is clear since removing a black bead can only decrease the number of ordered pairs with at least $j$ black beads between them
%
%$(d)$ Each of the $y_1y_2$ pairs $(o_1,o_2)$ for $o_1$ on the arc $v_1-v_2$ and $o_2$ on the arc $v_2-v_3$ will be in $P_j$ but not $P_{j+1}$ for some $1\leq j\leq J$. So, after the removal of $v_2$, each of these pairs will no longer be in $P_j$, so that they will contribute one less to $\mu_J(O')$.
%
%\end{proof}

Let $\Delta$ be a reduced diagram over a group associated to an $S$-machine $\textbf{S}$. Let $O$ be a circle partitioned by subarcs labeled by the edges of $\partial\Delta$. At the midpoint of a subarc labeled by a $\theta$-edge (respectively a $q$-edge), place a white bead (respectively a black bead). Then, define the \textit{mixture on $\Delta$} $\mu(\Delta)$ as the $J$-mixture of the corresponding necklace, i.e $\mu(\Delta)=\mu_J(O)$.

%%%%%%%%%%%%%%%%%%%%%%%%%%%%%%%%%%%%%%%%%%%%%%%%%%

\section{Diagrams without disks}

\subsection{$M$-minimal diagrams} \

In this section, we study diagrams over the canonical presentation of $M_\Omega(\textbf{M})$ homeomorphic to a disk, with the ultimate goal of bounding the `size' of such a diagram in terms of its perimeter. To do this, we first define a special class of diagrams for which this bound will hold.

A reduced diagram $\Delta$ over the canonical presentation of $M_\Omega(\textbf{M})$ is called \textit{$M$-minimal} if:

\begin{addmargin}[1em]{0em}

\begin{enumerate}[label=(MM{\arabic*})]

%\item it contains no $\theta$-annulus $S$ whose sides are labelled by letters of the tape alphabet of the `special' input sector,

\item for any $a$-cell $\pi$ and any $\theta$-band $\pazocal{T}$, at most half of the edges of $\partial\pi$ mark the start of an $a$-band that crosses $\pazocal{T}$, and

\item no maximal $a$-band ends on two different $a$-cells.

\end{enumerate}

\end{addmargin}

It follows immediately from the definition that a subdiagram of an $M$-minimal diagram is necessarily also $M$-minimal.

The following statement gives one fundamental consequence of this definition:

\begin{lemma}[Lemma 8.3 of \cite{W}] \label{M_a no annuli 2}

Let $\Delta$ be a reduced diagram over $M_\Omega(\textbf{M})$. 

\begin{enumerate}[label=({\arabic*})]

\item Suppose $\Delta$ contains a $\theta$-annulus $S$ and let $\Delta_S$ be the subdiagram of $\Delta$ bounded by the outer component of the contour of $S$. Then $\Delta_S$ contains no $(\theta,q)$-cells and $\lab(\partial\Delta_S)$ is a word over the tape alphabet of the `special' input sector.

\item If $\Delta$ is $M$-minimal, then it contains no $\theta$-annuli.

\end{enumerate}

\end{lemma}

\medskip

%%%%%%%%%%%%%%%%%%%%%%%%%%%%%%%%%%%%%%%%%%%%%%%%%%

\subsection{Transpositions of a $\theta$-band with an $a$-cell} \label{sec-transposition-a} \

Next, we describe a general surgery on diagrams, as introduced in \cite{W} and generalized in \cite{WMal}, over the relevant presentations which essentially pushes an $a$-cell through a $\theta$-band.

Let $\Delta$ be a reduced diagram over the disk presentation of $G_\Omega(\textbf{M})$ containing an $a$-cell $\pi$ and a $\theta$-band $\pazocal{T}$ subsequently crossing some of the $a$-bands starting at $\pi$. As the cells shared by these bands and $\pazocal{T}$ are $(\theta,a)$-cells, the domain of the rule $\theta$ corresponding to $\pazocal{T}$ must be nonempty in the `special' input sector. So, by the definition of the rules of $\textbf{M}$, the domain of $\theta$ in this sector is the entire alphabet.

Suppose there are no other cells between $\pi$ and the bottom of $\pazocal{T}$, i.e there is a subdiagram formed by $\pi$ and $\pazocal{T}$.

Let $\textbf{s}_1$ be the maximal subpath of $\partial\pi$ so that each edge is on the boundary of a $(\theta,a)$-cell of $\pazocal{T}$. Further, let $\textbf{s}_2$ be the complement of $\textbf{s}_1$ in $\partial\pi$ so that $\partial\pi=\textbf{s}_1\textbf{s}_2$ and let $\pazocal{T}'$ be the subband of $\pazocal{T}$ satisfying $\textbf{bot}(\pazocal{T}')=\textbf{s}_1$.

\begin{figure}[H]
\centering
\begin{subfigure}[b]{\textwidth}
\centering
\raisebox{0.3825in}{\includegraphics[scale=0.7]{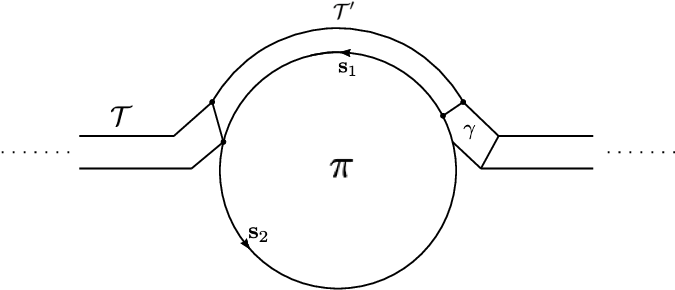}}
\caption{The subdiagram $\Gamma$}
\end{subfigure}
\begin{subfigure}[b]{\textwidth}
\centering
\includegraphics[scale=0.7]{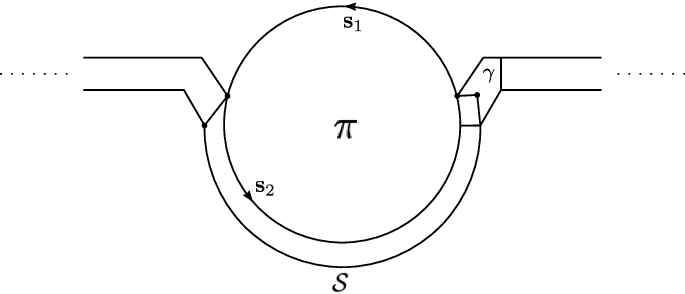}
\caption{The resulting subdiagram $\Gamma'$}
\end{subfigure}
\caption{The transposition of a $\theta$-band with an $a$-cell, $\gamma$ a $(\theta,q)$-cell}
\end{figure}

Let $V_1\equiv\text{Lab}(\textbf{s}_1)$ and $V_2\equiv\text{Lab}(\textbf{s}_2)$. Finally, let $\Gamma$ be the subdiagram formed by $\pi$ and $\pazocal{T}'$.

Then, we can construct the $\theta$-band $\pazocal{S}$ corresponding to $\theta$ consisting only of $(\theta,a)$-cells and with top label $V_2$. Let $\Gamma'$ be the subdiagram obtained by gluing a copy of $\pi$ to $\textbf{top}(\pazocal{S})$ in the clear way.

As $\text{Lab}(\textbf{top}(\pazocal{T}'))\equiv V_1^{-1}$, $\lab(\partial\Gamma)\equiv\lab(\partial\Gamma')$. So, we may replace the $\Gamma$ with $\Gamma'$, attaching the first and last cells of $\pazocal{S}$ to the complement of $\pazocal{T}'$ in $\pazocal{T}$ and making any necessary cancellations in the resulting band.

This process is called the \textit{transposition} of the $\theta$-band with the $a$-cell. 

Note that the diagram $\tilde{\Delta}$ resulting from the transposition has the same contour label as $\Delta$. Further, if a maximal $a$-band of $\Delta$ has one end on the $a$-cell $\pi$, then the other end is not changed by the transposition. 

Hence, if $\Delta$ is $M$-minimal, then $\tilde{\Delta}$ satisfies (MM2). However, $\tilde{\Delta}$ may not be $M$-minimal, as the transposed $\theta$-band may cross the maximal $a$-bands emanating from more than half of the $a$-edges on the boundary of the transposed $a$-cell.

Further, since the number of $(\theta,a)$-cells is altered by the transposition, the weight of the diagrams $\Delta$ and $\tilde{\Delta}$ may differ considerably. 

Despite these disadvantages, this process will prove valuable in forthcoming arguments.

\medskip

%%%%%%%%%%%%%%%%%%%%%%%%%%%%%%%%%%%%%%%%%%%%%%%%%%

\subsection{$a$-trapezia} \

We now generalize the concept of trapezium defined in \Cref{sec-trapezia} to the setting of $M$-minimal diagrams, allowing the existence of $a$-cells within the diagram.  This is done in just the same way as the introduction of this concept in \cite{W}.

To be specific, an \textit{$a$-trapezium} $\Delta$ is an $M$-minimal diagram with contour of the form $\textbf{p}_1^{-1}\textbf{q}_1\textbf{p}_2\textbf{q}_2^{-1}$, where each $\textbf{p}_i$ is the side of a $q$-band and each $\textbf{q}_i$ is the maximal subpath of the side of a $\theta$-band that starts and ends with $q$-edges. As with trapezia, the factorization $\textbf{p}_1^{-1}\textbf{q}_1\textbf{p}_2\textbf{q}_2^{-1}$ of the boundary is called the \textit{standard factorization} of $\partial\Delta$.

The \textit{history}, \textit{step history}, \textit{height}, and \textit{base} of an $a$-trapezium are defined in the same way they are defined for a trapezium. 

Note that the history of an $a$-trapezium must be reduced. Further, by Lemma \ref{trapezia are computations}, the base of an $a$-trapezium must be the base of an admissible word. So, in an $a$-trapezium $\Delta$, the subdiagram $\Gamma$ bounded by two consecutive $q$-bands is an $a$-trapezium with base $UV$ corresponding to these $q$-bands' makeups. In this case, $\Gamma$ is called a \textit{$UV$-sector} in $\Delta$. As with admissible words, an $a$-trapezium may contain sectors of the same name.

\begin{lemma}[Compare with Lemma 8.4 of \cite{W}] \label{a-cells sector}

Suppose $\Delta$ is an $a$-trapezium containing an $a$-cell $\pi$. Then $\pi$ is contained in a $(Q_0(1)Q_1(1))^{\pm1}$-, $Q_0(1)Q_0(1)^{-1}$-, or $Q_1(1)^{-1}Q_1(1)$-sector.  Moreover, if there exists a maximal $a$-band in $\Delta$ which ends on $\pi$ and does not end on $\partial\Delta$, then $\pi$ is contained in a $(Q_0(1)Q_1(1))^{\pm1}$- or a $Q_1(1)^{-1}Q_1(1)$-sector and step history of $\Delta$ contains the letter $(1)_1$.

\end{lemma}

\begin{lemma} \label{a-cell a-trapezia no boundary}

Let $\pi$ be an $a$-cell in an $a$-trapezium $\Delta$.  Suppose no maximal $a$-band that ends on $\pi$ also ends on a $(\theta,q)$-cell.  Then there exists an $M$-minimal diagram $\Delta_\pi$ consisting of an $a$-trapezium $\Delta_\pi'$ and a single $a$-cell $\pi'$ such that:

\begin{enumerate}

\item $\lab(\partial\Delta_\pi)\equiv\lab(\partial\Delta)$

\item $\text{wt}(\Delta_\pi)=\text{wt}(\Delta)$

\item $\lab(\partial\pi')\equiv\lab(\partial\pi)$

\item $\Delta_\pi'$ has the same base and history as $\Delta$, but contains one less $a$-cell

\item $\|\textbf{tbot}(\Delta_\pi')\|=\|\textbf{tbot}(\Delta)\|$ and $\|\textbf{ttop}(\Delta_\pi')\|=\|\textbf{ttop}(\Delta)\|$

\end{enumerate}

\end{lemma}

\begin{proof}

Enumerate the maximal $\theta$-bands $\pazocal{T}_1,\dots,\pazocal{T}_h$ of $\Delta$ from bottom to top and fix the index $j\in\{1,\dots,h-1\}$ such that $\pi$ lies between $\pazocal{T}_j$ and $\pazocal{T}_{j+1}$.

By \Cref{a-band on same a-cell} and (MM2), each of the $\|\partial\pi\|$ maximal $a$-bands which end on $\pi$ also end on $\partial\Delta$.  In particular, each such band either crosses every $\pazocal{T}_i$ for $i\geq j+1$ and ends on $\textbf{ttop}(\Delta)$ or crosses every $\pazocal{T}_i$ for $i\leq j$ and ends on $\textbf{tbot}(\Delta)$.

By (MM1), though, at most $\frac{1}{2}\|\partial\pi\|$ of these $a$-bands cross $\pazocal{T}_j$ and at most $\frac{1}{2}\|\partial\pi\|$ cross $\pazocal{T}_{j+1}$.  As such, exactly $\frac{1}{2}\|\partial\pi\|$ of these maximal $a$-bands end on $\textbf{ttop}(\Delta)$ and $\frac{1}{2}\|\partial\pi\|$ end on $\textbf{tbot}(\Delta)$.

Let $\Delta_1$ be the reduced diagram obtained by transposing $\pi$ and $\pazocal{T}_j$.  Further, let $\pi_1$ be the $a$-cell of $\Delta_1$ obtained from $\pi$ and $\pazocal{T}_j'$ be the maximal $\theta$-band corresponding to $\pazocal{T}_j$.

Since exactly half of the maximal $a$-bands ending on $\pi$ cross $\pazocal{T}_j$, by construction exactly half these $a$-bands cross $\pazocal{T}_j'$.  As no other maximal $a$-bands are affected by this transposition, it follows that $\Delta_1$ is $M$-minimal.  Moreover, the top and bottom lengths of $\pazocal{T}_j'$ are the same as those of $\pazocal{T}_j$, while $\lab(\partial\pi_1)\equiv\lab(\partial\pi)$, so that $\text{wt}(\Delta_1)=\text{wt}(\Delta)$.

Suppose $j=1$ and let $\Delta_1'$ be the subdiagram of $\Delta_1$ obtained by removing $\pi_1$.  Then $\Delta_1'$ is an $a$-trapezium with the same base and history as $\Delta$.  Moreover, $\|\textbf{tbot}(\Delta_1')\|=\|\textbf{tbot}(\pazocal{T}_1')\|=\|\textbf{tbot}(\pazocal{T}_1)\|=\|\textbf{tbot}(\Delta)\|$.  Hence, taking $\Delta_\pi=\Delta_1$ satisfies the statement.  

Otherwise, $\pi_1$ lies between $\pazocal{T}_{j-1}$ and $\pazocal{T}_j'$; but then transposing $\pi_1$ with $\pazocal{T}_{j-1}$ and iterating until the $a$-cell is transposed with $\pazocal{T}_1$ produces the desired diagram $\Delta_\pi$.

\end{proof}

\begin{lemma} \label{revolving trapezia no boundary}

Let $\Delta$ be an $a$-trapezium with revolving base $B$ and height $h$.  Suppose no maximal $a$-band of $\Delta$ which ends on an $a$-cell also ends on a $(\theta,q)$-cell.  Then: $$\text{wt}(\Delta)\leq c_4h\max(\|\textbf{tbot}(\Delta)\|,\|\textbf{ttop}(\Delta)\|)+\phi_{C_1}(|\textbf{tbot}(\Delta)|_a+|\textbf{ttop}(\Delta)|_a)$$

\end{lemma}

\begin{proof}

By \Cref{a-band on same a-cell} and (MM2), each $a$-edge on the boundary of an $a$-cell corresponds to an $a$-edge of either $\textbf{tbot}(\Delta)$ or of $\textbf{ttop}(\Delta)$.  As such, the sum of the (combinatorial) perimeters of the $a$-cells in $\Delta$ is at most $|\textbf{tbot}(\Delta)|_a+|\textbf{ttop}(\Delta)|_a$.  So, since $\phi_{C_1}$ is superadditive, the sum of the weights of all $a$-cells in $\Delta$ is then at most $\phi_{C_1}(|\textbf{tbot}(\Delta)|_a+|\textbf{ttop}(\Delta)|_a)$.

Now, iteratively applying \Cref{a-cell a-trapezia no boundary} produces an $M$-minimal diagram $\Delta'$ with the same boundary label and weight as $\Delta$ consisting of a trapezium $\Delta''$ with the same base, history, and trimmed top and bottom lengths as $\Delta$ and copies of the $a$-cells of $\Delta$ pasted to the top and bottom of $\Delta''$.

Let $H$ be the history of $\Delta$ and enumerating the maximal $\theta$-bands $\pazocal{T}_1'',\dots,\pazocal{T}_h''$ of $\Delta''$ from bottom to top.  Then \Cref{trapezia are computations} there exists a reduced computation $\pazocal{C}:W_0\to\dots\to W_h$ of $\textbf{M}$ with base $B$ and history $H$ such that $\lab(\textbf{tbot}(\pazocal{T}_i''))\equiv W_{i-1}$ and $\lab(\textbf{ttop}(\pazocal{T}_i''))\equiv W_i$ for all $i$.

Since $B$ is revolving, \Cref{revolving bound} then implies $\|W_i\|\leq c_4\max(\|W_0\|,\|W_h\|)$.

Note that if the history of $\pazocal{T}_i''$ is a positive rule ({\frenchspacing i.e. a letter} of $\Theta^+$), then the length of $\pazocal{T}_i''$ is $\|W_{i-1}\|$; conversely, if the history is a negative rule, then the length of $\pazocal{T}_i''$ is $\|W_i\|$.

Hence, the length of each maximal $\theta$-band $\pazocal{T}_i''$ is at most $\max_{0\leq j\leq h}(\|W_j\|)\leq c_4\max(\|W_0\|,\|W_h\|)$, so that $\text{wt}(\Delta'')=\text{Area}(\Delta'')\leq c_4h\max(\|W_0\|,\|W_h\|)=c_4h\max(\|\textbf{tbot}(\Delta)\|,\|\textbf{ttop}(\Delta)\|)$.

\end{proof}

%An $a$-trapezium is called \textit{standard} if its base is pararevolving and its history contains a controlled subword. Note that the subdiagram of a standard $a$-trapezium bounded by the $\theta$-bands corresponding to the controlled subword of the history is a trapezium.

Suppose $\Delta$ is an $a$-trapezium containing a maximal $a$-band which ends on both an $a$-cell and a $(\theta,q)$-cell.  Then $\Delta$ is called:

\begin{itemize}

\item \textit{big} if its base is revolving and its history contains a controlled subword.

\item \textit{exceptional} if its base is hyperfaulty and contains no subwords of the form $(P_i'(j)R_i'(j))^{\pm1}$ or $(P_i''(j)R_i''(j))^{\pm1}$.

\end{itemize}

Note that the base of a big $a$-trapezium is necessarily reduced.

\begin{lemma} \label{big a-trapezium accepted}

Let $\Delta$ be a big $a$-trapezium with base $\{t(1)\}B_4(1)\dots\{t(L)\}B_4(L)\{t(1)\}$.  Then there exist configurations $W_0$ and $W_1$ of $\textbf{M}$ such that:

\begin{enumerate}

\item $\ell(W_0),\ell(W_1)\leq1$.

\item The word $W_0t(1)$ differs from $\lab(\textbf{bot}(\Delta))$ by the word in the `special' input sector.

\item The word $W_1t(1)$ differs from $\lab(\textbf{top}(\Delta))$ by the word in the `special' input sector.

\end{enumerate}

\end{lemma}

\begin{proof}

Let $H$ be the history of $\Delta$ and let $\Delta'$ be the subdiagram of $\Delta$ bounded by the $\theta$-bands corresponding to the controlled subword of $H$.  By \Cref{a-cells sector}, $\Delta'$ must be a trapezium.  So, since the $Q_{s,r}''(L)t(1)$-sector is locked by all rules of $\textbf{M}$, \Cref{trapezia are computations} yields a reduced computation $\pazocal{C}':W_0't(1)\to\dots\to W_1't(1)$ such that $W_0'$ and $W_1'$ are configurations.  \Cref{M controlled} then implies $W_0'$ and $W_1'$ are accepted by one-machine computations.

%\begin{enumerate}
%
%\item $W_0'$ and $W_1'$ are accepted by one-machine computations.
%
%\item $\lab(\textbf{bot}(\Delta'))\equiv W_0't(1)$.
%
%\item $\lab(\textbf{top}(\Delta'))\equiv W_1't(1)$.
%
%\end{enumerate}

Let $V_0'$ and $V_1'$ be the admissible subwords of $W_0'$ and $W_1'$, respectively, with base $\{t(2)\}B_4(2)$.  Similarly, let $V_0$ and $V_1$ be the subwords of $\lab(\textbf{bot}(\Delta))$ and $\lab(\textbf{top}(\Delta))$, respectively, with base $\{t(2)\}B_4(2)$.

Now, let $\Delta_2$ be the subdiagram of $\Delta$ bounded by the $t$-bands corresponding to $\{t(2)\}$ and $Q_{s,r}''(2)$.  As above, \Cref{a-cells sector} implies $\Delta_2$ is a trapezium with base $\{t(2)\}B_4(2)$ and history $H$.  By \Cref{trapezia are computations}, there then exists a reduced computation $\pazocal{D}:V_0\to\dots\to V_1$ with history $H$ which contains a subcomputation $\pazocal{D}':V_0'\to\dots\to V_1'$.

Note that applying \Cref{extend one-machine} to $\pazocal{D}'$ produces the reduced computation $\pazocal{C}'':W_0'\to\dots\to W_1'$ given by restricting $\pazocal{C}'$ to the standard base.

Factor $H\equiv H_1\dots H_m$ such that each $H_i$ is the history of a one-machine subcomputation $\pazocal{D}_i$ of $\pazocal{D}$.  Then, applying \Cref{extend one-machine} to each $\pazocal{D}_i$ produces reduced computations $\pazocal{C}_i$ in the standard base such that $\pazocal{C}''$ is a subcomputation of some $\pazocal{C}_j$.

So, every configuration of $\pazocal{C}_j$ is accepted by a one-machine computation.  But by (a) and (b) in \Cref{extend one-machine}, this implies each configuration of each $\pazocal{C}_i$ is accepted by a one-machine computation.  Hence, letting $W_0$ be the initial configuration of $\pazocal{C}_1$ and $W_1$ the terminal configuration of $\pazocal{C}_m$ satisfies the statement.

\end{proof}

A \textit{partition} of an $a$-trapezium $\Delta$ is a (finite) collection of subdiagrams $\{\Delta_i\}_{i=1}^m$ such that each $\Delta_i$ consists of a number of sectors of $\Delta$, $\Delta_i\cap\Delta_j$ is either empty or a $q$-band for $i\neq j$, and each sector is a subdiagram of some $\Delta_i$. Note that $\|\textbf{tbot}(\Delta)\|=\sum_i\|\textbf{tbot}(\Delta_i)\|-m$ and, similarly, $\|\textbf{ttop}(\Delta)\|=\sum_i\|\textbf{ttop}(\Delta_i)\|-m$. Moreover, as $\textbf{tbot}(\Delta)$ and $\textbf{ttop}(\Delta)$ each have at least $m$ $q$-edges, $\sum_i\|\textbf{tbot}(\Delta_i)\|\leq2\|\textbf{tbot}(\Delta)\|$ and $\sum_i\|\textbf{ttop}(\Delta_i)\|\leq2\|\textbf{ttop}(\Delta)\|$.

Clearly, given a partition $\{\Delta_i\}$ of an $a$-trapezium $\Delta$, $\text{wt}(\Delta)\leq\sum_i\text{wt}(\Delta_i)$.

Let $\Delta$ be an $a$-trapezium with revolving base $B$ and let $B'$ be a cyclic permutation of $B$. Then, there exists an $a$-trapezium $\Delta'$ with revolving base $B'$ such that $\text{wt}(\Delta')=\text{wt}(\Delta)$. This diagram is constructed by cutting along a maximal $q$-band $\pazocal{Q}$ of $\Delta$, pasting together the left and right $q$-bands of $\Delta$, and pasting a copy of $\pazocal{Q}$ onto the side of the diagram. As with reduced computations, $\Delta'$ is called a cyclic permutation of $\Delta$.

Note that by Lemma \ref{lengths}(d), for any maximal $\theta$-band $\pazocal{T}$ in an $a$-trapezium $\Delta$, the length of $\pazocal{T}$ is at most $|\textbf{tbot}(\pazocal{T})|_q+3|\textbf{tbot}(\pazocal{T})|_a\leq 3\|\textbf{tbot}(\pazocal{T})\|$.

\begin{lemma} \label{revolving a-trapezia (1)}

Let $\Delta$ be an $a$-trapezium with revolving base $B$ and height $h$.  Suppose there exist subdiagrams $\Delta_1,\dots,\Delta_\ell$ of $\Delta$ each of which consists of a number of sectors that do not contain any $a$-cells and such that $c_0\sum_{i=1}^\ell (\|\textbf{tbot}(\Delta_i)\|+\|\textbf{ttop}(\Delta_i)\|)\geq h$.  Then:
$$\text{wt}(\Delta)\leq C_1h\max(\|\textbf{tbot}(\Delta)\|,\|\textbf{ttop}(\Delta)\|)+\phi_{C_2}(\|\textbf{tbot}(\Delta)\|+\|\textbf{ttop}(\Delta)\|)$$

\end{lemma}

\begin{proof}

Let $p$ be the number of maximal $q$-bands in $\Delta$ corresponding to the base letter $Q_1(1)^{\pm1}$ which bound a sector of $\Delta$ which contains $a$-cells.  As $B$ is revolving, $p\leq2$.

Let $\Delta_1',\dots,\Delta_r'$ be the sectors of $\Delta$ that contain $a$-cells.  By (MM2) and \Cref{a-band on same a-cell}, any maximal $a$-band in $\Delta$ that ends on an $a$-cell is contained in one $\Delta_i'$ and also ends either on $\textbf{tbot}(\Delta_i')$, on $\textbf{ttop}(\Delta_i')$, or on a $(\theta,q)$-cell of one of the $p$ maximal $q$-bands.  As the boundary of any $(\theta,q)$-cell of such a band contains at most one $a$-letter from the alphabet of the `special' input sector, it follows that the sum of the (combinatorial) perimeters of the $a$-cells in $\Delta$ is at most $\sum_{i=1}^r(|\textbf{tbot}(\Delta_i')|_a+|\textbf{ttop}(\Delta_i')|_a)+ph\leq 4c_0(\|\textbf{tbot}(\Delta)\|+\|\textbf{ttop}(\Delta)\|)$.

Since $\phi_{C_1}$ is superadditive, the sum of the weights of the $a$-cells in $\Delta$ must then be at most $\phi_{C_1}(4c_0(\|\textbf{tbot}(\Delta)\|+\|\textbf{ttop}(\Delta)\|))$.  But for any $n\in\N$, $\phi_{C_1}(4c_0n)=16c_0^2C_1n^2f_2(4c_0C_1n)$, which can be taken to be at most $\phi_{C_2}(n)$ by the parameter choices $C_2>>C_1>>c_0$.

Hence, it suffices to show that the sum of the lengths of the maximal $\theta$-bands in $\Delta$ is at most $C_1h\max(\|\textbf{tbot}(\Delta)\|,\|\textbf{ttop}(\Delta)\|)$.  In particular, given a maximal $\theta$-band $\pazocal{T}$ in $\Delta$, it suffices to show that the length of $\pazocal{T}$ is at most $C_1\max(\|\textbf{tbot}(\Delta)\|,\|\textbf{ttop}(\Delta)\|)$.

Given a sector $\Delta''$ of $\Delta$, let $\pazocal{T}''$ be the maximal $\theta$-band of $\Delta''$ which is a subband of $\pazocal{T}$.  For any $a$-edge of $\textbf{bot}(\pazocal{T}'')$, (MM2) and \Cref{a-cells sector} imply the corresponding maximal $a$-band of $\Delta''$ ends (at least once) on $\textbf{tbot}(\Delta'')$, on $\textbf{ttop}(\Delta'')$, or on a $(\theta,q)$-cell of $\Delta''$.  As $\Delta''$ contains two $q$-bands each of length $h$, this implies $|\textbf{tbot}(\Delta'')|_a\leq|\textbf{tbot}(\Delta'')|_a+|\textbf{ttop}(\Delta'')|_a+2h$.  Hence,
$$\|\textbf{tbot}(\pazocal{T})\|\leq\|\textbf{tbot}(\Delta)\|+\|\textbf{ttop}(\Delta)\|+2(\|B\|-1\|)h\leq\|\textbf{tbot}(\Delta)\|+\|\textbf{ttop}(\Delta)\|+4LNh$$
But $h\leq 2c_0(\|\textbf{tbot}(\Delta)\|+\|\textbf{ttop}(\Delta)\|)$ by hypothesis, so that the statement follows from the parameter choices.

%By \Cref{revolving trapezia no boundary}, it suffices to assume that there exists a maximal $a$-band in $\Delta$ which ends on both an $a$-cell and on a $(\theta,q)$-cell.  %\Cref{a-cells sector} then implies the base $B$ of $\Delta$ contains a subword of the form $(Q_0(1)Q_1(1))^{\pm1}$ or $Q_1(1)^{-1}Q_1(1)$.
%So, assuming $\Delta$ is not exceptional, $B$ contains a two-letter subword $UV$ of the form $(P_i'(j)R_i'(j))^{\pm1}$ or $(P_i''(j)R_i''(j))^{\pm1}$.  Let $\Delta'$ be the $UV$-sector of $\Delta$.
%
%As \Cref{a-cells sector} implies $\Delta'$ is a trapezium, \Cref{trapezia are computations} then yields a reduced computation $\pazocal{D}'$ with base $UV$ and history $H$.  As the step history is $(1)_1$, $\pazocal{D}'$ satisfies the hypotheses of \Cref{multiply two letters}, so that $h\leq\frac{1}{2}(|\textbf{tbot}(\Delta')|_a+|\textbf{ttop}(\Delta')|_a)$.
%
%Now, by (MM2) and \Cref{a-band on same a-cell}, every maximal $a$-band that ends on an $a$-cell must also end either on $\partial\Delta$ or on a $(\theta,q)$-cell.  As $a$-bands cannot cross $q$-bands, it then follows that the sum of the (combinatorial) perimeters of the $a$-cells in $\Delta$ is at most $|\textbf{tbot}(\Delta)|_a+|\textbf{ttop}(\Delta)|_a$.  Hence, the sum of the weights of the $a$-cells in $\Delta$ is at most $\phi_{C_1}(|\textbf{tbot}(\Delta)|_a+|\textbf{ttop}(\Delta)|_a)$.

\end{proof}

\begin{lemma} \label{revolving a-trapezia not exceptional or big}

Let $\Delta$ be an $a$-trapezium with revolving base $B$ and history $H$.  If $\Delta$ is neither exceptional nor big, then for $h=\|H\|$:
$$\text{wt}(\Delta)\leq C_1h\max(\|\textbf{tbot}(\Delta)\|,\|\textbf{ttop}(\Delta)\|)+\phi_{C_2}(\|\textbf{tbot}(\Delta)\|+\|\textbf{ttop}(\Delta)\|)$$

\end{lemma}

\begin{proof}

By \Cref{revolving trapezia no boundary} and the parameter choices $C_2>>C_1>>c_4$, it suffices to assume that there exists a maximal $a$-band in $\Delta$ which ends on both an $a$-cell and on a $(\theta,q)$-cell.  \Cref{a-cells sector} then implies the step history of $\Delta$ contains the letter $(1)_1$.

Assuming $\Delta$ is not exceptional, it then follows that $B$ contains a two-letter subword of the form $(P_i'(j)R_i'(j))^{\pm1}$ or of the form $(P_i''(j)R_i''(j))^{\pm1}$.  Further, as every rule with step history locks the $Q_{i,\ell}'(j)P_i'(j)$-, $R_i'(j)Q_{i,r}'(j)$-, $Q_{i,\ell}''(j)P_i''(j)$-, and $R_i''(j)Q_{i,r}''(j)$-sectors, there must exist a subword $B'$ of a cyclic permutation of $B$ which is of the form $(Q_{i,\ell}'(j)P_i'(j)R_i'(j)Q_{i,r}'(j))^{\pm1}$ or of the form $(Q_{i,\ell}''(j)P_i''(j)R_i''(j)Q_{i,r}''(j))^{\pm1}$.

Let $\Delta'$ be the maximal subdiagram of a cyclic permutation of $\Delta$ which is an $a$-trapezium with base $B'$ and history $H$.  \Cref{a-cells sector} then implies $\Delta'$ is a trapezium, so that \Cref{trapezia are computations} yields a reduced computation $\pazocal{C}:W_0\to\dots\to W_h$ with base $B'$ and history $H$.  Hence, by \Cref{revolving a-trapezia (1)} it suffices to show that $h\leq c_0(\|W_0\|+\|W_h\|)$.

Now, as the step history of $\pazocal{C}$ must contain the letter $(1)_1$, there exists a maximal nonempty subcomputation $\pazocal{C}_1:W_r\to\dots\to W_s$ of $\pazocal{C}$ which is a one-machine computation of the first machine such that the step history of $\pazocal{C}_1$ does not contain a letter of the form $(s)_1^{\pm1}$.  Moreover, as we assume $\Delta$ is not big, $H$ cannot contain a controlled subword.  

Hence, $\pazocal{C}_1$ can be identified with a reduced computation of $\textbf{M}_5$ satisfying the hypotheses of \Cref{M_5 PR}.  As such, $s-r\leq7\max(|W_r|_a,|W_s|_a)+6$ and neither $W_r$ nor $W_s$ is $\sigma(a)_1$-admissible.

Taking $c_0\geq7$ and perhaps passing to the inverse computation, it thus suffices to assume $r>0$.  In particular, $W_r$ must be $\sigma(s)_1^{-1}$-admissible, and so can be identified with an admissible subword of a tame input configuration of $\textbf{M}_5$.  \Cref{M_5 PR} then implies $W_s$ cannot be $\sigma(s)_1^{-1}$-admissible and $|W_r|_a\leq|W_s|_a$.  In particular, $s=h$ and, by extension, $h-r\leq7|W_h|_a+6$.

Taking $c_0\geq7$ implies $r\geq3$, while the subcomputation $W_{r-2}\to W_{r-1}\to W_r$ of $\pazocal{C}$ must have history $\sigma(s)_2^{-1}\sigma(s)_1$.  

Let $\pazocal{C}_2:W_\ell\to\dots\to W_{r-2}$ be the maximal subcomputation which is a one-machine computation of the second machine whose step history does not contain a letter of the form $(s)_2^{\pm1}$.  Then as above the inverse computation of $\pazocal{C}_2$ can be identified with a reduced computation of $\textbf{M}_5$ satisfying the hypotheses of \Cref{M_5 PR}, so that $\ell=0$ and $r-2\leq7|W_0|_a+6$.  

Thus, $h\leq7(|W_0|_a+|W_t|_a)+14$, so that the statement again follows for $c_0\geq7$.

\end{proof}

\begin{lemma} \label{exceptional a-trapezia (21)(1)}

Let $\Delta$ be an exceptional $a$-trapezium with base $B$ and history $H$.  Suppose $H$ has prefix $\sigma(21)_1$.  Then:

\begin{enumerate}[label=(\alph*)]

\item $\|\textbf{tbot}(\Delta)\|\leq c_0\|\textbf{ttop}(\Delta)\|$

\item $\text{wt}(\Delta)\leq C_2h\|\textbf{ttop}(\Delta)\|+\phi_{C_3}(\|\textbf{ttop}(\Delta)\|)$ for $h=\|H\|$.

\end{enumerate}

\end{lemma}

\begin{proof}

Since $H$ contains a letter of the form $\sigma(12)_1^{\pm1}$, Lemmas \ref{locked sectors} and \ref{trapezia are computations} imply that $B$ contains no two-letter subword of the form $Q_1(1)^{-1}Q_1(1)$.  The definition of exceptional $a$-trapezia and \Cref{a-cells sector} then imply the step history of $\Delta$ must contain the letter $(1)_1$ and $B$ must contain a two-letter subword of the form $(Q_0(1)Q_1(1))^{\pm1}$.  

As all rules lock the $\{t(1)\}Q_0(1)$- and $Q_{s,r}''(L)\{t(1)\}$-sectors and $\sigma(12)_1$ locks the $R_s''(L)Q_{s,r}''(L)$-sector, a cyclic permutation of $B$ must contain the subword $B'=Q_{s,r}''(L)^{-1}R_s''(L)^{-1}R_s''(L)Q_{s,r}''(L)$.

Let $\Delta'$ be the maximal subdiagram of a cyclic permutation of $\Delta$ which is an $a$-trapezium with base $B'$ and history $H$.  \Cref{a-cells sector} then implies $\Delta'$ is a trapezium, so that \Cref{trapezia are computations} yields a reduced computation $\pazocal{D}:V_0\to\dots\to V_h$ with base $B'$ and history $H$.

Let $\pazocal{D}_1:V_1\to\dots\to V_x$ be the maximal subcomputation with step history $(1)_1$.  The restriction of $\pazocal{D}_1$ to the $R_s''(L)^{-1}R_s''(L)$-sector then satisfies the hypotheses of \Cref{one alphabet historical words unreduced}, so that $x-1\leq\frac{1}{2}|V_x|_a$ and $V_x$ is not $\sigma(12)_1$-admissible.

So, if $h>x$, then there exists a maximal (perhaps empty) subcomputation $\pazocal{D}_0:V_{x+1}\to\dots\to V_y$ with step history $(0)_1$.  As $V_{x+1}$ is $\sigma(01)_1$-admissible, though, \Cref{primitive unreduced} implies $y=h$ and $|V_{x+1}|_a\leq|V_h|_a$.  Moreover, the restriction of $\pazocal{D}_0$ to the $(R_s''(L)Q_{s,r}''(L))^{\pm1}$-sector satisfies the hypotheses of \Cref{multiply one letter}, so that $h-x-1\leq\frac{1}{2}|V_h|_a$.

Hence, $h\leq|V_h|_a+2$.

Now, let $\Delta''$ be a sector in $\Delta$ with no copy contained in $\Delta'$.  By \Cref{M_a no annuli 1}, no maximal $a$-band of $\Delta''$ can end on $\textbf{tbot}(\Delta'')$ twice.  Moreover, since $\sigma(12)_1$ locks the `special' input sector, no maximal $a$-band of $\Delta''$ that ends on $\textbf{tbot}(\Delta'')$ can also end on an $a$-cell.  

Hence, any maximal $a$-band of $\Delta''$ that ends on $\textbf{tbot}(\Delta'')$ must also end on either $\textbf{ttop}(\Delta'')$ or on a $(\theta,q)$-cell in $\Delta''$, {\frenchspacing i.e. $|\textbf{tbot}(\Delta'')|_a\leq|\textbf{ttop}(\Delta'')|_a+2h$}.

As $B$ is hyperfaulty, (a) then follows from the parameter choice $c_0>>N$.  As a result, (b) follows from \Cref{revolving a-trapezia (1)} and the parameter choices $C_3>>C_2>>C_1>>c_0$.

\end{proof}

\begin{lemma} \label{revolving a-trapezia exceptional (12)}

Let $\Delta$ be an exceptional $a$-trapezium with base $B$ and history $H$.  Suppose $H$ contains a letter of the form $\sigma(12)_1^{\pm1}$.  Then for $h=\|H\|$:
$$\text{wt}(\Delta)\leq C_2h\max(\|\textbf{tbot}(\Delta)\|,\|\textbf{ttop}(\Delta)\|)+\phi_{C_3}(\|\textbf{tbot}(\Delta)\|+\|\textbf{ttop}(\Delta)\|)$$

\end{lemma}

\begin{proof}

As in the proof of \Cref{exceptional a-trapezia (21)(1)}, the presence of a letter of the form $\sigma(12)^{\pm1}$ in $H$ implies the existence of a subword $B'=Q_{s,r}''(L)^{-1}R_s''(L)^{-1}R_s''(L)Q_{s,r}''(L)$ in a cyclic permutation of $B$.  Letting $\Delta'$ be the maximal subdiagram of a cyclic permutation of $\Delta$ which is an $a$-trapezium with base $B'$ and history $H$, then Lemmas \ref{a-cells sector} and \ref{trapezia are computations} again yield a reduced computation $\pazocal{D}:V_0\to\dots\to V_h$ with base $B'$ and history $H$ corresponding to $\Delta'$.

Perhaps passing to the inverse computation, suppose the transition $V_s\to V_{s+1}$ has history $\sigma(21)_1$.  Then as in the proof of \Cref{exceptional a-trapezia (21)(1)}, the step history of the subcomputation $V_{s+1}\to\dots\to V_h$ is a prefix of $(1)_1(10)_1(0)_1$.

Now, let $\pazocal{D}_2:V_r\to\dots\to V_s$ be the maximal subcomputation whose history $H_2$ does not contain any letter of the form $\sigma(12)_1^{\pm1}$.  By \Cref{primitive unreduced}, the step history of $\pazocal{D}_2$ cannot contain the subword $(54)_1(4)_1(43)_1$ or $(34)_1(4)_1(43)_1$.  In particular, $\pazocal{D}_2$ must be a one-machine computation in the first machine whose step history consists only of the letters $(2)_1$, $(3)_1$, $(4)_1$, $(23)_1$, $(32)_1$, and $(43)_1$.

Let $\Delta_2$ be the maximal subdiagram of $\Delta$ which is an $a$-trapezium with history $H_2^{\pm1}$.  \Cref{a-cells sector} then implies $\Delta_2$ is a trapezium, so that Lemmas \ref{trapezia are computations} and \ref{one-machine hyperfaulty} imply 
$$\text{wt}(\Delta_2)=\text{Area}(\Delta_2)\leq c_2\|H_2\|\max(|\textbf{tbot}(\Delta_2)|_a,|\textbf{ttop}(\Delta_2)|_a)$$
Suppose $\textbf{top}(\Delta_2)\neq\textbf{top}(\Delta)$.  Then these paths bound an $a$-trapezium $\Delta_1$ whose history $H_1$ has prefix $\sigma(21)_1$.  But then \Cref{exceptional a-trapezia (21)(1)} implies $\|\textbf{ttop}(\Delta_2)\|=\|\textbf{tbot}(\Delta_1)\|\leq c_0\|\textbf{ttop}(\Delta)\|$ and $\text{wt}(\Delta_1)\leq C_2\|H_1\|\|\textbf{ttop}(\Delta)\|+\phi_{C_3}(\|\textbf{ttop}(\Delta)\|)$.

Similarly, if $\textbf{bot}(\Delta_2)\neq\textbf{bot}(\Delta)$, then these paths bound an $a$-trapezium $\Delta_0$ whose history $H_0$ has suffix $\sigma(12)_1$.  Applying \Cref{exceptional a-trapezia (21)(1)} to the reflection of $\Delta_0$ about its bottom then as above implies $\|\textbf{tbot}(\Delta_2)\|=\|\textbf{ttop}(\Delta_0)\|\leq c_0\|\textbf{tbot}(\Delta)\|$ and $\text{wt}(\Delta_0)\leq C_2\|H_0\|\|\textbf{tbot}(\Delta)\|+\phi_{C_3}(\|\textbf{tbot}(\Delta)\|)$.

Thus, the statement follows from the parameter choices $C_2>>c_2>>c_0$.

\end{proof}

\begin{lemma} \label{revolving a-trapezia exceptional transition}

Let $\Delta$ be an exceptional $a$-trapezium with base $B$ and history $H$.  Suppose $H$ contains a letter of the form $\sigma(01)_1^{\pm1}$.  Then for $h=\|H\|$:
$$\text{wt}(\Delta)\leq C_2h\max(\|\textbf{tbot}(\Delta)\|,\|\textbf{ttop}(\Delta)\|)+\phi_{C_3}(\|\textbf{tbot}(\Delta)\|+\|\textbf{ttop}(\Delta)\|)$$

\end{lemma}

\begin{proof}

As $H$ contains a letter of the form $\sigma(01)^{\pm1}$, Lemmas \ref{locked sectors} and \ref{trapezia are computations} implies $B$ contains no two-letter subword of the form $Q_{0,\ell}'(1)Q_{0,\ell}'(1)^{-1}$ nor one of the form $Q_i(1)Q_i(1)^{-1}$ for $i\geq1$.  Hence, since $\Delta$ is exceptional, \Cref{a-cells sector} implies a cyclic subword of $B$ must contain the subword $B'=Q_{0,\ell}'(1)P_0'(1)P_0'(1)^{-1}Q_{0,\ell}'(1)^{-1}$.

Let $\Delta'$ be the maximal subidagram of a cyclic permutation of $\Delta$ which is an $a$-trapezium with base $B'$ and history $H$.  Then Lemmas \ref{a-cells sector} and \ref{trapezia are computations} yield a reduced computation $\pazocal{C}:W_0\to\dots\to W_h$ with base $B'$ and history $H$ corresponding to $\Delta'$.  Thus, by \Cref{revolving a-trapezia (1)} and the parameter choices $C_3>>C_2>>C_1$, it suffices to show that $h\leq c_0(\|W_0\|+\|W_h\|)$.

Perhaps passing to the inverse computation, it may be assumed that there exists a transition $W_r\to W_{r+1}$ with history $\sigma(01)$ and such that there exists a nonempty maximal subcomputation $\pazocal{C}_1:W_{r+1}\to\dots\to W_s$ with step history $(1)_1$.

By \Cref{revolving a-trapezia exceptional (12)}, it suffices to assume that no letter of $H$ is of the form $\sigma(12)_1^{\pm1}$.  Further, by \Cref{M_5 step history 1}(a) the step history of $\pazocal{C}$ cannot contain the subword $(01)_1(1)_1(10)_1$.  Hence, $s=h$.

As the restriction of $\pazocal{C}_1$ to the $P_0'(1)P_0'(1)^{-1}$-sector satisfies the hypotheses of \Cref{one alphabet historical words unreduced}, it then follows that $h-r-1\leq|W_h|_a$.

Now, let $\pazocal{C}_0:W_\ell\to\dots\to W_r$ be the maximal (perhaps empty) subcomputation with step history $(0)_1$.  The inverse computation of $\pazocal{C}_0$ can then be identified with a reduced computation of a primitive machine satisfying the hypotheses of \Cref{primitive unreduced}, so that $W_\ell$ is not $\sigma(s)_1^{-1}$-admissible.  \Cref{M_5 step history 1}(b) then implies $\ell=0$.

The restriction of $\pazocal{C}_0$ to the $Q_{0,\ell}'(1)P_0'(1)$-sector then satisfies the hypotheses of \Cref{multiply one letter}, so that $r\leq|W_0|_a$.

But then $h\leq|W_0|_a+|W_h|_a+1$, so that the statement follows.

\end{proof}

\medskip

%%%%%%%%%%%%%%%%%%%%%%%%%%%%%%%%%%%%%%%%%%%%%%%%%%

\subsection{$G$-weight} \label{sec-G-weight} \

The goal of the remainder of this section is to bound the size of an $M$-minimal diagram in terms of its perimeter. However, this bound will not be given in terms of the area or weight of the diagram, but instead in terms of the artificial areal concept of \textit{$G$-weight} introduced in \cite{W}.

Let $B$ be the base of an admissible word of $\textbf{M}$. Suppose:

\begin{itemize}

\item Every letter of $B$ is of the form $Q_i(1)^{\pm1}$ or $Q_{0,\ell}'(1)^{\pm1}$

\item There is no two-letter subword of $B$ of the form $Q_{0,\ell}'(1)Q_{0,\ell}'(1)^{-1}$

\end{itemize}

Since the tape alphabet of the $\{t(1)\}Q_0(1)$-sector is empty, it then follows that $B$ can be identified with a base of an admissible word of $\textbf{Move}$.  If this base is full, then $B$ is called a \textit{Move-full base}.

Then, an $a$-trapezium $\Gamma$ with Move-full base and step history $(1)_1$ is called an \textit{impeding} $a$-trapezium.

Suppose $\Gamma$ is an $a$-trapezium which is either big or impeding.  Then the $G$-weight of $\Gamma$, denoted $\text{wt}_G(\Gamma)$ is defined to be the minimum of half its weight and:

\begin{itemize}

\item $h+\phi_{C_2}(\|\textbf{tbot}(\Gamma)\|+\|\textbf{ttop}(\Gamma)\|)$ if $\Gamma$ is big

\item $C_2h\max(\|\textbf{tbot}(\Gamma)\|,\|\textbf{ttop}(\Gamma)\|)+\phi_{C_3}(\|\textbf{tbot}(\Gamma)\|+\|\textbf{ttop}(\Gamma)\|)$ if $\Gamma$ is impeding

\end{itemize}

%In this case, let $\eta=\|H_1\|+n\|H_2\|+\|H_3\|$ and $h=\|H\|$. Then we define the $G$-weight of $\Gamma$, denoted $\text{wt}_G(\Gamma)$, to be the minimum of half its weight and:
%$$3h\max(\|\textbf{tbot}(\Gamma)\|,\|\textbf{ttop}(\Gamma)\|)+3C_1h\eta+C_1(|\textbf{tbot}(\Gamma)|_a+|\textbf{ttop}(\Gamma)|_a+2\eta)^2$$
%Similarly, if $\Gamma$ is a big $a$-trapezium with height $h$ then its $G$-weight is defined to be the minimum of half its weight and:
%$$c_5\max(\|\textbf{ttop}(\Gamma)\|,\|\textbf{tbot}(\Gamma)\|)+4C_1(\|\textbf{tbot}(\Gamma)\|+\|\textbf{ttop}(\Gamma)\|)^2$$

This notion is extended by assigning the $G$-weight of any single cell of a diagram to be its weight.

Now, given a reduced diagram $\Delta$ over $G_\Omega(\textbf{M})$, consider a family of subdiagrams $\textbf{P}$ such that:

\begin{itemize}

\item if $P\in\textbf{P}$, then $P$ is a single cell, a big $a$-trapezium, or an impeding $a$-trapezium,

\item every cell of $\Delta$ belongs to an element of $\textbf{P}$, and

\item if there exist $P_1,P_2\in\textbf{P}$ with nonempty intersection, then both $P_1$ and $P_2$ are $a$-trapezia and this intersection is a $q$-band.

\end{itemize}

In this case, $\textbf{P}$ is called a \textit{covering} of $\Delta$. The $G$-weight of $\textbf{P}$, $\text{wt}_G(\textbf{P})$, is defined to be the sum of the $G$-weights of its elements. 

Note that any reduced diagram over $G_\Omega(\textbf{M})$ has a covering, namely the one given by its cells. So, we may define the $G$-weight of $\Delta$, $\text{wt}_G(\Delta)$, as the minimum of the $G$-weights of its coverings.

Note that any cell of a diagram belongs to at most two elements of a covering, in which case these elements are big or impeding $a$-trapezia.  So, since $G$-weight of a big or impeding $a$-trapezium does not exceed half of its weight, the $G$-weight of any covering is at most the weight of the diagram.

\begin{remark}

The usage of the term `impeding $a$-trapezium' differs slightly from that in \cite{W}: In the context of the machine constructed here, the definition in \cite{W} requires the base of an impeding $a$-trapezium to be either $(Q_0(1)Q_1(1))^{\pm1}$ or $Q_1(1)^{-1}Q_1(1)$, {\frenchspacing i.e. a sector containing $a$-cells} rather than a Move-full base.

However, these two definitions are of the same spirit: By the construction of the submachine in \cite{W} which functions as a Move machine, the $G$-weight assigned to the `impeding trapezia' in \cite{W} implies a similar bound for the `impeding trapezia' with Move-full bases in the definition here (noting that $f_2(n)=1$ for all $n$ in the setting of \cite{W}).

\end{remark}

The next statement provides a simple tool for bounding the $G$-weight of reduced diagram.

\begin{lemma}[Lemma 8.10 of \cite{W}] \label{G-weight subdiagrams}

Let $\Delta$ be a reduced diagram over $G_\Omega(\textbf{M})$ and suppose every cell $\pi$ of $\Delta$ belongs in one of the subdiagrams $\Delta_1,\dots,\Delta_m$, where any nonempty intersection $\Delta_i\cap\Delta_j$ is a $q$-band. Then $\text{wt}_G(\Delta)\leq\sum_{i=1}^m\text{wt}_G(\Delta_i)$.

\end{lemma}

%\begin{proof}
%
%Let $\textbf{P}_1,\dots,\textbf{P}_m$ be coverings of $\Delta_1,\dots,\Delta_m$, respectively, so that the $G$-weight of $\textbf{P}_i$ is equal to that of $\Delta_i$. Then $\textbf{P}=\textbf{P}_1\cup\dots\cup\textbf{P}_m$ is a covering of $\Delta$ with $\text{wt}_G(\textbf{P})\leq\sum_{i=1}^m\text{wt}_G(\textbf{P}_i)$, implying the statement.
%
%\end{proof}

In particular, note that Lemma \ref{G-weight subdiagrams} implies that if $\{\Delta_i\}$ is a partition of the $a$-trapezium $\Delta$, then $\text{wt}_G(\Delta)\leq\sum\text{wt}_G(\Delta_i)$.

\begin{lemma} \label{revolving G-weight}

Suppose $\Delta$ is an $a$-trapezium with revolving base $B$ and height $h$. Then:
$$\text{wt}_G(\Delta)\leq C_3h\max(\|\textbf{tbot}(\Delta)\|,\|\textbf{ttop}(\Delta)\|)+\phi_{C_3}(\|\textbf{tbot}(\Delta)\|+\|\textbf{ttop}(\Delta)\|)$$

\end{lemma}

\begin{proof}

If $\Delta$ is neither big nor exceptional, then since $\text{wt}_G(\Delta)\leq\text{wt}(\Delta)$,the statement follows by \Cref{revolving a-trapezia not exceptional or big} and the parameter choices $C_3>>C_2>>C_1$.

Further, if $\Delta$ is big, then the statement follows from the assignment of the $G$-weight of an $a$-trapezium and the parameter choice $C_3>>C_2$.

Hence, it suffices to assume that $\Delta$ is an exceptional $a$-trapezium.  What's more, Lemmas \ref{a-cells sector}, \ref{revolving a-trapezia exceptional (12)}, and \ref{revolving a-trapezia exceptional transition} along with the parameter choice $C_3>>C_2$ imply that it suffices to assume the step history of $\Delta$ is $(1)_1$.

Now, by the definition of exceptional $a$-trapezia, there exists a factorization $B'\equiv B_1C_2B_2C_3$ of a cyclic permutation of $B$ such that:

\begin{itemize}

\item $B_1$ is nonempty and Move-full.

\item If $B_2$ is nonempty, then it is Move-full.

\item If $C_i$ is non-empty, then it is either:

\begin{itemize}

\item $Q_{0,\ell}'(1)P_0'(1)P_0'(1)^{-1}Q_{0,\ell}'(1)^{-1}$, or 

\item $\{t(1)\}^{-1}Q_{s,r}''(L)^{-1}R_s''(L)^{-1}R_s''(L)Q_{s,r}''(L)\{t(1)\}$

\end{itemize}

\end{itemize}

In particular, there exists a partition $\{\Delta_i\}$ of $\Delta$ such that each $\Delta_i$ is either an impeding $a$-trapezium or is a trapezium with base $C_i$.  

In the latter case, \Cref{a-cells sector} implies $\Delta_i$ is a trapezium, so that \Cref{trapezia are computations} yields a reduced computation $\pazocal{C}:W_0\to\dots\to W_h$ with base $C_i$ corresponding to $\Delta_i$.  As the step history of $\pazocal{C}$ is $(1)_1$, its restriction to the $P_0'(1)P_0'(1)^{-1}$- or $R_s''(L)^{-1}R_s''(L)$-sector satisfies the hypotheses of \Cref{unreduced base} while all other sectors are locked.  Hence, $|W_i|_a\leq\max(|W_0|_a,|W_h|_a)$ for all $i$, so that $\text{wt}_G(\Delta_i)\leq\text{wt}(\Delta_i)=\text{Area}(\Delta_i)\leq h\max(\|\textbf{tbot}(\Delta_i)\|,\|\textbf{ttop}(\Delta_i)\|)$.

Thus, the statement follows by noting $\text{wt}_G(\Delta)\leq\sum\text{wt}_G(\Delta_i)$ and $C_3>>C_2$.

\end{proof}

While the assignment of the $G$-weights of big and impeding $a$-trapezia imply the statement of \Cref{revolving G-weight} which will prove critical to the estimates made in the remainder of the paper, it remains to justify this assignment by connecting it to the weight (and so area) of a diagram with the same boundary label.

For big $a$-trapezia, this matter is handled in \Cref{sec-assignment-of-G-weight}, where \Cref{big G-weight} provides the desired justification.  Some hint for the reasoning, however, is provided by \Cref{big a-trapezium accepted}: 

Given a big $a$-trapezium $\Delta$, both $\lab(\textbf{tbot}(\Delta))$ and $\lab(\textbf{ttop}(\Delta))$ contain cyclic subwords which are, up to the insertion of some words in the `special' input sector, disk relations.  After showing that these words are necessarily $a$-relations, we then construct a reduced diagram $\Gamma$ with $\lab(\partial\Gamma)\equiv\lab(\partial\Delta)$ consisting of two disks, two $a$-cells, and a single $t$-band whose length is the height of $\Delta$.  Finally, we show that $\text{wt}(\Gamma)\leq\text{wt}_G(\Delta)$.

While $\Gamma$ has a desired bound on its weight, though, its two disks cause difficulties for analyzing the diagram obtained by excising $\Delta$ pasting $\Gamma$ in its place (this becomes evident in the definition of $(j)$-minimal diagrams).  This is the main reason for the introduction of the $G$-weight of a big $a$-trapezium.

A similar discussion is provided in the next section for impeding $a$-trapezia, related to the notion of move conditions.

\medskip

%%%%%%%%%%%%%%%%%%%%%%%%%%%%%%%%%%%%%%%%%%%%%%%%%%

\subsection{Move conditions} \label{sec-Move-conditions} \

We are now able to address the definition of a move condition for a Move machine.  We do so here in the specific context relevant to our setting:

%In the past several sections, we have taken the weight of an $a$-cell $\pi$ to be $\phi_{C_1}(\|\partial\pi\|)$ as defined in \Cref{sec-disks-and-weights}.  To begin this section, however, we instead take a more general approach, assuming the weights of $a$-cells are given and fixing an arbitrary non-decreasing function $f:\N\to\N$ such that $\text{wt}(\pi)\leq f(\|\partial\pi\|)$ for all $a$-cells $\pi$.

\begin{definition} \label{def-Move condition}

The Move machine $\textbf{Move}$ satisfies the $(n^2f_2(n),G)$-move condition if for any impeding $a$-trapezium $\Gamma$ with height $h$, there exists a reduced diagram $\Phi$ over $M_\Omega(\textbf{M})$ such that:

\begin{itemize}

\item $\lab(\partial\Phi)\equiv\lab(\partial\Gamma)$ 

\item $\text{wt}(\Phi)\leq C_2h\max(\|\textbf{tbot}(\Gamma)\|,\|\textbf{ttop}(\Gamma)\|)+\phi_{C_3}(\|\textbf{tbot}(\Gamma)\|+\|\textbf{ttop}(\Gamma)\|)$

\end{itemize}

\end{definition}

\begin{remark} 

Note that the move condition is defined as a property of the given move machine $\textbf{Move}$, yet is given by a property of diagrams over the group associated to the larger machine $\textbf{M}$ constructed from $\textbf{Move}$.  By \Cref{M_a no annuli 1} and \Cref{M_a no annuli 2}(a), though, any $(\theta,q)$- or $(\theta,a)$-cell in $\Phi$ corresponds to a rule of $\textbf{Move}$ operating on a copy of the machine's hardware.   Hence, the move condition is indeed a property of the move machine $\textbf{Move}$, dependent only on the function $n^2f_2(n)$ and the choice of group $G=\gen{X}$.

\end{remark}

\begin{remark} \label{rmk-generalized move condition}

While the move condition is defined in \Cref{def-Move condition} for the specific move machine $\textbf{Move}$, the specific function $n^2f_2(n)$, and the specific group $G=\gen{X}$, it is not difficult to generalize the definition to any move machine, function, and group generated by $X$.

Specifically, given a Move machine $\textbf{S}$ with input alphabet $X$, let $\textbf{M}_\textbf{S}$ be the machine constructed from $\textbf{S}$ in just the same way as $\textbf{M}$ is constructed from $\textbf{Move}$ in Sections \ref{sec-auxiliary} and \ref{sec-main-machine}.  Further, fix a non-decreasing function $f:\N\to\N$, a non-decreasing function $f':\N\to\N$ with $f'\sim f$, and a finitely generated group $H$.  Let $\Sigma$ be the set of all non-trivial cyclically reduced words over $X\cup X^{-1}$ which represent the identity in $H$.  Given a reduced diagram $\Delta$ over $M_\Sigma(\textbf{M}_\textbf{S})$, define the weight of each of any $a$-cell $\pi$ in $\Delta$ to be $f'(\|\partial\pi\|)$ and the weight of any other cell to be $1$.  Complete this by defining the weight of $\Delta$, $\text{wt}(\Delta)$, to be the sum of the weights of its cells.

Then $\textbf{S}$ satisfies the $(f,H)$-move condition if there exists a function $f''\sim f$ and a constant $K>0$ such that for any impeding $a$-trapezium $\Gamma$ with height $h$, there exists a reduced diagram $\Phi$ over $M_\Sigma(\textbf{M}_\textbf{S})$ such that:

\begin{itemize}

\item $\lab(\partial\Phi)\equiv\lab(\partial\Gamma)$

\item $\text{wt}(\Phi)\leq Kh\max(\|\textbf{tbot}(\Gamma)\|,\|\textbf{ttop}(\Gamma)\|)+f''(\|\textbf{tbot}(\Gamma)\|+\|\textbf{ttop}(\Gamma)\|)$

\end{itemize}

Note that this definition agrees with \Cref{def-Move condition} since $\phi_C\sim f$ for all $C>0$.

\end{remark}

We now make the connection between the move condition and the computational move condition:

\begin{lemma} \label{computational move is move}

If the move machine $\textbf{Move}$ satisfies the computational $c_0$-move condition, then it satisfies the $(n^2f_2(n),G)$-move condition.

\end{lemma}

\begin{proof}

Let $\Gamma$ be an impeding $a$-trapezium with height $h$, history $H$, and full base $B$.

Suppose first that $\Gamma$ contains no $a$-cells, {\frenchspacing i.e. $\Gamma$ is a trapezium}.  By \Cref{trapezia are computations}, there exists a reduced computation $\pazocal{C}:W_0\to\dots\to W_h$ of $\textbf{Move}$ with full base corresponding to $\Gamma$.  \Cref{Move faulty} (via (Mv7)) then implies $|W_i|_a\leq c_0\max(|W_0|_a,|W_h|_a)$ for all $i$, so that $$\text{wt}(\Gamma)=\text{Area}(\Gamma)\leq c_0h\max(\|\textbf{tbot}(\Gamma)\|,\|\textbf{ttop}(\Gamma)\|)$$ 
Taking $\Phi=\Gamma$ and $C_2>>c_0$, the conditions of \Cref{def-Move condition} are then satisfied.

Hence, it suffices to assume that $\Gamma$ contains an $a$-cell.  By \Cref{a-cells sector}, $B$ must then contain a two-letter subword of the form $(Q_0(1)Q_1(1))^{\pm1}$ or $Q_1(1)^{-1}Q_1(1)$.  Let $\{\Gamma_j\}$ be the collection of all sectors of $\Gamma$ corresponding to such two-letter subwords and let $\{\Gamma_j'\}$ be the collection of maximal subdiagrams of $\Gamma$ consisting of a number of sectors none of which are $\Gamma_j$.  By construction $\{\Gamma_j\}\cup\{\Gamma_j'\}$ is then a partition of $\Gamma$.

For any $j$, let $h_j$ be the number of $(\theta,q)$-cells of $\Gamma_j$ which is the crossing of a $\theta$-band corresponding to a moving rule and a $q$-band corresponding to $Q_1(1)^{\pm1}$.  Note that $\sum h_j=I(B)|H|_{mv}$.

By (MM2), \Cref{a-band on same a-cell}, and (Mv1), any maximal $a$-band in $\Gamma_j$ must have at least one end on $\text{tbot}(\Gamma_j)$, on $\textbf{ttop}(\Gamma_j)$, or on one of the $h_j$ $(\theta,q)$-cells.  In particular, the sum of the (combinatorial) perimeters of the $a$-cells in $\Gamma_j$ is at most $|\textbf{tbot}(\Gamma_j)|_a+|\textbf{ttop}(\Gamma_j)|_a+h_j$.  %Similarly, for any maximal $\theta$-band $\pazocal{T}_j$ in $\Gamma_j$, $|\textbf{tbot}(\pazocal{T})|_a\leq|\textbf{tbot}(\Gamma_j)|_a+|\textbf{ttop}(\Gamma_j)|_a+h_j$.

Now, let $V_0$ be the admissible word of $\textbf{Move}$ corresponding to $\lab(\textbf{tbot}(\Gamma))$.  As each $\Gamma_j'$ is a trapezium, it then follows from (Mv2) and \Cref{trapezia are computations} that there exists a reduced computation $\pazocal{D}:V_0\to\dots\to V_h$ of $\textbf{Move}$ whose history corresponds to $H$.  

Enumerate the maximal $\theta$-bands $\pazocal{T}_1,\dots,\pazocal{T}_h$ of $\Gamma$ from bottom to top.  For each $i$, let $\pazocal{T}_{i,j}$ ({\frenchspacing resp. $\pazocal{T}_{i,j}'$}) be the subband of $\pazocal{T}$ which is a maximal $\theta$-band in $\Gamma_j$ ({\frenchspacing resp. $\Gamma_j'$}).  Then, by \Cref{trapezia are computations} $\lab(\textbf{tbot}(\pazocal{T}_i))$ differs from the admissible word corresponding to $V_{i-1}$ only in the labels of $\textbf{tbot}(\pazocal{T}_{i,j})$.  In particular, $|V_{i-1}|_{NI}=\sum|\textbf{tbot}(\pazocal{T}_{i,j}')|_a$.  \Cref{Move faulty} (via (Mv8)) then implies $\sum|\textbf{tbot}(\pazocal{T}_{i,j}')|_a\leq c_0\max(\sum|\textbf{tbot}(\Gamma_j')|_a,\sum|\textbf{ttop}(\Gamma_j')|_a)$ for all $i$.

Further, as above in establishing the bound of the sum of the perimeters of the $a$-cells in $\Gamma_j$, $|\textbf{tbot}(\pazocal{T}_{i,j})|_a\leq|\textbf{tbot}(\Gamma_j)|_a+|\textbf{ttop}(\Gamma_j)|_a+h_j$ for all $i$.

Hence, $|\textbf{tbot}(\pazocal{T}_i)|_a\leq c_0(|\textbf{tbot}(\Gamma)|_a+|\textbf{ttop}(\Gamma)|_a)+I(B)|H|_{mv}$ and the sum of the combinatorial perimeters of the $a$-cells in $\Gamma$ is at most $\sum(|\textbf{tbot}(\Gamma_j)|_a+|\textbf{ttop}(\Gamma_j)|_a)+I(B)|H|_{mv}$.

As $\textbf{Move}$ satisfies the computational $c_0$-move condition, though, $\pazocal{D}$ must be $c$-narrow, i.e. 
$$I(B)|H|_{mv}\leq c_0\max(\sum|\textbf{tbot}(\Gamma_j')|_a,\sum|\textbf{ttop}(\Gamma_j')|_a)$$
Hence, $|\textbf{tbot}(\pazocal{T}_i)|_a\leq 2c_0(|\textbf{tbot}(\Gamma)|_a+|\textbf{ttop}(\Gamma)|_a)$ and the sum of the combinatorial perimeters of the $a$-cells in $\Gamma$ is at most $c_0(|\textbf{tbot}(\Gamma)|_a+|\textbf{ttop}(\Gamma)|_a)$.  

It then follows from the parameter choice $c_1>>c_0$ that the number of $(\theta,a)$- and $(\theta,q)$-cells in $\Gamma$ is at most $c_1h\max(\|\textbf{tbot}(\Gamma)\|,\|\textbf{ttop}(\Gamma)\|)$, while the superadditivity of $\phi_{C_1}$ and the parameter choices $C_2>>C_1>>c_0$ implies the sum of the weights of the $a$-cells in $\Gamma$ is at most $\phi_{C_2}(|\textbf{tbot}(\Gamma)|_a+|\textbf{ttop}(\Gamma)|_a)$.  Thus, the conditions are again satisfied for $\Phi=\Gamma$.

\end{proof}

\begin{remark} \label{rmk-computational-move-generalized}

Per the discussion in \Cref{rmk-generalized move condition}, \Cref{computational move is move} applies in the general setting of Move machines satisfying computational move conditions.  Specifically, as long as $f$ is equivalent to a superadditive function, a Move machine $\textbf{S}$ that satisfies the computational $c$-move condition necessarily satisfies the $(f,H)$-move condition for the appropriate choice of group $H=\gen{X}$.

\end{remark}

\begin{remark}

As with the discussion of big $a$-trapezia in the previous section, the move condition provides a clue to the reasoning behind and justification for the assignment of $G$-weight to impeding $a$-trapezia: Given an impeding $a$-trapezium $\Gamma$, we may excise $\Gamma$ and paste the reduced diagram $\Phi$ in its place.  However, while this surgery gives desirable bounds on the weight of the diagram, the resulting diagram may introduce new difficulties; for example, it may not be $M$-minimal, as was the case of the treatment in \cite{W}.

\end{remark}

\medskip

%%%%%%%%%%%%%%%%%%%%%%%%%%%%%%%%%%%%%%%%%%%%%%%%%%

\subsection{Combs and Subcombs} \

To aid with the arguments to come, we first discuss in this section a generalization of the notion of $a$-trapezia and analyze how they `cut' an $M$-minimal diagram into subdiagrams.

\begin{figure}[H]
\centering
\includegraphics[scale=1]{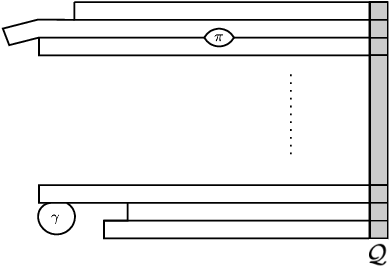}
\caption{Comb with handle $\pazocal{Q}$ containing $a$-cells $\pi$ and $\gamma$}
\end{figure}

Let $\Gamma$ be an $M$-minimal diagram containing a maximal $q$-band $\pazocal{Q}$ such that $\textbf{bot}(\pazocal{Q})$ is a subpath of $\partial\Delta$ and every maximal $\theta$-band of $\Delta$ ends at an edge of $\textbf{bot}(\pazocal{Q})$. Then $\Gamma$ is called a \textit{comb} and $\pazocal{Q}$ its \textit{handle}.

The number of cells in the handle of $\pazocal{Q}$ is the comb's \textit{height} and the maximal length of the bases of the $\theta$-bands its \textit{basic width}.

Note that every $a$-trapezium (or trapezium) may be viewed as a comb with either maximal side $q$-band its handle.

The next statement is adapted from \cite{W}, with the only difference arising from the assignment of the weight of $a$-cells.

\begin{lemma}[Lemma 8.8 of \cite{W}] \label{comb weights}

Let $\Gamma$ be a comb with height $h$, basic width $b$, and $|\partial\Gamma|_a=\a$. Let $\pazocal{T}_1,\dots,\pazocal{T}_h$ be the consecutive maximal $\theta$-bands of $\Gamma$ enumerated from bottom to top. Factor $\partial\Gamma=\textbf{y}\textbf{x}\textbf{z}$, where $\textbf{z}$ is the bottom of the handle of $\Gamma$ and $\textbf{x}$ is the maximal subpath below $\pazocal{T}_1$. Then: 

\begin{enumerate}[label=({\arabic*})]

\item $\text{wt}(\Gamma)\leq c_0bh^2+2\a h+\phi_{C_1}(bh+\a)$

\item $|\textbf{bot}(\pazocal{T}_1)|_a\leq|\textbf{y}|_a+4bh$

\end{enumerate}

\end{lemma}

A base word $B$ is \textit{tight} if it is of the form $uxvx$ for some letter $x$ and words $u$ and $v$, where: 

\begin{enumerate}[label=({\arabic*})]

\item $xvx$ is revolving, and

\item no letter from $u$ occurs in $xvx$.

\end{enumerate}

Note that any tight base has length at most $K_0=22L+1$, while any base with length at least $K_0$ must have a tight prefix.

A comb $\Delta$ is called \textit{tight} if:

\begin{enumerate}[label=(C{\arabic*})]

\item one of its maximal $\theta$-bands $\pazocal{T}$ has a tight base when read toward the handle, and

\item all maximal $\theta$-bands have tight bases or bases without tight prefixes

\end{enumerate}

\smallskip

If $\Delta$ is an $M$-minimal diagram over $M_\Omega(\textbf{M})$, then a subdiagram $\Gamma$ is a \textit{subcomb} of $\Delta$ if $\Gamma$ is a comb and its handle divides $\Delta$ into two parts, one of which is $\Gamma$.

Let $\Gamma$ be a comb with handle $\pazocal{C}$ and $\pazocal{B}$ be another maximal $q$-band in $\Gamma$. Then $\pazocal{B}$ cuts $\Gamma$ into two parts, where the part not containing $\pazocal{C}$ is a subcomb $\Gamma'$ with handle $\pazocal{B}$. Note that each maximal $\theta$-band $\pazocal{T}$ of $\Gamma$ crossing $\pazocal{B}$ has a subband $\pazocal{T}_0$ connecting $\pazocal{B}$ with $\pazocal{C}$. If $\pazocal{T}_0$ has no $(\theta,q)$-cells, then $\Gamma'$ is called a \textit{derivative subcomb} of $\Gamma$. 

Note that no maximal $\theta$-band of a comb can cross the handles of more than one derivative subcomb.

The next statement is proved in just the same way as in \cite{W}:

\begin{lemma}[Lemma 8.9 of \cite{W}] \label{tight subcomb}

Let $\Delta$ be an $M$-minimal diagram such that $|\partial\Delta|_\theta>0$ and every quasi-rim $\theta$-band has base of length at least $K$. %Assume that:
%\begin{addmargin}[1em]{0em}
%
%(1) $\Delta$ is a diagram over $M_a(\textbf{M})$, or
%
%(2) $\Delta$ has a subcomb of basic width at least $K_0$.
%
%\end{addmargin}
Then $\Delta$ contains a tight subcomb.

\end{lemma}

\medskip

%%%%%%%%%%%%%%%%%%%%%%%%%%%%%%%%%%%%%%%%%%%%%%%%%%

\subsection{Upper bound} \

Our goal throughout the rest of this section is to prove that for any $M$-minimal diagram $\Delta$, 
\reqnomode
\begin{equation} \label{counterexample}
\text{wt}_G(\Delta)\leq\phi_{N_2}(|\partial\Delta|)+N_1\mu(\Delta)f_2(N_1|\partial\Delta|)
\end{equation}
for the parameters $N_1$ and $N_2$.

We do this by arguing toward contradiction, considering a `minimal counterexample' diagram $\Delta$. In other words, $\Delta$ is an $M$-minimal diagram over $M_\Omega(\textbf{M})$ satisfying $$\text{wt}_G(\Delta)>\phi_{N_2}(|\partial\Delta|)+N_1\mu(\Delta)f_2(N_1|\partial\Delta|)$$ while (\ref{counterexample}) holds for all $M$-minimal diagrams $\Gamma$ over $M_\Omega(\textbf{M})$ satisfying $|\partial\Gamma|<|\partial\Delta|$. 

\begin{lemma}[Compare with Lemma 8.12 of \cite{W}] \label{no q-edge quadratic}

If $\Gamma$ is an $M$-minimal diagram over $M_\Omega(\textbf{M})$, with no $q$-edges on its boundary, then $\text{wt}_G(\Gamma)\leq \phi_{C_2}(|\partial\Gamma|)$.

\end{lemma}

\begin{proof}

As in the proof of the analogous statement in \cite{W}, $\text{wt}_G(\Gamma)=\text{wt}(\Gamma)$, the sum of the (combinatorial) perimeters of the $a$-cells is at most $\|\partial\Gamma\|$, and the number of $(\theta,a)$-cells is at most $\frac{1}{2}\|\partial\Gamma\|^2$.  As $\phi_{C_1}$ is superadditive, this implies $\text{wt}_G(\Gamma)\leq \phi_{C_1}(\|\partial\Gamma\|)+\frac{1}{2}\|\partial\Gamma\|^2$.

%Since any $q$-edge in $\Gamma$ would give rise to a maximal $q$-band which, by Lemma \ref{M_a no annuli 1}, can only end on the boundary of the diagram, $\Gamma$ cannot have any $q$-edges. So, $\Gamma$ is comprised entirely of $(\theta,a)$-cells and $a$-cells. 
%
%In particular, $\Gamma$ contains no $a$-trapezia (or trapezia), so that the only covering of $\Gamma$ is by single cells. Hence, $\text{wt}_G(\Gamma)=\text{wt}(\Gamma)$.
%
%Lemma \ref{a-band on same a-cell} and (MM2) then imply that any maximal $a$-band with one end on an $a$-cell must have its other end on the boundary, so that the sum of the (combinatorial) perimeters of the $a$-cells is at most $\|\partial\Gamma\|$. It follows that the sum of the weights of the $a$-cells is at most $C_1\|\partial\Gamma\|^2$. 
%
%Further, Lemma \ref{M_a no annuli 2} implies that any maximal $\theta$-band must start and end on $\partial\Gamma$, so that there are at most $\frac{1}{2}\|\partial\Gamma\|$ maximal $\theta$-bands in $\Gamma$. As Lemma \ref{M_a no annuli 1} implies that each maximal $a$-band must have at least one end on $\partial\Gamma$ and each $\theta$-band intersects each $a$-band in at most one cell, the length of each $\theta$-band is at most $\|\partial\Gamma\|$. So, the sum of the lengths of all maximal $\theta$-bands, and so the number of $(\theta,a)$-cells, is at most $\frac{1}{2}\|\partial\Gamma\|^2$.

By the definition of the functions $\phi_{C}$ for $C>0$, $\frac{1}{2}\|\partial\Gamma\|^2\leq\frac{1}{2}\|\partial\Gamma\|^2f_2(\|\partial\Gamma\|)\leq\phi_1(\|\partial\Gamma\|)$.  Hence, $\text{wt}_G(\Gamma)\leq2\phi_{C_1}(\|\partial\Gamma\|)$.

Taking into account the modified definition of perimeter, the statement follows from an appropriate choice of $C_2$ in terms of $C_1$ and $\delta$.

\end{proof}

The parameter choice $N_2>>C_2$ and Lemma \ref{no q-edge quadratic} thus allow us to assume that $\partial\Delta$ consists of at least two $q$-edges, i.e $|\partial\Delta|\geq2$.

\begin{lemma}[Compare with Lemma 8.13 of \cite{W}] \label{a-cell in counterexample}

Let $\pi$ be an $a$-cell contained in $\Delta$. Suppose $\partial\pi$ has a subpath $\textbf{s}$ shared with $\partial\Delta$. Then $\|\textbf{s}\|\leq\frac{2}{3}\|\partial\pi\|$.

\end{lemma}

\begin{proof}

Let $\partial\pi=\textbf{s}\textbf{t}$ and $\partial\Delta=\textbf{s}\textbf{s}_0$.

Assuming toward contradiction that $\|\textbf{s}\|>\frac{2}{3}\|\partial\pi\|$, then hypothesis (1) of \Cref{main-theorem} implies $\|\textbf{s}\|>\frac{2}{3}M>8$ by a parameter choice.  

Let $\Delta_0$ be the subdiagram of $\Delta$ obtained by removing $\pi$. So, $\partial\Delta_0=\textbf{t}^{-1}\textbf{s}_0$.

As in the proof of the analogous statement in \cite{W}, Lemma \ref{lengths}(c) implies $|\partial\Delta_0|\leq|\textbf{s}_0|+\delta\|\textbf{t}\|$ and $|\partial\Delta|\geq|\textbf{s}_0|+\delta(\|\textbf{s}\|-2)$, so that $|\partial\Delta|-|\partial\Delta_0|>\delta\|\textbf{s}\|/4\geq\delta\|\textbf{s}\|/8>0$.

The inductive hypothesis then applies to $\Delta_0$, yielding 
$$\text{wt}_G(\Delta_0)\leq \phi_{N_2}(|\partial\Delta_0|)+N_1\mu(\Delta_0)f_2(N_1|\partial\Delta_0|)\leq \phi_{N_2}(|\partial\Delta|-\delta\|\textbf{s}\|/8)+N_1\mu(\Delta_0)f_2(N_1|\partial\Delta_0|)$$
As $\delta\|\textbf{s}\|/8\leq|\partial\Delta|$, we have $\phi_{N_2}(|\partial\Delta|-\delta\|\textbf{s}\|/8)\leq \phi_{N_2}(|\partial\Delta|)-\frac{N_2\delta}{8}|\partial\Delta|\|\textbf{s}\|f_2(N_2|\partial\Delta|)$.

Lemma \ref{G-weight subdiagrams} implies $\text{wt}_G(\Delta)\leq\text{wt}_G(\Delta_0)+\text{wt}(\pi)$, while $\mu(\Delta)=\mu(\Delta_0)$ since the necklaces corresponding to $\partial\Delta$ and $\partial\Delta_0$ are identical. So, since the combinatorial perimeter of $\pi$ is $\|\textbf{s}\|+\|\textbf{t}\|$, Lemma \ref{G-weight subdiagrams} then implies:
$$\text{wt}_G(\Delta)\leq\phi_{N_2}(|\partial\Delta|)-\frac{N_2\delta}{8}\|\textbf{s}\||\partial\Delta|f_2(N_2|\partial\Delta|)+N_1\mu(\Delta)f_2(N_1|\partial\Delta|)+\phi_{C_1}(\|\textbf{s}\|+\|\textbf{t}\|)$$
So, we reach the contradiction $\text{wt}_G(\Delta)\leq N_2|\partial\Delta|^2+N_1\mu(\Delta)$ if 
$$\frac{N_2\delta}{8}\|\textbf{s}\||\partial\Delta|f_2(N_2|\partial\Delta|)\geq\phi_{C_1}(\|\textbf{s}\|+\|\textbf{t}\|)$$
As $\|\textbf{t}\|<\frac{1}{2}\|\textbf{s}\|$ implies $\phi_{C_1}(\|\textbf{s}\|+\|\textbf{t}\|)\leq\phi_{C_1}(\frac{3}{2}\|\textbf{s}\|)$, it suffices to show 
$$\frac{N_2\delta}{8}\|\textbf{s}\||\partial\Delta|f_2(N_2|\partial\Delta|)\geq\phi_{C_1}(3\|\textbf{s}\|/2)=\frac{9C_1}{4}\|\textbf{s}\|^2f_2\left(\frac{3C_1\|\textbf{s}\|}{2}\right)$$
But since $|\partial\Delta|\geq\frac{1}{4}\delta\|\textbf{s}\|$, this follows from the parameter choices $N_2>>C_1>>\delta^{-1}$.

\end{proof}

%The following is the direct analogue of Lemma 6.12 of [18] and Lemma 6.16 of [25]. The method of proof is identical to the ones presented in those sources, though many of the estimates differ.

\begin{lemma}[Compare with Lemma 8.14 of \cite{W}] \label{6.16} \

\begin{enumerate}[label=({\arabic*})]

\item $\Delta$ has no two disjoint subcombs $\Gamma_1$ and $\Gamma_2$ of basic widths at most $K$ with handles $\pazocal{B}_1$ and $\pazocal{B}_2$ such that some ends of these handles are connected by a subpath $\textbf{x}$ of $\partial\Delta$ with $|\textbf{x}|_q\leq c_0$.

\item If $\Gamma$ is a subcomb of $\Delta$ with basic width $s\leq K$, $|\partial\Gamma|_q=2s$.

\end{enumerate}

\end{lemma}

\renewcommand\thesubfigure{\arabic{subfigure}}
\begin{figure}[H]
\centering
\begin{subfigure}[b]{0.48\textwidth}
\centering
\includegraphics[scale=1]{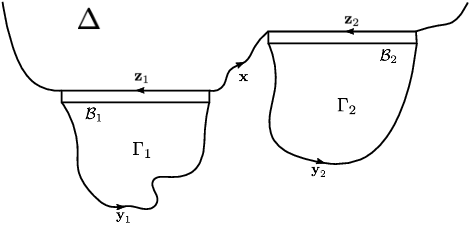}
\caption{ \ }
\end{subfigure}\hfill
\begin{subfigure}[b]{0.48\textwidth}
\centering
\includegraphics[scale=1.225]{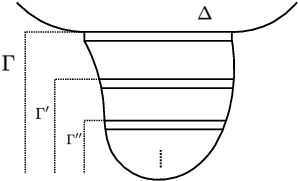}
\caption{ \ }
\end{subfigure}
\caption{Lemma \ref{6.16}}
\end{figure}

\begin{proof}

As in the proof of the analogous statement in \cite{W}, we prove (1) and (2) simultaneously, inducting on $W=\text{wt}(\Gamma_1)+\text{wt}(\Gamma_2)$ for (1) and $W=\text{wt}(\Gamma)$ for (2). In other words, we consider a counterexample to one of these two with minimal value of $W$.

Suppose the minimal counterexample is of the form (1). 

As $\text{wt}(\Gamma_i)<W$ for $i=1,2$, the inductive hypothesis implies that (2) holds for each. So, $\partial\Gamma_i$ has at most $2K$ $q$-edges.

Let $h_i$ be the height of $\Gamma_i$ and assume without loss of generality that $h_1\leq h_2$. For $i=1,2$, let $\partial\Gamma_i=\textbf{y}_i\textbf{z}_i$ where $\textbf{y}_i$ is a subpath of $\partial\Delta$ and $\textbf{z}_i=\textbf{bot}(\pazocal{B}_i)$. Without loss of generality, assume $\textbf{y}_1\textbf{x}\textbf{y}_2$ is a subpath of $\partial\Delta$.

Then each $\theta$-edge of $\textbf{y}_1$ is separated in $\partial\Delta$ from each $\theta$-edge of $\textbf{y}_2$ by at most $4K+c_0$ $q$-edges, and so by at most $J$ $q$-edges by the choice of parameters. Hence, each (correctly ordered) pair of such edges (or the white beads corresponding to these edges) makes a contribution to $\mu(\Delta)$.

Let $\Delta'$ be the diagram obtained by removing the subdiagram $\Gamma_1$ from $\Delta$. When passing from $\partial\Delta$ to $\partial\Delta'$, one replaces each $\theta$-edge of $\textbf{y}_1$ with the corresponding $\theta$-edge of $\textbf{z}_1$ belonging to the same $\theta$-band. But since $\pazocal{B}_1$ is removed, there is at least one less $q$-edge separating any of the $h_1h_2$ (correctly ordered) pairs of $\theta$-edges described above. So, $\mu(\Delta)-\mu(\Delta')\geq h_1h_2$ by Lemma \ref{mixtures}(d).

Letting $|\partial\Gamma_1|_a=\a$, Lemma \ref{comb weights} yields $\text{wt}_G(\Gamma_1)\leq\text{wt}(\Gamma_1)\leq c_0Kh_1^2+2\a h_1+\phi_{C_1}(Kh_1+\a)$.

Letting $\textbf{s}$ be the complement of $\textbf{y}_1$ in $\partial\Delta$, then as in the proof of the analogous statement in \cite{W}, $|\partial\Delta'|\leq h_1+|\textbf{s}|$ and $|\partial\Delta|\geq h_1+|\textbf{s}|+2-2\delta+\max(0,(\a-2h_1)\delta)$.

%By Lemma \ref{lengths}(b), we have $|\textbf{z}_1|=h_1$. Moreover, each of the $h_1$ $(\theta,q)$-cells of $\pazocal{B}_1$ contributes at most one $a$-edge to $\textbf{z}_1$.
%
%So, $\textbf{y}_1$ consists of $h_1$ $\theta$-edges, at least two $q$-edges, and at least $\max(0,\a-h_1)$ $a$-edges. Lemma \ref{lengths}(a) then implies $|\textbf{y}_1|\geq\max(h_1+2,h_1+2+(\a-2h_1)\delta)$.
%
%Letting $\textbf{s}$ be the complement of $\textbf{y}_1$ in $\partial\Delta$, $\textbf{s}$ is also the complement of $\textbf{z}_1^{-1}$ in $\partial\Delta'$. So, Lemma \ref{lengths}(c) implies that $|\partial\Delta'|\leq|\textbf{z}_1|+|\textbf{s}|=h_1+|\textbf{s}|$ and 
%$$|\partial\Delta|\geq|\textbf{y}_1|+|\textbf{s}|-2\delta\geq h_1+|\textbf{s}|+2-2\delta+\max(0,(\a-2h_1)\delta)$$ 
Hence, taking $\delta^{-1}>2$, we have
\begin{equation} \label{6.16 difference in perimeter}
|\partial\Delta|-|\partial\Delta'|\geq\gamma=\max(1,(\a-2h_1)\delta)
\end{equation}
In particular, $|\partial\Delta'|<|\partial\Delta|$, so that the inductive hypothesis implies
\begin{align*}
\text{wt}_G(\Delta')&\leq \phi_{N_2}(|\partial\Delta'|)+N_1\mu(\Delta')f_2(N_1|\partial\Delta'|) \\
&\leq\phi_{N_2}(|\partial\Delta|-\gamma)+N_1(\mu(\Delta)-h_1h_2)f_2(N_1|\partial\Delta|)
\end{align*}
Noting that $\gamma\leq|\partial\Delta|$, we have $\phi_{N_2}(|\partial\Delta|-\gamma)\leq\phi_{N_2}(|\partial\Delta|)-N_2\gamma|\partial\Delta|f_2(N_2|\partial\Delta|)$, so that
$$\text{wt}_G(\Delta')\leq\phi_{N_2}(|\partial\Delta|)-N_2\gamma|\partial\Delta|f_2(N_2|\partial\Delta|)+N_1\mu(\Delta)f_2(N_1|\partial\Delta|)-N_1h_1h_2f_2(N_1|\partial\Delta|)$$
Combining this with the $G$-weight of $\Gamma_1$, Lemma \ref{G-weight subdiagrams} then implies that it suffices to show:
%$$\text{wt}_G(\Delta)\leq\text{wt}_G(\Delta')+c_0Kh_1^2+2\a h_1+\phi_{C_1}(Kh_1+\a)$$
%So, in order to reach the contradiction $\text{wt}_G(\Delta)\leq N_2|\partial\Delta|^2+N_1\mu(\Delta)$, it suffices to show:
$$c_0Kh_1^2+2\a h_1+\phi_{C_1}(Kh_1+\a)\leq N_2\gamma|\partial\Delta|f_2(N_2|\partial\Delta|)+N_1h_1h_2f_2(N_1|\partial\Delta|)$$
If $\a\leq4h_1$, then taking $K\geq4$, $\phi_{C_1}(Kh_1+\a)\leq\phi_{C_1}(2Kh_1)=4C_1K^2h_1^2f_2(2C_1Kh_1)$.  Since $|\partial\Delta|\geq h_1$, it then follows from the parameter choices $N_1>>C_1>>K$ that 
$$c_0Kh_1^2+2\a h_1+\phi_{C_1}(Kh_1+\a)\leq N_1h_1^2f_2(N_1|\partial\Delta|)$$
But then the desired inequality follows from the assumption that $h_1\leq h_2$.

Otherwise, taking $\a>4h_1$, it follows that $\gamma\geq\frac{1}{2}\delta\a$.  As $|\partial\Delta|\geq\gamma$, it then suffices to show:
$$c_0Kh_1^2+\frac{1}{2}\a^2+\phi_{C_1}(K\a)\leq \frac{1}{4}N_2\delta^2\a^2f_2(N_2|\partial\Delta|)+N_1h_1^2f_2(N_1|\partial\Delta|)$$
The parameter choices $N_1>>K>>c_0$ and the hypothesis that $f_2(1)\geq1$ immediately imply $c_0Kh_1^2\leq N_1h_1^2f_2(N_1|\partial\Delta|)$.  Similarly, $N_2>>\delta^{-1}$ implies $\frac{1}{2}\a^2\leq\frac{1}{8}N_2\delta^2\a^2f_2(N_2|\partial\Delta|)$.  Hence, it suffices to show:
$$8C_1K^2\a^2f_2(K\a)\leq N_2\delta^2\a^2f_2(N_2|\partial\Delta|)$$
But $|\partial\Delta|\geq\frac{1}{2}\delta\a$, and so this follows from the parameter choices $N_2>>C_1>>\delta^{-1}>>K$.

Conversely, if we have a minimal counterexample of the form (2), then a verbatim argument to the one presented in \cite{W} for the proof of the analogous statement implies the statement.

%As each derivative subcomb of $\Gamma$ is connected with the handle $\pazocal{B}$ of $\Gamma$ by $\theta$-bands, they can be ordered in a natural way. 
%
%Consider two neighbor derivative subcombs, $\Gamma_1$ and $\Gamma_2$. The handle of $\Gamma_i$ is intersected by two disjoint collections of $\theta$-bands which connect them with $\pazocal{B}$. If there is any $\theta$-band between these two collections, then it cannot intersect any $q$-bands except for $\pazocal{B}$, as otherwise it intersects a derivative subcomb between $\Gamma_1$ and $\Gamma_2$. So, the subpath $\textbf{x}$ of $\partial\Delta$ between the handles of $\Gamma_1$ and $\Gamma_2$ satisfies $|\textbf{x}|_q=0$. 
%
%Hence, $\Gamma_1$ and $\Gamma_2$ form a contradiction to (1). However, $\text{wt}_G(\Gamma_1)+\text{wt}_G(\Gamma_2)<\text{wt}_G(\Gamma)=W$ since they contain no cells of $\pazocal{B}$, contradicting the minimality of the counterexample.
%
%Thus, $\Gamma$ contains at most one derivative subcomb $\Gamma'$. In turn, $\Gamma'$ contains at most one derivative subcomb $\Gamma''$, and so on. Thus, there are $s$ maximal $q$-bands in $\Gamma$, so that Lemma \ref{M_a no annuli 1} implies that $|\partial\Gamma|_q=2s$.

\end{proof}

%Similarly, the next statement is a direct analogue of Lemma 6.14 in [18] and Lemma 6.17 in [25] with altered estimates.

\begin{lemma}[Compare with Lemma 8.15 of \cite{W}] \label{6.17}

Suppose $\Gamma$ is a subcomb of $\Delta$ whose basic width is at most $K_0$ and whose handle $\pazocal{B}$ has length $\ell$. If $\Gamma'$ is a subcomb of $\Gamma$ with handle $\pazocal{B}'$ of length $\ell'$, then $\ell'>\ell/2$.

\end{lemma}

\begin{proof}

Assume toward contradiction that $\Gamma'$ is a subcomb of $\Gamma$ whose handle $\pazocal{B}'$ has length $\ell'\leq\ell/2$. Then, we can choose $\Gamma'$ so that $\ell'$ is minimal for all subcombs in $\Gamma$ and so that $\Gamma'$ has no proper subcombs, i.e the basic width of $\Gamma'$ is 1. Letting $\a=|\partial\Gamma'|_a$, Lemma \ref{comb weights} implies
$$\text{wt}_G(\Gamma')\leq\text{wt}(\Gamma')\leq c_0(\ell')^2+2\a\ell'+\phi_{C_1}(\ell'+\a)$$
Letting $\Delta'$ be the diagram obtained from $\Delta$ by removing $\Gamma'$, then analogous to (\ref{6.16 difference in perimeter}):
\begin{equation} \label{6.17 difference in perimeter}
|\partial\Delta|-|\partial\Delta'|\geq\gamma=\max(1,(\a-2\ell')\delta)
\end{equation}
In particular, $|\partial\Delta'|<|\partial\Delta|$, so that
\begin{equation} \label{6.17 G-weight 1}
\text{wt}_G(\Delta')\leq\phi_{N_2}(|\partial\Delta'|)+N_1\mu(\Delta')f_2(N_1|\partial\Delta'|)\leq\phi_{N_2}(|\partial\Delta|-\gamma)+N_1\mu(\Delta')f_2(N_1|\partial\Delta|)
\end{equation}

As in the proof of the analogous statement in \cite{W}, $\pazocal{B}$ can be decomposed into three subbands: $\pazocal{C}$ which consists of all $\theta$-bands which cross $\pazocal{B}'$ handles $\pazocal{B}_1$ and $\pazocal{B}_2$ of combs $E_1$ and $E_2$, respectively, in $\Gamma'$ (see \Cref{fig-bigsubcomb}).  Letting $\ell_i$ be the height of $E_i$, then again $\mu(\Delta)-\mu(\Delta')\geq\ell'(\ell_1+\ell_2)$.

So, combining (\ref{6.17 difference in perimeter}) and (\ref{6.17 G-weight 1}) and applying Lemma \ref{G-weight subdiagrams}, we then obtain:
$$\text{wt}_G(\Delta)\leq\phi_{N_2}(|\partial\Delta|-\gamma)+N_1(\mu(\Delta)-\ell'(\ell_1+\ell_2))f_2(N_1|\partial\Delta|)+c_0(\ell')^2+2\a\ell'+\phi_{C_1}(\ell'+\a)$$

\begin{figure}[H]
\centering
\includegraphics[scale=1]{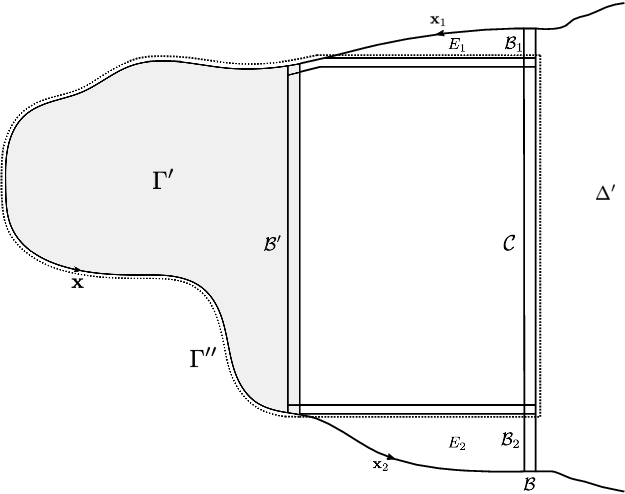}
\caption{Lemma \ref{6.17}}
\label{fig-bigsubcomb}
\end{figure}

As $|\partial\Delta|\geq\gamma$, so that $\phi_{N_2}(|\partial\Delta|-\gamma)\leq\phi_{N_2}(|\partial\Delta|)-N_2\gamma|\partial\Delta|f_2(N_2|\partial\Delta|)$. 
%Factoring in $\Gamma'$ and applying Lemma \ref{G-weight subdiagrams} then yields
%$$
%\text{wt}_G(\Delta)\leq N_2|\partial\Delta|^2-N_2\gamma|\partial\Delta|+N_1\mu(\Delta)-N_1\ell'(\ell_1+\ell_2)+c_0(\ell')^2+2\a\ell'+C_1(\ell'+\a)^2
%$$
Hence, it suffices to show
\begin{equation} \label{suffices 1}
c_0(\ell')^2+2\a\ell'+\phi_{C_1}(\ell'+\a)\leq N_2\gamma|\partial\Delta|f_2(N_2|\partial\Delta|)+N_1\ell'(\ell_1+\ell_2)f_2(N_1|\partial\Delta|)
\end{equation}
As $\ell'\leq\ell/2$ implies $\ell_1+\ell_2\geq\ell'$, it follows that $N_1\ell'(\ell_1+\ell_2)f_2(N_1|\partial\Delta|)\geq\phi_{N_1}(\ell')$.  So, if $\a\leq4\ell'$, then (\ref{suffices 1}) follows from the parameter choices $N_1>>C_1>>c_0$.
%
%Suppose $\a\leq4\ell'$. Then $c_0(\ell')^2+2\a\ell'+C_1(\ell'+\a)^2\leq c_0(\ell')^2+8(\ell')^2+C_1(5\ell')^2$. As $\ell_1+\ell_2\geq\ell'$, (\ref{suffices 1}) then follows from the parameter choices $N_1>>C_1>>c_0$.

Otherwise, if $\a>4\ell'$, then $\gamma\geq\frac{1}{2}\delta\a$. Hence, $|\partial\Delta|\geq\gamma$ and (\ref{suffices 1}) imply that it suffices to show
$$
c_0(\ell')^2+\a^2/2+\phi_{C_1}(5\a/4)\leq\phi_{N_2}(\delta\a/2)+N_1(\ell')^2f_2(N_1|\partial\Delta|)
$$
But this follows from the parameter choices $N_2>>N_1>>C_1>>\delta^{-1}>>c_0$.

\end{proof}

\begin{lemma}[Compare to Lemma 8.16 of \cite{W}] \label{6.18}

If $\pazocal{T}$ is a quasi-rim $\theta$-band in $\Delta$, then the base of $\pazocal{T}$ has length $s>K$.

\end{lemma}

\begin{proof}

Suppose $\pazocal{T}$ is a quasi-rim $\theta$-band in $\Delta$ with base of length $s\leq K$. Without loss of generality, say that any cell between $\textbf{top}(\pazocal{T})$ and $\partial\Delta$ is an $a$-cell. %Let $\textbf{P}_1$ be the set of such $a$-cells.

Let $\textbf{u}$ be the subpath of $\partial\Delta$ bounded by the two end $\theta$-edges of $\pazocal{T}$ and $\textbf{v}$ be its complement in $\partial\Delta$.  Further, let $\Delta'$ be the diagram obtained by $\Delta$ by cutting along $\textbf{bot}(\pazocal{T})$ and let $\textbf{u}'$ be the complement of $\textbf{v}$ in $\partial\Delta'$.  Then, as in the proof of the analogous statement in \cite{W}, we construct the diagram $\Delta''$ from $\Delta'$ by pasting copies of the $a$-cells above $\textbf{top}(\pazocal{T})$ to $\textbf{u}'$ (see \Cref{fig-6.18}).

%For $\pi\in\textbf{P}_1$, factor $\partial\pi=\textbf{p}_\pi\textbf{p}_\pi'$ where $\textbf{p}_\pi$ is a subpath of $\textbf{u}$ and $\textbf{p}_\pi'$ is a subpath of $\textbf{top}(\pazocal{T})$. 
%
%Let $b_\pi$ be the number of edges of $\textbf{p}_\pi'$ that are on the boundary of a $(\theta,q)$-cell of $\pazocal{T}$. 

%By Lemma \ref{a-cell in counterexample}, $\|\textbf{p}_\pi'\|\geq\frac{1}{3}\|\partial\pi\|$. Further, by Lemma \ref{Lemma 6.2}, $\|\partial\pi\|\geq(1-\b)n\geq n/2$ by the parameter choice for $\b$ (see Section 2.8). As a result, $\|\textbf{p}_\pi'\|\geq3$ and $b_\pi\leq2$, so that $\textbf{p}_\pi'$ has a maximal subpath $\textbf{p}_\pi''$ consisting of edges on the boundary of $(\theta,a)$-cells of $\pazocal{T}$.
%
%Consider the diagram $\Delta'$ obtained from $\Delta$ by cutting along $\textbf{bot}(\pazocal{T})$, removing $\pazocal{T}$ and the $a$-cells of $\textbf{P}_1$. For $\pi\in\textbf{P}_1$, the subpath $\textbf{p}_\pi''$ can be identified with a subpath of $\textbf{bot}(\pazocal{T})$, so that we may paste $\pi$ to $\Delta'$ along this subpath. 
%
%Let $\Delta''$ be the diagram obtained by pasting all cells of $\textbf{P}_1$ to $\Delta'$. Note that $\textbf{v}$ can be identified with a subpath of $\partial\Delta''$. Let $\textbf{u}''$ be the complement of $\textbf{v}$ in $\partial\Delta''$. 

\renewcommand\thesubfigure{\alph{subfigure}}
\begin{figure}[H]
\centering
\begin{subfigure}[b]{0.48\textwidth}
\centering
\includegraphics[scale=1.25]{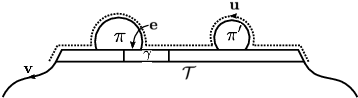}
\caption{$\Delta$, $\gamma$ a $(\theta,q)$-cell}
\end{subfigure}\hfill
\begin{subfigure}[b]{0.48\textwidth}
\centering
\includegraphics[scale=1.25]{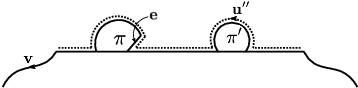}
\caption{$\Delta''$}
\end{subfigure}
\caption{Lemma \ref{6.18}}
\label{fig-6.18}
\end{figure}

%For any $\pi\in\textbf{P}_1$, the edges of $\partial\pi$ contributing to $b_\pi$ belong to $\textbf{u}''$ after this pasting. So, at least $\|\textbf{p}_\pi\|+b_\pi\geq\frac{1}{2}\|\partial\pi\|$ edges of $\partial\pi$ are shared with $\partial\Delta''$. It is thus clear from construction that $\Delta''$ is $M$-minimal.

%Meanwhile, by Lemma \ref{simplify rules}, each $(\theta,q)$-cell of $\pazocal{T}$ contributes at most two $a$-edges to $\textbf{bot}(\pazocal{T})$. Any other edge of $\textbf{u}''$ corresponds to an edge of $\textbf{u}$.

%As each $a$-edge contributing to $b_\pi$ for some $\pi\in\textbf{P}$ is labelled by a letter from the alphabet of the `special' input sector and each $(\theta,q)$-relation has at most one such letter, $\sum b_\pi\leq s$. So, since two $\theta$-edges are removed from $\textbf{u}$, Lemma \ref{lengths} implies
%$$|\textbf{u}|-|\textbf{u}''|\geq2-(2s+2)\delta-\delta\sum b_\pi\geq2-(3s+2)\delta\geq2-(3K+2)\delta\geq1$$
%The parameter choice $\delta^{-1}>>K$ and Lemma \ref{lengths} then imply
Let $\textbf{u}''$ be the complement of $\textbf{v}$ in $\partial\Delta''$.  By \Cref{a-cell in counterexample} and hypothesis (1) of \Cref{main-theorem}, the parameter choice $\delta^{-1}>>K$ and \Cref{lengths} again imply:
$$|\partial\Delta|-|\partial\Delta''|\geq(|\textbf{u}|+|\textbf{v}|-\delta)-(|\textbf{u}''|+|\textbf{v}|)\geq2-(3K+3)\delta\geq1$$
Hence, the inductive hypothesis may be applied to $\Delta''$, so that
$$\text{wt}_G(\Delta'')\leq \phi_{N_2}(|\partial\Delta''|)+N_1\mu(\Delta'')f_2(N_1|\partial\Delta''|)\leq \phi_{N_2}(|\partial\Delta|-1)+N_1\mu(\Delta'')f_2(N_1|\partial\Delta|)$$
Note that the necklace corresponding to $\Delta''$ is obtained from that corresponding to $\Delta$ by the removal of two white beads. Lemma \ref{mixtures}(a) then yields $\mu(\Delta'')\leq\mu(\Delta)$.  Further, note that $|\partial\Delta|\geq1$, so that $\phi_{N_2}(|\partial\Delta|-1)\leq\phi_{N_2}(|\partial\Delta|)-N_2|\partial\Delta|f_2(N_2|\partial\Delta|)$.

%Let $\textbf{P}''$ be a minimal covering of $\Delta''$. As each $a$-cell of $\textbf{P}_1$ has a boundary edge shared with $\partial\Delta''$, it cannot be contained in a trapezium in $\Delta''$. So, $\textbf{P}_1\subset\textbf{P}''$.

%Let $\textbf{P}$ be the covering of $\Delta$ given by $\textbf{P}''$ and the cells of $\pazocal{T}$. Then for $\ell$ the length of $\pazocal{T}$,
As in the analogous setting of \cite{W}, it then follows that for $\ell$ the length of $\pazocal{T}$:
$$\text{wt}_G(\Delta)\leq \phi_{N_2}(|\partial\Delta|)-N_2|\partial\Delta|f_2(N_2|\partial\Delta|)+N_1\mu(\Delta)f_2(N_1|\partial\Delta|)+\ell$$
Hence, it suffices to show that $N_2|\partial\Delta|f_2(N_2|\partial\Delta|)\geq \ell$. 

%For $\pi\in\textbf{P}_1$, (MM1) implies $\|\textbf{p}_\pi'\|\leq\|\textbf{p}_\pi\|+2b_\pi$. So, since $\sum b_\pi\leq s$, $|\textbf{top}(\pazocal{T})|_a\leq|\textbf{u}|_a+2s$.
%
%By Lemma \ref{lengths}(d), $\ell\leq|\textbf{top}(\pazocal{T})|_a+3|\textbf{top}(\pazocal{T})|_q\leq|\textbf{u}|_a+5s$. As each $a$-edge of $\textbf{u}$ contributes at least $\delta$ to $|\Delta|$ and there are $s$ $q$-edges of $\textbf{u}$, $\ell\leq\delta^{-1}(|\partial\Delta|-s)+5s\leq\delta^{-1}|\partial\Delta|$.

But since $f_2(N_2|\partial\Delta|)\geq1$, this again follows from the parameter choice $N_2>>\delta^{-1}$.

\end{proof}

Thus, Lemmas \ref{tight subcomb} and \ref{6.18} imply that there exists a tight subcomb $\Gamma$ in $\Delta$. By the definition of tight combs, the basic width of $\Gamma$ is at most $K_0$.  We decompose $\Gamma$ into subdiagrams $\Gamma_1,\dots,\Gamma_4$ and define paths as in \cite{W} (see \Cref{fig-Mcounterexample}).  

As in \cite{W}, let $\ell$, $\ell'$, $\ell_3$, and $\ell_4$ be the heights of the combs $\Gamma$, $\Gamma_1$, $\Gamma_3$, and $\Gamma_4$, respectively. By Lemma \ref{6.17}, $\ell'>\ell/2$.

\begin{figure}[H]
\centering
\includegraphics[height=3.5in]{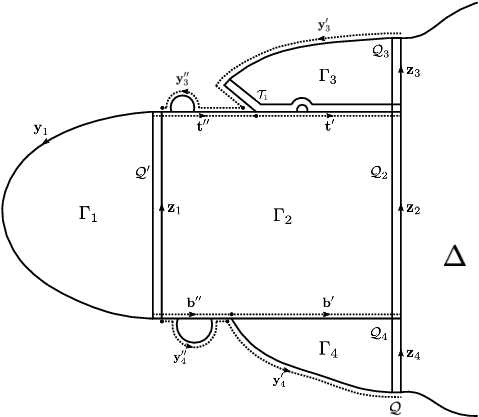}
\caption{Tight subcomb $\Gamma$}
\label{fig-Mcounterexample}
\end{figure}

The following follows in just the same way as in \cite{W}:

\begin{lemma}[Lemma 8.17 of \cite{W}] \label{counterexample combs} \

\begin{enumerate}[label=({\arabic*})] 

\item $|\textbf{t}''|_a\leq|\textbf{y}_3''|_a+4$ and $|\textbf{b}''|_a\leq|\textbf{y}_4''|_a+4$

%\item $|\textbf{y}_3''|_a+|\partial\Gamma_3|_a\geq|\textbf{ttop}(\Gamma_2)|_a-K$ and $|\textbf{y}_4''|_a+|\partial\Gamma_4|_a\geq|\textbf{tbot}(\Gamma_2)|_a-K$

\item $|\textbf{y}_3|_a\geq|\textbf{t}|_a-2\ell_3K-4$ and $|\textbf{y}_4|_a\geq|\textbf{b}|_a-2\ell_4K-4$.

\end{enumerate}

\end{lemma}

\begin{lemma}[Compare with Lemma 8.18 of \cite{W}] \label{H>M}

Set $M'=\max(|\textbf{b}|_a,|\textbf{t}|_a)$. Then $2K\ell>M'$.

\end{lemma}

\begin{proof}

As in the proof of the analogous statement in \cite{W}, $M'\leq|\textbf{y}|_a+3K\ell/2$.  So, assuming the statement is false, then $|\textbf{y}|_a\geq\frac{1}{2}K\ell\geq K_0\ell$.

%As $|\textbf{y}_i|_a\leq|\textbf{y}|_a$ and $\ell_i\leq\ell/2$ for $i=3,4$, Lemma \ref{counterexample combs}(2) implies
%$$M\leq|\textbf{y}|_a+K\ell+4\leq|\textbf{y}|_a+3K\ell/2$$
%Assuming that $2K\ell\leq M$, we then have $|\textbf{y}|_a\geq \frac{1}{2}K\ell\geq K_0\ell$. 

By Lemma \ref{comb weights}, $\text{wt}(\Gamma)\leq c_0K_0\ell^2+2(|\textbf{y}|_a+\ell)\ell+\phi_{C_1}(K_0\ell+|\textbf{y}|_a+\ell)$.
%&\leq(c_0K_0+2+C_1(K_0+1)^2)\ell^2+(2+2C_1(K_0+1))|\textbf{y}|_a\ell+C_1|\textbf{y}|_a^2
%\end{align*}
So, as $C_2$ is chosen after $C_1$, $K_0$, and $c_0$, we have
$\text{wt}(\Gamma)\leq \phi_{C_2}(|\textbf{y}|_a)$.

Since $|\textbf{y}|_\theta=\ell$ and $|\textbf{y}|_q\geq2$, Lemma \ref{lengths}(a) implies that $|\textbf{y}|\geq2+\ell+(|\textbf{y}|_a-\ell)\delta\geq\ell+2+\frac{1}{2}\delta|\textbf{y}|_a$.

Letting $\Delta'$ be the $M$-minimal diagram formed from $\Delta$ by removing $\Gamma$, then as in \cite{W} we have $|\partial\Delta|-|\partial\Delta'|\geq\gamma=\max(1,\frac{1}{2}\delta|\textbf{y}|_a)$.

%Let $\Delta'$ be the $M$-minimal diagram formed from $\Delta$ by removing $\Gamma$. Then in $\partial\Delta'$, $\textbf{y}$ is replaced with $\textbf{z}$. Lemma \ref{lengths}(b) implies that $|\textbf{z}|=\ell$.

%Letting $\textbf{s}$ be the complement of $\textbf{z}$ in $\partial\Delta'$, Lemma \ref{lengths}(c) implies $|\partial\Delta'|\leq|\textbf{s}|+\ell$ and $|\partial\Delta|\geq|\textbf{s}|+|\textbf{y}|-2\delta$. So, $|\partial\Delta|-|\partial\Delta'|\geq\gamma=\max(1,\frac{1}{2}\delta|\textbf{y}|_a)$.

Hence, the inductive hypothesis applies to $\Delta'$, yielding
$$\text{wt}_G(\Delta')\leq \phi_{N_2}(|\partial\Delta'|)+N_1\mu(\Delta')f_2(N_1|\partial\Delta'|)\leq \phi_{N_2}(|\partial\Delta|-\gamma)+N_1\mu(\Delta')f_2(N_1|\partial\Delta|)$$
As $\gamma\leq|\partial\Delta|$, $\phi_{N_2}(|\partial\Delta|-\gamma)\leq\phi_{N_2}(|\partial\Delta|)-N_2\gamma|\partial\Delta|f_2(N_2|\partial\Delta|)$. Lemma \ref{mixtures} further implies that $\mu(\Delta')\leq\mu(\Delta)$. So, adding in the weight of $\Gamma$, Lemma \ref{G-weight subdiagrams} implies:
$$\text{wt}_G(\Delta)\leq\phi_{N_2}(|\partial\Delta|)-N_2\gamma|\partial\Delta|f_2(N_2|\partial\Delta|)+N_1\mu(\Delta)f_2(N_1|\partial\Delta|)+\phi_{C_2}(|\textbf{y}|_a)$$
So, it suffices to show that $N_2\gamma|\partial\Delta|f_2(N_2|\partial\Delta|)\geq\phi_{C_2}(|\textbf{y}|_a)$. 

But $\frac{1}{2}\delta|\textbf{y}|_a\leq\gamma\leq|\partial\Delta|$, so that $\gamma|\partial\Delta|\geq\frac{1}{4}\delta^2|\textbf{y}|_a^2$. Hence, the desired inequality follows from the parameter choices $N_2>>C_2>>\delta^{-1}$.

\end{proof}

Finally, we reach the desired contradiction:

\begin{lemma}[Compare with Lemma 8.19 of \cite{W}] \label{diskless}

The counterexample diagram $\Delta$ does not exist.

\end{lemma}

\begin{proof}

Let $\Delta_1$ be the diagram obtained from $\Delta$ by removing $\Gamma\setminus\pazocal{Q}$. As in the proof of Lemma 8.19 in \cite{W}, we then construct the diagram $\Delta_0$ by pasting $\Gamma_1$ to $\Delta_1$ along $\pazocal{Q}'$ (see \Cref{fig-diskless}).  Again, $\Delta_0$ must be $M$-minimal.

\begin{figure}[H]
\centering
\includegraphics[height=3in]{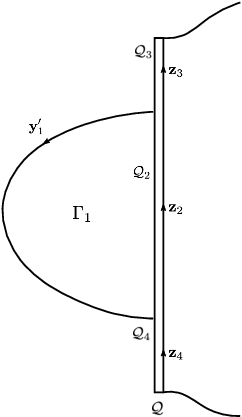}
\caption{The construction of $\Delta_0$}
\label{fig-diskless}
\end{figure}

%As $B_2$ is revolving, the first letter of $B_1$ appears in $B_2$. But then the base of the maximal $\theta$-band of $\Gamma$ containing $\pazocal{T}$ has a tight prefix, contradicting the assumption that $\Gamma$ is a tight comb.

%So, for any covering of $\Delta_0$, each element is either contained completely in $\Gamma_1$ or completely in $\Delta_1$. Hence, given a minimal covering $\textbf{P}$ of $\Delta_0$, we may construct coverings $\textbf{P}'$ and $\textbf{P}''$ of $\Gamma_1$ and $\Delta_1$, respectively, by including only the elements belonging to these subdiagrams and perhaps adding in the cells of $\pazocal{Q}'$ or $\pazocal{Q}_2$. As at most the $\ell'$ cells of $\pazocal{Q}_2$ are counted twice in these coverings, we have $\text{wt}_G(\Delta_0)\geq\text{wt}_G(\Gamma_1)+\text{wt}_G(\Delta_1)-\ell'$.

As in \cite{W}, $\text{wt}_G(\Delta_0)\geq\text{wt}_G(\Delta_1)-\ell'$, and so \Cref{G-weight subdiagrams} implies

%Lemma \ref{G-weight subdiagrams} then implies:
\begin{equation} \label{G-weight}
\text{wt}_G(\Delta)\leq\text{wt}_G(\Delta_1)+\text{wt}_G(\Gamma)\leq\text{wt}_G(\Delta_0)+\text{wt}_G(\Gamma_2)+\text{wt}_G(\Gamma_3)+\text{wt}_G(\Gamma_4)+A+\ell'
\end{equation}
where $A$ is the sum of the weights of the $a$-cells attached to $\textbf{t}''$ or $\textbf{b}''$.

Moreover, an identical argument to that in \cite{W} implies:
%For $i=3,4$, note that the subpath $\textbf{y}_i''$ has no $\theta$-edges, while $\textbf{y}_i'$ consists of $\ell_i$ $\theta$-edges and at least one $q$-edge. So, Lemma \ref{lengths}(a) implies $|\textbf{y}_i|\geq1+\ell_i+\delta\max(0,|\textbf{y}_i|_a-\ell_i,|\textbf{y}_i''|_a-1)$.
%
%Letting $\textbf{s}$ be the complement of $\textbf{y}$ in $\partial\Delta$, Lemma \ref{lengths}(c) then yields
%$$|\partial\Delta|\geq|\textbf{s}|+|\textbf{y}|-2\delta\geq|\textbf{s}|+|\textbf{y}_3|+|\textbf{y}_1|+|\textbf{y}_4|-4\delta$$
%Next, let $\textbf{y}_1'$ be the subpath of $\textbf{y}_1$ not containing the first or last edge. Note that both of these edges are $q$-edges corresponding to $\pazocal{Q}'$, so that $|\textbf{y}_1|=|\textbf{y}_1'|+2$.
%
%Then, Lemma \ref{lengths}(c) implies $|\partial\Delta_0|\leq|\textbf{s}|+1+|\textbf{top}(\pazocal{Q}_3)|+|\textbf{y}_1'|+|\textbf{top}(\pazocal{Q}_4)|+1$. Further, Lemma \ref{lengths}(b) implies $|\textbf{top}(\pazocal{Q}_i)|=\ell_i$ for $i=3,4$. So, $|\partial\Delta_0|\leq|\textbf{s}|+|\textbf{y}_1|+\ell_3+\ell_4$.
%
%Hence,
\begin{align*}
|\partial\Delta|-|\partial\Delta_0|\geq\gamma&=(|\textbf{y}_3|-\ell_3)+(|\textbf{y}_4|-\ell_4)-4\delta \\
&\geq2-4\delta+\delta\max(0,|\textbf{y}_3|_a-\ell_3,|\textbf{y}_3''|_a-1)+\delta\max(0,|\textbf{y}_4|_a-\ell_4,|\textbf{y}_4''|_a-1)
\end{align*}
So, taking $\delta^{-1}\geq4$, $|\partial\Delta|-|\partial\Delta_0|\geq\gamma\geq2-4\delta\geq1$.

%2-4\delta+\delta(\max(0,|\textbf{y}_3|_a-\ell_3)+\max(0,|\textbf{y}_4|_a-\ell_4))\geq1$$

Hence, we may apply the inductive hypothesis to $\Delta_0$, so that
$$\text{wt}_G(\Delta_0)\leq\phi_{N_2}(|\partial\Delta_0|)+N_1\mu(\Delta_0)f_2(N_1|\partial\Delta_0|)\leq N_2(|\partial\Delta|-\gamma)^2+N_1\mu(\Delta_0)f_2(N_1|\partial\Delta|)$$
Using \Cref{mixtures}, the analysis in \cite{W} of the beads removed in passing from $\Delta$ to $\Delta_0$ implies
%In $\textbf{y}$, any $\theta$-edge of $\textbf{y}_1$ is separated from a $\theta$-edge of $\textbf{y}_3$ or $\textbf{y}_4$ by a $q$-edge at the end of $\pazocal{Q}'$. Moreover, since the basic width of $\Gamma$ is at most $K_0$, the parameter choice $J>>K_0$ implies that each of these (correctly ordered) pairs contribute to $\mu(\Delta)$. But the black beads corresponding to $\pazocal{Q}'$ are removed in the formation of the necklace for $\Delta_0$, so that Lemma \ref{mixtures}(d) implies
$\mu(\Delta)-\mu(\Delta_0)\geq\ell'(\ell_3+\ell_4)$.  Noting that $\gamma\leq|\partial\Delta|$, we then have:
$$\text{wt}_G(\Delta_0)\leq\phi_{N_2}(|\partial\Delta|)-N_2\gamma|\partial\Delta|f_2(N_2|\partial\Delta|)+N_1\mu(\Delta)f_2(N_1|\partial\Delta|)-N_1\ell'(\ell_3+\ell_4)f_2(N_1|\partial\Delta|)$$
Hence, by (\ref{G-weight}), it suffices to show that:
\begin{equation} \label{diskless suffices 1}
N_2\gamma|\partial\Delta|f_2(N_2|\partial\Delta|)+N_1\ell'(\ell_3+\ell_4)f_2(N_1|\partial\Delta|)\geq\text{wt}_G(\Gamma_2)+\text{wt}_G(\Gamma_3)+\text{wt}_G(\Gamma_4)+A+\ell'
\end{equation}
Setting $\nu_i=|\partial\Gamma_i|_a$ for $i=3,4$, Lemma \ref{comb weights} and the parameter choices $C_2>>C_1>>K_0>>c_0$ imply:
$$\text{wt}_G(\Gamma_i)\leq\text{wt}(\Gamma_i)\leq c_0K_0\ell_i^2+2\nu_i\ell_i+\phi_{C_1}(K_0\ell_i+\nu_i)\leq\phi_{C_2}(\ell_i+\nu_i)$$
By Lemmas \ref{revolving G-weight},
%\begin{align*}
$$\text{wt}_G(\Gamma_2)\leq C_3\ell'\max(\|\textbf{t}\|,\|\textbf{b}\|)+\phi_{C_3}(\|\textbf{t}\|+\|\textbf{b}\|)\leq C_3\ell'(M'+K_0)+\phi_{C_3}(2M'+K)$$
%\end{align*}
As $a$-cell $\pi$ whose weight contributes to $A$, $\pi$ is attached to either $\textbf{t}''$ or $\textbf{b}''$. \Cref{a-cell in counterexample} then implies $A\leq\phi_{C_1}(3|\textbf{t}''|_a+3|\textbf{b}''|_a)\leq \phi_{C_1}(6M')$, so that Lemmas \ref{H>M} and \ref{6.17} and a parameter choice yield $A\leq 36C_1(M')^2f_2(6C_1M')\leq C_2\ell'M'f_2(C_2M')$.  As a result, the parameter choices $C_4>>C_3>>C_2>>K$ imply:
%$M\leq2K\ell\leq4K\ell'$. So, since $C_2>>C_1>>K$, $A\leq C_2\ell'M$.
%Hence, the parameter choices $C_3>>C_2>>K>>K_0$ imply
\begin{align*}
\text{wt}_G(\Gamma_2)+A+\ell'&\leq C_3\ell'M'+C_3K_0\ell'+\phi_{C_3}(2M'+K)+C_2\ell'M'f_2(C_2M')+\ell' \\
&\leq C_4\ell'M'f_2(C_4\ell')+C_4\ell'
%&\leq 2C_2\ell'M+16C_2K\ell'M+C_2K_0\ell'+32C_2K_0K\ell'+\ell'+4C_2K_0^2 \\
%&\leq C_3\ell'M+C_3\ell'+C_3
\end{align*}
Now, without loss of generality suppose $\nu_4\leq\nu_3$.  Then as in the proof of the analogous statement in \cite{W} we have $M'\leq\max(|\textbf{t}''|_a,|\textbf{b}''|_a)+\nu_3\leq\max(|\textbf{y}_3''|_a,|\textbf{y}_4''|_a)+\nu_3+4$.  So, 
$$\text{wt}_G(\Gamma_2)+A+\ell'\leq C_4\ell'\left(\max(|\textbf{y}_3''|_a,|\textbf{y}_4''|_a)+\nu_3+5\right)f_2(C_4\ell')$$

%So, by (\ref{diskless suffices 1}), it suffices to show that:
%\begin{equation} \label{diskless suffices 2}
%N_2\gamma|\partial\Delta|+N_1\ell'(\ell_3+\ell_4)\geq C_3\ell'M+C_3\ell'+C_3+C_2(\ell_3+\nu_3)^2+C_2(\ell_4+\nu_4)^2
%\end{equation}
%Without loss of generality, assume $\nu_4\leq\nu_3$.

%Note $M=\max(|\textbf{t}|_a,|\textbf{b}|_a)=\max(|\textbf{t}''|_a+|\textbf{t}'|_a,|\textbf{b}''|_a+|\textbf{b}'|_a)\leq\max(|\textbf{t}''|_a,|\textbf{b}''|_a)+\max(|\textbf{t}'|_a,|\textbf{b}'|_a)$. Since $\textbf{t}'$ and $\textbf{b}'$ are subpaths of $\partial\Gamma_3$ and $\partial\Gamma_4$, respectively, we then have $M\leq\max(|\textbf{t}''|_a,|\textbf{b}''|_a)+\nu_3$.

%Lemma \ref{counterexample combs}(1) then yields $M\leq\max(|\textbf{y}_3''|_a,|\textbf{y}_4''|_a)+\nu_3+4$.

Meanwhile, $|\partial\Delta|\geq\ell'$, the parameter choice $N_2>>C_4$ allows us to assume 
$$\frac{1}{3}N_2\gamma|\partial\Delta|f_2(N_2|\partial\Delta|)\geq 5C_4\ell'f_2(C_4\ell')$$

Moreover, if $\max(|\textbf{y}_3''|_a,|\textbf{y}_4''|_a)\leq1$, then since $|\partial\Delta\geq\ell'$ and $\gamma\geq1$, we have
\begin{equation} \label{diskless inequality}
\frac{1}{3}N_2\gamma|\partial\Delta|f_2(N_2|\partial\Delta|)\geq C_4\ell'\max(|\textbf{y}_3''|_a,|\textbf{y}_4''|_a)f_2(C_4\ell')
\end{equation}
Otherwise, recall that $\gamma\geq\delta(|\textbf{y}_3''|_a-1)+\delta(|\textbf{y}_4''|_a-1)\geq\delta(\max(|\textbf{y}_3''|_a,|\textbf{y}_4''|_a)-1)\geq\frac{1}{2}\delta\max(|\textbf{y}_3''|_a,|\textbf{y}_4''|_a)$. So, since $|\partial\Delta|\geq\ell'$, the parameter choices $N_2>>C_4>>\delta^{-1}$ allow us to again assume (\ref{diskless inequality}) holds.

%As $\gamma\geq1$, $|\partial\Delta|\geq\ell'$, and $|\partial\Delta|\geq2$, the parameter choice $N_2>>C_3$ allows us to assume that $$\frac{1}{3}N_2\gamma|\partial\Delta|\geq5C_3\ell'+C_3$$
Hence, it suffices to show that 
\begin{equation} \label{diskless suffices 3}
\frac{1}{3}N_2\gamma|\partial\Delta|f_2(N_2|\partial\Delta|)+N_1\ell'(\ell_3+\ell_4)f_2(N_1|\partial\Delta|)\geq C_4\ell'\nu_3f_2(C_4\ell')+\sum_{i=3}^4\phi_{C_2}(\ell_i+\nu_i)
\end{equation}
%\begin{equation} \label{diskless suffices 3}
%\frac{2}{3}N_2\gamma|\partial\Delta|f_2(N_2|\partial\Delta|)+N_1\ell'(\ell_3+\ell_4)f_2(N_1|\partial\Delta|)\geq C_4\ell'\left(\max(|\textbf{y}_3''|_a,|\textbf{y}_4''|_a)+\nu_3\right)f_2(C_4\ell')+\phi_{C_2}(\ell_3+\nu_3)+\phi_{C_2}(\ell_4+\nu_4)
%\end{equation}
%Suppose $\max(|\textbf{y}_3''|_a,|\textbf{y}_4''|_a)\leq1$. Then since $\gamma\geq1$, $|\partial\Delta|\geq\ell'$, and $N_2>>C_4$, we may take 
%\begin{equation} \label{diskless inequality}
%\frac{1}{3}N_2\gamma|\partial\Delta|\geq C_4\ell'\max(|\textbf{y}_3''|_a,|\textbf{y}_4''|_a)
%\end{equation}
%Otherwise, recall that $\gamma\geq\delta(|\textbf{y}_3''|_a-1)+\delta(|\textbf{y}_4''|_a-1)\geq\delta(\max(|\textbf{y}_3''|_a,|\textbf{y}_4''|_a)-1)\geq\frac{1}{2}\delta\max(|\textbf{y}_3''|_a,|\textbf{y}_4''|_a)$. So, since $|\partial\Delta|\geq\ell'$, the parameter choices $N_2>>C_3>>\delta^{-1}$ allow us to again assume (\ref{diskless inequality}) holds.

%Thus, by (\ref{diskless suffices 3}), it suffices to show that
%\begin{equation} \label{diskless suffices 4}
%\frac{1}{3}N_2\gamma|\partial\Delta|+N_1\ell'(\ell_3+\ell_4)\geq C_3\ell'\nu_3+C_2(\ell_3+\nu_3)^2+C_2(\ell_4+\nu_4)^2
%\end{equation}

We now proceed in cases:

\textbf{1.} Suppose $\nu_3\leq 3J(\ell_3+\ell_4)$.

Then for $i=3,4$, $\ell_i+\nu_i\leq4J(\ell_3+\ell_4)$. As Lemma \ref{6.17} implies $\ell_3+\ell_4\leq\ell'$, this implies $\phi_{C_2}(\ell_i+\nu_i)\leq16C_2J^2\ell'(\ell_3+\ell_4)f_2(4C_2J\ell')$.

So, the parameter choices $C_3>>C_2>>J$ imply $\phi_{C_2}(\ell_i+\nu_i)\leq C_3\ell'(\ell_3+\ell_4)f_2(C_3\ell')$.

But $|\partial\Delta|\geq\ell'$ and the parameter choices $N_1>>C_3>>J$ imply $$N_1\ell'(\ell_3+\ell_4)f_2(N_1|\partial\Delta|)\geq C_4\ell'\nu_3f_2(C_4\ell')+C_3\ell'(\ell_3+\ell_4)f_2(C_3\ell')$$

\smallskip

Thus, we may assume that $\nu_3>3J(\ell_3+\ell_4)$.

For $i=3,4$, this implies $\ell_i+\nu_i\leq\ell_3+\ell_4+\nu_3\leq2\nu_3$, so that $\phi_{C_2}(\ell_i+\nu_i)\leq\phi_{C_2}(2\nu_3)$.

It then follows from (\ref{diskless suffices 3}) that it suffices to show that
\begin{equation} \label{diskless suffices 4}
\frac{1}{3}N_2\gamma|\partial\Delta|f_2(N_2|\partial\Delta|)+N_1\ell'(\ell_3+\ell_4)f_2(N_1|\partial\Delta|)\geq C_4\ell'\nu_3f_2(C_4\ell')+\phi_{C_3}(\nu_3)
\end{equation}

\textbf{2.} Suppose $\nu_3\leq16$.

Then $C_4\ell'\nu_3+\phi_{C_3}(\nu_3)\leq16C_4\ell'+\phi_{C_4}(1)$ by the parameter choice $C_4>>C_3$.

But $|\partial\Delta|\geq\max(2,\ell')$ and $\gamma\geq1$, so that the parameter choices $N_2>>C_3>>C_2$ yield
$$\frac{1}{3}N_2\gamma|\partial\Delta|f_2(N_2|\partial\Delta|)\geq C_4\ell'+\phi_{C_4}(1)$$

\textbf{3.} Thus, it suffices to assume that $\nu_3>\max(3J(\ell_3+\ell_4),16)$ and show that (\ref{diskless suffices 4}) holds.

As in the proof of the analogous statement in \cite{W}, it follows that $\frac{7}{12}\nu_3\leq2|\textbf{y}_3|_a$ and $\gamma\geq\frac{1}{4}\delta\nu_3$.

%As $a$-edges of $\Gamma_3$ must be edges of $\textbf{y}_3'$, $\textbf{t}'$, or $\textbf{bot}(\pazocal{Q}_3)$, it follows that $\nu_3\leq|\textbf{y}_3|_a+|\textbf{t}|_a+\ell_3$.

%As any $a$-edge of $\partial\Gamma_3$ is part of $\textbf{y}_3'$, $\textbf{t}'$, or $\textbf{bot}(\pazocal{Q}_3)$, we have 
%$\nu_3\leq|\textbf{y}_3|_a+|\textbf{t}|_a+\ell_3$.
%By Lemma \ref{counterexample combs}(2), this implies $\nu_3\leq2|\textbf{y}_3|_a+2K\ell_3+\ell_3+4\leq2|\textbf{y}_3|_a+J\ell_3+4$.  Note that $J\ell_3+4<\frac{1}{3}\nu_3+\frac{1}{4}\nu_3=\frac{5}{12}\nu_3$, so that $\frac{7}{12}\nu_3\leq2|\textbf{y}_3|_a$.

%Recall that $\gamma\geq\delta(|\textbf{y}_3|_a-\ell_3)$. So, since $\ell_3<\frac{1}{3J}\nu_3\leq\frac{1}{24}\nu_3$ by taking $J\geq8$, we have $\gamma\geq\frac{1}{4}\delta\nu_3$.

As $|\partial\Delta|\geq\ell'$, it then follows that $\gamma|\partial\Delta|\geq\delta\ell'\nu_3$, so that the parameter choices $N_2>>C_4>>\delta^{-1}$ imply $\frac{1}{6}N_2\gamma|\partial\Delta|f_2(N_2|\partial\Delta|)\geq C_4\ell'\nu_3f_2(C_4\ell')$.

Further, since $|\partial\Delta|\geq\gamma$, we have $\gamma|\partial\Delta|\geq\frac{1}{16}\delta^2\nu_3^2$, so that $N_2>>C_2>>\delta^{-1}$ imply
$$\frac{1}{6}N_2\gamma|\partial\Delta|f_2(N_2|\partial\Delta|)\geq\phi_{C_3}(\nu_3)$$
%Thus, $$\frac{1}{3}N_2\gamma|\partial\Delta|f_2(N_2|\partial\Delta|)\geq C_4\ell'\nu_3f_2(C_4\ell')+\phi_{C_3}(\nu_3)$$

Thus, (\ref{diskless suffices 4}) is satisfied, and so the statement is proved.

\end{proof}

\medskip

%%%%%%%%%%%%%%%%%%%%%%%%%%%%%%%%%%%%%%%%%%%%%%%%%%%%%%%%%%%%%%%%%

\section{Diagrams with disks} \label{sec-disks}

In this section we study some of the basic properties of diagrams over the canonical presentation of $G_\Omega(\textbf{M})$.  For this purpose, we assume throughout that our diagrams are all homeomorphic to a disk {\frenchspacing (i.e. they are an Kampen diagrams}); however, it should be noted that many of the notions have obvious analogues for diagrams on various surfaces.  

Unless explicitly stated otherwise (notably in \Cref{sec-transposition}), the treatment herein is analogous to (and sometimes an exact copy of) that found in Section 9 of \cite{W}.

\subsection{Diminished, Minimal, and $D$-minimal diagrams} \

A $q$-letter of the form $t(i)$ for $2\leq i\leq L$ is called a \textit{$t$-letter}. Accordingly, a $(\theta,q)$-relation corresponding to a $t$-letter is called a \textit{$(\theta,t)$-relation}. Note that for each rule $\theta$ and each $t$-letter, the corresponding $(\theta,t)$-relation is of the simple form $\theta_jt(i)=t(i)\theta_{j+1}$.

Now, we modify the definition of a reduced diagram over the canonical presentation of $M_\Omega(\textbf{M})$ or over the disk presentation of $G_\Omega(\textbf{M})$. To this end, we introduce the \textit{signature} of such a diagram $\Delta$ as the four-tuple $s(\Delta)=(\a_1,\a_2,\a_3,\a_4)$ where: 

\begin{addmargin}[1em]{0em}

$\bullet$ $\a_1$ is the number of disks in $\Delta$ (of course, this is zero if $\Delta$ is a diagram over $M_\Omega(\textbf{M})$), 

$\bullet$ $\a_2$ is the number of $(\theta,t)$-cells, 

$\bullet$ $\a_3$ is the total number of $(\theta,q)$-cells, and

$\bullet$ $\a_4$ is the number of $a$-cells.

\end{addmargin}

The signatures of reduced diagrams over the disk presentation of $G_\Omega(\textbf{M})$ are ordered lexicographically. In particular, if $\Delta$ and $\Gamma$ are such diagrams with $s(\Delta)=(\a_1,\a_2,\a_3,\a_4)$ and $s(\Gamma)=(\b_1,\b_2,\b_3,\b_4)$, then $s(\Delta)\leq s(\Gamma)$ if:

\begin{addmargin}[1em]{0em}

$\bullet$ $\a_1\leq\b_1$

$\bullet$ for $i\in\{2,3,4\}$, if $\a_j=\b_j$ for all $j<i$, then $\a_i\leq\b_i$

\end{addmargin}

A reduced diagram $\Delta$ over the disk presentation of $G_\Omega(\textbf{M})$ is called \textit{diminished} if for any reduced diagram $\Gamma$ with $\lab(\partial\Delta)\equiv\lab(\partial\Gamma)$, we have $s(\Delta)\leq s(\Gamma)$.

Given a reduced diagram $\Delta$ over the disk presentation of $G_\Omega(\textbf{M})$ with $s(\Delta)=(\a_1,\a_2,\a_3,\a_4)$, the \textit{2-signature} of $\Delta$ is the ordered pair $s_2(\Delta)=(\a_1,\a_2)$. The \textit{1-signature} $s_1(\Delta)$ is defined similarly and can be interpreted simply as the number of disks in $\Delta$ with the natural order on the natural numbers.

A reduced diagram $\Delta$ over the disk presentation of $G_\Omega(\textbf{M})$ is called \textit{$D$-minimal} if for any reduced diagram $\Gamma$ with $\lab(\partial\Gamma)\equiv\lab(\partial\Delta)$, $s_1(\Delta)\leq s_1(\Gamma)$. By the definition of the lexicographic order, a diminished diagram is necessarily $D$-minimal.

Finally, reduced diagram $\Delta$ over the disk presentation of $G_\Omega(\textbf{M})$ is called \textit{minimal} if:

\begin{addmargin}[1em]{0em}

\begin{enumerate}[label=(M{\arabic*})]

%\item it contains no $\theta$-annulus $S$ whose sides are labelled by letters of the tape alphabet of the `special' input sector,

\item for any $a$-cell $\pi$ and any $\theta$-band $\pazocal{T}$, at most half of the edges of $\partial\pi$ mark the start of an $a$-band that crosses $\pazocal{T}$,

\item no maximal $a$-band ends on two different $a$-cells, and

\item for any reduced diagram $\Gamma$ with $\text{Lab}(\partial\Delta)\equiv\text{Lab}(\partial\Gamma)$, $s_2(\Delta)\leq s_2(\Gamma)$.

\end{enumerate}

\end{addmargin}

Note that conditions (M1) and (M2) are equivalent to the conditions (MM1) and (MM2) in the definition of $M$-minimal. As a result, a minimal diagram containing no disks is necessarily $M$-minimal. Further, a diminished diagram necessarily satisfies (M3).

As with $M$-minimal diagrams, a subdiagram of a diminished (resp minimal, $D$-minimal) diagram is necessarily diminished (resp minimal, $D$-minimal).

In what follows, it is taken implicitly that any diminished, minimal, or $D$-minimal diagram over $G_\Omega(\textbf{M})$ is formed over its disk presentation (rather than its canonical presentation).

The next statement, proved in exactly the same way as its analogue in \cite{W}, provides a useful refinement of the van Kampen Lemma \cite{v-K} in this setting:

\begin{lemma}[Lemma 9.1 of \cite{W}] \label{diminished exist}

A word $W$ over $\pazocal{Y}$ represents the trivial element of $M_\Omega(\textbf{M})$ if and only if there exists a diminished diagram $\Delta$ over $M_\Omega(\textbf{M})$ such that $\lab(\partial\Delta)\equiv W$ and $\Delta$ contains no $\theta$-annuli.

\end{lemma}

Similarly, the next statement is proved in exactly the same was as its analogue:

\begin{lemma}[Lemma 9.2 of \cite{W}] \label{a-bands between a-cells}

Every diminished diagram satisfies (M2).

\end{lemma}

Note that Lemma \ref{a-bands between a-cells} implies that a diminished diagram satisfying (M1) is minimal.

\medskip

\subsection{$t$-spokes} \

When considering diminished, minimal, or $D$-minimal diagrams in what follows, many arguments rely on the $q$-bands corresponding to $t$-letters. To distinguish these from bands corresponding to other parts of the base, we adopt the convention of previous literature ({\frenchspacing e.g. \cite{O18}, \cite{OS19}, \cite{W}, \cite{WMal}) and refer to them as \textit{$t$-bands}. Note that the top and the bottom of such a band are each labelled by a copy of the band's history.

In a diminished, minimal, or $D$-minimal diagram, a maximal $q$-band with one end on a disk $\Pi$ is called a \textit{spoke} of $\Pi$. A \textit{$t$-spoke} is then defined in the natural way.

The pairs $\{t(2),t(3)\},\dots,\{t(L-1),t(L)\},\{t(L),t(2)\}$ are called \textit{adjacent} $t$-letters. Two $t$-spokes of the same disk are called \textit{consecutive} if they correspond to adjacent $t$-letters.

\begin{lemma}[Lemma 9.3 of \cite{W}] \label{extend}

For $i\in\{2,\dots,L\}$, let $\pazocal{C}:A(i)\to\dots\to A(i)$ be a reduced computation of $\textbf{M}$ with history $H$. Then there exists a reduced diagram $\Delta$ over $M_\Omega(\textbf{M})$ with contour label $H(0)^{-1}W_{ac}H(0)W_{ac}^{-1}$, where $H(0)$ is the copy of $H$ in $F(R)$ obtained by adding the subscript 0 to each letter.

\end{lemma}

\begin{proof}

Consider the factorization $H\equiv H_1\cdots H_\ell$ for $\ell\geq2$ given by Lemma \ref{projected end to end}.

Define $H_i(0)$ as the word in $F(R)$ obtained from $H_i$ by adding a subscript $0$ to each letter. As in the proof of the analogous statement in \cite{W}, Lemma \ref{computations are trapezia} implies that for each $1\leq j\leq \ell$, there exists a trapezium $\Delta_j$ with contour label $$H_j(0)^{-1}W_{j-1}^{(z_j)}H_j(0)(W_j^{(z_j)})^{-1}$$

By construction, for $1\leq j\leq\ell-1$, $W_j^{(z_j)}$ differs from $W_j^{(z_{j+1})}$ only by the insertion/deletion of words in $\pazocal{R}_1$ in the `special' input sector, while $W_0^{(z_1)}\equiv W_\ell^{(z_\ell)}\equiv W_{ac}$. Note that every word of $\pazocal{R}_1$ represents the trivial element of $G$, so that $\pazocal{R}_1\subset\Omega$. For $1\leq j\leq\ell-1$, let $\tilde{\Delta}_j$ be the diagram obtained from pasting the $a$-cell corresponding to this element of $\pazocal{R}_1$ to the top of $\Delta_j$, so that the `top' label of $\tilde{\Delta}_j$ is $W_j^{(z_{j+1})}$.

Then, letting $\tilde{\Delta}_\ell=\Delta_\ell$, we may glue the top of $\tilde{\Delta}_j$ to the bottom of $\tilde{\Delta}_{j+1}$. Letting $\Delta$ be the reduced diagram that results from these pastings, $\text{Lab}(\partial\Delta)\equiv H(0)^{-1}W_{ac}H(0)W_{ac}^{-1}$.

\end{proof}

\begin{lemma}[Compare with Lemma 9.4 of \cite{W}] \label{extend 2}

Let $\pazocal{C}:V_0\to\dots\to V_t$ be a reduced computation of $\textbf{M}$ with history $H$ and base $\{t(i)\}B_4(i)$ for some $i\in\{2,\dots,L\}$. Suppose there exist configurations $W_0$ and $W_t$ of $\textbf{M}$ each of which is accepted by a one-machine computation such that $W_0(i)\equiv V_0$ and $W_t(i)\equiv V_t$. Then there exists a reduced diagram $\Delta$ over $M_\Omega(\textbf{M})$ with contour label $H(0)^{-1}W_0H(0)W_t^{-1}$, where $H(0)$ is the copy of $H$ in $F(R)$ obtained by adding the subscript 0 to each letter.

\end{lemma}

\begin{proof}

Let $\pazocal{C}_0$ and $\pazocal{C}_t$ be one-machine computations accepting $W_0$ and $W_t$, respectively.  Letting $H_0$ and $H_t$ be the histories of these computations, then there exists a reduced computation $\pazocal{D}:A(i)\to\dots\to A(i)$ whose history $H'$ is freely equal to the product $H_0^{-1}HH_t$.

\Cref{extend} then implies there exists a reduced diagram $\Delta'$ over $M_\Omega(\textbf{M})$ with contour label $H'(0)^{-1}W_{ac}H'(0)W_{ac}^{-1}$, where $H'(0)$ is the copy of $H'$ in $F(R)$ obtained by adding the subscript $0$ to each letter.

Now, let $\Delta_0$ and $\bar{\Delta}_t$ be the trapezia corresponding to the reduced computation $\pazocal{C}_0$ and the inverse $\bar{\pazocal{C}}_t$ of the computation $\pazocal{C}_t$, respectively.  Then by construction $\lab(\partial\Delta_0)\equiv H_0(0)^{-1}W_0H_0(0)W_{ac}^{-1}$ and $\lab(\partial\bar{\Delta}_t)\equiv H_t(0)W_{ac}H_t(0)^{-1}W_t^{-1}$.

So, pasting the top of $\Delta_0$ to the `bottom' of $\Delta'$ and the bottom of $\bar{\Delta}_t$ to the `top' of $\Delta'$, we obtain a (possibly unreduced) diagram $\Delta_1$ over $M_\Omega(\textbf{M})$ with $$\lab(\partial\Delta_1)\equiv H_t(0)H'(0)^{-1}H_0(0)^{-1}W_0H_0(0)H'(0)H_t(0)^{-1}W_t^{-1}$$
Thus, the reduced diagram $\Delta$ obtained from $\Delta_0$ by removing cancellable cells and reducing the boundary labels satisfies the statement.

\end{proof}

Using \Cref{extend 2}, the next statement follows in just the same way as the analogous statement in \cite{W}:

\begin{lemma}[Lemma 9.5 of \cite{W}] \label{t-spokes between disks}

Let $\Delta$ be a $D$-minimal diagram over the disk presentation of $G_\Omega(\textbf{M})$. Suppose there exist two disks $\Pi_1$ and $\Pi_2$ in $\Delta$ so that $\pazocal{Q}_1$ and $\pazocal{Q}_2$ are consecutive $t$-spokes of both. Let $\Psi$ be the subdiagram bounded by the sides of $\pazocal{Q}_i$ and the subpaths of $\partial\Pi_i$ such that $\Psi$ does not contain $\Pi_1$ or $\Pi_2$. Then $\Psi$ contains a disk.

\end{lemma}

\renewcommand\thesubfigure{\alph{subfigure}}
\begin{figure}[H]
\centering
\begin{subfigure}[b]{\textwidth}
\centering
\includegraphics[scale=1]{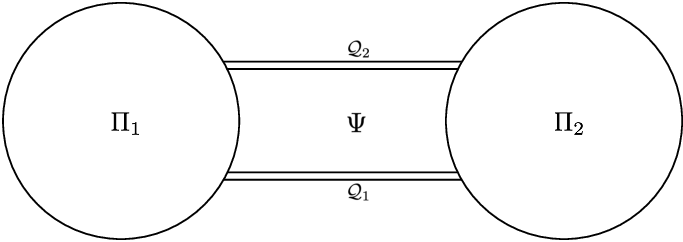}
\caption{Adjacent $t$-letters are $\{t(i),t(i+1)\}$ for $2\leq i\leq L-1$}
\end{subfigure} \\ \vspace{0.2in}
\begin{subfigure}[b]{\textwidth}
\centering
\includegraphics[scale=1]{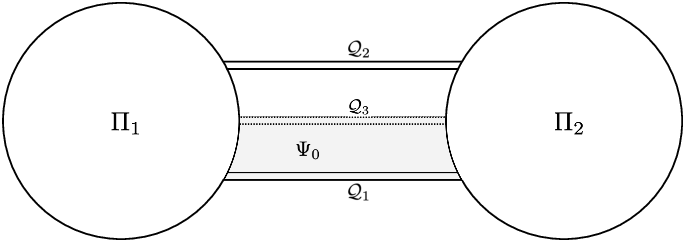}
\caption{Adjacent $t$-letters are $\{t(L),t(2)\}$           }
\end{subfigure}
\caption{Lemma \ref{t-spokes between disks}}
\end{figure}

Given a reduced diagram $\Delta$ over the disk presentation of $G_\Omega(\textbf{M})$, define the corresponding planar graph $\Gamma\equiv\Gamma(\Delta)$ as follows:

\begin{enumerate}[label=({\arabic*})]

\item $V(\Gamma)=\{v_0,v_1,\dots,v_\ell\}$ where each $v_i$ for $i\geq1$ corresponds to one of the $\ell$ disks of $\Delta$ and $v_0$ is one exterior vertex

\item For $i,j\geq1$, each shared $t$-spoke of the disks corresponding to $v_i$ and $v_j$ corresponds to an edge $(v_i,v_j)\in E(\Gamma)$

\item For $i\geq1$, each $t$-spoke of the disk corresponding to $v_i$ which ends on $\partial\Delta$ corresponds to an edge $(v_0,v_i)\in E(\Gamma)$

\end{enumerate}

As the degree of each interior vertex of $\Gamma$ is $L-1$, taking $L$ sufficiently large and applying \Cref{t-spokes between disks} implies the following statement, which possesses analogues in \cite{W}, \cite{WMal}, \cite{O18}, \cite{OS19}, and several others.

\begin{lemma}[Lemma 9.6 of \cite{W}] \label{graph}

Suppose $\Delta$ is a $D$-minimal diagram containing at least one disk. Then $\Delta$ contains a disk $\Pi$ such that $L-4$ consecutive $t$-spokes $\pazocal{Q}_1,\dots,\pazocal{Q}_{L-4}$ of $\Pi$ end on $\partial\Delta$ and such that every subdiagram $\Gamma_i$ bounded by $\pazocal{Q}_i$, $\pazocal{Q}_{i+1}$, $\partial\Pi$, and $\partial\Delta$ ($i=1,\dots,L-5$) contains no disks.

\end{lemma}

\begin{figure}[H]
\centering
\includegraphics[scale=1]{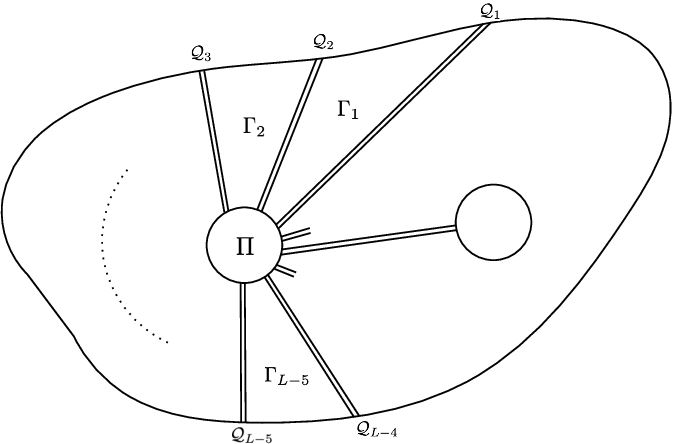}
\caption{Lemma \ref{graph}}
\label{fig-graph}
\end{figure}

%Applying induction on the number of hubs then implies:
%
%\begin{lemma} \label{number of q-edges}
%
%\textit{(Lemma 5.19 of [26])} Let $\Delta$ be a reduced diagram over the disk presentation of $G_a(\textbf{M})$ satisfying (M3). If $\Delta$ contains $\ell\geq1$ disks, then the number of spokes ending on $\partial\Delta$, and therefore the number of $q$-edges of $\partial\Delta$, is greater than $\ell LN/2$.
%
%\end{lemma}

\medskip

%%%%%%%%%%%%%%%%%%%%%%%%%%%%%%%%%%%%%%%%%%%%%%%%%%

\subsection{Transposition of a $\theta$-band and a disk} \label{sec-transposition} \

As in \cite{W}, we now describe a procedure for moving a $\theta$-band about a disk that is similar in spirit to the procedure for moving a $\theta$-band about an $a$-cell described in \Cref{sec-transposition-a}.  However, as there are differences in the way disk relations are defined in this setting, the approach outlined here takes some minor deviations.

Let $\Delta$ be a $D$-minimal diagram containing a disk $\Pi$ and a $\theta$-band $\pazocal{T}$ subsequently crossing the $t$-spokes $\pazocal{Q}_1,\dots,\pazocal{Q}_\ell$ of $\Pi$. Assume $\ell\geq2$ is maximal for $\Pi$ and $\pazocal{T}$.

First, suppose there are no other cells between $\Pi$ and the bottom of $\pazocal{T}$, i.e there is a subdiagram formed by $\Pi$ and $\pazocal{T}$.

Let $\pazocal{T}'$ be the subband of $\pazocal{T}$ whose bottom path, $\textbf{s}_1^{-1}$, starts with the $t$-edge corresponding to the start of $\pazocal{Q}_1$ and ends with that of $\pazocal{Q}_\ell$. Further, let $\textbf{s}_2$ be the complement of $\textbf{s}_1$ in $\partial\Pi$ so that $\partial\Pi=\textbf{s}_1\textbf{s}_2$. Then as any sector of the standard base containing a $t$-letter has empty tape alphabet, $\lab(\textbf{s}_2)$ is an admissible word.

Let $W\equiv\text{Lab}(\partial\Pi)^{\pm1}$, $V\equiv\text{Lab}(\textbf{s}_1)$, and $\theta$ be the rule corresponding to $\pazocal{T}$. Further, let $\Gamma$ be the subdiagram formed by $\Pi$ and $\pazocal{T}'$. Then, by Lemma \ref{trapezia are computations}, $V^{-1}$ is $\theta$-admissible with $V^{-1}\cdot\theta\equiv\lab(\textbf{ttop}(\pazocal{T}'))=\lab(\textbf{top}(\pazocal{T}'))$.  In particular, $W(j)$ is $\theta$-admissible for some $j\in\{2,\dots,L\}$.

Suppose $W$ is $\theta$-admissible and $W\cdot\theta$ is a disk relator.  Then as it is a subword of $W$, $\lab(\textbf{s}_2)$ is $\theta$-admissible.  Let $\bar{\Pi}$ be a disk with contour labelled by $W\cdot\theta$ and let $\pazocal{T}''$ be the auxiliary $\theta$-band corresponding to $\theta$ whose top is labelled by $\lab(\textbf{s}_2)\cdot\theta$. Then, let $\bar{\Gamma}$ be the diagram obtained from attaching $\pazocal{T}''$ to $\bar{\Pi}$. Finally, let $\bar{\Delta}$ be the reduced diagram obtained from excising $\Gamma$ from $\Delta$ and pasting $\bar{\Gamma}$ in its place, attaching the first and last cells of $\pazocal{T}''$ to the complement of $\pazocal{T}'$ in $\pazocal{T}$ and perhaps making cancellations in the resulting $\theta$-band. Note that $\bar{\Delta}$ has the same contour label as that of $\Delta$.

Next, suppose $W$ is not $\theta$-admissible.  Then since $W(j)$ is $\theta$-admissible, \Cref{projected not admissible}(1) implies $\theta=\theta(s)_2$ and $W\equiv I(w,H)$ for some $w\in\pazocal{R}_1$ and $H\in F(\Phi^+)$.  
%Moreover, since $V^{-1}$ is $\theta$-admissible, the admissible subword of $W$ corresponding to the `special' input sector is a subword of $\lab(\textbf{s}_2)^{\pm1}$.
In this case, we replace $\Pi$ with a subdiagram consisting of a disk $\Pi_0$ with boundary label $J(w,H)^{\pm1}$ and an $a$-cell $\pi$ with boundary label $w^{\pm1}$.  As the admissible subword of $W$ corresponding to the `special' input sector is a subword of $\lab(\textbf{s}_2)^{\pm1}$, $\textbf{bot}(\pazocal{T}')$ is still a subpath of $\partial\Pi_0^{-1}$, so that we may move $\pazocal{T}'$ about $\Pi_0$ as above, producing a diagram $\bar{\Delta}$ as above (but with one extra $a$-cell introduced).

%Assuming $W$ is accepted by a one-machine computation of the $j$-th machine, since $\lab(\textbf{s}_1)$ being $\theta$-admissible implies $W(i)$ is $\theta$-admissible for some $i\in\{2,\dots,L\}$, then \Cref{one-machine accepted is admissible} implies $\theta\notin\Theta_j$.  So, $W$ must be either an accepted start or end configuration.  

Finally, suppose $W$ is $\theta$-admissible but $\ell(W\cdot\theta)=2$.  Then \Cref{projected not admissible}(2) implies $\theta=\theta(s)_1$ and $W\equiv J(w,H)$ for some $w\in\pazocal{R}_1$ and $H\in F(\Phi^+)$.  But then we may introduce an $a$-cell as above to move $\pazocal{T}$ about $\Pi$ and construct $\bar{\Delta}$.

%Then Lemma \ref{projection admissible configuration not} applies to $W$, so that $\lab(\textbf{s}_2)$ contains the `special' input sector and would be $\theta$-admissible with the insertion/deletion of some $u^n\in\pazocal{L}$. So, after attaching to $\Pi$ an $a$-cell corresponding to $u^n$, we may construct the disk $\bar{\Pi}$ and the auxiliary $\theta$-band $\pazocal{T}''$ as above. Attaching the mirror $a$-cell on the other side of $\pazocal{T}''$ then produces a diagram $\bar{\Delta}$ with the same contour label as $\Delta$.

The procedure of excising $\Gamma$ from $\Delta$ to create $\bar{\Delta}$ is called the \textit{transposition} of the disk $\Pi$ and the $\theta$-band $\pazocal{T}$ in $\Delta$.

\renewcommand\thesubfigure{\alph{subfigure}}
\begin{figure}[H]
\centering
\begin{subfigure}[b]{0.48\textwidth}
\centering
\raisebox{0.2in}{\includegraphics[width=3in]{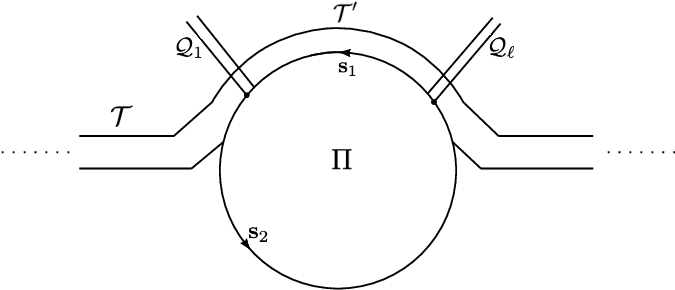}}
\caption{The subdiagram $\Gamma$}
\end{subfigure}\hfill
\begin{subfigure}[b]{0.48\textwidth}
\centering
\includegraphics[width=3in]{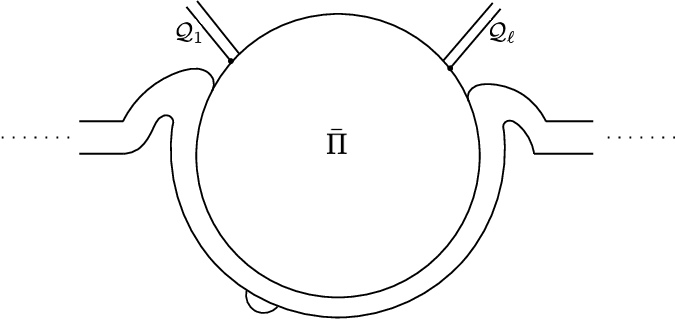}
\caption{The resulting subdiagram $\bar{\Gamma}$}
\end{subfigure}
\caption{The transposition of a $\theta$-band with a disk}
\end{figure}

Now, consider the situation where there are cells between the $\theta$-band and the disk, each of which is an $a$-cell.

Suppose the pair of adjacent $t$-letters corresponding to $\pazocal{Q}_1$ and $\pazocal{Q}_2$ is $\{t(i),t(i+1)\}$ for some $2\leq i\leq L-1$. Let $\pazocal{T}_1'$ be the subband of $\pazocal{T}'$ between $\pazocal{Q}_1$ and $\pazocal{Q}_2$. Then, let $\Psi$ be the subdiagram of $\Delta$ bounded by $\pazocal{T}_1'$ and $\partial\Pi$. By Lemma \ref{diminished exist}, there exists a diminished diagram $\Lambda$ over $M_\Omega(\textbf{M})$ with $\lab(\partial\Psi)\equiv\lab(\partial\Lambda)$. Lemmas \ref{a-band on same a-cell} and \ref{a-bands between a-cells} then imply that $\Lambda$ contains no $a$-cell, so that Lemma \ref{M(S) annuli} implies that $\Lambda$ consists of a single $\theta$-band. Hence, by Lemma \ref{trapezia are computations}, $W(i)$ is $\theta$-admissible.

Otherwise, if the pair of adjacent $t$-letters is $\{t(L),t(2)\}$, then the same argument applies to the subdiagram bounded by the $t$-band corresponding to $t(L)$, the $q$-band corresponding to $t(1)$, and $\pazocal{T}'$. As a result, $W(L)$ is $\theta$-admissible.

Performing this procedure to every corresponding subband of $\pazocal{T}'$, we may then assume that any $a$-cells sit below a subband with band corresponding to the `special' input sector.  As above, by \Cref{projected not admissible} perhaps replacing $\Pi$ with a subdiagram consisting of one disk and one $a$-cell, we may assume $W$ is $\theta$-admissible.

In this case, $V^{-1}\cdot\theta$ can differ from $\lab(\textbf{ttop}(\pazocal{T}')$ only by the tape word in the `special' input sector.  Hence, we may perform an analogous procedure as above to produce the transposition $\bar{\Delta}$.

%As above, by \Cref{projected not admissible} perhaps replacing $\Pi$ with a subdiagram consisting of a  we may construct a new disk and auxiliary band that, perhaps after attaching another $a$-cell, functions as the transposition of $\Pi$ with $\pazocal{T}$.

Note that the reduced diagram $\bar{\Delta}$ arising from the transposition has the same number of disks and contour label as $\Delta$, and so is $D$-minimal.

However, the minimality of the 2-signature (and so the signature) need not be preserved by a transposition. This is because many $(\theta,t)$-cells may be added through transposition.

%Note that the definition of transposition above differs from that in [18] and [25] only by the presence of $a$-cells.

Even so, the following statement follows just as its analogue in \cite{W}:

\begin{lemma}[Lemma 9.7 of \cite{W}] \label{G_a theta-annuli}

Let $\Delta$ be a reduced diagram over the disk presentation of $G_\Omega(\textbf{M})$ satisfying (M3).

\begin{enumerate}[label=({\arabic*})]

\item Suppose a $\theta$-band $\pazocal{T}$ crosses $\ell$ $t$-spokes of a disk $\Pi$ and there are no disks in the subdiagram bounded by these spokes, $\pazocal{T}$, and $\partial\Pi$. Then $\ell\leq(L-1)/2$.

\item Suppose $\pazocal{T}$ and $\pazocal{T}'$ are disjoint $\theta$-bands crossing $\ell$ and $\ell'$ $t$-spokes, respectively, of a disk $\Pi$. Suppose further that every cell between the bottom of $\pazocal{T}$ (of $\pazocal{T}'$) and $\Pi$ is an $a$-cell. Further, suppose these bands correspond to the same rule $\theta$ if the history is read toward the disk. Then $\ell+\ell'\leq(L-1)/2$.

\item If $S$ is a $\theta$-annulus in $\Delta$ and $\Delta_S$ is the subdiagram bounded by the outer contour of $S$, then $\Delta_S$ is a diagram over $M_\Omega(\textbf{M})$.

\end{enumerate}

\end{lemma}

%\begin{figure}[H]
%\centering
%\includegraphics[scale=1.5]{twothetas.eps}
%\caption{Lemma \ref{G_a theta-annuli}(2)}
%\end{figure}

%\begin{proof}
%
%(1) Lemma \ref{M_a no annuli 1} implies that there exists a $\theta$-band $\pazocal{T}_0$ crossing all $\ell$ spokes such that the only cells between it and $\Pi$ are $a$-cells. If $\ell>(L-1)/2$, then the transposition of $\Pi$ and $\pazocal{T}_0$ in $\Delta$ then yields a diagram with the same contour label, the same number of disks, and strictly less $(\theta,t)$-cells. This contradicts the minimality of $s_2(\Delta)$.
%
%(2) The transposition of $\pazocal{T}$ and $\Pi$ removes $\ell$ $(\theta,t)$-cells and adds $(L-1)-\ell$ new $(\theta,t)$-cells in the resulting band. However, $\ell'$ of these cells form cancellable pairs with cells of $\pazocal{T}'$, so that it is possible to cancel $2\ell'$ cells. Hence, the change in the number of $(\theta,t)$-cells is $(L-1)-2\ell-2\ell'$, so that the relation $\ell+\ell'>(L-1)/2$ would contradict the minimality of $s_2(\Delta)$.
%
%(3) Suppose $\Delta_S$ contains a disk. Then, since $\Delta$ is $D$-minimal, Lemma \ref{graph} gives a disk $\Pi$ in $\Delta_S$ with $L-4$ consecutive $t$-spokes that end on $\partial\Delta_S$ and such that the subdiagram of $\Delta_S$ bounded by these spokes contains no disks. But then taking $L>7$, $S$ and $\Pi$ contradict (1).
%
%\end{proof}

The following is an immediate consequence of Lemmas \ref{M_a no annuli 2}(2) and \ref{G_a theta-annuli}(3).

\begin{lemma} \label{minimal theta-annuli}

A minimal diagram $\Delta$ contains no $\theta$-annuli.

\end{lemma}

Thus, the following strengthened version of van Kampen's Lemma for minimal diagrams over $G_\Omega(\textbf{M})$ follows in the same way as in \cite{W}:

\begin{lemma}[Lemma 9.9 of \cite{W}] \label{minimal exist}

A word $W$ over $\pazocal{Y}$ represents the trivial element of $G_\Omega(\textbf{M})$ if and only if there exists a minimal diagram $\Delta$ such that $\lab(\partial\Delta)\equiv W$.

\end{lemma}

Lemma \ref{minimal exist} immedately implies the following strengthened version of van Kampen's Lemma for $M$-minimal diagrams.

\begin{lemma} \label{M-minimal exist}

A word $W$ over $\pazocal{Y}$ represents the trivial element of $M_\Omega(\textbf{M})$ if and only if there exists an $M$-minimal diagram $\Delta$ such that $\lab(\partial\Delta)\equiv W$.

\end{lemma}

\medskip

%%%%%%%%%%%%%%%%%%%%%%%%%%%%%%%%%%%%%%%%%%%%%%%%%%

\subsection{Quasi-trapezia} \

Next, the concept of trapezium is generalized to the setting of minimal diagrams over $G_\Omega(\textbf{M})$. 

A \textit{quasi-trapezium} is a minimal diagram defined in much the same way as an $a$-trapezium except that it is permitted to contain disks. In other words, a quasi-trapezium is a minimal diagram whose boundary can be factored as $\textbf{p}_1^{-1}\textbf{q}_1\textbf{p}_2\textbf{q}_2^{-1}$, where each $\textbf{p}_i$ is the side of a $q$-band and each $\textbf{q}_i$ is the maximal subpath of the side of a $\theta$-band where the subpath starts and ends with a $q$-edge.

The \textit{(step) history} of a quasi-trapezium is defined in the same way as for an $a$-trapezium, as are the \textit{base}, the \textit{height}, and the \textit{standard factorization}.

Note that a quasi-trapezium containing no disks is an $a$-trapezium, while one without any disks or $a$-cells is a trapezium. 

Indeed, an $a$-trapezium is necessarily a quasi-trapezium. To see that an $a$-trapezium satisfies (M3), note that Lemmas \ref{M_a no annuli 1} and \ref{M_a no annuli 2} imply that in any minimal diagram with the same contour label, any maximal $\theta$-band must cross each maximal $q$-band exactly once.

With the tool of transposition developed in the previous section, the next statement follows in just the same way as its analogue in \cite{W}:

\begin{lemma}[Lemma 9.13 of \cite{W}] \label{quasi-trapezia}

Let $\Gamma$ be a quasi-trapezium with standard factorization of its contour $\textbf{p}_1^{-1}\textbf{q}_1\textbf{p}_2\textbf{q}_2^{-1}$. Then there exists a reduced diagram $\Gamma'$ such that:

\begin{enumerate}[label=({\arabic*})]

\item $\partial\Gamma'=(\textbf{p}_1')^{-1}\textbf{q}_1'\textbf{p}_2'(\textbf{q}_2')^{-1}$, where $\lab(\textbf{p}_j')\equiv\lab(\textbf{p}_j)$ and $\lab(\textbf{q}_j')\equiv\lab(\textbf{q}_j)$ for $j=1,2$

\item the number of disks 
%and $(\theta,q)$-cells 
in $\Gamma'$ is the same as the number of disks in $\Gamma$

\item there exists a simple path $\textbf{s}_1$ (respectively $\textbf{s}_2$) connecting the vertices $(\textbf{p}_1')_-$ and $(\textbf{p}_2')_-$ (respectively $(\textbf{p}_1')_+$ and $(\textbf{p}_2')_+$) such that

\begin{enumerate}

\item $(\textbf{p}_1')^{-1}\textbf{s}_1\textbf{p}_2'\textbf{s}_2^{-1}$ is the standard factorization of the boundary of an $a$-trapezium $\Gamma_2$ and

\item any cell above $\textbf{s}_2$ or below $\textbf{s}_1$ is a disk or an $a$-cell

\end{enumerate}

\item there exists $m\in\N$ such that any maximal $\theta$-band of $\Gamma$ contains $m$ $(\theta,t)$-cells and any maximal $\theta$-band of $\Gamma_2$ contains $m$ $(\theta,t)$-cells.

\end{enumerate}

\end{lemma}

\begin{remark}

Note that the proof of the statement analogous to \Cref{quasi-trapezia} in \cite{W} relies on a preliminary statement (namely Lemma 9.12 in that setting).  However, there is a verbatim analogue of that statement in this setting which may be proved with an identical approach.

\end{remark}

\begin{remark}

As in \cite{W}, the concept of $D$-minimal diagram exists specifically for the treatment of \Cref{quasi-trapezia}, as $D$-minimality is both sufficient as a hypothesis for Lemma \ref{graph} and preserved under transposition (whereas, for example, (M3) satisfies the first condition but not the second).

\end{remark}

\medskip

%%%%%%%%%%%%%%%%%%%%%%%%%%%%%%%%%%%%%%%%%%%%%%%%%%%%%%%%%%%%%%%%%

\subsection{Shafts and Designs} \

We now introduce a concept that, as it was in \cite{O18}, \cite{OS19}, and \cite{W}, will be used to define a valuable measure on minimal diagrams.

%A trapezium $\Delta$ over the canonical presentation of $M(\textbf{M})$ is called \textit{standard} if its base is the standard base (or its inverse) and its history $H$ contains a subword $H_0^{\pm1}$ for some controlled history $H_0$. Naturally, a history $H$ (i.e $H\in F(\Theta^+)$) is called \textit{standard} if there exists a standard trapezium with history $H$. By Lemmas \ref{M controlled} and \ref{computations are trapezia}, a standard trapezium is uniquely determined by its history.

Let $\Pi$ be a disk contained in a minimal diagram and $\pazocal{B}$ be a $t$-spoke of $\Pi$. Suppose there is a subband $\pazocal{C}$ of $\pazocal{B}$ starting on $\Pi$ whose history $H$ contains a controlled subword. For $W$ the configuration corresponding to $\lab(\partial\Pi)$, suppose $W(i)$ is $H$-admissible for $i\geq2$. Then the $t$-band $\pazocal{C}$ is called a \textit{shaft} of $\Pi$.

Note that, as in \cite{W}, this definition differs from that used in \cite{O18} and \cite{OS19}, where it was required that the entire configuration $W$ be $H$-admissible.  This modification allows for `flexibility' in the `special' input sector, as $W(1)$ need not be $H$-admissible.

For a disk $\Pi$, a shaft $\pazocal{C}$ of $\Pi$ is called a \textit{$\lambda$-shaft} of $\Pi$ if for every factorization $H\equiv H_1H_2H_3$ satisfying $\|H_1\|+\|H_3\|\leq\lambda\|H\|$, $H_2$ contains a controlled subword. Note that a shaft is a 0-shaft.

The next statement then follows in just the same way as its analogue in \cite{W}.

\begin{lemma}[Lemma 9.15 of \cite{W}] \label{shafts}

Let $\Pi$ be a disk in a minimal diagram $\Delta$ and $\pazocal{C}$ be a $\lambda$-shaft at $\Pi$ with history $H$. Then $\pazocal{C}$ has no factorization $\pazocal{C}=\pazocal{C}_1\pazocal{C}_2\pazocal{C}_3$ such that

\begin{enumerate}[label=({\arabic*})]

\item the sum of the lengths of $\pazocal{C}_1$ and $\pazocal{C}_3$ do not exceed $\lambda\|H\|$ and

\item $\Delta$ contains a quasi-trapezium $\Gamma$ such that the bottom (or top) of $\Gamma$ has $L$ $t$-edges and $\pazocal{C}_2$ starts on the bottom and ends on the top of $\Gamma$.

\end{enumerate}

\end{lemma}

Shafts are particularly useful for studying minimal diagrams through the measure introduced for such purpose in \cite{O18}.

Let $\pazocal{D}$ be a disk in the Euclidean plane, $\textbf{T}$ be a finite set of disjoint chords, and $\textbf{Q}$ be a finite set of disjoint simple curves in $\pazocal{D}$, called \textit{arcs} (as to differentiate them from the chords). 

Assume that arcs belong to the open disk $\pazocal{D}^\circ$ and that each chord crosses any arc transversely and at most one, with the intersection not coming at either of the arc's endpoints.

With these assumptions, the pair $(\textbf{T},\textbf{Q})$ is called a \textit{design} on the disk.

The length of an arc $C\in\textbf{Q}$, denoted $|C|$, is the number of chords crossing it. \textit{Subarcs} are defined in the natural way, so that the inequality $|D|\leq|C|$ is clear for $D$ a subarc of $C$.

An arc $C_1$ is \textit{parallel} to an arc $C_2$, denoted $C_1 \ \| \ C_2$, if every chord crossing $C_1$ also crosses $C_2$. Note that this relation is reflexive and transitive, but not symmetric.

For the parameter $\lambda\in(0,1/2)$ (see \Cref{sec-parameters}) and $m$ a positive integer, a design $(\textbf{T},\textbf{Q})$ is said to satisfy property $P(\lambda,m)$ if for any collection of $m$ distinct arcs $C_1,\dots,C_m\in\textbf{Q}$, there are no subarcs $D_1,\dots,D_m$, respectively, such that $|D_i|>(1-\lambda)|C_i|$ for all $i$ and $D_1 \ \| \ D_2 \ \| \dots \| \ D_m$.

For a design $(\textbf{T},\textbf{Q})$, define the length of $\textbf{Q}$, $\ell(\textbf{Q})$, to be $\ell(\textbf{Q})=\sum\limits_{C\in\textbf{Q}}|C|$.

\begin{lemma}[Lemma 8.2 of \cite{O18}] \label{design}

There is a constant $c$ dependant on $\lambda$ and $m$ such that for any design $(\textbf{T},\textbf{Q})$ satisfying property $P(\lambda,m)$, $\ell(\textbf{Q})\leq c(\#\textbf{T})$.

\end{lemma}

Let $\Delta$ be a minimal diagram and $\pazocal{Q}$ be a $t$-spoke of a disk $\Pi$ in $\Delta$. Let $\pazocal{Q}_\Pi$ be the subband of $\pazocal{Q}$ which is a $\lambda$-shaft at $\Pi$ of maximal length. Then, define $\sigma_\lambda(\Delta)$ as the sum of the lengths of the $\lambda$-shafts $\pazocal{Q}_\Pi$ for all disks $\Pi$ and $t$-spokes $\pazocal{Q}$.

If $\Delta$ is a minimal diagram, then identify $\Delta$ with a disk and construct the design $(\textbf{T},\textbf{Q})$ as follows: Let the middle lines of maximal $\theta$-bands be the chords and the middle lines of maximal $\lambda$-shafts be the arcs. 

Note that there is a subtle hindrance to this construction: If a maximal $t$-spoke connects two disks, then it may contain a $\lambda$-shaft at each disk, and these $\lambda$-shafts may overlap. However, this issue can be remedied simply by `making room' in the spoke for both arcs to fit and be disjoint. 

Note that the length $|C|$ of an arc with respect to this design is the number of cells in the $\lambda$-shaft and $\#\textbf{T}=\frac{1}{2}|\partial\Delta|_\theta\leq\frac{1}{2}|\partial\Delta|$ since every maximal $\theta$-band ends twice on $\partial\Delta$.

The next statement then follows in the same way as its analogue in \cite{W}.

\begin{lemma}[Lemma 9.17 of \cite{W}, adapted from Lemma 8.5 of \cite{O18}] \label{G_a design}

If $\Delta$ is a minimal diagram, then $\sigma_\lambda(\Delta)\leq C_1|\partial\Delta|_\theta\leq C_1|\partial\Delta|$.

\end{lemma}

%\begin{proof}
%
%By Lemma \ref{design} and the parameter choices $C_1>>L>>\lambda^{-1}$, it suffices to prove that the design $(\textbf{T},\textbf{Q})$ satisfies Property $P(\lambda,2L-1)$.
%
%Arguing toward contradiction, there are $2L-1$ maximal $\lambda$-shafts $\pazocal{C}_1,\dots,\pazocal{C}_{2L-1}$ such that for some subband $\pazocal{D}$ of $\pazocal{C}_1$, $|\pazocal{D}|>(1-\lambda)|\pazocal{C}_1|$ and every maximal $\theta$-band crossing $\pazocal{D}$ also crosses each of $\pazocal{C}_2,\dots,\pazocal{C}_{2L-1}$. So, since at most two of these $\lambda$-shafts correspond to any particular $t$-spoke, each of the $\theta$-bands crossing $\pazocal{D}$ crosses at least $L$ $t$-bands. 
%
%But then the $\lambda$-shaft $\pazocal{C}_1$ crosses a quasi-trapezium of height $|\pazocal{D}|>(1-\lambda)|\pazocal{C}_1|$ whose base has at least $L$ $t$-letters, contradicting Lemma \ref{shafts}.
%
%\end{proof}

%%%%%%%%%%%%%%%%%%%%%%%%%%%%%%%%%%%%%%%%%%%%%%%%%%

\section{Upper bound on the weight of minimal diagrams} \label{sec-disks-upper-bound}

\subsection{Weakly minimal diagrams} \

\begin{figure}[H]
\centering
\includegraphics[scale=0.75]{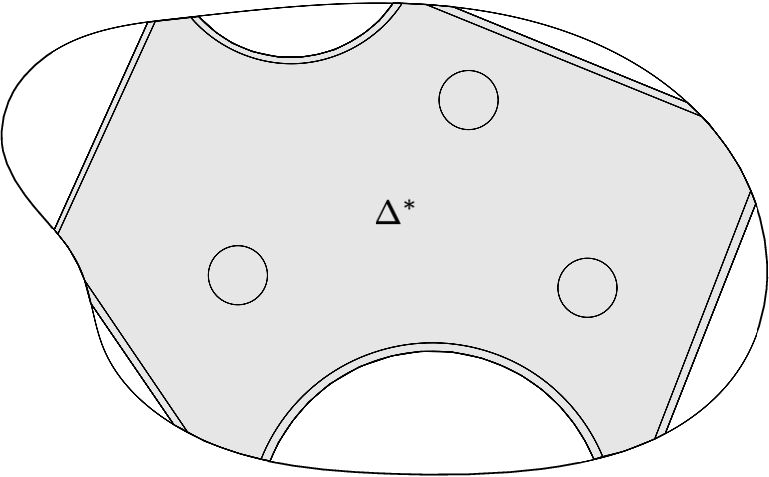}
\caption{The stem of a reduced diagram $\Delta$ containing disks}
\end{figure}

The goal in this section is to bound the $G$-weight of all minimal diagrams $\Delta$. In light of Lemma \ref{diskless}, it suffices to restrict our attention to minimal diagrams containing disks. However, as in \cite{W}, it proves necessary to consider a larger class of diagrams over the disk presentation of $G_\Omega(\textbf{M})$, called `weakly minimal'.

Let $\Delta$ be a reduced diagram over the disk presentation of $G_\Omega(\textbf{M})$ which contains a disk. Then, let $\pazocal{C}$ be a cutting $q$-band of $\Delta$, i.e $\pazocal{C}$ ends twice on the boundary of $\Delta$. Then $\pazocal{C}$ is called a \textit{stem band} if it is either a rim band of $\Delta$ or both components of $\Delta\setminus\pazocal{C}$ contain disks. The unique maximal subdiagram of $\Delta$ satisfying the property that every cutting $q$-band is a stem band is called the \textit{stem} of $\Delta$ and denoted $\Delta^*$.

If $\pazocal{C}$ is a cutting $q$-band that is not a stem band, then exactly one component $\Gamma$ of $\Delta\setminus\pazocal{C}$ contains no disks. In this situation, the cells of $\Gamma$ are called \textit{crown} cells. Note that one can construct $\Delta^*$ from $\Delta$ simply by cutting off all of the crown cells.

Finally, a reduced diagram $\Delta$ over the disk presentation of $G_\Omega(\textbf{M})$ which contains a disk is called \textit{weakly minimal} if:

\begin{addmargin}[1em]{0em}

\begin{enumerate}[label=(WM{\arabic*})]

%\item it contains no $\theta$-annulus $S$ whose sides are labelled by letters of the tape alphabet of the `special' input sector,

\item for any $a$-cell $\pi$ and any $\theta$-band $\pazocal{T}$, at most half of the edges of $\partial\pi$ mark the start of an $a$-band that crosses $\pazocal{T}$,

\item no maximal $a$-band ends on two different $a$-cells, and 

\item its stem $\Delta^*$ is a minimal diagram. 

\end{enumerate}

\end{addmargin}

Note that conditions (WM1) and (WM2) are identical to conditions (MM1) and (MM2) in the definition of $M$-minimal. As a result, any subdiagram of a weakly minimal diagram which contains no disks is $M$-minimal. 

Conversely, any minimal diagram containing a disk is weakly minimal.

\begin{lemma}[Lemma 10.1 of \cite{W}] \label{weakly minimal} \

\begin{enumerate}[label=({\alph*})]

\item If $\Delta_1$ is a subdiagram of a weakly minimal diagram $\Delta$ and contains a disk, then $\Delta_1$ is weakly minimal, $\Delta_1^*\subset\Delta^*$, and $\sigma_\lambda(\Delta_1^*)\leq\sigma_\lambda(\Delta^*)$.

\item For every weakly minimal diagram $\Delta$, $\sigma_\lambda(\Delta^*)\leq C_1|\partial\Delta|$.

\item A weakly minimal diagram $\Delta$ contains no $\theta$-annuli.

\item Let $\pazocal{C}$ be a cutting $q$-band of a reduced diagram $\Delta$ over the disk presentation of $G_\Omega(\textbf{M})$ and let $\Delta_1$, $\Delta_2$ be the components of $\Delta\setminus\pazocal{C}$. Suppose $\Delta_1\cup\pazocal{C}$ is $M$-minimal (over $M_\Omega(\textbf{M})$) and $\Delta_2\cup\pazocal{C}$ is weakly minimal. Then $\Delta$ is weakly minimal.

\end{enumerate}

\end{lemma}

\medskip

%%%%%%%%%%%%%%%%%%%%%%%%%%%%%%%%%%%%%%%%%%%%%%%%%%%%%%%%%%%%%%%%%

\subsection{Definition of the minimal counterexample and cloves} \

The objective of the rest of this section is to exhibit an upper bound for the $G$-weight of a weakly minimal diagram in terms of its perimeter.  In particular, for any weakly minimal diagram $\Gamma$ letting $p(\Gamma)=|\partial\Gamma|+\sigma_\lambda(\Gamma^*)$, we will prove the inequality $$\text{wt}_G(\Gamma)\leq \phi_{N_4}(p(\Gamma))+N_3\mu(\Gamma)f_2(N_3p(\Gamma))$$ holds for large enough choices of the parameters $N_4$ and $N_3$. 

The proof of this bound follows in much the same way as that presented in Section 10 \cite{W}, except:

\begin{itemize}

\item There are some minor changes in the parameters, arising from the differences in the constructions of the main machines. 

\item In the setting of \cite{W}, the proof proceeded for the specific function $f_2(x)=1$ (and so $\phi_{N_4}(x)=x^2$).

\end{itemize} 

With that said, many of statements are proved in almost identical ways to their analogues in \cite{W}.  As such, we omit the proofs of such statements and present only those which require some more technical modifications.

%As such, many of the proofs of the statements that follow will either be omitted or amount to a remark on how it differs from the corresponding proof in [25].

Let $\Delta$ be a `minimal counterexample' diagram with respect to $p(\Delta)$, i.e a weakly minimal diagram satisfying
$$\text{wt}_G(\Delta)>\phi_{N_4}(p(\Delta))+N_3\mu(\Delta)f_2(N_3p(\Delta))$$ 
while for any weakly minimal diagram $\Gamma$ such that $p(\Gamma)<p(\Delta)$, we have
$$\text{wt}_G(\Gamma)\leq\phi_{N_4}(p(\Gamma))+N_3\mu(\Gamma)f_2(N_3p(\Gamma))$$ 
As with Lemma \ref{a-cell in counterexample} (and proved in an analogous way), the following statement is an immediate consequence of the inductive hypothesis.

\begin{lemma} \label{a-cell in counterexample 2}

Let $\pi$ be an $a$-cell contained in $\Delta$. Suppose $\partial\pi$ has a subpath $\textbf{s}$ shared with $\partial\Delta$. Then $\|\textbf{s}\|\leq\frac{2}{3}\|\partial\pi\|$.

\end{lemma}

Since $\Delta^*$ contains every disk of $\Delta$ and is minimal, $\Delta$ is a $D$-minimal diagram. So, Lemma \ref{graph} guarantees that it contains a disk $\Pi$ with $L-4$ consecutive $t$-spokes $\pazocal{Q}_1,\dots,\pazocal{Q}_{L-4}$ ending on $\partial\Delta$ and bounding $L-5$ diskless subdiagrams (see \Cref{fig-graph}).

For $1\leq i<j\leq L-4$, the subdiagram of $\Delta$ bounded by $\partial\Pi$, $\pazocal{Q}_i$, and $\pazocal{Q}_j$ (and not containing $\Pi$) is called a \textit{clove} and is denoted $\Psi_{ij}$. The maximal clove $\Psi_{1,L-4}$ is simply denoted $\Psi$.

\begin{lemma} [Compare to Lemma 10.3 of \cite{W}] \label{rim theta-bands}

Let $\pazocal{T}$ be a quasi-rim $\theta$-band in $\Delta$. Then the base of $\pazocal{T}$ has length $s>K$.

\end{lemma}

\begin{proof}

Assume toward contradiction that $\pazocal{T}$ is a quasi-rim $\theta$-band with base of length $s\leq K$. Define $\Delta'$ and $\Delta''$ as in the proof of Lemma \ref{6.18}. As in that setting, $\Delta''$ satisfies (WM1) and (WM2) and $|\partial\Delta''|\leq|\partial\Delta|-1$.

Since $\Delta'$ is a subdiagram of $\Delta$, it follows from Lemma \ref{weakly minimal}(a) that it is weakly minimal with $\sigma_\lambda((\Delta')^*)\leq\sigma_\lambda(\Delta^*)$. 

Since the diagram $\Delta''$ is formed from $\Delta'$ through the addition of $a$-cells, the 2-signatures of $(\Delta')^*$ and $(\Delta'')^*$ are equal. Hence, $\Delta''$ is a weakly minimal diagram.

Further, every $\lambda$-shaft of $(\Delta'')^*$ is at most as long as the corresponding $\lambda$-shaft of $\Delta^*$, so that $\sigma_\lambda((\Delta'')^*)\leq\sigma_\lambda(\Delta^*)$. Consequently, $p(\Delta'')\leq p(\Delta)-1$, and so the inductive hypothesis may be applied to $\Delta''$.

The proof of Lemma \ref{6.18} then adapts naturally to this setting, providing a contradiction.

\end{proof}

\medskip

%%%%%%%%%%%%%%%%%%%%%%%%%%%%%%%%%%%%%%%%%%%%%%%%%%

\subsection{Properties of the cloves of $\Delta$} \

The following statement is an adaptation of Lemma \ref{6.16} to this setting and is proved in exactly the same way.

\begin{lemma}[Lemma 10.4 of \cite{W}] \label{7.20} \

\begin{enumerate}[label=({\arabic*})]

\item $\Delta$ has no two disjoint subcombs $\Gamma_1$ and $\Gamma_2$ contained in $\Psi$ with basic widths at most $K$ and handles $\pazocal{B}_1$ and $\pazocal{B}_2$ such that some ends of these handles are connected by a subpath $\textbf{x}$ of $\partial\Delta$ with $|\textbf{x}|_q\leq c_0$.

\item If $\Gamma$ is a subcomb of $\Delta$ contained in $\Psi$ with basic width $s\leq K$, $|\partial\Gamma|_q=2s$.

\end{enumerate}

\end{lemma}

%As such, the next statement follows in the same way as its analogue in \cite{W}.

\begin{lemma}[Lemma 10.5 of \cite{W}] \label{7.21}

The basic width of a subcomb of $\Delta$ contained in $\Psi$ is at most $K_0$.

\end{lemma}

%\begin{proof}
%
%Assume toward contradiction that there exists a subcomb of $\Delta$ contained in $\Psi$ with basic width $s>K_0$. Then, using Lemma \ref{rim theta-bands}, an identical proof to the one presented in Lemma \ref{tight subcomb} implies that there exists a tight subcomb $\Gamma$ of $\Delta$ contained in $\Psi$.
%
%Further, an analogous proof to that presented in Lemma \ref{6.17} implies that any subcomb of $\Gamma$ has height greater than $\ell'/2$. Indeed, other than switching the parameters and using Lemma \ref{7.20} in place of Lemma \ref{6.18}, the only necessary alteration to the proof of Lemma \ref{6.17} is in the application of the inductive hypothesis, where we must use the inequality $\sigma_\lambda((\Delta')^*)\leq\sigma_\lambda(\Delta^*)$ arising from Lemma \ref{weakly minimal}(a).
%
%But then similar analogues of Lemmas \ref{counterexample combs}-\ref{diskless} yield a contradiction in the same way. Only one major alteration is needed: In the adaptation of Lemma \ref{diskless}, the diagram $\Delta_0$ is weakly minimal by Lemma \ref{weakly minimal}(d) and satisfies $\sigma_\lambda(\Delta_0^*)=\sigma_\lambda(\Delta^*)$ since the handle of the tight subcomb $\Gamma$ is a non-stem cutting $q$-band.
%
%\end{proof}

\begin{remark}

Note that the proof of \Cref{7.21} proceeds by applying an adaptation of \Cref{diskless} to a subcomb with basic width greater than $K_0$.  As in \cite{W}, the reason we consider weakly minimal diagrams lies in this proof, as the adaptation of Lemma \ref{diskless} fails for minimal diagrams since there is no analogue of Lemma \ref{weakly minimal}(d) for minimal diagrams.

\end{remark}

%Using \Cref{7.21}, the next statement follows in just the same way as its analogue in \cite{W}.

\begin{lemma}[Lemma 10.7 of \cite{W}] \label{7.22} \

\begin{enumerate}[label=({\arabic*})]

\item Every maximal $\theta$-band of $\Psi$ crosses either $\pazocal{Q}_1$ or $\pazocal{Q}_{L-4}$

\item There exists an $r$ satisfying $(L-1)/2-3\leq r\leq (L-1)/2$ such that the $\theta$-bands of $\Psi$ crossing $\pazocal{Q}_{L-4}$ do not cross $\pazocal{Q}_r$ and the $\theta$-bands of $\Psi$ crossing $\pazocal{Q}_1$ do not cross $\pazocal{Q}_{r+1}$

\end{enumerate}

\end{lemma}

\medskip

%%%%%%%%%%%%%%%%%%%%%%%%%%%%%%%%%%%%%%%%%%%%%%%%%%

\subsection{Paths in the cloves} \

For $1\leq i<j\leq L-4$, denote $\textbf{p}_{ij}$ as the shared subpath of $\partial\Psi_{ij}$ and $\partial\Delta$. For simplicity, denote the path $\textbf{p}_{1,L-4}$ associated to the maximal clove simply as $\textbf{p}$.

Let $\bar{\Delta}$ be the subdiagram of $\Delta$ consisting of $\Pi$ and $\Psi$. Then, let $\bar{\textbf{p}}=\textbf{bot}(\pazocal{Q}_1)^{-1}\textbf{u}^{-1}\textbf{top}(\pazocal{Q}_{L-4})$ where $\textbf{u}$ is a subpath of $\partial\Pi$ and such that cutting along $\bar{\textbf{p}}$ separates $\Delta$ into two components, one of which is $\bar{\Delta}$. Denote the other component $\Psi'$.

Similarly, for $1\leq i<j\leq L-4$, define the the path $\bar{\textbf{p}}_{ij}=\textbf{bot}(\pazocal{Q}_i)^{-1}\textbf{u}_{ij}^{-1}\textbf{top}(\pazocal{Q}_j)$ and the subdiagrams $\bar{\Delta}_{ij}$ and $\Psi_{ij}'$ (see \Cref{fig-clove-paths}).

\begin{figure}[H]
\centering
\includegraphics[scale=0.75]{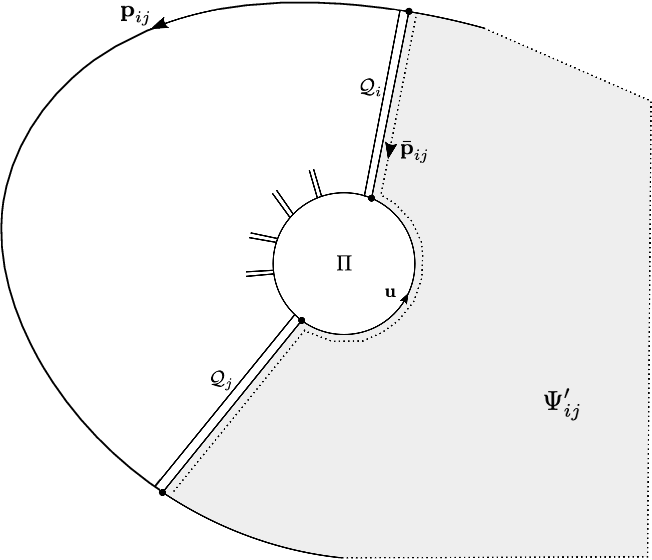}
\caption{Subdiagrams and paths in $\Delta$}
\label{fig-clove-paths}
\end{figure}

%The next statement then follows in just the same way as its analogue in \cite{W}.

\begin{lemma}[Lemma 10.8 of \cite{W}] \label{7.23} 

For $1\leq i\leq L-5$, $|\textbf{p}_{i,i+1}|_q<3K_0$.

\end{lemma}

Let $H_1,\dots,H_{L-4}$ be the histories of the spokes $\pazocal{Q}_1,\dots,\pazocal{Q}_{L-4}$, respectively, read starting from the disk $\Pi$. Further, let $h_i=\|H_i\|$ for all $i$. Lemma \ref{7.22} then implies the inequalities
$$h_1\geq h_2\geq\dots\geq h_r; \ \ h_{r+1}\leq\dots\leq h_{L-4}$$ 

where $(L-1)/2-3\leq r\leq(L-1)/2$. It then follows that $H_{i+1}$ is a prefix of $H_i$ for $i=1,\dots,r-1$ while $H_j$ is a prefix of $H_{j+1}$ for $j=r+1,\dots,L-5$.

Let $W$ be the accepted configuration corresponding to $\lab(\partial\Pi)$. Then, $W\equiv W(1)W(2)\dots W(L)$, where $W(2),\dots,W(L)$ are all copies of the same configuration $V$ of $\textbf{M}_5$. 

As $\ell(W)\leq1$, the definition of the rules immediately implies $|W(1)|_a\leq|V|_a$.  As such, the proof of the next statement adapts naturally from that of the analogous statement in \cite{W}.

\begin{lemma}[Compare with Lemma 10.9 of \cite{W}] \label{7.24} Suppose $i\leq r$ and $j\geq r+1$.

\begin{enumerate}[label=({\arabic*})]

\item $|\textbf{p}_{ij}|\geq|\textbf{p}_{ij}|_\theta+|\textbf{p}_{ij}|_q\geq h_i+h_j+N(j-i)+1$

\item $|\bar{\textbf{p}}_{ij}|\leq h_i+h_j+N(L-j+i)+(L-j+i)\delta|V|_a-1$

\end{enumerate}

\end{lemma}

%\begin{proof}
%
%(1) Lemma \ref{7.22}(2) implies that $\textbf{p}_{ij}$ contains $h_i+h_j$ $\theta$-edges. Further, as $q$-bands cannot cross, every spoke starting on the complement $\bar{\textbf{u}}_{ij}$ of $\textbf{u}_{ij}$ in $\partial\Pi$ must end on $\textbf{p}_{ij}$, so that $\textbf{p}_{ij}$ contains at least $N(j-i)+1$ $q$-edges. The inequality thus follows.
%
%(2) By parts (b) and (c) of Lemma \ref{lengths}, it suffices to show that $$|\textbf{u}_{ij}|\leq11(L-j+i)+(L-j+i)\delta|V|_a-1$$
%As $\partial\Pi$ consists only of $q$-edges and $a$-edges, $|\textbf{u}_{ij}|=|\partial\Pi|-|\bar{\textbf{u}}_{ij}|$. 
%
%As $|\partial\Pi|=NL+\delta\sum_{i=1}^L|W(i)|_a\leq NL+\delta L|V|_a$ and $\lab(\bar{\textbf{u}}_{ij})$ consists of at least $j-i$ copies of $V^{\pm1}$ and one more $t$-letter, $|\bar{\textbf{u}}_{ij}|\geq N(j-i)+(j-i)\delta|V|_a+1$. Thus, the inequality follows.
%
%\end{proof}

%Using \Cref{7.22} and \Cref{lengths}, the next statement follows in the same way as in \cite{W}.

\begin{lemma}[Lemma 10.10 of \cite{W}] \label{7.25}

If $1\leq i<j\leq L-4$ such that $j-i\geq L/2$, then $\mu(\Delta)-\mu(\Psi_{ij}')>-2J|\partial\Delta|(h_i+h_j)\geq-2J|\partial\Delta||\textbf{p}_{ij}|$.

\end{lemma}

\begin{lemma}[Lemma 10.11 of \cite{W}] \label{7.26}

If $1\leq i<j\leq L-4$ such that $j-i\geq L/2$, then $$|\textbf{p}_{ij}|+\sigma_\lambda(\bar{\Delta}_{ij}^*)\leq|\textbf{p}_{ij}|+\sigma_\lambda(\Delta^*)-\sigma_\lambda((\Psi_{ij}')^*)<(1+\eps)|\bar{\textbf{p}}_{ij}|$$ for $\eps=1/\sqrt{N_4}$.

\end{lemma}

\begin{proof}

Set $y=|\textbf{p}_{ij}|+\sigma_\lambda(\Delta^*)-\sigma_\lambda((\Psi_{ij}')^*)$ and $d=y-|\bar{\textbf{p}}_{ij}|$. Suppose $d\geq\eps|\bar{\textbf{p}}_{ij}|>0$.

As in the proof of the analogous statement in \cite{W}, \Cref{lengths} implies $d\geq\frac{\eps y}{2}$ and $p(\Delta)-p(\Psi_{ij}')\geq d$.

%Then $d\geq y-\eps^{-1}d$, so that $d\geq(1+\eps^{-1})^{-1}y\geq\frac{\eps y}{2}$ as $N_4\geq1$. 
%
%As $\Psi_{ij}'$ and $\bar{\Delta}_{ij}$ are disjoint, the definition of the design on a minimal diagram implies 
%$$\sigma_\lambda(\bar{\Delta}_{ij}^*)+\sigma_\lambda((\Psi_{ij}')^*)\leq\sigma_\lambda(\Delta^*)$$
%Let $\textbf{s}$ be the complement of $\textbf{p}_{ij}$ in $\partial\Delta$. As $\textbf{p}_{ij}$ starts and ends with $q$-edges, $|\partial\Delta|=|\textbf{p}_{ij}|+|\textbf{s}|$. Further, by Lemma \ref{lengths}(c), $|\partial\Psi_{ij}'|\leq|\bar{\textbf{p}}_{ij}|+|\textbf{s}|$. 
%
%So, these relations imply
%\begin{align*}
%(|\partial\Delta|+\sigma_\lambda(\Delta^*))-(|\partial\Psi_{ij}'|+\sigma_\lambda((\Psi_{ij}')^*))&\geq|\partial\Delta|-|\partial\Psi_{ij}'|+\sigma_\lambda(\Delta^*)-\sigma_\lambda((\Psi_{ij}')^*)\ \\
%&\geq|\textbf{p}_{ij}|-|\bar{\textbf{p}}_{ij}|+\sigma_\lambda(\Delta^*)-\sigma_\lambda((\Psi_{ij}')^*) \\
%&=d>0
%\end{align*}
Hence, if $\Psi_{ij}'$ contains a disk, then we may apply the inductive hypothesis to it. Conversely, if $\Psi_{ij}'$ contains no disks, then we may apply Lemma \ref{diskless}. In either case,
$$\text{wt}_G(\Psi_{ij}')\leq \phi_{N_4}(p(\Delta)-d)+N_3\mu(\Psi_{ij}')f_2(N_3p(\Psi_{ij}'))$$
Noting that $d\leq p(\Delta)$, Lemma \ref{7.25} then implies
\begin{equation} \label{7.26 G-weight of Psi'}
\text{wt}_G(\Psi_{ij}')\leq \phi_{N_4}(p(\Delta))-N_4dp(\Delta)f_2(N_4p(\Delta))+N_3\left(\mu(\Delta)+2J|\partial\Delta||\textbf{p}_{ij}|\right)f_2(N_3p(\Delta))
\end{equation}
Note that $|\partial\Pi|\leq L|\bar{\textbf{p}}_{ij}|\leq Ly$, so that
\begin{equation} \label{7.26 weight of Pi}
\text{wt}(\Pi)\leq C_1L^2y^2
\end{equation}
Further, as $j-i\geq L/2$ implies $i\leq r<r+1\leq j$ by Lemma \ref{7.22}(2). So, Lemma \ref{7.24} implies 
$$|\bar{\textbf{p}}_{ij}|<|\textbf{p}_{ij}|+NL+(L-1)\delta|V|_a\leq|\textbf{p}_{ij}|+|\partial\Pi|$$ 
and hence $|\partial\Psi_{ij}|<2|\textbf{p}_{ij}|+|\partial\Pi|\leq(L+2)y$ by Lemma \ref{lengths}(c). Thus, by Lemma \ref{diskless},
\begin{equation} \label{7.26 G-weight of Psi}
\text{wt}_G(\Psi_{ij})\leq\phi_{N_2}((L+2)y)+N_1\mu(\Psi_{ij})f_2(N_1(L+2)y)
\end{equation}

As Lemma \ref{G-weight subdiagrams} implies $\text{wt}_G(\Delta)\leq\text{wt}_G(\Psi_{ij}')+\text{wt}(\Pi)+\text{wt}_G(\Psi_{ij})$, it follows from (\ref{7.26 G-weight of Psi'}), (\ref{7.26 weight of Pi}), and (\ref{7.26 G-weight of Psi}) that it suffices to show that 
%
%By (\ref{7.26 G-weight of Psi'}), (\ref{7.26 weight of Pi}), and (\ref{7.26 G-weight of Psi}), Lemma \ref{G-weight subdiagrams} implies
%$$
%\text{wt}_G(\Delta)\leq N_4x^2-N_4xd+N_3\mu(\Delta)+2N_3Jy|\partial\Delta|+N_2(L+3)^2y^2+N_1\mu(\Psi_{ij})+C_1(L+1)^2y^2 
%$$
%Hence, in order to reach a contradiction, it suffices to show that
\begin{equation} \label{7.26 suffices 1}
N_4dp(\Delta)f_2(N_4p(\Delta))\geq 2N_3Jy|\partial\Delta|f_2(N_3p(\Delta))+2\phi_{N_2}((L+2)y)+N_1\mu(\Psi_{ij})f_2(N_1(L+2)y)
\end{equation}
Note that $p(\Delta)=|\partial\Delta|+\sigma_\lambda(\Delta^*)\geq|\textbf{p}_{ij}|+\sigma_\lambda(\Delta^*)$, so that $p(\Delta)\geq\max(|\partial\Delta|,y)$. Hence, $$N_4dp(\Delta)\geq\frac{\eps}{2}N_4y\max(|\partial\Delta|,y)=\frac{1}{2}\sqrt{N_4}\max(y|\partial\Delta|,y^2)$$
The parameter choices $N_4>>N_3>>N_2>>C_1>>J>>L$ then imply
$$\frac{1}{2}N_4dp(\Delta)f_2(N_4p(\Delta))\geq2N3Jy|\partial\Delta|f_2(N_3p(\Delta))+2\phi_{N_2}((L+2)y)$$ 
Hence, by (\ref{7.26 suffices 1}) it suffices to show that
\begin{equation} \label{7.26 suffices 2}
N_4dp(\Delta)f_2(N_4p(\Delta))\geq2N_1\mu(\Psi_{ij})f_2(N_1(L+2)y)
\end{equation}
As each $\theta$-edge of $\partial\Psi_{ij}$ must be in its own factor of any decomposition of $\lab(\Psi_{ij})$, the number of white beads on the necklace corresponding to $\partial\Psi_{ij}$ is at most $|\partial\Psi_{ij}|\leq(L+2)y$. So, Lemma \ref{mixtures}(a) implies
$$\mu(\Psi_{ij})\leq J(L+2)^2y^2$$
But as above, $N_4dp(\Delta)\geq\frac{1}{2}\sqrt{N_4}y^2$, so that the parameter choices $N_4>>J>>L$ imply (\ref{7.26 suffices 2}), and so the statement.

\end{proof}

For $i=1,\dots,L-5$, if the pair of adjacent $t$-letters associated to $\pazocal{Q}_i$ and to $\pazocal{Q}_{i+1}$ are $t(L)$ and $t(2)$ (or vice versa), then $\Psi_{i,i+1}$ is called the \textit{distinguished clove}. As $q$-bands cannot cross, the distinguished clove contains a cutting $q$-band $\pazocal{Q}_i'$ formed by the $q$-spoke of $\Pi$ corresponding to the base letter $t(1)$. Let $\Lambda_{i,i+1}'$ (respectively $\Lambda_{i,i+1}''$) be the subdiagram bounded by $\pazocal{Q}_i'$ and the $t$-spoke corresponding to $t(L)$ (respectively $t(2)$). Define $\textbf{p}_{i,i+1}'$ (respectively $\textbf{p}_{i,i+1}''$) as the subpath of $\partial\Lambda_{i,i+1}'$ (respectively $\partial\Lambda_{i,i+1}''$) shared with $\partial\Delta$, so that $\textbf{p}_{i,i+1}$ is the concatenation of these two paths along a shared $q$-edge.

Suppose $\Psi_{i,i+1}$ is not the distinguished clove. Then, let $\textbf{q}_{i,i+1}$ be the shortest path in $\Psi_{i,i+1}$ homotopic to $\textbf{p}_{i,i+1}$ and having the same first and last edges. 

If $\Psi_{i,i+1}$ is the distinguished clove, then define $\textbf{q}_{i,i+1}'$ and $\textbf{q}_{i,i+1}''$ as the analgous shortest paths in $\Lambda_{i,i+1}'$ and $\Lambda_{i,i+1}''$. Then, let $\textbf{q}_{i,i+1}$ be the concatenation of $\textbf{q}_{i,i+1}'$ and $\textbf{q}_{i,i+1}''$ along their shared $q$-edge.

For $1\leq i<j\leq L-4$, let $\textbf{q}_{ij}$ be the concatenation of the paths $\textbf{q}_{i,i+1},\dots,\textbf{q}_{j-1,j}$ along their shared $q$-edges.

Then, let $\Psi_{ij}^0$ be the diagram obtained from $\Psi_{ij}$ by replacing $\textbf{p}_{ij}$ in the contour with $\textbf{q}_{ij}$, i.e by removing any cells between $\textbf{q}_{ij}$ and $\textbf{p}_{ij}$. Similarly define $\Psi^0$, $(\Lambda_{i,i+1}')^0$,  and $(\Lambda_{i,i+1}'')^0$.

The next statement, a direct analogue of \Cref{7.27}(1), follows in just the same way as the analogous statements in \cite{O18}, \cite{OS19}, and \cite{W}.

%The following is the direct analogue of Lemma \ref{7.24}(1) and is proved in exactly the same way.

\begin{lemma}[Lemma 10.12 of \cite{W}] \label{7.27} 

If $i\leq r$ and $j\geq r+1$, then $|\textbf{q}_{ij}|\geq h_i+h_j+N(j-i)+1$.

\end{lemma}

Similarly, the next two statements are proved in exactly the same way as the analogous statements in \cite{W}.

\begin{lemma}[Lemma 10.13 of \cite{W}] \label{7.28} \

\begin{enumerate}[label=({\arabic*})]

\item Every maximal $q$-band of $\Psi^0$ corresponds to a spoke of $\Pi$.

\item No two $\theta$-edges of $\textbf{q}_{i,i+1}$ are part of the same $\theta$-band of $\Psi_{i,i+1}$.

\end{enumerate}

\end{lemma}

\begin{lemma}[Lemma 10.14 of \cite{W}] \label{non-distinguished a-cells} \

\begin{enumerate}[label=({\arabic*})]

\item If $\Psi_{i,i+1}$ is not the distinguished clove, then $\Psi_{i,i+1}^0$ contains no $a$-cells. 

\item If $\Psi_{i,i+1}$ is the distinguished clove, then $(\Lambda_{i,i+1}')^0$ contains no $a$-cells.

\end{enumerate}

\end{lemma}

%\begin{proof}
%
%(1) Suppose $\pi$ is an $a$-cell contained in $\Psi_{i,i+1}^0$.
%
%By Lemma \ref{7.28}(1), no maximal $q$-band of $\Psi_{i,i+1}^0$ corresponds to a base letter with coordinate 1. So, the contour of any $(\theta,a)$- or $(\theta,q)$-cell has no $a$-edge labelled by a letter from the alphabet of the `special' input sector.
%
%As a result, any edge of $\partial\pi$ must be shared with $\partial\Psi_{i,i+1}^0$. In particular, $\partial\pi$ must be a subpath of $\textbf{q}_{i,i+1}$. 
%
%But then removing this subpath produces a path homotopic to $\textbf{q}_{i,i+1}$ that contradicts its definition.
%
%(2) is proved analogously, as the only base letter with coordinate 1 present in $(\Lambda_{i,i+1}')^0$ is $\{t(1)\}$.
%
%\end{proof}

\medskip

%%%%%%%%%%%%%%%%%%%%%%%%%%%%%%%%%%%%%%%%%%%%%%%%%%

\subsection{Trapezia and combs in the cloves} \label{sec-trapezia-combs-cloves} \

For $1\leq i\leq r-1$, suppose $\Psi_{i,i+1}$ is not the distinguished clove. Then Lemma \ref{7.22} implies that all maximal $\theta$-bands of $\Psi_{i,i+1}$ crossing $\pazocal{Q}_{i+1}$ must also cross $\pazocal{Q}_i$. So, these $\theta$-bands bound an $a$-trapezium $\Gamma_i$ in $\Psi_{i,i+1}^0$ with height $h_{i+1}$. The base of $\Gamma_i$ (or its inverse) is $\{t(\ell)\}B_4(\ell)\{t(\ell+1)\}$ for some $2\leq\ell\leq L-1$. Lemma \ref{a-cells sector} then implies that $\Gamma_i$ is a trapezium. Set $\textbf{y}_i=\textbf{bot}(\Gamma_i)$ and $\textbf{z}_i=\textbf{top}(\Gamma_i)$. Note that $\textbf{y}_i^{-1}$ is shared with $\partial\Pi$. 

For $2\leq i\leq r-1$, suppose neither $\Psi_{i-1,i}$ nor $\Psi_{i,i+1}$ is the distinguished clove. Then $\lab(\textbf{y}_{i-1})$ and $\lab(\textbf{y}_i)$ are coordinate shifts of one another while $H_{i+1}$ is a prefix of $H_i$. So, $h_{i+1}$ $\theta$-bands of $\Gamma_{i-1}$ form a copy of $\Gamma_i$, $\Gamma_i'$, contained in $\Gamma_{i-1}$. Set $\textbf{y}_i'=\textbf{bot}(\Gamma_i')$ and $\textbf{z}_i'=\textbf{top}(\Gamma_i')$. Note that $\textbf{y}_i'=\textbf{y}_{i-1}$.

For $1\leq i\leq r-1$, if $\Psi_{i,i+1}$ is not the distinguished clove, then denote by $E_i$ (respectively $E_i^0$) the maximal comb in $\Psi_{i,i+1}$ (respectively $\Psi_{i,i+1}^0$) containing the maximal $\theta$-bands that cross the $t$-spoke $\pazocal{Q}_i$ but not the $t$-spoke $\pazocal{Q}_{i+1}$. The handle $\pazocal{C}_i$ of these combs has height $h_i-h_{i+1}$ and is contained in $\pazocal{Q}_i$. Any cell of $\Psi_{i,i+1}$ (respectively $\Psi_{i,i+1}^0$) not contained in $\Gamma_i$ or $E_i$ (respectively $E_i^0$) must be an $a$-cell attached to either $\textbf{z}_i$ or $\pazocal{Q}_{i+1}$. By the structure of the relations, such an $a$-cell must share every boundary edge with $\partial\Delta$. But this contradicts Lemma \ref{a-cell in counterexample 2}. Hence, $E_i$ (respectively $E_i^0$) is the complement of $\Gamma_i$ in $\Psi_{i,i+1}$ (respectively $\Psi_{i,i+1}^0$). 

Now suppose $\Psi_{i,i+1}$ is the distinguished clove for $1\leq i\leq r-1$. 

First, suppose $\pazocal{Q}_i$ corresponds to the base letter $\{t(L)\}$, so that the subdiagram $\Lambda_{i,i+1}'$ is bounded by $\pazocal{Q}_i$ and $\pazocal{Q}_i'$ (see \Cref{fig-distinguished-trapezia}(a)). By Lemma \ref{7.22}, every maximal $\theta$-band of $\Lambda_{i,i+1}'$ crossing $\pazocal{Q}_i'$ must also cross $\pazocal{Q}_i$. So, these $\theta$-bands bound an $a$-trapezia $\Gamma_i$ contained in $(\Lambda_{i,i+1}')^0$. As above, Lemma \ref{a-cells sector} implies that $\Gamma_i$ must be a trapezium. The base of $\Gamma_i$ (or its inverse) is $\{t(L)\}B_4(L)\{t(1)\}$, while the height is the length $h_i'$ of the band $\pazocal{Q}_i'$.

\begin{figure}[H]
\centering
\includegraphics[scale=1.25]{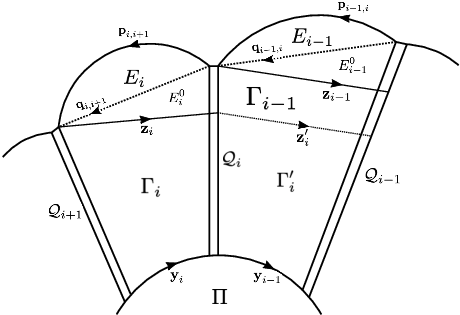}
\caption{Trapezia and combs if neither $\Psi_{i-1,i}$ and $\Psi_{i,i+1}$ are distinguished}
\end{figure}

Otherwise, $\pazocal{Q}_{i+1}$ corresponds to the base letter $\{t(L)\}$, so that the subdiagram $\Lambda_{i,i+1}'$ is bounded by $\pazocal{Q}_i'$ and $\pazocal{Q}_{i+1}$ (see \Cref{fig-distinguished-trapezia}(2)). Lemma \ref{7.22} then implies that every maximal $\theta$-band of $\Lambda_{i,i+1}'$ crossing $\pazocal{Q}_{i+1}$ must also cross $\pazocal{Q}_i'$, so that these $\theta$-bands bound an $a$-trapezium $\Gamma_i$ contained in $(\Lambda_{i,i+1}')^0$. Again, $\Gamma_i$ must be a trapezium whose base (or its inverse) is $\{t(L)\}B_4(L)\{t(1)\}$. In this case, the height of $\Gamma_i$ is $h_{i+1}$.

In either case, we define $\textbf{y}_i=\textbf{bot}(\Gamma_i)$ and $\textbf{z}_i=\textbf{top}(\Gamma_i)$. If $i\geq2$, then again $\lab(\textbf{y}_i)$ is a coordinate shift of $\lab(\textbf{y}_{i-1})$ and there exists a copy $\Gamma_i'$ of $\Gamma_i$ in $\Gamma_{i-1}$ with $\textbf{bot}(\Gamma_i')=\textbf{y}_i'=\textbf{y}_{i-1}$. Similarly, if $i\leq r-2$, then $\lab(\textbf{y}_i)$ is a coordinate shift of $\lab(\textbf{y}_{i+1})$ and there exists a copy $\Gamma_{i+1}'$ of $\Gamma_{i+1}$ in $\Gamma_i$.

Suppose $\Lambda_{i,i+1}'$ is bounded by $\pazocal{Q}_i$ and $\pazocal{Q}_i'$. Then denote by $E_i$ (respectively $E_i^0$) the maximal comb in $\Lambda_{i,i+1}'$ (respectively $(\Lambda_{i,i+1}')^0$) containing the maximal $\theta$-bands that cross the $t$-spoke $\pazocal{Q}_i$ but not the $q$-spoke $\pazocal{Q}_i'$. The handle $\pazocal{C}_i$ of these combs has height $h_i-h_i'$ and is contained in $\pazocal{Q}_i$. As above, $E_i$ (respectively $E_i^0$) is the complement of $\Gamma_i$ in $\Lambda_{i,i+1}'$ (respectively $(\Lambda_{i,i+1}')^0$).

Otherwise, $\Lambda_{i,i+1}'$ is bounded by $\pazocal{Q}_i'$ and $\pazocal{Q}_{i+1}$. In this case denote by $E_i$ (respectively $E_i^0$) the maximal comb in $\Lambda_{i,i+1}'$ (respectively $(\Lambda_{i,i+1}')^0$) containing the maximal $\theta$-bands that cross the $q$-spoke $\pazocal{Q}_i'$ but not the $t$-spoke $\pazocal{Q}_{i+1}$. The handle $\pazocal{C}_i$ of these combs has height $h_i'-h_{i+1}$ and is contained in $\pazocal{Q}_i'$. Again, $E_i$ (respectively $E_i^0$) is the complement of $\Gamma_i$ in $\Lambda_{i,i+1}'$ (respectively $(\Lambda_{i,i+1}')^0$).

Note that no $a$-trapezium or comb has been defined in the subdiagram $\Lambda_{i,i+1}''$. Though such subdiagrams exist, their consideration is not necessary for the rest of the proof. As a result, one may view the indexing as `skipping over' the portion of the clove between the base letters $\{t(1)\}$ and $\{t(2)\}$.

For $r+1\leq i\leq L-5$, the trapezium $\Gamma_i$, the combs $E_i$ and $E_i^0$, and the paths $\textbf{y}_i$ and $\textbf{z}_i$ are defined symmetrically.

The next two statements, which study the makeup of these trapezia and combs, are proved in exactly the same manner as their analogues in \cite{W}.

\begin{lemma} [Lemma 10.15 of \cite{W}] \label{7.29}

%\textit{(Compare with Lemma 9.15 of [18] and Lemma 7.29 of [25])}

For $i\in\{2,\dots,r-1\}$, suppose a maximal $a$-band $\pazocal{B}$ of $E_i^0$ starts on $\textbf{z}_i$ and ends on a side of a maximal $q$-band $\pazocal{C}$. Let $\nabla$ be the comb bounded by $\pazocal{B}$, a part of $\pazocal{C}$, and a subpath $\textbf{x}$ of $\textbf{z}_i$. Then there is a copy of the comb $\nabla$ in the trapezium $\Gamma=\Gamma_{i-1}\setminus\Gamma_i'$.

\end{lemma}

%\begin{proof}
%
%By Lemma \ref{non-distinguished a-cells}, $\nabla$ contains no $a$-cells.
%
%Let the $a$-edge $\textbf{e}$ and the $q$-edge $\textbf{f}$ be the first and last edge of $\textbf{x}$, respectively. Since $\textbf{z}_i'$ is a copy of $\textbf{z}_i$ in the trapezium $\Gamma_{i-1}$, it contains a subpath $\textbf{x}'$ that is a copy of $\textbf{x}$ and starts with an $a$-edge $\textbf{e}'$ and ends with a $q$-edge $\textbf{f}'$. If $\pi$ is the $(\theta,q)$-cell attached to $\textbf{f}$ in $\nabla$, then the $(\theta,q)$-cell $\pi'$ attached to $\textbf{f}'$ is a copy since it corresponds to the same letter of the history. Moving from $\textbf{f}$ to $\textbf{e}$, the whole maximal $\theta$-band of $\nabla$ containing $\pi$ has a copy in $\Gamma_{i-1}$. Moving up, we find a copy of every maximal $\theta$-band of $\nabla$ in $\Gamma_{i-1}$, forming a copy of $\nabla$ in $\Gamma_{i-1}$.
%
%\end{proof}

\begin{lemma}[Compare with Lemma 10.16 of \cite{W}] \label{7.30}

At most $2N$ $a$-bands starting on the path $\textbf{y}_i$ (or $\textbf{z}_i$) can end on $(\theta,q)$-cells of the same $\theta$-band.

\end{lemma}

%\begin{proof}
%
%Assume each of the $a$-bands $\textbf{A}_1,\dots,\textbf{A}_m$ starts from an edge of $\textbf{y}_i$ and ends on some $(\theta,q)$-cell of a $\theta$-band $\pazocal{T}$.  Let $\pazocal{T}_0$ be the minimal subband of $\pazocal{T}$ such that the $a$-bands $\textbf{A}_2,\dots,\textbf{A}_{m-1}$ end on $\pazocal{T}_0$. Then, let $\bar{\textbf{y}}_i$ the minimal subpath of $\textbf{y}_i$ where the $a$-bands $\textbf{A}_1,\dots,\textbf{A}_m$ start (see \Cref{fig-cloveaband}).
%
%By Lemma \ref{M_a no annuli 1}, each $q$-band starting on $\bar{\textbf{y}}_i$ has to cross $\pazocal{T}_0$ and vice versa. So, the base of $\pazocal{T}_0$ is a subword of a reduced pararevolving base not containing the `special' input sector. 
%
%As a result, we can identify this base with a subword of the standard base of $\textbf{M}_5$ (or its inverse). By the structure of the rules of $\textbf{M}_5$, an application of any rule inserts/deletes at most $2(N-1)$ $a$-letters in a configuration. Thus, $m-2\leq2(N-1)$, so that the statement follows.
%
%An analogous argument applies for $a$-bands starting from $\textbf{z}_i$.
%
%\end{proof}
%
%\begin{figure}[H]
%\centering
%\includegraphics[scale=2]{cloveaband.eps}
%\caption{}
%\label{fig-cloveaband}
%\end{figure}

\renewcommand\thesubfigure{\alph{subfigure}}
\begin{figure}[H]
\centering
\begin{subfigure}[b]{0.48\textwidth}
\centering
\includegraphics[scale=1]{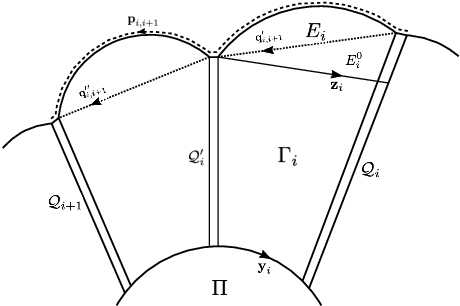}
\caption{$\pazocal{Q}_i$ and $\pazocal{Q}_i'$ bound $\Lambda_{i,i+1}'$}
\end{subfigure}\hfill
\begin{subfigure}[b]{0.48\textwidth}
\centering
\includegraphics[scale=1]{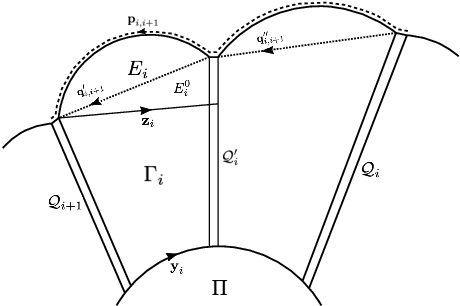}
\caption{$\pazocal{Q}_i$ and $\pazocal{Q}_i'$ bound $\Lambda_{i,i+1}''$}
\end{subfigure}
\caption{Trapezia in the distinguished clove}
\label{fig-distinguished-trapezia}
\end{figure}

Note that the bound given in \Cref{7.30} differs from its analogue in \cite{W}.  The reason for this is a difference in how a transition can operate on an admissible word with pararevolving base: A transition of the main machine of \cite{W} can insert/delete at most 4 $a$-letters in such an admissible word, whereas in this setting a transition of $\textbf{M}$ can insert/delete at most $2(N-1)$ $a$-letters.

By the parameter choice $L>>L_0$ and Lemma \ref{7.22}, we may assume that $L_0+1\leq r$ and $L-L_0-4\geq r+1$. Suppose without loss of generality that $h\defeq h_{L_0+1}\geq h_{L-L_0-4}$.

Using Lemmas \ref{7.24}, \ref{7.26}, \ref{7.27}, \ref{7.28}(2), and \ref{7.30}, the next statement follows in much the same way as its analogue in \cite{W}.

\begin{lemma}[Compare with Lemma 10.17 in \cite{W}] \label{7.31} 

Let $I$ be the subset of the set of indices $i\in[L_0+1,r-1]\cup[r+1,L-L_0-5]$ such that $|\textbf{z}_i|_a\geq|V|_a/c_3N$. If $h\leq L_0^2|V|_a$, then $\# I\leq L/5$.

\end{lemma}

The next statement then follows from Lemmas \ref{G_a theta-annuli}(2), \ref{7.24}, \ref{7.26}, and \ref{7.31}.

\begin{lemma}[Lemmas 10.18-10.19 of \cite{W}] \label{7.32-7.33}

If $h\leq L_0^2|V|_a$, then:

\begin{enumerate}

\item $H_1$ and $H_{L-4}$ have different first letters

\item $|V|_a>\frac{NL}{4\delta L_0}$

\end{enumerate}

\end{lemma}

Finally, we reach a contradiction to the assumption that $h\leq L_0^2|V|_a$.  The proof follows in much the same way as its analogue in \cite{W}, applying Lemmas \ref{M projected long history controlled}, \ref{7.31}, and \ref{7.32-7.33} to deduce a contradiction to \Cref{7.26}.

\begin{lemma}[Lemma 10.20 of \cite{W}] \label{7.34}

The inequality $h>L_0^2|V|_a$ must be true.

\end{lemma}

\begin{proof}

Assuming the statement is false, Lemma \ref{7.31} implies that for at least $L-5-L/5-2L_0>3L/4$ indices $j\in\{1,\dots,L-5\}$,  $|\textbf{z}_j|_a<|V|_a/c_3N$. So, we can choose two such indices, $i$ and $j$, such that $L_0+1\leq i\leq r<r+1\leq j\leq L-L_0-5$, $j-i\geq 3L/5$, and neither $\Psi_{i,i+1}$ nor $\Psi_{j,j+1}$ is the distinguished clove. 

Since $H_{i+1}$ (respectively $H_j$) is a prefix of $H_1$ (respectively $H_{L-4}$), it follows from Lemma \ref{7.32-7.33}(1) that the first letters of $H_{i+1}$ and $H_j$ are different.

Since $\lab(\textbf{y}_i)$ and $\lab(\textbf{y}_j)$ are coordinate shifts of one another (and are copies of $V$), we can construct an auxiliary trapezium $E$ by pasting the mirror of a coordinate shift of $\Gamma_j$ to $\Gamma_i$ along $\textbf{y}_i$. The history of $E$ is $H_j^{-1}H_{i+1}$, which is a reduced word since the first letter of $H_{i+1}$ is different from the first letter of $H_j$.

The top and the bottom of $E$ are copies of $\textbf{z}_i$ and $\textbf{z}_j$, respectively, and so have $a$-lengths less than $|V|_a/c_3N$. Without loss of generality, assume $h_{i+1}\geq h_j$, and so $h_{i+1}\geq t/2$ for $t$ the height of $E$.

Note that $|V|_a-|V|_a/c_3N>|V|_a/2$, and so $h_{i+1},h_j>|V|_a/4N$ since any rule of $\textbf{M}_5$ alters the $a$-length of a configuration by at most $2(N-1)$.

By Lemma \ref{7.32-7.33}(2), $|V|_a/4N>\frac{L\delta^{-1}}{16L_0}\geq Nc_2$ since $\delta^{-1}$ and $L$ are chosen after $L_0$, $c_2$, and $N$. Further, letting $W_0$ and $W_t$ be the bottom and top labels of $E$, respectively, $|V|_a/4N\geq \frac{c_3}{4}\max(|W_0|_a,|W_t|_a)>c_2\max(|W_0|_a,|W_t|_a)$ since $c_3>>c_2$.

As a result,
$$t=h_{i+1}+h_j>|V|_a/2N>c_2\max(|W_0|_a,|W_t|_a)+Nc_2\geq c_2\max(\|W_0\|,\|W_t\|)$$
Let $\pazocal{C}$ be the computation associated to $E$ through Lemma \ref{trapezia are computations}. Then the restriction of $\pazocal{C}$ (or its inverse) to $\{t(\ell)\}B_4(\ell)$ for the appropriate $\ell\geq2$ satisfies the hypotheses of Lemma \ref{M projected long history controlled}. 

Setting $\lambda<1/10$, every factorization $H'H''H'''$ of $H_{i+1}$ with $\|H'\|+\|H'''\|\leq\lambda h_{i+1}$ satisfies $\|H''\|>0.4t$. So, applying Lemma \ref{M projected long history controlled}, $H''$ contains a controlled subword. Further, since all $\theta$-bands crossing $\pazocal{Q}_{i+1}$ must cross $\pazocal{Q}_i$, $W(\ell)$ is $H_{i+1}$-admissible. Hence, $\pazocal{Q}_{i+1}$ is a $\lambda$-shaft. 

Lemma \ref{7.24}(1) then implies that $|\textbf{p}_{i+1,j}|+\sigma_\lambda(\bar{\Delta}_{i+1,j}^*)\geq 2h_{i+1}+h_j$.

As $h_{i+1}>|V|_a/4N$, it follows that $\delta L|V|_a\leq4N\delta Lh_{i+1}<\frac{1}{4}h_{i+1}$ by the parameter choices $\delta^{-1}>>L>>N$. Further, Lemma \ref{7.32-7.33}(2) and and the parameter choice $\delta^{-1}>>L_0>>N$ imply $NL<4\delta L_0|V|_a<16N\delta L_0h_{i+1}\leq\frac{1}{4}h_{i+1}$.

So, Lemma \ref{7.24}(2) yields $|\bar{\textbf{p}}_{i+1,j}|\leq \frac{3}{2}h_{i+1}+h_j$.

Hence, $$\frac{|\textbf{p}_{i+1,j}|+\sigma_\lambda(\bar{\Delta}_{i+1,j}^*)}{|\bar{\textbf{p}}_{i+1,j}|}\geq\frac{2h_{i+1}+h_j}{\frac{3}{2}h_{i+1}+h_j}\geq\frac{6}{5}$$ since $h_j\leq h_{i+1}$.

Taking $N_4$ sufficiently large, $\eps=1/\sqrt{N_4}<0.2$. However, as $j-(i+1)\geq3L/5-1\geq L/2$, the above inequality contradicts Lemma \ref{7.26}.

\end{proof}

The next lemma consists of several statements then follow from \Cref{7.34}; their proofs are all the same as those of the analogous statements in \cite{W}, where again a contradiction to \Cref{7.26} is deduced.  %Note that the presence of the constant $c_2$ in (3) and (4) is indicative of the dependence on \Cref{M projected long history controlled}.

\begin{lemma}[Lemmas 10.21-10.25 of \cite{W}] \label{lem-combined}

Let $i\in\{1,\dots,L_0-1\}$.

\begin{enumerate}

\item $h_i>\delta^{-1}$

\item The spoke $\pazocal{Q}_i$ does not contain a $\lambda$-shaft of $\Pi$ of length at least $\delta h$

\item $|\textbf{z}_i|_a>h_{i+1}/2c_2$

\item $h_{i+1}<(1-\frac{1}{30c_2})h_i$

\item $|\textbf{z}_i|_a\leq 8h_i$

\end{enumerate}

\end{lemma}

Note that if $\Psi_{i,i+1}$ is the distinguished clove for $i\leq r-1$, then $H_{i+1}$ need not be the history of $\Gamma_i$. To account for this, let $H_{i+1}'$ be the history of $\Gamma_i$. Note that $H_{i+1}$ is always a prefix of $H_{i+1}'$.

\begin{lemma}[Lemma 10.26 of \cite{W}] \label{no one-step}

For $2\leq i\leq L_0-2$, let $H_i'=H_{i+1}'H'=H_{i+2}'H''H'$ and $\pazocal{C}$ be the computation corresponding to the trapezium $\Gamma_{i-1}$. Suppose the subcomputation $\pazocal{D}$ of $\pazocal{C}$ with history $H''H'$ has step history of length 1. Then there is no two-letter subword $Q'Q$ of the base of $\Gamma_{i-1}$ such that every rule of $\pazocal{D}$ inserts one letter to the left of $Q$.

%Then the subcomputation $\pazocal{D}$ of $\pazocal{C}$ with history $H''H'$ cannot have step history of length 1 so that there exists a sector $QQ'$ such that one of either $Q$ or $Q'$ has a letter inserted next to it to increase the length of this sector for each rule of $\pazocal{D}$.

\end{lemma}

\begin{proof}

%Supposing the statement is false, Lemma \ref{M one-step} allows us to assume without loss of generality that the rules of $\pazocal{D}$ write letters to the right of $Q$.

Let $\pazocal{Q}$ be the maximal $q$-band of $E_i^0$ that is a subband of the $q$-spoke of $\Pi$ corresponding to a coordinate shift the state letter $Q$. Similarly, let $\pazocal{Q}'$ be the maximal $q$-band corresponding to a coordinate shift of $Q'$, so that $\pazocal{Q}'$ and $\pazocal{Q}$ are neighbor $q$-bands. Let $\textbf{x}$ be the subpath of $\textbf{z}_i$ between $\pazocal{Q}'$ and $\pazocal{Q}$.

Since $\Gamma_i$ contains a copy $\Gamma_{i+1}'$ of the trapezium $\Gamma_{i+1}$, the bottom of the trapezium $\Gamma_i\setminus\Gamma_{i+1}'$ is a copy $\textbf{z}_{i+1}'$ of $\textbf{z}_{i+1}$, while the top is $\textbf{z}_i$. This trapezium has history $H''$, so that the corresponding computation inserts one $a$-letter to the left of the state letter corresponding to $\pazocal{Q}$ at each transition. As a result, $|\textbf{x}|_a\geq\|H''\|\geq h_{i+1}-h_{i+2}$. 

By \Cref{lem-combined}(4), $h_{i+1}-h_{i+2}>\frac{1}{30c_2}h_{i+1}$. As $h_{i+1}\geq h$, Lemma \ref{7.34} and the parameter choice $L_0>>c_2$ imply
$$|\textbf{x}|_a\geq\frac{h}{30c_2}>\frac{L_0^2|V|_a}{30c_2}>10L_0|V|_a$$
If an $a$-band starting on $\textbf{x}$ ended on a $(\theta,q)$-cell of $\pazocal{Q}$, then Lemma \ref{7.26} implies that there is a copy of this in the trapezium $\Gamma_{i-1}\setminus\Gamma_i'$. By Lemma \ref{trapezia are computations}, though, this would imply there exists a transition of $\pazocal{D}$ that deletes a letter from the left of $\pazocal{Q}$, contradicting the hypothesis.

\begin{figure}[H]
\centering
\includegraphics[scale=1]{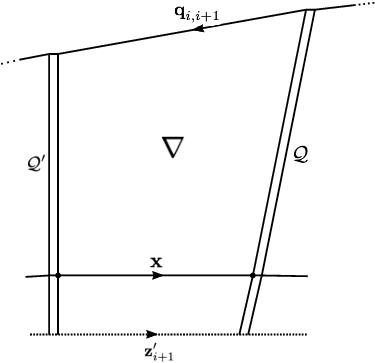}
\caption{}
\label{fig-nabla}
\end{figure}

Now, consider the comb $\nabla$ contained in $E_i^0$ bounded by $\pazocal{Q}'$, $\pazocal{Q}$, $\textbf{x}$, and $\textbf{q}_{i,i+1}$ (see \Cref{fig-nabla}). Set $s$ and $s'$ as the lengths of $\pazocal{Q}$ and $\pazocal{Q}'$, respectively. Lemmas \ref{7.22}(2) and \ref{7.28}(2) imply $s'\leq s$. So, by Lemma \ref{7.28}(1), there are $|\textbf{x}|_a+s$ maximal $a$-bands starting on $\textbf{x}$ or $\pazocal{Q}$ and ending on $\pazocal{Q}'$ or on $\textbf{q}_{i,i+1}$. Since only $s'$ $a$-bands can end on $\pazocal{Q}'$, at least $|\textbf{x}|_a+s-s'$ of them end on the segment of $\textbf{q}_{i,i+1}$ between $\pazocal{Q}$ and $\pazocal{Q}'$. By Lemmas \ref{7.22}(2) and \ref{7.28}(2), the same segment contains $s-s'$ $\theta$-edges, meaning at least $|\textbf{x}|_a$ of them contribute $\delta$ to its length. 

As in the proof of the analogous statement in \cite{W}, \Cref{7.24}, \Cref{7.34}, and the parameter choice $L_0>>c_2$ then imply $$|\textbf{p}_{i,L-L_0-4}|-|\bar{\textbf{p}}_{i,L-L_0-4}|\geq\delta h_{i+1}/50c_2$$
%So, by Lemma \ref{7.24}(1),
%\begin{align*}
%|\textbf{p}_{i,L-L_0-4}|&\geq h_i+h_{L-L_0-4}+11L/2+\delta\frac{h_{i+1}}{30c_3} \\
%&\geq h_i+h_{L-L_0-4}+11L/2+10\delta L_0|V|_a
%\end{align*}
%Also by Lemma \ref{7.24}(2) and \ref{7.34},
%$$|\bar{\textbf{p}}_{i,L-L_0-4}|\leq h_i+h_{L-L_0-4}+11(3L_0)+3\delta L_0|V|_a\leq h_i+h_{L-L_0-4}+11(3L_0)+3\delta h/L_0$$
%So,
%\begin{align*}
%|\textbf{p}_{i,L-L_0-4}|-|\bar{\textbf{p}}_{i,L-L_0-4}|&\geq11(L/2-3L_0)+\delta(h_{i+1}/30c_3-3h/L_0) \\
%&>\delta h_{i+1}(1/30c_3-3/L_0) \\
%&>\delta h_{i+1}/50c_3
%\end{align*}
%by again taking $L_0>>c_3$.

Let $\textbf{s}$ be the complement of the $\textbf{p}_{i,L-L_0-4}$ in $\partial\Delta$. Then, since $\textbf{p}_{i,L-L_0-4}$ starts and ends with $q$-edges, Lemma \ref{lengths}(c) implies
\begin{align*}
|\partial\Psi'_{i,L-L_0-4}|&\leq|\textbf{s}|+|\bar{\textbf{p}}_{i,L-L_0-4}| \\
&<|\textbf{s}|+|\textbf{p}_{i,L-L_0-4}|-\delta h_{i+1}/50c_2 \\
&=|\partial\Delta|-\delta h_{i+1}/50c_2 \numberthis
\end{align*}
Lemma \ref{weakly minimal}(1) implies that $\Psi'_{i,L-L_0-4}$ is weakly minimal with $\sigma_\lambda((\Psi'_{i,L-L_0-4})^*)\leq\sigma_\lambda(\Delta^*)$.

Hence, $p(\Delta)-p(\Psi_{i,L-L_0-4}')\geq\delta h_{i+1}/50c_2>0$, so that applying either the inductive hypothesis or \Cref{diskless} implies
% $\Psi_{i,L-L_0-4}'$ contains a disk, then we may apply the inductive hypothesis to it. Otherwise, we may apply Lemma \ref{diskless} to $\Psi_{i,L-L_0-4}'$. In either case, this implies

\begin{align*} \label{G-weight Psi'}
\text{wt}_G(\Psi'_{i,L-L_0-4}) &\leq \phi_{N_4}(p(\Psi_{i,L-L_0-4}'))+N_3\mu(\Psi_{i,L-L_0-4}')f_2(N_3p(\Psi_{i,L-L_0-4}')) \\
&\leq\phi_{N_4}(p(\Delta))-\frac{N_4\delta h_{i+1}}{50c_2}p(\Delta)f_2(N_4p(\Delta))+N_3\mu(\Psi_{i,L-L_0-4}')f_2(N_3p(\Delta)) \numberthis
\end{align*}
%Taking $x=|\partial\Delta|+\sigma_\lambda(\Delta^*)$, we have
%$$0\leq|\Psi_{i,L-L_0-4}'|+\sigma_\lambda((\Psi_{i,L-L_0-4}')^*)\leq x-\delta h_{i+1}/50c_3$$
%So, since $\delta h_{i+1}/50c_3\leq x$,
%\begin{equation} \label{G-weight Psi'}
%\text{wt}_G(\Psi_{i,L-L_0-4}')<N_4x^2-N_4\delta xh_{i+1}/50c_3+N_3\mu(\Psi_{i,L-L_0-4}')
%\end{equation}
Next, note that $|V|_a\leq h_i/L_0^2$ by Lemma \ref{7.34}, $h_i>\delta^{-1}>100NL_0$ by \Cref{lem-combined}(1), and $h_{L-L_0-4}\leq h\leq h_i$. So, for sufficiently large $L_0$, we have
$$|\bar{\textbf{p}}_{i,L-L_0-4}|\leq 2h_i+3h_i/100N+3\delta h_i/L_0\leq 2.1h_i$$
As Lemma \ref{7.26} implies $|\textbf{p}_{i,L-L_0-4}|\leq(1+\eps)|\bar{\textbf{p}}_{i,L-L_0-4}|$, taking $N_4$ sufficiently large yields 
$$|\partial\Psi_{i,L-L_0-4}|\leq|\textbf{p}_{i,L-L_0-4}|+|\bar{\textbf{p}}_{i,L-L_0-4}|+|\partial\Pi|\leq4.5h_i+|\partial\Pi|$$
The parameter choices $\delta^{-1}>>L>>N$ imply
$$|\partial\Pi|\leq NL+L\delta|V|_a\leq\delta^{-1}/4+h_i/L_0^2\leq h_i/2$$
As $\Psi_{i,L-L_0-4}$ contains no disks, combining the above inequalities and applying Lemma \ref{diskless} yields
$$\text{wt}_G(\Psi_{i,L-L_0-4})\leq \phi_{N_2}(5h_i)+N_1\mu(\Psi_{i,L-L_0-4})f_2(5N_1h_i)$$
while the assignment of weight implies
$$\text{wt}(\Pi)=C_1|\partial\Pi|^2\leq C_1h_i^2$$
Combining these two inequalities, the parameter choice $N_2>>C_1$ and Lemma \ref{G-weight subdiagrams} imply
$$\text{wt}_G(\bar{\Delta}_{i,L-L_0-4})\leq 2\phi_{N_2}(5h_i)+N_1\mu(\Psi_{i,L-L_0-4})f_2(5N_1h_i)$$
By Lemma \ref{7.22}, $|\partial\Psi_{i,L-L_0-4}|_\theta=2(h_i+h_{L-L_0-4})\leq4h_i$. So, by Lemma \ref{mixtures}(1), 
$$\mu(\Psi_{i,L-L_0-4})\leq16Jh_i^2$$
Hence, taking $N_2>>N_1>>J$,
\begin{equation} \label{G-weight bar Delta}
\text{wt}_G(\bar{\Delta}_{i,L-L_0-4})\leq2\phi_{N_2}(5h_i)+16N_1Jh_i^2f_2(5N_1h_i)\leq3\phi_{N_2}(5h_i)
\end{equation}
Thus, combining (\ref{G-weight Psi'}) and (\ref{G-weight bar Delta}) through Lemma \ref{G-weight subdiagrams}, it suffices to show that
%$$
%\text{wt}_G(\Delta)\leq N_4x^2-N_4\delta xh_{i+1}/50c_3+N_3\mu(\Psi_{i,L-L_0-4}')+30N_2h_i^2
%$$
%Hence, to reach a contradiction, it suffices to show that
\begin{equation} \label{7.40 suffices 1}
\frac{N_4\delta h_{i+1}}{50c_2}p(\Delta)f_2(N_4p(\Delta))\geq N_3\left(\mu(\Psi_{i,L-L_0-4}')-\mu(\Delta)\right)f_2(N_3p(\Delta))+3\phi_{N_2}(5h_i)
\end{equation}
Now consider the diagram $\Psi_{i+1,L-L_0-4}'$. When passing from $\partial\Psi_{i+1,L-L_0-4}'$ to $\partial\Psi_{i,L-L_0-4}'$, a subpath $\textbf{t}$ is replaced with $\textbf{bot}(\pazocal{Q}_i)^{-1}$ (see \Cref{fig-patht}). The subpath $\textbf{t}$ consists of:

\begin{itemize}

\item the subpath $\textbf{p}_{i,i+1}'$ of $\partial\Delta$ obtained from $\textbf{p}_{i,i+1}$ by removing the end of $\pazocal{Q}_{i+1}$,

\item $\textbf{bot}(\pazocal{Q}_{i+1})^{-1}$, and

\item a subpath of the inverse of $\partial\Pi$

\end{itemize}

\begin{figure}[H]
\centering
\includegraphics[scale=1.4]{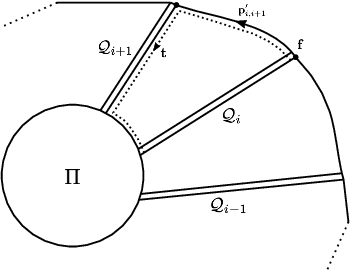}
\caption{}
\label{fig-patht}
\end{figure}

By Lemma \ref{7.22}, there is a correspondence between the $\theta$-edges of $\textbf{t}$ and those of $\textbf{bot}(\pazocal{Q}_i)$. So, since $\textbf{bot}(\pazocal{Q}_i)$ contains no $q$-edges, the necklace corresponding to $\Psi_{i,L-L_0-4}'$ may be obtained from that of $\Psi_{i+1,L-L_0-4}'$ by the removal of the black beads corresponding to the $q$-edges of $\textbf{t}$.

Consider the $q$-edge $\textbf{f}$ on the end of $\pazocal{Q}_i$. In $\partial\Psi_{i+1,L-L_0-4}'$, $\textbf{f}$ separates the $h_{i-1}-h_i$ $\theta$-edges of $\textbf{p}_{i-1,i}$ from the $h_i+h_{L-L_0-4}$ $\theta$-edges of the path $\textbf{p}_{i,i+1}'\bar{\textbf{p}}_{i+1,L-L_0-4}$.

By Lemma \ref{7.23}, $|\textbf{p}_{i,i+1}'|_q<3K_0$, while $|\bar{\textbf{p}}_{i+1,L-L_0-4}|_q\leq NL$. So, as $J>>K>>L$, there are at most $J$ $q$-edges of $\partial\Psi_{i+1,L-L_0-4}'$ between any pair of $\theta$-edges separated by $\textbf{f}$ mentioned above. As such, Lemma \ref{mixtures}(d) implies 
$$\mu(\Psi_{i+1,L-L_0-4}')-\mu(\Psi_{i,L-L_0-4}')\geq (h_{i-1}-h_i)(h_i+h_{L-L_0-4})$$
Meanwhile, Lemma \ref{7.25} implies
$$\mu(\Delta)-\mu(\Psi_{i+1,L-L_0-4}')\geq-2J|\partial\Delta|(h_{i+1}+h_{L-L_0-4})$$
Combining these and noting that $h_{L-L_0-4}\leq h\leq h_{i+1}\leq h_i$, we have
\begin{equation} \label{contradiction-mu}
\mu(\Delta)-\mu(\Psi_{i,L-L_0-4}')\geq 2h_i(h_{i-1}-h_i)-4J|\partial\Delta|h_{i+1}
\end{equation}
Note that since $p(\Delta)\geq|\partial\Delta|$, the parameter choices $N_4>>N_3>>\delta^{-1}>>J>>c_2$ imply
\begin{equation} \label{contradiction-N_4}
\frac{N_4\delta h_{i+1}}{50c_2}p(\Delta)f_2(N_4p(\Delta))\geq4N_3J|\partial\Delta|h_{i+1}f_2(N_3p(\Delta))
\end{equation}
Moreover, by \Cref{lem-combined}(4), $h_{i-1}-h_i>h_{i-1}/30c_2\geq h_i/30c_2$, and so by the parameter choices $N_3>>N_2>>c_2$ we have
\begin{equation} \label{contradiction-N_3}
2N_3h_i(h_{i-1}-h_i)f_2(N_3p(\Delta))>\frac{N_3}{15c_2}h_i^2f_2(N_3p(\Delta))\geq3\phi_{N_2}(h_i)
\end{equation}
Thus, combining (\ref{contradiction-mu}), (\ref{contradiction-N_4}), and (\ref{contradiction-N_3}), we obtain (\ref{7.40 suffices 1}).

\end{proof}

\begin{remark}

We assumed without loss of generality that $h=h_{L_0+1}$, {\frenchspacing i.e. $h_{L-L_0-4}\leq h_{L_0+1}$}. If instead $h_{L-L_0-4}>h_{L_0+1}$, then the symmetric statement to Lemma \ref{no one-step} will be needed for $L-L_0-2\leq i\leq L-1$, {\frenchspacing i.e. so that there is no two-letter subword} $Q'Q$ such that every rule of $\pazocal{D}$ inserts one letter to the right of $Q'$. This statement can be proved analogously.

\end{remark}

Finally, we reach the main contradiction of this setting, achieving the desired upper bound on $G$-weights.

\begin{lemma} \label{contradiction}

The counterexample diagram $\Delta$ cannot exist.

\end{lemma}

\begin{proof}

First, fix an integer $\eta\geq2$ dependant on $c_2$ such that $(1-\frac{1}{30c_2})^\eta<\frac{1}{64c_2}$. Note that, although $\eta$ is not listed as one of the parameters of \Cref{sec-parameters}, we may take $L_0>>\eta$ since $L_0$ is chosen after $c_2$.

For $i=1,\dots,L_0-1$, Lemma \ref{lem-combined}(4) implies $h_{i+1}<(1-\frac{1}{30c_2})h_i$. So, if $1\leq i<j\leq L_0-1$ with $j-i-1\geq\eta$, then $h_j<(1-\frac{1}{30c_2})^\eta h_{i+1}<\frac{1}{64c_2}h_{i+1}$.

For such $i,j$, Lemma \ref{lem-combined}(3) then implies that $|\textbf{z}_i|_a\geq h_{i+1}/2c_2>32h_j$. As \Cref{lem-combined}(5) implies $8h_j\geq|\textbf{z}_j|_a$, we then have $|\textbf{z}_i|_a>4|\textbf{z}_j|_a$.

Now, as $L_0>>\eta$ and $L_0>>c_0$, there exist indices $2\leq j_1<j_2<\dots<j_m\leq L_0-1$ such that $m\geq c_0$ and $j_{i+1}-j_i-1\geq\eta$. So, $|\textbf{z}_{j_i}|_a>4|\textbf{z}_{j_{i+1}}|_a$ and $h_{j_i+1}\geq64c_2h_{j_{i+1}}$.

Let $\pazocal{C}:W_0\to\dots\to W_t$ be the computation corresponding to the trapezium $\Gamma_{j_2}$ by Lemma \ref{trapezia are computations}. As $\Gamma_{j_2}$ contains a copy of $\Gamma_{j_2+1}$, which in turn contains a copy of $\Gamma_{j_2+2}$ and so on, there exist words $V_i$ in $\pazocal{C}$ for $i=1,\dots,m$ that are coordinate shifts of the labels of $\textbf{z}_{j_i}$. By the inequalities above, $|V_{i+1}|_a>4|V_i|_a$.

If for some $i$ the subcomputation $V_{i+2}\to\dots\to V_i$ is a one-step computation, then by Lemma \ref{M one-step} there exists a sector $Q'Q$ such that:

\begin{enumerate}

\item the sector's length increases with each transition of the subcomputation

\item each transition of the subcomputation inserts a letter to the left of $Q$ (or to the right of $Q'$ if $h=h_{L-L_0-4}$)

\end{enumerate} 

But since $\eta\geq2$, there would then exist a subcomputation contradicting Lemma \ref{no one-step}. 

Hence, the subcomputation $\pazocal{C}':V_m\to\dots\to W_t$ of $\pazocal{C}$ must contain at least $c_0/2\geq9$ distinct one-step computations. Lemma \ref{long step history} then implies that the step history of $\pazocal{C}'$ contains a subword of the form $(34)_j(4)_j(45)_j$ or $(54)_j(4)_j(43)_j$. Let $\pazocal{C}''$ be the subcomputation of $\pazocal{C}'$ with this step history.

Then, we may factor $H_{j_2+1}'\equiv H'H''H'''$ where $H''$ is a controlled history. Further, since the subcomputation $\pazocal{C}''$ repeats $k$ copies of a controlled history, taking $k\geq3$ allows us to assume $\|H''\|\leq\|H'\|$.

Since $h_{j_1+1}>64c_2h_{j_2}$, $H_{j_1+1}'$ has prefix $K\equiv H'H''H'''_1$ where $\|H'''_1\|=\|H'\|\geq\|H''\|$. Set $\pazocal{B}$ as the subband of the spoke $\pazocal{Q}_{j_1}$ with history $K$. Then, for any factorization $\pazocal{B}_1\pazocal{B}_2\pazocal{B}_3$ such that the sum of the lengths of $\pazocal{B}_1$ and $\pazocal{B}_3$ is at most $\frac{1}{3}\|K\|$, the history of $\pazocal{B}_2$ must contain $H''$. So, since all $\theta$-bands crossing $\pazocal{Q}_{j_1}$ must cross $\pazocal{Q}_{j_1-1}$, taking $\lambda<1/3$ implies $\pazocal{B}$ is a $\lambda$-shaft with length $\|K\|$.

However, note that the subcomputation $W_0\to\dots\to V_m$ has length at least $h_{L_0-1}\geq h$, so that $\|K\|\geq\|H'\|\geq h>\delta h$. Thus, the existence of $\pazocal{B}$ in $\pazocal{Q}_{j_1}$ contradicts Lemma \ref{lem-combined}(2).

\end{proof}

\medskip

%%%%%%%%%%%%%%%%%%%%%%%%%%%%%%%%%%%%%%%%%%%%%%%%%%

\section{Proof of \Cref{main-theorem}}

We now use the construction outlined in the previous sections to complete the proof of \Cref{main-theorem}.  

The first step toward this is to justify the assignments made throughout the construction.

\subsection{Assignment of $a$-relations} \

As mentioned in the introduction to the groups of interest in \Cref{sec-associated-groups}, the set of $a$-relators of interest, $\Omega$, is the set of non-trivial cyclically reduced words in the letters $X\cup X^{-1}$ whose value in $G$ is 1.

The following Lemma sheds some light on why these particular relations are adjoined to the group presentation.

\begin{lemma} \label{a-relations are relations}

For any word $w\in\pazocal{R}_1$, the relation $w=1$ holds in the group $G(\textbf{M})$.

\end{lemma}

\begin{proof}

By \Cref{M language}(1) there exists $H_w\in F(\Phi^+)$ such that $I(w,H_w)$ and $J(w,H_w)$ are accepted by one-machine computations of the first and second machines, respectively.  \Cref{disks are relations} then implies that $I(w,H)$ and $J(w,H)$ are both trivial over the group $G(\textbf{M})$.  These two words differ only by the insertion of the word $w$ in the `special' input section, so that $w=1$ in $G(\textbf{M})$.

%Lemmas \ref{M language} and \ref{disks are relations} imply that the words corresponding to the configurations $I(w,H)$ and $J(w,H)$ are trivial over the group $G(\textbf{M})$. These two words differ only by the insertion of the word $u^n$ in the `special' input sector, so that $u^n=1$ in $G(\textbf{M})$.

\end{proof}

\begin{lemma}[Lemma 11.2 of \cite{W}] \label{G isomorphic to G_a}

The groups $G(\textbf{M})$ and $G_\Omega(\textbf{M})$ are isomorphic.

\end{lemma}

\begin{proof}

Consider the map $\varphi:X\to G(\textbf{M})$ sending each letter to its natural copy in the tape alphabet of the `special' input sector.  As $\pazocal{P}_1=\gen{X\mid\pazocal{R}_1}$ is a presentation for $G$, \Cref{a-relations are relations} and the theorem of von Dyck \cite{vonDyck} imply that $\varphi$ extends to a homomorphism $G\to G(\textbf{M})$.  In particular, for any word $w\in\Omega$, $w=1$ in $G(\textbf{M})$.

%Identify $B(\pazocal{A},n)$ with the presentation $\gen{\pazocal{A}\mid w=1, w\in\pazocal{L}}$.

%Then let $\varphi:\pazocal{A}\to G(\textbf{M})$ be the map sending each letter to its natural copy in the tape alphabet of the `special' input sector. By the theorem of von Dyck (Theorem 4.5 in [16]), Lemma \ref{a-relations are relations} implies that $\varphi$ extends to a homomorphism $B(\pazocal{A},n)\to G(\textbf{M})$. So, for any word $w$ corresponding to an $a$-relation $w=1$, the relation $w=1$ holds in $G(\textbf{M})$. 

The theorem of von Dyck then implies that the map sending each generator of the canonical presentation of $G(\textbf{M})$ to the corresponding generator of the disk presentation of $G_\Omega(\textbf{M})$ extends to an isomorphism between the two groups.

\end{proof}

\begin{lemma} \label{embedding}

The group $G$ embeds in the group $G(\textbf{M})$.

\end{lemma}

\begin{proof}

Consider the map $\varphi:X\to G_\Omega(\textbf{M})$ sending the elements of $X$ to their copies in the tape alphabet of the `special' input sector. The theorem of von Dyck implies that this extends to a homomorphism $\varphi:G\to G_\Omega(\textbf{M})$.

Now suppose the reduced word $w$ over $X$ satisfies $\varphi(w)=1$. Then by Lemma \ref{minimal exist}, there exists a minimal diagram $\Delta$ over $G_\Omega(\textbf{M})$ satisfying $\text{Lab}(\partial\Delta)\equiv w$. By Lemmas \ref{graph} and \ref{G_a theta-annuli}, every cell of $\Delta$ must be an $a$-cell. But then this is a diagram over the presentation $\gen{X\mid\pazocal{R}_1}$ for $G$, so that $w=1$ in $G$.

Thus, $\varphi:G\to G_\Omega(\textbf{M})$ is an embedding. Lemma \ref{G isomorphic to G_a} then implies the statement.

\end{proof}

%%%%%%%%%%%%%%%%%%%%%%%%%%%%%%%%%%%%%%%%%%%%%%%%%%

\subsection{Assignment of weights and $G$-weight} \label{sec-assignment-of-G-weight} \

Now we wish to justify our assignment of weights to $a$-cells and disks over the disk presentation of $G_\Omega(\textbf{M})$. To do so, we first study areas of a diagram over the canonical presentation of $G(\textbf{M})$ with contour label corresponding to a disk relation.

\begin{lemma} \label{disks are quadratic} \

(1) For any configuration $W$ accepted by $\textbf{M}$, there exists a reduced diagram $\Delta$ over the canonical presentation of $G(\textbf{M})$ such that $\lab(\partial\Delta)\equiv W$ and $\text{Area}(\Delta)\leq C_1|W|^2$.

(2) For any $w\in\pazocal{R}_1$, there exists a reduced diagram $\Delta$ over the canonical presentation of $G(\textbf{M})$ with $\lab(\partial\Delta)\equiv w$ and $\text{Area}(\Delta)\leq Kf_1(\|w\|)^2+K$.

\end{lemma}

\begin{proof}

(1) If $W=W_{ac}$, then $W$ is a relator in $G(\textbf{M})$, so that the statement is satisfied for $\Delta$ a single cell.

Otherwise, there exists a one-machine computation $\pazocal{C}:W\equiv W_0\to\dots\to W_t\equiv W_{ac}$ accepting $W$, so that \Cref{M projected long history one-machine} implies $t\leq c_1\|W(i)\|$ for all $i\geq2$. \Cref{simplify rules} and the parameter choice $c_2>>c_1$ then imply $\|W_j\|\leq c_2\|W\|$ for all $j$.

By Lemma \ref{computations are trapezia}, we can then build a trapezium $\Gamma$ over $M(\textbf{M})$ corresponding to $\pazocal{C}$, so that $\lab(\textbf{tbot}(\Gamma))\equiv W$ and $\lab(\textbf{ttop}(\Gamma))\equiv W_{ac}$ with $\text{Area}(\Gamma)\leq c_1c_2\|W\|^2$.

As the $Q_{s,r}''(L)\{t(1)\}$-sector is locked by every rule, the sides of $\Gamma$ are labelled identically and no trimming was necessary. So, we may glue these sides together and paste a hub into the middle of the diagram. This produces a reduced diagram $\Delta$ over the canonical presentation of $G(\textbf{M})$ with $\lab(\partial\Delta)\equiv W$ and satisfying $\text{Area}(\Delta)\leq c_1c_2\|W\|^2+1$. 

The statement then follows as we choose the parameter $C_1$ after $c_2$, $c_1$, and $\delta$.

(2) By \Cref{M language}(1), there exists $H_w\in F(\Phi^+)$ with $\|H_w\|\leq c_1f_1(c_1\|w\|)+c_1$ such that $I(w,H_w)$ and $J(w,H_w)$ are disk relators.

As in the previous case, we can then build reduced diagrams $\Delta_1$ and $\Delta_2$ over the canonical presentations of $G(\textbf{M})$ where $\Delta_j$ is made of a hub and a trapezium satisfying: 
\begin{itemize}
\item $\text{Lab}(\Delta_1)\equiv I(w,H_w)$ and $\text{Area}(\Delta_1)\leq c_1c_2\|I(w,H)\|^2+1$

\item $\text{Lab}(\Delta_2)\equiv J(w,H_w)$ and $\text{Area}(\Delta_2)\leq c_1c_2\|J(w,H)\|^2+1$
\end{itemize}

Note that $\|I(w,H_w)\|,\|J(w,H_w)\|\leq L(N+\|w\|+\|H_w\|)\leq L(N+(c_1+1)f_1(c_1\|w\|))$. So, since $K$ is chosen after $c_2$, $c_1$, $L$, and $n$, we can assume that $\text{Area}(\Delta_j)\leq \frac{1}{2}K(f_1(\|w\|)^2+1)$ for $j=1,2$.

Gluing $\Delta_1$ and $\Delta_2$ along their common contours (and making any possible cancellations) then yields a diagram $\Delta$ satisfying the statement.

\end{proof}

\begin{lemma} \label{a-cells are quadratic} 

If $w$ is a reduced word over the alphabet $X$ such that $w=1$ in $G$, then there exists a reduced diagram $\Delta$ over the canonical presentation of $G(\textbf{M})$ with $\lab(\partial\Delta)\equiv w$ and satisfying $\text{Area}(\Delta)\leq\phi_{C_1}(\|w\|)$.

\end{lemma}

\begin{proof}

Let $\Delta_0$ be a van Kampen diagram over the presentation $\pazocal{P}_2=\gen{X\mid\pazocal{R}_2}$ of with $\lab(\partial\Delta_0)\equiv w$. For each cell $\Pi_0$ in $\Delta_0$, $\lab(\partial\Pi_0)\in\pazocal{R}_2\subset\pazocal{R}_1$. Setting $\lab(\partial\Pi_0)\equiv u(\Pi_0)$, Lemma \ref{disks are quadratic}(2) then implies that there exists a diagram $\Pi$ over the canonical presentation of $G(\textbf{M})$ satisfying $\lab(\partial\Pi)\equiv u(\Pi_0)$ and $\text{Area}(\Pi)\leq C_1f_1(\|u(\Pi_0)\|)^2+C_1$.

Pasting $\Pi$ in place of $\Pi_0$ for each cell of $\Delta_0$ then produces a diagram $\Delta$ over the canonical presentation of $G(\textbf{M})$ satsifying $\lab(\partial\Delta)\equiv w$ and $$\text{Area}(\Delta)=\sum\text{Area}(\Pi)\leq\sum\limits_{\Pi_0\in\Delta_0} Kf_1(\|u(\Pi_0)\|)^2+K$$
As $1\notin\pazocal{R}_2$ implies $\|u(\Pi_0)\|\geq1$, $f_1(\|u(\Pi_0)\|)\geq1$.  Hence,
$$\text{Area}(\Delta)\leq\sum_{\Pi_0\in\Delta_0} 2Kf_1(\|u(\Pi_0)\|)^2$$
But then the mass condition given in hypothesis (4) of the statement of \Cref{main-theorem} and the parameter choices $C_1>>K$ imply $$\sum\limits_{\Pi_0\in\Delta_0}2Kf_1(\|u(\Pi_0)\|)^2\leq\phi_{C_1}(\|w\|)$$

\end{proof}

Our next goal is to justify the choices of $G$-weight assigned to reduced diagrams over $G_\Omega(\textbf{M})$.  Given the definition laid out in \Cref{sec-G-weight}, we first focus on impeding and big $a$-trapezia.

\begin{lemma} \label{impeding G-weight}

Let $\Delta$ be an impeding $a$-trapezium. Then there exists a reduced diagram $\tilde{\Delta}$ over the canonical (finite) presentation of $G(\textbf{M})$ such that $\lab(\partial\tilde{\Delta})\equiv\lab(\partial\Delta)$ and $\text{Area}(\tilde{\Delta})\leq2\text{wt}_G(\Delta)$.

\end{lemma}

\begin{proof}

First, suppose $\text{wt}_G(\Delta)=\frac{1}{2}\text{wt}(\Delta)$.  

For any $a$-cell $\pi$ in $\Delta$, \Cref{a-cells are quadratic} allows us to excise $\pi$ and paste in its place a diagram $\tilde{\pi}$ over the canonical presentation of $G(\textbf{M})$ with $\lab(\partial\tilde{\pi})\equiv\lab(\partial\pi)$ and $\text{Area}(\tilde{\pi})\leq\text{wt}(\pi)$.  Carrying out this surgery for every $a$-cell and making any necessary cancellations then produces a reduced diagram $\tilde{\Delta}$ over the canonical presentation of $G(\textbf{M})$ satisfying the statement.

Hence, we may assume 
$$\text{wt}_G(\Delta)=C_2h\max(\|\textbf{tbot}(\Delta)\|,\|\textbf{ttop}(\Delta)\|)+\phi_{C_3}(\|\textbf{tbot}(\Delta)\|+\|\textbf{ttop}(\Delta)\|)$$
Note that by hypothesis (2) of \Cref{main-theorem}, we assume the Move machine $\textbf{Move}$ satisfies the $(n^2f_2(n),G)$-move condition.  By \Cref{def-Move condition} there then exists a reduced diagram $\Phi$ over $M_\Omega(\textbf{M})$ with $\lab(\partial\Phi)\equiv\lab(\partial\Delta)$ and $\text{wt}(\Phi)\leq \text{wt}_G(\Delta)$.

As above, replacing every $a$-cell of $\Phi$ with the corresponding reduced diagram given by \Cref{a-cells are quadratic} (and making any necessary cancellations) then produces reduced diagram $\tilde{\Delta}$ over the canonical presentation of $G(\textbf{M})$ with $\lab(\partial\tilde{\Delta})\equiv\lab(\partial\Phi)\equiv\lab(\partial\Delta)$ and $\text{Area}(\tilde{\Delta})\leq\text{wt}(\Phi)\leq\text{wt}_G(\Delta)$.

\end{proof}

\begin{lemma}[Compare to Lemma 11.7 of \cite{W}] \label{big G-weight}

For every big $a$-trapezium $\Delta$, there is a reduced diagram $\tilde{\Delta}$ over the canonical (finite) presentation of $G(\textbf{M})$ such that $\lab(\partial\tilde{\Delta})\equiv\lab(\partial\Delta)$ and $\text{Area}(\tilde{\Delta})\leq2\text{wt}_G(\Delta)$.

\end{lemma}

\begin{proof}

As in the proof of Lemma \ref{impeding G-weight}, if $\text{wt}_G(\Delta)=\frac{1}{2}\text{wt}(\Delta)$, then we may construct $\tilde{\Delta}$ simply by replacing all $a$-cells with the corresponding diagram given by \Cref{a-cells are quadratic} and making any necessary cancellations. Hence, it suffices to assume that $$\text{wt}_G(\Delta)=h+\phi_{C_2}(\|\textbf{tbot}(\Delta)\|+\|\textbf{ttop}(\Delta)\|)$$
As $\Delta$ is big, its base must be reduced.  Hence, perhaps passing to a cyclic permutation and the opposite orientation, we may assume without loss of generality that the base of $\Delta$ begins and ends with $\{t(1)\}$.  Let $\Delta_1,\dots,\Delta_L$ be the subdiagrams bounded by consecutive maximal $t$-bands, so that each is an $a$-trapezium with pararevolving base.

Let $\Delta'$ be the maximal subdiagram of $\Delta$ which is an $a$-trapezium with controlled history.  As the `special' input sector is locked by every rule of a controlled history, $\Delta'$ must be a trapezium.  Hence, Lemmas \ref{M controlled} and \ref{trapezia are computations} yield disk relators $W_0'$ and $W_1'$ such that $W_0't(1)\equiv\lab(\textbf{bot}(\Delta'))$ and $W_1't(1)\equiv\lab(\textbf{top}(\Delta'))$.

On the other hand, \Cref{big a-trapezium accepted} produces disk relators $W_0$ and $W_1$ such that $W_0t(1)$ and $W_1t(1)$ differ from $\lab(\textbf{bot}(\Delta))$ and $\lab(\textbf{top}(\Delta))$, respectively, only in the tape word of the `special' input sector.

By \Cref{M controlled} and the parallel nature of the rules of $\textbf{M}$, it then follows that $W_0t(1)$ and $W_ht(1)$ can only differ from $\lab(\textbf{bot}(\Delta))$ and $\lab(\textbf{top}(\Delta))$, respectively, in the tape word of the `special' input sector.

Let $w_0\in F(X)$ be the word by which $W_0t(1)$ and $\lab(\textbf{bot}(\Delta))$ differ (perhaps $w_0=1$).  Consider the diagram $\Delta_0'$ obtained by passing to the subdiagram bounded by $\textbf{bot}(\Delta')$ and $\textbf{bot}(\Delta)$ and removing the final $t$-band.  Letting $H_0$ be the history of this $a$-trapezium, $\Delta_0'$ has boundary label $H_0^{-1}\lab(\textbf{bot}(\Delta))t(1)^{-1}H_0(W_0')^{-1}$.  As $W_0'$ is an disk relation, it follows from \Cref{embedding} that $w_0$ represents the identity in $G$, {\frenchspacing i.e. $w_0$ is freely conjugate to a word $w_0'\in\Omega\cup\{1\}$ with $\|w_0'\|\leq\|w_0\|$}.

Note that
$$\|w_0\|\leq|\textbf{bot}(\Delta_1)|_a+|W_0(1)|_a\leq|\textbf{bot}(\Delta_1)|_a+|W_0(2)|_a=|\textbf{bot}(\Delta_1)|_a+|\textbf{bot}(\Delta_2)|_a\leq|\textbf{bot}(\Delta)|_a$$
By the symmetric argument, for $w_1\in F(X)$ such that $W_1t(1)$ and $\lab(\textbf{top}(\Delta))$ differ, then $w_1$ is freely conjugate to a word $w_1'\in\Omega\cup\{1\}$ with $\|w_1'\|\leq\|w_1\|\leq|\textbf{top}(\Delta)|_a$.

As a result, we may construct a reduced diagram $\tilde{\Delta}_1$ over the disk presentation of $G_\Omega(\textbf{M})$ with $\lab(\partial\tilde{\Delta}_1)\equiv\lab(\partial\Delta)$ consisting of:

\begin{itemize}

\item disks corresponding to $W_0$ and $W_1$

\item $a$-cells corresponding to $w_0'$ and $w_1'$ (should they be nontrivial, and so elements of $\Omega$)

\item a single $t$-band corresponding to $\{t(1)\}$ of length $h$.

\end{itemize}

Hence, $\text{wt}(\tilde{\Delta}_1)\leq C_1|W_0|^2+C_1|W_1|^2+\phi_{C_1}(|\textbf{bot}(\Delta)|_a+|\textbf{top}(\Delta)|_a)+h$.

Now, note that $W_0(j)t(j+1)\equiv\lab(\textbf{bot}(\Delta_j)$ for all $j\in\{2,\dots,L\}$.  So, since $|W_0(1)|_a\leq|W_0(2)|_a$, $\|W_0\|\leq2\|\textbf{bot}(\Delta)\|$.

Similarly, $\|W_1\|\leq2\|\textbf{top}(\Delta)\|$.

So, since $C_2>>C_1>>\delta^{-1}$, it follows that $\text{wt}(\tilde{\Delta})\leq\text{wt}_G(\Delta)$.

Thus, we obtain the desired result by constructing the reduced diagram $\tilde{\Delta}$ over the canonical presentation of $G(\textbf{M})$ from $\tilde{\Delta}_1$ as follows:

\begin{itemize}

\item Excise any disk and paste in its place the diagram arising from \Cref{disks are quadratic}.

\item Excise any $a$-cell and paste in its place the diagram arising rom \Cref{a-cells are quadratic}.

\item Make any necessary cancellations.

\end{itemize}

\end{proof}

\medskip

%%%%%%%%%%%%%%%%%%%%%%%%%%%%%%%%%%%%%%%%%%%%%%%%%%

\subsection{Dehn function bound} \

We now obtain the desired bound on the Dehn function of $G(\textbf{M})$.

%Finally, we complete the proof of Theorem \ref{main theorem}.

%As Lemma \ref{embedding} implies that $G(\textbf{M})$ contains an infinite torsion subgroup, the Dehn function of $G(\textbf{M})$ is at least quadratic. Thus, it suffices to prove a quadratic upper bound bound.

Let $w\in F(\pazocal{Y})$ such that $w=1$ in $G(\textbf{M})$. By Lemma \ref{G isomorphic to G_a}, $w$ is also trivial over the group $G_\Omega(\textbf{M})$, so that Lemma \ref{minimal exist} yields a minimal diagram $\Delta_a$ over $G_\Omega(\textbf{M})$ with $\lab(\partial\Delta_a)\equiv w$. By Lemma \ref{contradiction}, we have
$$\text{wt}_G(\Delta_a)\leq\phi_{N_4}(|w|+\sigma_\lambda(\Delta_a^*))+N_3\mu(\Delta_a)f_2(N_3|w|+N_3\sigma_\lambda(\Delta_a^*))$$
Lemma \ref{G_a design} implies that $\sigma_\lambda(\Delta_a^*)\leq C_1|w|$, while Lemma \ref{mixtures}(a) implies $\mu(\Delta_a)\leq J|w|^2$. So, as $|w|\leq\|w\|$, we can choose $N_5$ large enough so that
$$\text{wt}_G(\Delta_a)\leq\frac{1}{2}\phi_{N_5}(\|w\|)$$
Now, let $\textbf{P}$ be a minimal covering of $\Delta_a$ and construct the reduced diagram $\Delta$ over the canonical presentation of $G(\textbf{M})$ by:

\begin{itemize}

\item excising any impeding $a$-trapezium $P\in\textbf{P}$ and pasting in its place the reduced diagram given in Lemma \ref{impeding G-weight} with the same contour label and area at most $2\text{wt}_G(P)$

\item excising any big $a$-trapezium $P\in\textbf{P}$ and pasting in its place the reduced diagram given in Lemma \ref{big G-weight} with the same contour label and area at most $2\text{wt}_G(P)$

\item excising any disk $\Pi\in\textbf{P}$ and pasting in its place the reduced diagram given in Lemma \ref{disks are quadratic} with the same contour label and area at most $C_1|\partial\Pi|^2$

\item excising any $a$-cell $\pi\in\textbf{P}$ and pasting in its place the reduced diagram given in Lemma \ref{a-cells are quadratic} with the same contour label and area at most $\phi_{C_1}(\|\partial\pi\|)$

\end{itemize}

By the definition of $G$-weight, it follows that $\text{Area}(\Delta)\leq2\text{wt}_G(\Delta_a)\leq \phi_{N_5}(\|w\|)$.

Therefore, the Dehn function of $G(\textbf{M})$ is at most $\phi_{N_5}(n)\sim n^2f_2(n)$.

\medskip

%%%%%%%%%%%%%%%%%%%%%%%%%%%%%%%%%%%%%%%%%%%%%%%%%%

\subsection{$g$-diagrams and $g$-minimal diagrams} \

To complete the proof of \Cref{main-theorem}, we must display that the embedding given by $\varphi$ is bi-Lipschitz.  To do so, we follow the proof outlined in \cite{W}.

%By Lemma \ref{embedding}, every $g\in G$ can be identified with an element of $G_\Omega(\textbf{M})$, namely $\varphi(g)$. %For $g\in G$, define $|g|_X$ as the smallest number of letters comprising a word over $X$ whose value in $B(\pazocal{A},n)$ is $g$.

For $g\in G$, a minimal diagram $\Delta$ is called a \textit{$g$-diagram} if $\partial\Delta=\textbf{st}$, $\lab(\textbf{t})$ is a word over $X$ whose value in $G$ is $g^{-1}$, and $\|\textbf{t}\|=|g|_X$. 

A $g$-diagram $\Delta$ is called \textit{$g$-minimal} if $p(\Delta)=|\partial\Delta|+\sigma_\lambda(\Delta)$ is minimal amongst all $g$-diagrams.

The next lemma consists of two statements which follow in exactly the same way as their analogues in \cite{W}.

\begin{lemma} [Lemmas 12.1-12.2 of \cite{W}] \label{12.1-12.2}

Let $\Delta$ be a $g$-minimal diagram for some $g\in G$.

\begin{enumerate}

\item $p(\Delta)\leq 2\delta|g|_X$

\item No $q$-band of $\Delta$ has two ends on $\partial\Delta$

\end{enumerate}

\end{lemma}

Let $\Delta$ be a $g$-minimal diagram for some $g\in G$ and decompose $\partial\Delta=\textbf{st}$ as in the definition of $g$-diagram. Suppose $\Delta$ contains a quasi-rim $\theta$-band $\pazocal{T}$. Since $\textbf{t}$ is comprised entirely of $a$-edges, $\pazocal{T}$ must end twice on $\textbf{s}$. Let $\textbf{s}_0$ be the subpath of $\partial\Delta$ bounded by the two ends of $\pazocal{T}$ such that, per the definition of quasi-rim $\theta$-band, any cell between $\textbf{bot}(\pazocal{T})$ (or $\textbf{top}(\pazocal{T})$) and $\textbf{s}_0$ is an $a$-cell. If $\textbf{s}_0$ is a subpath of $\textbf{s}$, then $\pazocal{T}$ is called a \textit{$g$-rim $\theta$-band}.

Again, the following statements are proved by identical arguments as the ones employed in \cite{W} for the analogous statements.

\begin{lemma} [Lemmas 12.3-12.4 of \cite{W}] \label{12.3-12.4}

Let $\Delta$ be a $g$-minimal diagram for some $g\in G$ and decompose $\partial\Delta=\textbf{st}$ as in the definition of $g$-diagram.

\begin{enumerate}

\item The base of any $g$-rim $\theta$-band in $\Delta$ has length greater than $K$.

\item If $\Delta$ contains no disks, then $|\textbf{s}|=|\textbf{t}|=\delta|g|_X$.

\end{enumerate}

\end{lemma}

%\medskip
%
%%%%%%%%%%%%%%%%%%%%%%%%%%%%%%%%%%%%%%%%%%%%%%%%%%%
%
%\subsection{$g$-minimal diagrams containing disks} \

\begin{lemma}[Lemma 12.5 of \cite{W}] \label{distortion graph}

Let $\Delta$ be a $g$-minimal diagram for some $g\in G$ containing at least one disk. Decompose $\partial\Delta=\textbf{st}$ as in the definition of $g$-diagram. Then $\Delta$ contains a disk $\Pi$ such that:

\begin{enumerate}[label=({\alph*})]

\item $L-6$ consecutive $t$-spokes $\pazocal{Q}_1,\dots,\pazocal{Q}_{L-6}$ of $\Pi$ end on $\partial\Delta$

\item for $i=1,\dots,L-7$, the subdiagram $\Psi_{i,i+1}$ of $\Delta$ bounded by $\pazocal{Q}_i$, $\pazocal{Q}_{i+1}$, $\partial\Pi$, and $\partial\Delta$ contains no disks, and

\item $\textbf{t}$ is not a subpath of $\partial\Psi_{i,i+1}$ for any $i=1,\dots,L-7$.

\end{enumerate}

\end{lemma}

We are now able to adapt many of the arguments of \Cref{sec-disks-upper-bound} to the setting of a $g$-minimal diagram containing at least one disk:

Recall that the arguments presented in \Cref{sec-disks-upper-bound} center on the disk $\Pi$ given by \Cref{graph} in a weakly minimal diagram $\Delta$.  In general, these arguments use $\Pi$ to find subdiagrams $\Gamma$ with $p(\Gamma)<p(\Delta)$ which allow us to estimate $\text{wt}_G(\Delta)$ in an advantageous way.

If the $g$-minimal diagram $\Delta$ contains a disk, then the disk $\Pi$ given by \Cref{distortion graph} can be treated analogously.  Thus, as $g$-minimal diagrams are defined to be those with minimal parameter $p(\Delta)$, using \Cref{12.1-12.2}(2) in place of \Cref{7.21} (and altering some estimates) allows us to adapt many of the arguments of \Cref{sec-disks-upper-bound}.  

This is done in much the same way as in Section 12.2 of \cite{W}, which arrives at the following statement:

\begin{lemma}[Lemmas 12.6-12.21 of \cite{W}] \label{distortion contradiction}

For $g\in G$, a $g$-minimal diagram contains no disks.

\end{lemma}
%
%\begin{proof}
%
%The proof follows the same outline as that of Lemma \ref{contradiction}.
%
%For $\eta\geq2$ an integer such that $\left(1-\frac{1}{30c_3}\right)^\eta<\frac{1}{64c_3}$, Lemma \ref{distortion h_i+1} implies that if $1\leq i<j\leq L_0-1$ with $j-i-1\geq\eta$, then $h_j<\frac{1}{64c_3}h_{i+1}$.
%
%For such $i,j,$ Lemma \ref{distortion big h_i+1} implies that $|\textbf{z}_i|_a>32h_j$, while Lemma \ref{distortion number of cloves} implies $8h_j\geq|\textbf{z}_j|_a$. As a result, $|\textbf{z}_i|_a>4|\textbf{z}_j|_a$.
%
%Taking $L_0>>\eta$ and $L_0>>c_0$, there exist indices $2\leq j_1<j_2<\dots<j_m\leq L_0-1$ such that $m\geq c_0$ and $j_{i+1}-j_i-1\geq\eta$. So, $|\textbf{z}_{j_i}|_a>4|\textbf{z}_{j_{i+1}}|_a$ and $h_{j_i+1}\geq64c_3h_{j_i+1}$.
%
%Let $\pazocal{C}:W_0\to\dots\to W_t$ be the computation corresponding to the trapezium $\Gamma_{j_2}$. As $\Gamma_{j_2}$ contains a copy of $\Gamma_{j_2+1}$, which in turn contains a copy of $\Gamma_{j_2+2}$ and so on, there exist words $V_i$ in $\pazocal{C}$ for $i=1,\dots,m$ that are coordinate shifts of $\lab(\textbf{z}_{j_i})$. Note that $|V_{i+1}|_a>4|V_i|_a$.
%
%As in the proof of Lemma \ref{contradiction}, the subcomputations $V_{i+2}\to\dots\to V_i$ cannot be one-step, as an application of Lemma \ref{M one-step} would lead to a contradiction of Lemma \ref{distortion no one-step}.
%
%But then this implies that $\pazocal{Q}_{j_1}$ contains a $\lambda$-shaft of length at least $h_{L_0}$, contradicting Lemma \ref{distortion no shafts}.
%
%\end{proof}

\medskip

%%%%%%%%%%%%%%%%%%%%%%%%%%%%%%%%%%%%%%%%%%%%%%%%%%

\subsection{Distortion} \

Since $X$ can be identified with a subset of the generators $\pazocal{Y}$ of $G(\textbf{M})$ (namely those corresponding to the tape alphabet of the `special' input sector), it is immediate from the definition of the word norm that $|g|_\pazocal{Y}\leq|g|_X$ for all $g\in G$.  %Hence, it suffices to find a constant $K>0$ such that $|g|_X\leq K|g|_\pazocal{Y}$ for any $g\in G$.

Conversely, fix $g\in G$ and let $\Delta$ be a $g$-minimal diagram. 

Let $w\in F(\pazocal{Y})$ so that the value of $w$ in $G_\Omega(\textbf{M})$ is $g$ and $|w|$ is minimal for all such words. Further, let $\Gamma$ be a minimal diagram over $G_\Omega(\textbf{M})$ such that $\lab(\partial\Gamma)\equiv wv^{-1}$ for $v\in F(X)$ such that the value of $v$ in $G$ is $g$ and $\|v\|=|g|_X$.

Then, $\Gamma$ is a $g$-diagram, so that $p(\Gamma)\geq p(\Delta)$.

By Lemmas \ref{distortion contradiction} and \ref{12.3-12.4}(2), $p(\Delta)=|\partial\Delta|=2\delta|g|_X$.  On the other hand, Lemma \ref{lengths}(c) implies $|\partial\Gamma|\leq|w|+\delta|g|_X$. Hence, $|w|+\sigma_\lambda(\Gamma)\geq\delta|g|_X$.

As Lemma \ref{G_a design} implies $\sigma_\lambda(\Gamma)\leq C_1|\partial\Gamma|_\theta\leq C_1|w|_\theta\leq C_1|w|$, it follows that $|w|\geq\delta(C_1+1)^{-1}|g|_X$.

But $\|u\|\geq|u|$ for any word $u$ in the alphabet $\pazocal{Y}\cup\pazocal{Y}^{-1}$, so that the minimality of $|w|$ implies $|g|_\pazocal{Y}\geq|w|\geq\delta(C_1+1)^{-1}|g|_X$.

Thus, $|g|_X\leq\delta^{-1}(C_1+1)|g|_\pazocal{Y}$ for all $g\in G$, completing the proof of \Cref{main-theorem}.

\medskip

%%%%%%%%%%%%%%%%%%%%%%%%%%%%%%%%%%%%%%%%%%%%%%%%%%

\section{Proof of \Cref{main-corollary}}

To prove \Cref{main-corollary}, we invoke the construction of the `initial embedding' outlined in Section 3 of \cite{WMal} (adapted from \cite{OSQ}) to embed $G$ into a group $\tilde{G}$ suitable for the application of \Cref{main-theorem}.

Letting $X=\{x_1,\dots,x_k\}$, define the alphabet $\tilde{X}=\{x_{1,1},\dots,x_{1,M},\dots,x_{k,1},\dots,x_{k,M}\}$.  For each $1\leq i\leq k$, denote the reduced word $\tilde{x}_i=x_{i,1}\dots x_{i,M}$ over $\tilde{X}$.  

Given a word $w=w(x_1,\dots,x_k)$ over $X$, define the corresponding word $\tilde{w}=w(\tilde{x}_1,\dots,\tilde{x}_k)$ over $\tilde{X}$.  For example, if $w=x_1x_k^{-1}$, then $\tilde{w}=\tilde{x}_1\tilde{x}_k^{-1}=x_{1,1}\dots x_{1,M}x_{k,M}^{-1}\dots x_{1,M}^{-1}$.

Now, for $i=1,2$, define the set $\tilde{\pazocal{R}}_i=\{\tilde{r}_i\mid r_i\in\pazocal{R}_i\}$ and the presentation $\tilde{\pazocal{P}}_i=\gen{\tilde{X}\mid\tilde{\pazocal{R}}_i}$ for some finitely generated group $\tilde{G}_i$.

The next statements then follow in the same way as their analogues in \cite{WMal}.

\begin{lemma}[Lemma 3.2 of \cite{W}] \label{Blow-up Embedding}

There exists an embedding $\varphi:G\to\tilde{G}_i$ given by the map $w\mapsto\tilde{w}$.

\end{lemma}

\begin{lemma}[Lemma 3.3 of \cite{W}] \label{Blow-up Greendlinger}

Let $W$ be a non-trivial cyclically reduced word over $\tilde{X}$ which represents the identity in $\tilde{G}_i$.  Then there exists a non-trivial subword of a cyclic permutation of $W$ which is a cyclic permutation of a word $\tilde{w}$ for some word $w$ over $X$ which represents the identity in $G$.  In particular, $\|W\|\geq \|\tilde{w}\|\geq M$.

\end{lemma}

\begin{lemma}[Lemma 3.4 of \cite{W}] \label{Blow-up bi-Lipschitz}

For any $g\in G$, $|\varphi(g)|_{\tilde{X}}=M|g|_X$.  In particular, $\varphi$ is a bi-Lipschitz embedding.

\end{lemma}

\begin{remark}

Note that in the setting of \cite{WMal}, there are additional assumptions made about the set of relators $\pazocal{R}_i$.  In particular, it is assumed that $\pazocal{R}_i$ consists of positive words and is recursive as a subset of the positive words in $X$.  However, these assumptions are not necessary for the diagrammatic arguments that imply Lemmas \ref{Blow-up Embedding}-\ref{Blow-up bi-Lipschitz}.

\end{remark}

\begin{lemma} \label{Blow-up iso}

The identity map on $\tilde{X}$ induces an isomorphism $\tilde{G}_1\cong\tilde{G}_2$.

\end{lemma}

\begin{proof}

As $\pazocal{R}_2\subseteq\pazocal{R}_1$ implies $\tilde{\pazocal{R}}_2\subseteq\tilde{\pazocal{R}}_1$, it suffices to show that any non-trivial cyclically reduced word $W$ over $\tilde{X}$ which represents the identity in $\tilde{G}_1$ also represents the identity in $\tilde{G}_2$.

By \Cref{Blow-up Greendlinger}, there exists a cyclic permutation $W'$ of $W$ with a non-trivial prefix $v'$ which is a cyclic permutation of a word $\tilde{w}$ for some word $w$ over $X$ which represents the identity in $G$.  By definition, $\tilde{w}$ represents the identity in $\tilde{G}_2$, and so $v'$ does also.  Letting $W'\equiv v'W''$, it then suffices to show that $W''$ represents the identity in $\tilde{G}_2$. 
 
As $v'$ and $W'$ both represent the identity in $\tilde{G}_1$, $W''$ also represents the identity in $\tilde{G}_1$.  So, letting $W_1$ be a cyclically reduced word over $\tilde{X}$ which is freely conjugate to $W''$, it follows that $W_1$ represents the identity in $\tilde{G}_1$ and $\|W_1\|<\|W\|$.

But an inductive argument then implies $W_1$ represents the identity in $\tilde{G}_2$, so that $W''$ does also.

\end{proof}

Hence, \Cref{Blow-up iso} implies $\tilde{\pazocal{P}}_1$ and $\tilde{\pazocal{P}}_2$ are presentations for the same finitely generated group $\tilde{G}$, while \Cref{Blow-up Embedding} and \Cref{Blow-up bi-Lipschitz} imply the existence of a bi-Lipschitz embedding of $G$ into $\tilde{G}$.

\begin{remark}

The embedding of $G$ into $\tilde{G}$ can also be shown to satisfy several other properties.  For example, we can show that the embedding is Frattini (see Lemma 3.5 of \cite{WMal}) and satisfies the Congruence Extension Property (see Section 15 of \cite{WMal}).  As such, the embedding of \Cref{main-corollary} also satisfies these properties since that of \Cref{main-theorem} does.

\end{remark}

\begin{lemma} \label{Blow-up S-machine}

There exists an $S$-machine $\tilde{\textbf{S}}$ that recognizes $\tilde{\pazocal{R}}_1$ and satisfies $\TM_{\tilde{\textbf{S}}}\preceq f_1$.

\end{lemma}

\begin{proof}

By hypothesis, there exists an $S$-machine $\textbf{S}$ which recognizes $\pazocal{R}_1$ and satisfies $\TM_\textbf{S}\preceq f_1$.  Without loss of generality, suppose $\textbf{S}$ satisfies the statement of \Cref{simplify rules}.

Consider the machine $\tilde{\textbf{S}}$ obtained from $\textbf{S}$ as follows:

The standard base of $\tilde{\textbf{S}}$ is the same as that of $\textbf{S}$.  Further, the tape alphabets are all the same except for the input alphabet, which is $\tilde{X}$ (instead of $X$).

Any positive rule of $\textbf{S}$ which does not multiply by a letter in the input tape has a corresponding rule in $\tilde{\textbf{S}}$.  However, any positive rule $\theta$ of $\textbf{S}$ that does multiply by a letter in the input tape corresponds to $M$ positive rules $\tilde{\theta}_1,\dots,\tilde{\theta}_M$ of $\tilde{\textbf{S}}$, where:

\begin{itemize}

\item $\tilde{\theta}_1$ operates in the non-input sectors in the same way that $\theta$ does

\item If $\theta$ multiplies the input sector on the right by $x_i$ (respectively on the left by $x_i^{-1}$), then each $\tilde{\theta}_j$ multiplies the input sector on the right by $x_{i,j}$ (respectively on the left by $x_{i,j}^{-1}$)

\item If $\theta$ multiplies the input sector on the left by $x_i$ (respectively on the right by $x_i^{-1}$), then each $\tilde{\theta}_j$ multiplies the input sector on the left by $x_{i,M-j+1}$ (respectively on the right by $x_{i,M-j+1}^{-1}$)

\end{itemize}

It then follows that $\tilde{\textbf{S}}$ accepts $\tilde{\pazocal{R}}_1$ with $\TM_{\textbf{S}}(w)\leq\TM_{\tilde{\textbf{S}}}(\tilde{w})\leq M\TM_{\textbf{S}}(w)$ for all $w\in\pazocal{R}_1$.

\end{proof}

\begin{lemma} \label{Blow-up mass}

The presentation $\tilde{\pazocal{P}}_2$ satisfies the $(f_1(n)^2,n^2f_2(n))$-mass condition.

\end{lemma}

\begin{proof}

Fix $C_1>0$ and let $W$ be a word over $\tilde{X}$ which represents the identity in $\tilde{G}$.  Our goal is to bound the $f_1(C_1n)^2$-mass of a van Kampen diagram over $\tilde{\pazocal{P}}_2$ with boundary label $W$.

As such, we may assume without loss of generality that $W$ is non-trivial and cyclically reduced.  Moreover, by \Cref{Blow-up Greendlinger} we may assume that $W=\tilde{w}$ for some word $w$ over $X$ representing the identity in $G$.

By hypothesis, $\pazocal{P}_2$ satisfies the $(f_1(n)^2,n^2f_2(n))$-mass condition.  Hence, there exists a constant $C_2>0$ and a van Kampen diagram $\Delta$ over $\pazocal{P}_2$ with $\lab(\partial\Delta)\equiv w$ such that the $f_1(C_1Mn)^2$-mass of $\Delta$ is at most $C_2(C_2\|w\|)^2f_2(C_2\|w\|)$.

Consider the diagram $\tilde{\Delta}$ obtained from $\Delta$ by subdividing each edge $\textbf{e}$ into $M$ edges $\tilde{\textbf{e}}_1,\dots,\tilde{\textbf{e}}_M$ such that if $\lab(\textbf{e})\equiv x_i^\eps$, then $\lab(\tilde{\textbf{e}}_1)\dots\lab(\tilde{\textbf{e}}_M)\equiv\tilde{x}_i^\eps$.  It then follows that $\lab(\partial\tilde{\Delta})\equiv\tilde{w}=W$.

Moreover, note that for any cell $\pi$ of $\Delta$, the corresponding cell $\tilde{\pi}$ of $\tilde{\Delta}$ satisfies $\|\partial\tilde{\pi}\|=M\|\partial\pi\|$.  Hence, the $f_1(C_1n)^2$-mass of $\tilde{\Delta}$ is at most the $f_1(C_1Mn)^2$-mass of $\Delta$, and so is at most 
$$C_2(C_2\|w\|)^2f_2(C_2\|w\|)=C_2(C_2\|W\|/M)^2f_2(C_2\|W\|/M)\leq C_2(C_2\|W\|)^2f_2(C_2\|W\|)$$

\end{proof}

By Lemmas \ref{Blow-up Embedding} and \ref{Blow-up bi-Lipschitz}, it suffices to find a bi-Lipschitz embedding of $\tilde{G}$ into a finitely presented group $H$ with $\delta_H(n)\preceq n^2f_2(n)$.  Hence, it suffices to show that the presentations $\tilde{\pazocal{P}}_1$ and $\tilde{\pazocal{P}}_2$ for $\tilde{G}$ satisfy the hypotheses of \Cref{main-theorem}.

Hypotheses (1), (3), and (4) are shown to be satisfied by Lemmas \ref{Blow-up Greendlinger}, \ref{Blow-up S-machine}, and \ref{Blow-up mass}, respectively.  As such, it suffices to find a $(\tilde{\pazocal{R}}_1,f_1,C)$-Move machine satisfying the $(n^2f_2(n),\tilde{G})$-move condition.

But by hypothesis (1) of \Cref{main-corollary}, there exists a $(F(\tilde{X}),f_1,C)$-Move machine satisfying the computational $c$-move condition for some constant $c$.  The statement thus follows by \Cref{computational move is move} in the general setting described in \Cref{rmk-computational-move-generalized}.

\medskip

%%%%%%%%%%%%%%%%%%%%%%%%%%%%%%%%%%%%%%%%%%%%%%%%%%

\section{Proof of \Cref{theorem-higman}}

Let $Z=X\sqcup \bar{X}$, where $\bar{X}$ is a copy of $X$ identified with the set $X^{-1}$.  Define $\pazocal{L}_G$ to be the set of (positive) words in $Z$ that, when the elements of $\bar{X}$ are replaced by the corresponding elements of $X^{-1}$, represent the identity in $G$.  

\begin{lemma} \label{lem-presentation}

$\pazocal{P}=\gen{Z\mid\pazocal{L}_G}$ is a presentation for $G$.

\end{lemma}

\begin{proof}

Let $G'$ be the group with presentation $\pazocal{P}$.

For any $x\in X$, letting $\bar{x}$ be the corresponding element of $\bar{X}$, note that $x\bar{x}\in\pazocal{L}_G$.

Now, for every $w\in\pazocal{L}_G$, let $w'$ be the reduced word in $X\cup X^{-1}$ obtained by replacing every element of $\bar{X}$ with the corresponding element of $X^{-1}$.  Then, letting $\pazocal{L}_G'$ be the set obtained by replacing every $w\in\pazocal{L}_G$ with $w'$, it follows that the presentation $\pazocal{P}'=\gen{X\mid\pazocal{L}_G'}$ is obtained from $\pazocal{P}$ by Tietze transformations, and so is a presentation of $G'$.

By the definition of $\pazocal{L}_G$, though, it follows from the theorem of von Dyck that the identity map on $X$ extends to inverse homomorphisms between $G$ and $G'$.

\end{proof}

Note that, by hypothesis, the Turing machine $\pazocal{T}$ can be viewed to have input alphabet $Z$ and enumerates $\pazocal{L}_G$.  By \cite{CW} there then exists an $S$-machine $\textbf{S}$ with input alphabet $Z$ enumerating $\pazocal{L}_G$ and such that $\TM_\textbf{S}\preceq f^{1+\eps/2}$.  Since $f$ is superadditive, $f^{1+\eps/2}$ is also.

Define $f_1(n)=n^2+f(n)^{1+\eps/2}$ and $g(n)=\frac{f(n)^{1+\eps/2}}{n}$ for all $n\in\N$.  

Further, define $g_2(n)=\max_{i\leq n}g(i)$ and $f_2'(n)=g_2(n)^2$ for all $n\in\N$.

\begin{lemma} \label{lem-equivalence}

$f_2'(n)\geq1$ for all $n\in\N$ and $n^2f_2'(n)\sim f(n)^{2+\eps}$.

\end{lemma}

\begin{proof}

Since $f(n)\geq n$, $g(n)\geq1$ for all $n$.  As such, $g_2(n)\geq 1$ for all $n$, so that $f_2'(n)\geq 1$ for all $n$.

Fix $n\in\N$ and let $m$ be the maximal index such that $g_2(n)=g(m)$.  Since $f^{1+\eps/2}$ is superadditive, it follows that $g(2m)\geq g(m)$.  Hence, the maximality of $m$ implies $n<2m$.

As a result, $f(n)^{1+\eps/2}=ng(n)\leq ng_2(n)=ng(m)\leq 2f(m)^{1+\eps/2}\leq 2f(n)^{1+\eps/2}$, and so
$$f(n)^{2+\eps}\leq n^2f_2'(n)\leq 4f(n)^{2+\eps}$$

\end{proof}

Now, define $f_2(n)=n^2+f_2'(n)$ for all $n\in\N$.  Note then that $$f_1(n)^2\leq4\max(n^4,f(n)^{2+\eps})\leq 4n^2\max(n^2,f_2'(n))\leq4n^2f_2(n)$$
for all $n\in\N$.

\begin{lemma} \label{lem-mass}

The presentation $\pazocal{P}$ satisfies the $(f_1(n)^2,n^2f_2(n))$-mass condition.

\end{lemma}

\begin{proof}

Let $W$ be a non-trivial word over $Z$ which represents the identity in $G$ and fix $C_1>0$.

First, suppose $W$ is a positive word in $Z$.  By definition, $W\in\pazocal{L}_G$, and so there exists a van Kampen diagram $\Delta$ over $\pazocal{P}$ with boundary label $W$ consisting of a single cell.  Hence, the $f_1(C_1n)^2$-mass of $\Delta$ is $f_1(C_1\|W\|)^2\leq 4C_1^2\|W\|^2f_2(C_1\|W\|)$.

Otherwise, let $W\equiv z_1^{\eps_1}\dots z_k^{\eps_k}$ where $z_i\in Z$ and $\eps_i\in\{\pm1\}$.  

Suppose $\eps_i=-1$.  If $z_i\in X$, then let $\bar{z}_i$ be the corresponding letter of $\bar{X}$; otherwise, let $\bar{z}_i$ be the corresponding letter of $X$.  In either case, $z_i\bar{z}_i\in\pazocal{L}_G$, so that the letter $z_i^{-1}$ in $W$ can be replaced with $\bar{z}_i$ without changing the value of the word in $G$.

Doing this for all $i$, we then obtain a positive word $W'$ in $Z$ with $\|W'\|=\|W\|$ which represents the identity in $G$.  As above, we may then construct a van Kampen diagram $\Delta$ over $\pazocal{P}$ with boundary label $W'$ whose $f_1(C_1n)^2$-mass is at most $4C_1^2\|W\|^2f_2(C_1\|W\|)$.

But by the same argument, for all $i$ with $\eps_i=-1$ there exist van Kampen diagrams $\Delta_i$ with boundary label $z_i\bar{z}_i$ and $f_1(C_1n)^2$-mass at most $16C_1^2f_2(2C_1)$.

Pasting the $\Delta_i$ to $\Delta'$ then produces a van Kampen diagram $\Delta$ over $\pazocal{P}$ with boundary label $W$ and $f_1(C_1n)^2$-mass at most $$4C_1^2\|W\|^2f_2(C_1\|W\|)+16C_1^2\|W\|f_2(2C_1)\leq20C_1^2\|W\|^2f_2(2C_1\|W\|)$$
Thus, the statement is satisfied by setting $C_2=2C_1+5$.

\end{proof}

Now, take $\pazocal{P}_1=\pazocal{P}_2=\pazocal{P}$ and extend $f_1$ and $f_2$ in any way so that they are non-decreasing functions on the nonnegative reals.  Then, it follows by \Cref{move_1 comp move} and \Cref{lem-mass} that the hypotheses of \Cref{main-corollary} are satisfied for $\pazocal{P}_i$ and the Move machine $\textbf{Move}_{1,Y}$.  Thus, by \Cref{lem-equivalence}, there exists a bi-Lipschitz embedding of $G$ into a finitely presented group $H$ such that $$\delta_H(n)\preceq n^2f_2(n)=n^4+n^2f_2'(n)\sim n^4+f(n)^{2+\eps}$$

\bigskip

\bibliographystyle{plain}
\bibliography{biblio}

\end{document}